\documentclass[a4paper,10pt]{amsart}

\usepackage[T1]{fontenc}
\usepackage{lmodern}
\usepackage{amsmath}
\usepackage{amsfonts}
\usepackage{amssymb}
\usepackage{amsthm}
\usepackage{thmtools}
\usepackage{mathtools}
\usepackage{enumitem}
\usepackage{placeins}
\usepackage{multirow}
\usepackage{tikz}
\usetikzlibrary{cd,shapes.geometric}
\usepackage[pdftex,
            pdfauthor={},
            pdftitle={Vertex Superalgebras for Hypertoric Varieties and 3d Abelian Gauge Theories},
            pdfsubject={},
            pdfkeywords={},
            pdfproducer={},
            pdfcreator={}]{hyperref}
\usepackage{caption}
\usepackage{adjustbox}

\theoremstyle{plain}
\newtheorem{thm}{Theorem}[section]
\newtheorem{prop}[thm]{Proposition}
\newtheorem{cor}[thm]{Corollary}
\newtheorem{lem}[thm]{Lemma}
\newtheorem{conj}[thm]{Conjecture}

\newtheorem*{thm*}{Theorem}
\newtheorem*{conj*}{Conjecture}
\newtheorem*{claim*}{Claim}
\newtheorem*{prop*}{Proposition}

\theoremstyle{definition}
\newtheorem{defi}[thm]{Definition}

\newtheorem{conv}[thm]{Convention}
\newtheorem*{nota*}{Notation}
\newtheorem{rem}[thm]{Remark}
\newtheorem{ex}[thm]{Example}

\newtheoremstyle{introTheorems}
  {}
  {}
  {\itshape}
  {}
  {\bfseries}
  {}
  { }
  {\thmname{#1}
  \textnormal{\bf \thmnote{#3}.$\hspace{-.1cm}$}
  }
\theoremstyle{introTheorems}
\newtheorem{introTheorem}{Theorem}

\newtheorem{introConj}{Conjecture}

\newcommand{\Q}{\mathbb{Q}}
\newcommand{\Z}{\mathbb{Z}}
\newcommand{\Ns}{\mathbb{Z}_{>0}}
\newcommand{\N}{\mathbb{Z}_{\geq0}}
\newcommand{\C}{\mathbb{C}}
\newcommand{\R}{\mathbb{R}}

\newcommand{\wt}{\operatorname{wt}}
\newcommand{\tr}{\operatorname{tr}}
\renewcommand{\i}{\mathrm{i}}

\newcommand{\Hom}{\operatorname{Hom}}

\newcommand{\rk}{\operatorname{rk}}
\newcommand{\im}{\operatorname{im}}

\newcommand{\voa}{vertex operator algebra}

\newcommand{\svoa}{vertex operator superalgebra}

\newcommand{\SVOA}{Vertex Operator Superalgebra}
\newcommand{\vac}{\textbf{1}}
\newcommand{\ch}{\operatorname{ch}}

\newcommand{\id}{\operatorname{id}}
\newcommand{\eps}{\varepsilon}

\newcommand{\SLZ}{\operatorname{SL}_2(\mathbb{Z})}
\newcommand{\GL}{\operatorname{GL}}
\newcommand{\gl}{\mathfrak{gl}}

\newcommand{\g}{\mathfrak{g}}
\newcommand{\h}{\mathfrak{h}}

\renewcommand{\sl}{\mathfrak{sl}}
\newcommand{\psl}{\mathfrak{psl}}

\newcommand{\T}{\mathcal{T}}

\newcommand{\strat}{strongly rational}

\newcommand{\Com}{\operatorname{Com}}
\newcommand{\no}{\,{\raise0.25em\hbox{$\mathop{\hphantom{\cdot}}\limits^{_{\circ}}_{^{\circ}}$}}\,}

\newcommand{\diag}{\operatorname{diag}}
\newcommand{\Spec}{\operatorname{Spec}}
\newcommand{\Specm}{\operatorname{Specm}}
\newcommand{\Proj}{\operatorname{Proj}}
\newcommand{\qquot}{/\!\!/}
\newcommand{\hatotimes}{\operatorname{\hat{\otimes}}}
\newcommand{\HL}{V}

\newcommand{\hooklongrightarrow}{\lhook\joinrel\longrightarrow}

\begin{document}

\title[Vertex Superalgebras for Hypertoric Varieties]{Vertex Superalgebras for Hypertoric Varieties and 3d Abelian Gauge Theories}
\author[Tomoyuki Arakawa, Andrea~E.~V. Ferrari and Sven Möller]{Tomoyuki Arakawa,\textsuperscript{\lowercase{a},\lowercase{b}} Andrea~E.~V. Ferrari\textsuperscript{\lowercase{c},\lowercase{d},\lowercase{e}} and Sven Möller\textsuperscript{\lowercase{f},\lowercase{g}}}
\thanks{\textsuperscript{a}{Research Institute for Mathematical Sciences, Kyoto University, Kyoto, Japan}}
\thanks{\textsuperscript{b}{Okinawa Institute of Science and Technology, Onna, Okinawa, Japan}}
\thanks{\textsuperscript{c}{Deutsches Elektronen-Synchrotron DESY, Hamburg, Germany}}
\thanks{\textsuperscript{d}{School of Mathematics, University of Edinburgh, Edinburgh, United Kingdom}}
\thanks{\textsuperscript{e}{DAMTP, University of Cambridge, Cambridge, United Kingdom}}
\thanks{\textsuperscript{f}{Fachbereich Mathematik, Universität Hamburg, Hamburg, Germany}}
\thanks{\textsuperscript{g}{Department of Mathematics and Statistics, Maynooth University, Maynooth, Ireland}}
\thanks{Email: \scriptsize \href{tomoyuki.arakawa@oist.jp}{\nolinkurl{tomoyuki.arakawa@oist.jp}}, \href{mailto:andrea.e.v.ferrari@gmail.com}{\nolinkurl{andrea.e.v.ferrari@gmail.com}}, \href{mailto:math@moeller-sven.de}{\nolinkurl{math@moeller-sven.de}}}

\begin{abstract}
Hypertoric (or toric hyperkähler) varieties are a class of symplectic singularities and their resolutions, obtained as Hamiltonian reductions of a symplectic vector space acted on by a torus. In physics, they appear as Higgs (and Coulomb) branches of 3d $\mathcal{N}=4$ supersymmetric quantum field theories with abelian gauge group.

In this work, we construct an $\hbar$-adic (in the sense of microlocalisation) sheaf of \svoa{}s over a given smooth hypertoric variety. Its global sections give the $A$-twisted boundary \svoa{} of the corresponding 3d gauge theory. We use this to prove that the associated affine variety of this hypertoric \svoa{} recovers the singular hypertoric variety. This proves the 3d Higgs branch conjecture for a large class of boundary \svoa{}s. In particular, these \svoa{}s are quasi-lisse.

This is in contrast to the (purely even) hypertoric \voa{}s (and their $\hbar$-adic localisations) constructed previously by Kuwabara as global sections of sheaves on families of universal Poisson deformations of the hypertoric varieties. These are generally not quasi-lisse. We show that the \svoa{}s defined in this paper are (fermionic) simple-current extensions of those defined by Kuwabara, and investigate the consequences for symplectic duality and characters. We observe that the latter are upgraded from partial (or false) theta functions to quasimodular forms.
\end{abstract}

\maketitle

\setcounter{tocdepth}{1}
\tableofcontents
\setcounter{tocdepth}{2}


\section{Introduction}

Toric hyperkähler varieties, or hypertoric varieties for short, are hyperkähler (or quaternionic) analogues of toric varieties. They were introduced (as manifolds) by Bielawski and Dancer \cite{BD00} by applying the hyperkähler quotient construction of \cite{HKLR87} to a torus $G=\smash{(\C^\times)^M}$ acting on a quaternionic vector space $T^*V=\C^{2N}$. Here, we treat them as algebraic varieties (equipped with the Zariski topology), as described in \cite{HS02,PW07,Pro08},
\begin{equation*}
Y_\delta(\Delta)=\mu^{-1}(0)\qquot\!_\delta G=\Proj\Bigl(\bigoplus_{m\in\N}\C[\mu^{-1}(0)]^{G,m\delta}\Bigr)
\end{equation*}
obtained as algebraic symplectic or GIT quotients associated with the algebraic action of $G$ on the symplectic vector space $T^*V$, encoded in the integer weight matrix $\Delta\colon\Z^N\to\Z^M$. Here, $\delta\colon G\to\C^\times$ is a choice of effective stability parameter. The hypertoric variety $Y_\delta(\Delta)$ is projective over the affine quotient
\begin{equation*}
Y_0(\Delta)\coloneqq\mu^{-1}(0)\qquot G=\Spec\C[\mu^{-1}(0)]^G.
\end{equation*}
The latter is a symplectic singularity in the sense of \cite{Bea00}. If $\Delta$ is unimodular and $\delta$ generic, $Y_\delta(\Delta)$ is smooth and the natural projective morphism $\pi\colon Y_\delta(\Delta)\to Y_0(\Delta)$ is a conical symplectic resolution. We assume this from now on. In particular, $Y_0(\Delta)$ is normal and $\pi$ is birational, which implies that $\smash{\mathcal{O}_{Y_\delta(\Delta)}(Y_\delta(\Delta))=\C[Y_0(\Delta)]}$, i.e.\ $Y_\delta(\Delta)$ and $Y_0(\Delta)$ have the same global functions.

\medskip

Analogously to the celebrated Beilinson-Bernstein quantisation-localisation result \cite{BB91} for the Springer resolution $T^*(G/B)\to\mathcal{N}$ and following the idea of Kashiwara and Rouquier \cite{KR08} to quantise the Hamiltonian reduction process (which they applied to the symplectic resolution $\operatorname{Hilb}^n(T^*\C)\to(T^*\C)^n/S_n$ given by the punctual Hilbert scheme), one can consider also quantisations of hypertoric varieties $Y_\delta(\Delta)\to Y_0(\Delta)$. These were considered in \cite{MV98}, and a certain (micro)localisation was constructed in \cite{BK12,BLPW12} (see also \cite{Los12,Los22}). That is, there is a sheaf $\mathcal{A}_\hbar$ of noncommutative, $\hbar$-adic algebras on $Y_\delta(\Delta)$ whose algebra of ($\C^\times$-invariant) global sections $\smash{[\mathcal{A}_\hbar(Y_\delta(\Delta))]^{\C^\times}}$ equals the quantisation of $\C[Y_0(\Delta)]$. These examples fit into the framework of quantisations of conical symplectic resolutions in the sense of \cite{BLPW16a,BLPW16b}.

As suggested in \cite{KR08}, in the symplectic setting the localisation needs to be upgraded to a microlocalisation in the sense of microdifferential operators with homogenising parameter $\hbar$. To give a minimal example, in order to localise the quantisation of $T^*\C$, the Weyl algebra, on all of $T^*\C$ and not just on the pushforward to $\C$ (i.e.\ to quantise $\mathcal{O}_{T^*\C}$), we need to replace the differential operators $\mathcal{D}_\C$ by their ($\hbar$-adic) deformation quantisation version. For instance when considering the global sections, one may remove the $\hbar$-adic structure (which is too big in some sense) by considering a certain $\C^\times$-invariant structure.

Moreover, one may re-interpret the quantum Hamiltonian reduction to obtain quantised hypertoric varieties in terms of BRST cohomology \cite{Kuw15}.

\medskip

Vertex (super)algebras $V$ can be viewed as chiral quantisations of affine Poisson schemes \cite{Ara12}. In good cases, they quantise the arc space (or $\infty$-jet scheme) of the associated Poisson scheme $\tilde{X}_V$. In this work, we mostly focus on the underlying Poisson variety $\smash{X_V=(\tilde{X}_V)_\mathrm{red}}$.

Now, motivated by the 3d/2d-correspondence \cite{Gai19,CG19,CCG19} and the 3d Higgs branch conjecture \cite{BF25} (which we comment upon below), in this work we lift the quantisation-localisation of hypertoric varieties to vertex algebras. That is, we construct a sheaf $\smash{\mathcal{V}^\hbar_{\delta,\Delta}}$ of $\hbar$-adic vertex algebras in the sense of \cite{AKM15,Li04} on $Y_\delta(\Delta)$ by quantum Hamiltonian reduction (BRST cohomology).

There is a naive choice of BRST cohomology complex to quantise the Hamiltonian reduction construction of hypertoric varieties. However, the naive BRST differential does not square to zero. To correct this anomaly, we introduce a fermionic extension of the BRST complex, now with a well-defined differential. That is, following \cite{AKM23} (see also \cite{CSYZ25}), we construct $\smash{\mathcal{V}^\hbar_{\delta,\Delta}}$ as a sheaf of $\hbar$-adic vertex \emph{super}algebras on $\smash{Y_\delta(\Delta)}$.
\begin{thm*}[\autoref{sec:sheafBRST}]
There is a sheaf of $\hbar$-adic vertex superalgebras $\smash{\mathcal{V}^\hbar_{\delta,\Delta}}$ on $\smash{Y_\delta(\Delta)}$. It quantises $\smash{\bar{\mathcal{O}}_{J_\infty Y_\delta(\Delta)}}$, the arc space of a certain super-analogue $\smash{\bar{\mathcal{O}}_{Y_\delta(\Delta)}}$ of the hypertoric variety $\smash{\mathcal{O}_{Y_\delta(\Delta)}}$.
\end{thm*}
Our construction is in contrast to \cite{Kuw21}, where the same anomaly cancellation problem is resolved by constructing a sheaf of (purely even) $\hbar$-adic vertex algebras on a universal family of Poisson deformations $\smash{\tilde{Y}_\delta(\Delta)\to\g^*}$ of the hypertoric variety. Concretely, whilst in \cite{Kuw21} the naive Weyl vertex algebra $\smash{\mathcal{D}^\mathrm{ch}(T^*V)}$ on $T^*V$ (and its $\hbar$-adic localisation) entering the BRST cohomology is tensored by a Heisenberg vertex algebra to cancel the anomaly, we tensor it by the Clifford vertex superalgebra $\mathcal{C}\ell(\Pi T^*V)$.

\medskip

The $\C^\times$-invariant global sections of $\smash{\mathcal{V}^\hbar_{\delta,\Delta}}$,
\begin{equation*}
V(\Delta)=\bigl[\mathcal{V}_{\delta,\Delta}^{\hbar}(Y_\delta(\Delta))\bigr]^{\C^\times}
\end{equation*}
carry the structure of a vertex superalgebra in the usual sense. It inherits from the free-field \svoa{} $M=\mathcal{D}^\mathrm{ch}(T^*V)\otimes\mathcal{C}\ell(\Pi T^*V)$ in the input of the BRST reduction a natural conformal structure of central charge $c=-2N$; see \autoref{cor:conformalvector}. We call $V(\Delta)$ \emph{boundary hypertoric \svoa{}}. They already appear in \cite{BF25,BCDN23,Niu23} (see also \cite{AM21,Yos23,FS24,Sas25} for some special cases). We prove:
\begin{introTheorem}[\ref{thm:var}]
The boundary hypertoric \svoa{} $V(\Delta)$ is simple, self-contragredient and of CFT-type. Its associated Poisson variety
\begin{equation*}
X_{V(\Delta)}\cong Y_0(\Delta)
\end{equation*}
recovers the affine hypertoric variety $Y_0(\Delta)$. In particular, $V(\Delta)$ is quasi-lisse.
\end{introTheorem}
The quasi-lisseté (in the sense of \cite{AK18}) follows because symplectic singularities have finitely many symplectic leaves \cite{Kal06}.

By contrast, the $\C^\times$-invariant global sections of the sheaf constructed in \cite{Kuw21}, which we call \emph{minimal hypertoric \voa{}s} and denote by $V_\mathrm{min}(\Delta)$ in this text, are in general not quasi-lisse. This is because the universal family $\tilde{Y}_\delta(\Delta)$ does not have finitely many symplectic leaves to begin with. We shall explain in detail how $V(\Delta)$ can be viewed as a (fermionic) simple-current extension of $V_\mathrm{min}(\Delta)$ tensored with a lattice \svoa{}; see \autoref{sec:summary}. From this perspective, the fact that $V(\Delta)$ has the zero fibre of the universal family of Poisson deformations $\smash{\tilde{Y}_\delta(\Delta)\to\g^*}$ as associated variety follows from a Sugawara-like construction for the conformal vector, based on which we can prove that its image in $\smash{(R_{V(\Delta)})_{\mathrm{red}}}$ is nilpotent; see \autoref{lem:nilpotent-stress}.

The boundary hypertoric \svoa{} can also be constructed directly (see \autoref{defi:svoa}) by quantum Hamiltonian reduction of the (nonlocalised) free-field \svoa{} $M=\mathcal{D}^\mathrm{ch}(T^*V)\otimes\mathcal{C}\ell(\Pi T^*V)$.
\begin{introTheorem}[\ref{thm:glob-sec}]
There is a natural isomorphism of \svoa{}s
\begin{equation*}
V(\Delta)\cong H_\mathrm{BRST}^{\infty/2+\bullet}(\g,M).
\end{equation*}
\end{introTheorem}
This result uses a faithfulness theorem from \cite{ADS26} (see \autoref{prop:faithful}) and relies on the fact (see \autoref{prop:specialelements}) that we can find certain invariant polynomials $f,g\in\C[T^*V]^G$ whose principal open subsets $\mathfrak{D}_f$ and $\mathfrak{D}_g$ are contained in the semistable locus $\mathfrak{X}=(T^*V)_\delta^\mathrm{ss}$ used in the construction of $Y_\delta(\Delta)$ and that behave well with respect to certain auxiliary conformal structures.


\subsubsection*{Physical Background}

The study of vertex (super)algebras associated with symplectic singularities has been enriched by its connection to supersymmetric quantum field theory, and the fruitfulness of this interaction continues to grow.

According to the 4d/2d-correspondence~\cite{BLLPRR15}, it is possible to canonically assign a vertex algebra to a 4d $\mathcal{N}=2$ superconformal quantum field theory. Moreover, the Higgs branch conjecture~\cite{BR18} states that the associated Poisson variety of the vertex algebra arising from such a theory is isomorphic to the Higgs branch of the theory, which is expected to be a symplectic singularity. One can then speculate that at least in certain cases, the vertex algebras arising from 4d $\mathcal{N}=2$ theories are global sections of sheaves of $\hbar$-adic vertex algebras, in the sense of~\cite{AKM15}, on the respective (resolved) Higgs branches. The case of pure type-A gauge theories motivated the construction of sheaves of $\hbar$-adic vertex algebras on the Hilbert scheme of points given by two of the authors in \cite{AKM23}.

\medskip

More recently, an analogue of the Higgs branch conjecture has been investigated in a closely related physical context, boundary vertex algebras of topologically twisted 3d $\mathcal{N}=4$ supersymmetric quantum field theories. Indeed, in what is called the 3d/2d-correspondence, from any 3d $\mathcal{N}=4$ quantum field theory $\mathcal{T}$ one can produce two cohomological topological field theories, the $A$- and $B$-twist, which when put on a half-space support vertex algebras
\begin{equation*}
\smash{V^A(\mathcal{T})}\quad\text{and}\quad\smash{V^B(\mathcal{T})},
\end{equation*}
respectively, at their boundaries \cite{Gai19,CG19,CCG19}. This can be thought of as a nonunitary and generically logarithmic generalisation of the famous Chern-Simons/Wess-Zumino-Witten correspondence \cite{Wit89}. The parent 3d $\mathcal{N}=4$ theory $\mathcal{T}$ has two distinguished branches of vacua, the Higgs $\mathcal{M}_H(\mathcal{T})$ and Coulomb branch $\mathcal{M}_H(\mathcal{T})$, which are expected to be closely related to the boundary vertex algebras.

In the 3d setting, in good cases, both the Higgs and Coulomb branch are expected to be symplectic singularities (see, e.g., \cite{Kam22,WY23} for recent surveys). Mathematically, the Higgs branch $\mathcal{M}_H(\mathcal{T})$ is straightforwardly defined as a holomorphic symplectic quotient by the gauge group $G$ \cite{HKLR87}. The Coulomb branch $\mathcal{M}_C(\mathcal{T})$, on the other hand, is much more difficult to construct. The celebrated work of Braverman, Finkelberg and Nakajima \cite{BFN18} gives a definition based on the geometry of affine Grassmannians.

Now, the 3d Higgs branch conjecture \cite{BF25} formulated by one of the authors states that, under favourable circumstances (and with the right choice of boundary conditions), the associated variety of the boundary vertex algebra $V^A(\mathcal{T})$,
\begin{equation}\tag{\ref{eq:Higgs}}
X_{V^A(\mathcal{T})}\cong\mathcal{M}_H(\mathcal{T})
\end{equation}
recovers the Higgs branch $\mathcal{M}_H(\mathcal{T})$ of $\mathcal{T}$. This also leads to the expectation that the boundary vertex algebras $V^A(\mathcal{T})$ are quasi-lisse.

The hypertoric vertex superalgebras $V(\Delta)$ constructed in this work are the boundary vertex superalgebras of 3d $\mathcal{N}=4$ gauge theories with abelian gauge group $G=(\C^\times)^M$ \cite{BF25,BCDN23,Niu23}. Hence, \autoref{thm:var} states:
\begin{thm*}
The Higgs branch conjecture holds for all 3d $\mathcal{N}=4$ gauge theories with abelian gauge group and unimodular charge matrix $\Delta$.
\end{thm*}
This conjecture was proved in \cite{FS24} for the special case of minimal nilpotent orbit closures for $\sl_N$, which were identified with the associated varieties of $L_1(\psl_{N|N})$; see below.

Using symplectic duality (which reduces to Gale duality) and 3d mirror symmetry, this also shows an analogous result for the Coulomb branch $\smash{\mathcal{M}_C(\mathcal{T})}$ and the $B$-twisted boundary vertex algebra $\smash{V^B(\mathcal{T})}$.


\subsubsection*{Free-Field Realisations and Screening Kernels}

Given the construction of a vertex algebra as global sections of a sheaf of $\hbar$-adic vertex algebras on some (smooth) variety, one can obtain free-field realisations by restricting to nice open subsets. Examples of such free-field embeddings are given in \cite{Kuw21,AKM23} based on local trivialisations of the defining Hamiltonian actions.

This picture is particularly elegant in the case of the symplectic resolutions $\pi\colon Y_\delta(\Delta)\to Y_0(\Delta)$ given by hypertoric varieties. In fact, in \cite{BF25,BCDN23} explicit and rigorous free-field realisations for the hypertoric \svoa{}s $V(\Delta)$ were obtained in terms of a ``bosonisation'' of the Weyl vertex algebra $\mathcal{D}^\mathrm{ch}(T^*V)$ in the input of the BRST reduction,
\begin{equation*}
\mathcal{D}^\mathrm{ch}(T^*V)\hookrightarrow\mathcal{D}^\mathrm{ch}(T^*(\C^\times)^N).
\end{equation*}
This reproduces the well-known embedding of the Weyl vertex algebra into a half-lattice vertex algebra (or Heisenberg-lattice vertex algebra). It was shown that the BRST cohomology can be easily performed at the level of the bosonised fields. As a consequence, the vertex superalgebra $V(\Delta)$ can be realised by applying screening kernels to $\mathcal{D}^{\mathrm{ch}}(T^*(\C^\times)^{N-M})$ tensored with free fermions; see \autoref{prop:vanish_main2}.

Now, in principle our sheaf construction should enable us to prove the expectation of \cite{BF25} that these free fields can be identified with the local sections of the sheaves of $\hbar$-adic vertex superalgebras on openly embedded subsets $T^*(\C^{\times})^{N-M}\hookrightarrow Y_{\delta}(\Delta)$; see our remarks in \autoref{sec:freefieldrealisation}.


\subsubsection*{Symplectic Duality}

One of the most intriguing aspects of the study of symplectic singularities is that these come in pairs with many matching properties \cite{BLPW16b}. Many examples of symplectic dual pairs are provided by the Higgs and Coulomb branch $\mathcal{M}_H(\mathcal{T})$ and $\mathcal{M}_C(\mathcal{T})$ of the same 3d $\mathcal{N}=4$ theory $\mathcal{T}$. This phenomenon is closely related to 3d mirror symmetry in physics \cite{IS96}, which predicts the existence of a dual theory $\mathcal{T}^\vee$ that has Higgs and Coulomb branches swapped compared to $\mathcal{T}$. In this work, we explore how (some aspects of) symplectic duality can be upgraded to vertex algebras in the special case of hypertoric varieties.

For hypertoric varieties, symplectic duality is perhaps best understood, where it reduces to Gale duality \cite{BLPW10,BLPW12}; see \autoref{sec:sym-dual}. Indeed, to any hypertoric variety $Y_0(\Delta)$ one can associate another hypertoric variety, its symplectic dual $Y_0(\Delta\!^!)$. As main examples, we consider in this text the Gale dual symplectic singularities $\smash{Y_0(\Delta)=\overline{\mathbb{O}_\mathrm{min}(\sl_N)}=\{Z\in\sl_N\mid\rk(Z)\leq1\}}$, the minimal nilpotent orbit closure of $\sl_N$, and $\smash{Y_0(\Delta\!^!)=X_{A_{N-1}}=\C^2/{(\Z/N\Z)}}$, the Kleinian (or du Val) singularity of type $A_{N-1}$; see below.

\medskip

Physically, whilst $\smash{Y_0(\Delta)=\mathcal{M}_H(\mathcal{T})}$ is the Higgs branch of the 3d $\mathcal{N}=4$ supersymmetric field theory $\mathcal{T}$ with (abelian) gauge group action defined by the datum~$\Delta$, $Y_0(\Delta\!^!)=\mathcal{M}_H(\mathcal{T}^\vee)$ is the Higgs branch of the dual theory $\mathcal{T}^\vee$ defined by $\Delta\!^!$. By 3d mirror symmetry, this is then isomorphic to the Coulomb branch $\mathcal{M}_C(\mathcal{T})$ of $\mathcal{T}$,
\begin{equation*}
\mathcal{M}_H(\mathcal{T})=Y_0(\Delta)\quad\leftrightarrow\quad\mathcal{M}_C(\mathcal{T})=\mathcal{M}_H(\mathcal{T}^\vee)=Y_0(\Delta\!^!).
\end{equation*}
In general, the Coulomb branch $\mathcal{M}_C(\mathcal{T})$ of a 3d $\mathcal{N}=4$ gauge theory $\mathcal{T}$ is difficult to construct mathematically \cite{BFN18}, and the same is expected to be true for the $B$-twisted boundary vertex algebras $\smash{V^B(\mathcal{T})}$, which (in analogy to the better understood Higgs branch picture) should be understood as some kind of chiral quantisations of the Coulomb branch. Consequently, mathematical constructions of $\smash{V^B(\mathcal{T})}$ are very rare, unless they can be identified with the Higgs branch of a mirror dual theory.

For the special case of hypertoric varieties, a construction of $V^B(\mathcal{T})$ was proposed in \cite{BCDN23}, where an isomorphism $V^A(\T)\cong V^B(\T^\vee)$ was also proved. This was possible because in this case we are in the fortunate position that
\begin{equation*}
V^A(\mathcal{T})=V(\Delta)\quad\leftrightarrow\quad V^B(\mathcal{T})=V^A(\mathcal{T}^\vee)= V(\Delta\!^!).
\end{equation*}
In particular, we can construct $\smash{V^B(\T)}$ (and the corresponding sheaf) in the same way as $\smash{V^A(\T)}$, namely with the localisation-quantisation approach for Hamiltonian torus reductions pursued in this work. This is because the symplectic dual variety of a hypertoric variety is again (by 3d mirror symmetry) a hypertoric variety.

\medskip

Now, it is expected that many features of symplectic duality, here between $Y_0(\Delta)$ and $Y_0(\Delta\!^!)$, lift to the corresponding boundary vertex algebras, here $V(\Delta)$ and $V(\Delta\!^!)$. In this work, we explore some aspects of this.

One basic feature of symplectic duality is that the Lie algebra of torus Hamiltonian isometries of $Y_0(\Delta)$ is isomorphic to the Lie algebra of stability parameters of $Y_0(\Delta\!^!)$. It was claimed in \cite{BF25,BCDN23} that these are reflected by $\C^{N-M}$ inner and noninner automorphisms of $V(\Delta)$ and $V(\Delta\!^!)$, respectively, arising from the inner automorphisms of the tensor factor $\mathcal{C}\ell (\Pi T^*V)$ in the input of the BRST reduction. It was further pointed out in~\cite{BF25} (based on observations in \cite{CCG19}) that whenever a subalgebra $\C^\ell \subset \C^{N-M}$ of the infinitesimal torus Hamiltonian isometries of $Y(\Delta\!^!)$ extends to an $\sl_\ell$ infinitesimal Hamiltonian isometry, it is in bijection with $\sl_\ell$ noninner automorphisms of $V(\Delta\!^!)$, see \autoref{prop:outer}. In this paper we complete this picture and show that these are themselves in bijection with $L_1(\sl_\ell)$ affine subalgebras of $V(\Delta)$, see \autoref{prop:inner}. This can be thought of as a \svoa{} enhancement of the aforementioned symplectic duality statement.


\subsubsection*{Examples}

The results of this text are exemplified by two infinite families, which are mutual symplectic duals; see \autoref{sec:examples}. We consider the symplectic resolutions for the minimal nilpotent orbit closures
\begin{equation*}
T^*\mathbb{P}^{N-1}\to\overline{\mathbb{O}_\mathrm{min}(\sl_N)}\subset\sl_N^*
\end{equation*}
and the Kleinian singularities
\begin{equation*}
\widetilde{X}_{A_{N-1}}\to X_{A_{N-1}}=\C^2/{(\Z/N\Z)}.
\end{equation*}
The former is a symplectic reduction of $T^*\C^N$ by $G=\C^\times$ with $\Delta=(1,\dots,1)$ and the latter by $G^!=(\C^\times)^{N-1}$ with the Gale dual matrix $\Delta\!^!$.

The corresponding boundary hypertoric vertex superalgebras are, with the exception of certain special cases for small $N$,
\begin{equation*}
V(\Delta)=L_1(\psl_{N|N})\quad\text{and}\quad V(\Delta\!^!)=\tilde{\mathcal{W}}_{-N-1}(\sl_N,f_\mathrm{sub});
\end{equation*}
see \autoref{prop:ex1} and \autoref{prop:ex2}. The former is the simple quotient of the affine \svoa{} for $\psl_{N|N}$ at level $k=\pm1$, as was already shown in \cite{FS24}. The latter is a certain fermionic extension of the subregular $W$-algebra $\smash{\mathcal{W}_{-N-1}(\sl_N,f_\mathrm{sub})}$ for $\sl_N$ at level $k=-N-1$ \cite{FS04}.

Recall that $V(\Delta)$ can be understood as a simple-current extension of $V_\mathrm{min}(\Delta)$ from \cite{Kuw21} tensored with some lattice \svoa{} $V_{J^\bot}$. Concretely, in the above examples, there are the conformal embeddings
\begin{align*}
L_{-1}(\sl_N)\otimes V_{A_{N-1}}&\hookrightarrow L_1(\psl_{N|N}),\\
\intertext{which was already described in \cite{AM21,FS24},
and}
\mathcal{W}_{-N-1}(\sl_N,f_\mathrm{sub})\otimes V_{\Z(N)}&\hookrightarrow\tilde{\mathcal{W}}_{-N-1}(\sl_N,f_\mathrm{sub}),
\end{align*}
where $\Z(N)$ is the odd standard lattice $\Z$ with quadratic form scaled by $N$. In the last example, we note that $\mathcal{W}_{-N-1}(\sl_N,f_\mathrm{sub})$ is not exactly $V_\mathrm{min}(\Delta)$, but rather the fixed points under the (finite) Namikawa-Weyl group $\mathbb{W}$.

On the level of the associated varieties, these fermionic simple-current extensions correspond to intersecting with the nilpotent cone $\mathcal{N}\subset\sl_N^*$, i.e.
\begin{align*}
\overline{\mathbb{S}_\mathrm{min}(\sl_N)}&\hookleftarrow \overline{\mathbb{O}_\mathrm{min}(\sl_N)}\cap\mathcal{N}=\overline{\mathbb{S}_\mathrm{min}(\sl_N)}
\intertext{and}
\mathcal{S}_{f_\mathrm{sub}}&\hookleftarrow\mathcal{S}_{f_\mathrm{sub}}\cap\mathcal{N}\cong X_{A_{N-1}},
\end{align*}
respectively. Here, $\smash{\overline{\mathbb{S}_\mathrm{min}(\sl_N)}}$ is the closure of the unique minimal Dixmier sheet containing $\mathbb{O}_\mathrm{min}(\sl_N)$ and $\mathcal{S}_{f_\mathrm{sub}}\subset\sl_N^*$ is a Slodowy slice. Neither $\smash{\overline{\mathbb{S}_\mathrm{min}(\sl_N)}}$ nor $\mathcal{S}_{f_\mathrm{sub}}$ are symplectic varieties.


\subsubsection*{Characters}

By using a vanishing result for nonzero cohomologies \cite{BF25,Vor94} (see \autoref{prop:vanish}) and the Euler-Poincaré principle, \autoref{thm:glob-sec}, we can compute the characters of the boundary hypertoric vertex superalgebras $V(\Delta)$ and of their minimal versions $V_\mathrm{min}(\Delta)$ using the Euler-Poincaré principle. The former coincide with the half-indices of the 3d $\mathcal{N}=4$ abelian gauge theories to which they correspond via the 3d/2d-correspondence.

\medskip

As the minimal hypertoric vertex algebras $V_{\mathrm{min}}(\Delta)$ from \cite{Kuw21} are not quasi-lisse, not much can be concluded about the modular properties of their characters on general grounds. We show that they can be written as infinite sums of partial (or false) theta functions; see \autoref{prop:min-char}.

\medskip

However, when the $V_{\mathrm{min}}(\Delta)$ are extended to the quasi-lisse \svoa{}s $V(\Delta)$, one can conclude that their (super)characters must satisfy a modular linear differential equation \cite{AK18,Li23}. In fact, we expect them to be (quasi)modular forms (of mixed, nonnegative weights).

It is easy to show this for the supercharacters, which are of the simple form
\begin{equation*}
\operatorname{sch}_{V(\Delta)}(q)=\eta(q)^{2M},
\end{equation*}
where $\eta(q)$ is the Dedekind $\eta$-function; see \autoref{prop:boundary-schar}.

The characters take a more complicated form; see \autoref{prop:boundary-char} and \autoref{prop:boundarychar}. We only explore their modular properties in the two aforementioned series of examples. Concretely, in the case of the Kleinian singularity $Y_0(\Delta)=X_{A_{N-1}}$, we show in \autoref{prop:kleinian_char} that
\begin{equation*}
\ch_{V(\Delta)}(q)=\frac{\eta(q^2)^{2N}}{\eta(q)^{2N+2}}\cdot\Bigl(1+2^{N+1}\sum_{n=1}^\infty\frac{q^{Nn/2}}{(1+q^n)^N}\Bigr)
\end{equation*}
is quasimodular when $N$ is even, by expanding the Lambert-like series in terms of Eisenstein series. A similar analysis can be performed for the case of the minimal nilpotent orbit closures (see \autoref{ex:minimal_nilp_char}), where the quasimodularity of the character of $V(\Delta)=L_1(\psl_{N|N})$ is established up to $N=5$ (see also \cite{AM21}).

A more comprehensive analysis of the modularity of the hypertoric characters is left for future work.
\begin{introConj}[\ref{conj:quasimodular}]
The characters of the boundary hypertoric \svoa{}s $V(\Delta)$ are quasimodular, possibly of mixed nonnegative weights for some congruence subgroup of $\SLZ$ of level~$2$ with character.
\end{introConj}


\subsection*{Acknowledgements}

We are happy to thank Christopher Beem, Tudor Dimofte, Ian Grojnowski, Reimundo Heluani, Justin Hilburn, Thibault Juillard, Shashank Kanade, Andrew Linshaw, Toshiro Kuwabara, Matthias Storzer and Leonardo Rastelli for valuable discussions.

Tomoyuki Arakawa was partially supported by JSPS KAKENHI Grant Numbers 21H04993, 25K21659 and 26H019. Andrea E. V. Ferrari acknowledges support from the EPSRC grant EP/W020939/1 \emph{3d N=4 TQFTs}. Sven Möller acknowledges support from the DFG through the Emmy Noether Programme and the CRC 1624 \emph{Higher Structures, Moduli Spaces and Integrability}, project numbers 460925688 and 506632645.

This manuscript was submitted to the arXiv in coordination with the authors of \cite{CSYZ26}. We thank them for their cooperation.


\subsection*{Notation and Conventions}

All vector spaces, algebras, Lie algebras, vertex algebras, etc.\ are over the base field $\C$ unless otherwise noted. Throughout, a variety will be an integral, separated scheme of finite type over $\C$.

Let $G=(\C^\times)^M$ be a torus and $V$ a $G$-module. As usual, we denote the $G$-invariant elements of $V$ by $V^G=\{v\in V\mid g\cdot v=v\text{ for all }g\in G\}$. For a character $\delta\in\Hom(G,\C^\times)$, we let $V^{G,\delta}=\{v\in V\mid g\cdot v=\delta(g)v\text{ for all }g\in G\}$ denote the $(G,\delta)$-semi-invariant elements. We also consider fractional characters $\delta\in\Hom(G,\C^\times)\otimes_\Z\Q$. In that case, $V^{G,\delta}$ is zero unless $\delta\in\Hom(G,\C^\times)$.

A matrix $\Delta$ is called unimodular if all the maximal minors take values in $\{-1,0,1\}$.

Let $\hatotimes$ denote the completion of the tensor product $\smash{\otimes_{\C[[\hbar]]}}$ with respect to the $\hbar$-adic topology.

Following standard notation, let $(a;q)_\infty=\prod_{n=0}^\infty(1-aq^n)$ denote the infinite $q$-Pochhammer symbol and by $\eta(q)=\smash{q^{1/24}(q;q)_\infty}$ the Dedekind eta function. Moreover, denote by $\smash{\vartheta(z,q)=(q;q)_\infty(-q^{1/2}z;q)_\infty(-q^{1/2}z^{-1};q)_\infty=\sum_{n\in\Z}z^nq^{n^2/2}}$ the Jacobi theta function. It specialises to $\vartheta_3(q)=\vartheta(1,q)=\sum_{n\in\Z}q^{n^2/2}$, which is the theta series of the odd standard lattice $\Z$.


\section{Hypertoric and Quiver Varieties}

We review hypertoric (or toric hyperkähler) varieties, the main geometric objects of this text \cite{HKLR87,BD00} (see also \cite{HS02,PW07,Pro08}). They are a family of conical symplectic singularities \cite{Bea00} (and their resolutions) that can be realised as Hamiltonian reduction of a symplectic vector space $T^*V=T^*\C^N$ by the action of an algebraic torus $G=(\C^\times)^M$. We also refer to \cite{BK12,Kuw21} for a more thorough introduction in a similar context.

Special cases of hypertoric varieties are quiver varieties \cite{Nak94,Nak98} associated with quivers with abelian gauge nodes (see \autoref{sec:quiver}). Examples of these quiver hypertoric varieties include minimal nilpotent orbit closures for $\sl_N$ (see \autoref{sec:var1}) and Kleinian singularities of type $A_{N-1}$ (see \autoref{sec:var2}).


\subsection{Conical Symplectic Resolutions}\label{sec:conical}

Hypertoric varieties are in particular conical symplectic singularities that admit conical symplectic resolutions \cite{Bea00,BLPW16a,BLPW16b} (see also \cite{Kam22}). Recall that a \emph{symplectic resolution} is a morphism $\pi\colon X\to X_0$ of complex algebraic varieties such that
\begin{enumerate}
\item $X$ is smooth and has a symplectic structure,
\item $X_0$ is affine, normal and has a Poisson structure,
\item $\pi$ is projective, birational and Poisson.
\end{enumerate}
Moreover, a symplectic resolution is called \emph{conical} if there are $\C^\times$-actions on $X_0$ and $X$, compatible via $\pi$, such that $\C^\times$ contracts $X_0$ to a single point $0\in X_0$. We assume that the symplectic form has weight~$2$ under the $\C^\times$-action.

A conical symplectic resolution is \emph{Hamiltonian} if there are Hamiltonian actions of a torus $T$ on $X_0$ and $X$ such that $\pi$ is $T$-equivariant.


\subsection{Hypertoric Varieties and Their Symplectic Resolutions}\label{sec:hypertoric}

In the following, we introduce hypertoric varieties \cite{HKLR87,BD00}. Most of the results presented here can be found in \cite{HS02} and \cite{BK12}, or in the survey \cite{Pro08}. We mostly follow the notation used in \cite{Kuw21}.

\medskip

Let $M,N\in\Ns$ with $M\leq N$. We consider the $N$-dimensional vector space $V\cong\C^N$ and the $M$-dimensional torus $G=(\C^\times)^M$. We fix a basis of $V$ with corresponding coordinate functions $x_1,\dots,x_N\in V^*\subset\C[V]\cong\C[x_1,\dots,x_N]$ and assume that $G$ acts on $V\cong\C^N$ by transforming the coordinates as
\begin{equation*}
(t_1,\dots,t_M)\cdot x_j=t_1^{\Delta_{1j}}\dots t_M^{\Delta_{Mj}}x_j
\end{equation*}
for $(t_1,\dots,t_M)\in G$ and $j=1,\dots,N$, where $\Delta=(\Delta_{ij})_{1\leq i\leq M,1\leq j\leq N}$ is an $(M\times N)$-matrix with integer values $\Delta_{ij}\in\Z$.

In other words, we consider an algebraic $G$-action on $V$ and choose the basis such that the above coordinate functions are weight vectors with respect to this action. Indeed, if we denote by $\smash{\Delta_j=(\Delta_{ij})_{1\leq i\leq M}}$ the $j$-th column of the matrix~$\Delta$, then $\Delta_j$ is the weight of $x_j$ under the $G$-action for $j=1,\dots,N$. We call $\Delta\in\Z^{M\times N}$ the weight matrix.

The map $\smash{\Z^N\overset{\Delta}{\longrightarrow}\Z^M}$ defined by $\Delta$ is surjective if and only if the maximal minors of $\Delta$ (i.e.\ those of size $M\times M$) are relatively prime. We assume this in the following (cf.\ \autoref{conv:main}).

\medskip

We then consider the cotangent bundle $T^*V=V\oplus V^*$ of $V$. Let the coordinates $\smash{y_1,\dots,y_N\in V\subset\C[V^*]\cong\C[y_1,\dots,y_N]}$ be dual to $x_1,\dots,x_N$. The cotangent bundle $T^*V\cong\C^{2N}$ is naturally a symplectic vector space with the symplectic form $\sum_{i=1}^N dx_i\wedge dy_i$ and hence the coordinate ring $\C[T^*V]\cong\C[x_1,\dots,x_N,y_1,\dots,y_N]$ a Poisson algebra with Poisson bracket $\{y_i,x_j\}=\delta_{ij}$ and $\{x_i,x_j\}=\{y_i,y_j\}=0$ for $i,j=1,\dots,N$. The action of $G$ on $V$ induces an action on the symplectic vector space $T^*V\cong\C^{2N}$ characterised by the weight matrix $(\Delta,-\Delta)$ of size $M\times 2N$. In other words, $G$ acts on $\C^N\cong V^*\subset T^*V$ by transforming the coordinates as
\begin{equation*}
(t_1,\dots,t_M)\cdot y_j=t_1^{-\Delta_{1j}}\dots t_M^{-\Delta_{Mj}}y_j
\end{equation*}
for $(t_1,\dots,t_M)\in G=(\C^\times)^M$ and $j=1,\dots,N$.

Let $\g=\operatorname{Lie}(G)\cong\C^M$ denote the $M$-dimensional abelian Lie algebra corresponding to $G$. The above $G$-action is Hamiltonian, i.e.\ it can be characterised by a moment map, a morphism of affine varieties $\mu\colon T^*V\cong\C^{2N}\to\g^*\cong\C^M$, which here is of the form
\begin{equation*}
\mu((x_1,\dots,x_N,y_1,\dots,y_N))=\Bigl(\sum_{j=1}^N\Delta_{ij}x_jy_j\Bigr)_{1\leq i\leq M}.
\end{equation*}
The corresponding comoment map $\mu^*\colon\g\to\C[T^*V]$ is
\begin{equation*}
a_i\mapsto\sum_{j=1}^N\Delta_{ij}x_jy_j,
\end{equation*}
where $\{a_1,\dots,a_M\}$ is the standard basis of $\g\cong\C^M$. The comoment map satisfies
\begin{equation*}
\frac{d}{dt}(\exp(ta)\cdot f)\Bigr|_{t=0}\!\!=\{\mu^*(a),f\}
\end{equation*}
and
\begin{equation*}
0=\mu^*([a,b])=\{\mu^*(a),\mu^*(b)\}
\end{equation*}
for all $a,b\in\g$ and $f\in\C[T^*V]$.

Let $\{\cdot,\cdot\}$ denote the Poisson bracket on the structure sheaf $\smash{\mathcal{O}_{T^*V}}$ of the symplectic vector space $T^*V$. Then the induced $\g$-action on $\smash{\mathcal{O}_{T^*V}}$ is given by the comoment map via $a\mapsto\{\mu^*(a),\cdot\}$ for $a\in\g$. In other words, the $G$-action on $T^*V$ induces an action of $G$ on the structure sheaf $\smash{\mathcal{O}_{T^*V}}$.

\medskip

First, we consider the affine quotient, which is singular:
\begin{defi}[Affine Hypertoric Variety]
Consider the Hamiltonian action of $G=(\C^\times)^M$ on $T^*V=T^*\C^N$ determined by the weight matrix $\Delta\in\Z^{M\times N}$ with relatively prime maximal minors. The \emph{affine hypertoric variety} $Y_0(\Delta)$ associated with $\Delta$ is
\begin{equation*}
Y_0 (\Delta) \coloneqq \mu^{-1}(0) \qquot G = \Spec \left(\C [\mu^{-1}(0)]^G\right).
\end{equation*}
\end{defi}
By definition, the coordinate ring of $Y_0(\Delta)$ is $\smash{\C[Y_0(\Delta)]=\C[\mu^{-1}(0)]^G}$, the subring of $\smash{\C[\mu^{-1}(0)]=\C[T^*V]/(\{\sum_{j=1}^N\Delta_{ij}x_jy_j\}_{i=1}^M)=\C[T^*V]/(\{\mu^*(a_i)\}_{i=1}^M)}$ invariant under $G$.

\medskip

In the following we define conical symplectic resolutions of the hypertoric variety $Y_0(\Delta)$ as GIT quotients. We identify the fractional characters $\Hom(G,\C^\times)\otimes_\Z\Q$ of $G=(\C^\times)^M$ with $\Q^M$. A choice of $\delta\in\Q^M$ is called \emph{stability parameter}.

A point $p\in T^*V$ is called $\delta$-semistable if there is an $m\in\Ns$ and a function $f\in\C[T^*V]^{G,m\delta}$ (of weight $m\delta$ under the $G$-action) with $f(p)\neq 0$. If additionally the stabiliser subgroup $G_p$ is finite, then $p\in T^*V$ is called $\delta$-stable. We denote by $(T^*V)_\delta^\mathrm{s}$ and $\mathfrak{X}\coloneqq(T^*V)_\delta^\mathrm{ss}$ the subset of $\delta$-stable and $\delta$-semistable points of $T^*V$, respectively. The stability parameter $\delta$ is called effective if $(T^*V)_\delta^\mathrm{ss}\neq\emptyset$. There is a rational polyhedral fan $\Delta(G,T^*V)$ in $\Q^M$, the GIT fan, whose support is the set of all effective stability parameters $\delta\in\Q^M$ such that $\mathfrak{X}\neq\{0\}$ and whose walls are given by the stability parameters such that $\mathfrak{X}\neq(T^*V)_\delta^\mathrm{s}$. With the assumptions on $\Delta$, the maximal cones of $\Delta(G,T^*V)$ are $M$-dimensional. We also denote by $\mathfrak{Z}\coloneqq T^*V\setminus\mathfrak{X}$ the unstable locus, a closed subset of $T^*V$.

The semistable locus $\mathfrak{X}=(T^*V)_\delta^\mathrm{ss}$, an open subset of $T^*V$, has the following presentation, as shown for example in~\cite{Kon07} (see also \cite{Nag21}).
\begin{prop}\label{prop:stable-locus}
Let $\delta=(\delta_1,\dots,\delta_M)^t\in\mathbb{Q}^M$ be an effective stability parameter. Then a point $(x_1,\dots,x_N,y_1,\dots,y_N)\in T^*V$ lies in the semistable locus $\mathfrak{X}=(T^*V)_\delta^\mathrm{ss}$ if and only if there exist $a_1,\dots,a_N\in\Q_{\geq0}$ and $b_1,\dots,b_N\in\Q_{\geq0}$ such that 
\begin{equation*}
\delta_i=\sum_{\substack{j=1\\x_j\neq0}}^Na_j\Delta_{ij}-\sum_{\substack{j=1\\y_j\neq0}}^Nb_j\Delta_{ij}
\end{equation*}
for all $i=1,\dots,M$.
\end{prop}

Notice that the $G$-action is closed on the inverse image $\mu^{-1}(0)$ under the moment map $\mu$. For a subset $S\subset T^*V$, we say that two points $p,q\in S$ are $S$-equivalent if the $G$-orbit closures $\overline{G\cdot p}$ and $\overline{G\cdot q}$ intersect in $S$.

\medskip

We then define (projective) hypertoric varieties associated with effective stability parameters $\delta\in\Q^M$ as GIT quotients:
\begin{defi}[Hypertoric Variety]
For an effective stability parameter $\delta\in\Q^M$, the \emph{hypertoric variety} $Y_\delta(\Delta)$ is the categorical quotient of the semistable locus $\mathfrak{X}$,
\begin{equation*}
Y_\delta(\Delta)\coloneqq\mu^{-1}(0)\qquot\!_\delta G= (\mu^{-1}(0)\cap\mathfrak{X})/{\sim},
\end{equation*}
where $\sim$ denotes $(\mu^{-1}(0)\cap\mathfrak{X})$-equivalence.
\end{defi}
It follows from the definition that the hypertoric variety $Y_\delta(\Delta)$ is a projective scheme over $Y_0(\Delta)$,
\begin{equation*}
Y_\delta(\Delta)\cong\Proj\Bigl(\bigoplus_{m\in\N}\C[\mu^{-1}(0)]^{G,m\delta}\Bigr)
\end{equation*}
and that there is a projective morphism
\begin{equation*}
\pi\colon Y_\delta(\Delta)\to Y_0(\Delta).
\end{equation*}

We recall some important properties of the hypertoric varieties $Y_\delta(\Delta)$ from \cite{HS02,BK12}. If the stability parameter $\delta$ is in the interior of a maximal cone of $\Delta(G,\mu^{-1}(0))$, then the hypertoric variety $Y_\delta(\Delta)$ is an orbifold (and coincides with the geometric quotient $T^*V/G$). In this case, we say that $\delta$ is \emph{generic}. $Y_\delta(\Delta)$ is smooth if and only if $\delta$ is generic and $\Delta$ is unimodular. In that case, $Y_\delta(\Delta)$ has complex dimension $2(N-M)$. Recall that we say that the matrix $\Delta$ is unimodular if every $(M\times M)$-minor takes values in $\{-1,0,1\}$. Moreover, the moment map $\mu$ is flat, and $\mu^{-1}(0)$ is a reduced complete intersection in $T^*V$.
\begin{conv}\label{conv:main}
Throughout the paper, we assume that the matrix $\Delta\in\Z^{M\times N}$ is surjective, unimodular and that $\delta\in\Q^M$ is a generic effective stability parameter.
\end{conv}
The symplectic structure on $T^*V$ induces a symplectic structure on $Y_\delta(\Delta)$, and a Poisson structure on $Y_0(\Delta)$. The morphism $\pi$ preserves these Poisson structures, i.e.\ there is a homomorphism of Poisson algebras $\smash{\mathcal{O}_{Y_0(\Delta)}\to\mathcal{O}_{Y_\delta(\Delta)}}$. There is also a conical $\C^\times$-action on $Y_\delta(\Delta)$ and $Y_0(\Delta)$ inherited from the linear scaling action of $\C^\times$ on $T^*V$. Overall, one can show that under the unimodularity assumption, $\pi\colon Y_\delta(\Delta)\to Y_0(\Delta)$ is a (conical and Hamiltonian) symplectic resolution.

The structure sheaf $\mathcal{O}_{Y_\delta(\Delta)}$ of $Y_\delta(\Delta)$ can also be obtained by Hamiltonian reduction of the structure sheaf $\mathcal{O}_{\mathfrak{X}}$ of $\mathfrak{X}=(T^*V)_\delta^\mathrm{ss}$. Indeed,
\begin{equation*}
\mathcal{O}_{Y_\delta(\Delta)}=\Bigl(p_{*}\bigl(\mathcal{O}_{\mathfrak{X}}\bigm/\sum_{i=1}^{M}\mathcal{O}_{\mathfrak{X}}\mu^*(a_i)\bigr)\Bigr)^\g
\end{equation*}
with the projection map $p\colon\mu^{-1}(0)\cap\mathfrak{X}\to Y_\delta(\Delta)$. We also note that $Y_\delta(\Delta)$ and $Y_0(\Delta)$ have the same global functions, i.e.
\begin{equation*}
\mathcal{O}_{Y_\delta(\Delta)}(Y_\delta(\Delta))=\C[Y_0(\Delta)]=\C[\mu^{-1}(0)]^G
\end{equation*}
In \autoref{thm:glob-sec}, we prove a vertex algebraic analogue of this statement.


\subsection{Torus Actions}\label{sec:torus-act}

Since $\Z^M$ is a free $\Z$-module, the surjective map $\Z^N\to\Z^M$ defined by the $(M\times N)$-matrix $\Delta$ can be extended to a short exact sequence
\begin{equation}\label{eq:SES}
\{0\}\longrightarrow\Z^{N-M}\overset{E}{\longrightarrow}\Z^N\overset{\Delta}{\longrightarrow}\Z^M\longrightarrow\{0\}.
\end{equation}
In other words, the columns of the $(N\times (N-M))$-matrix $E$ form a basis of the null space of $\Delta$. It follows that, given $\Delta$, the matrix $E$ is only well-defined up to a change of basis of this null space. However, this will not affect the corresponding hypertoric variety up to isomorphism (see \autoref{subsec:iso} for further comments on isomorphisms). Note that $E$ is unimodular if and only if $\Delta$ is. Hence, following \autoref{conv:main}, we assume that both $E$ and $\Delta$ are unimodular.

By applying the contravariant functor $\Hom(\,\cdot\,,\C^\times)$ to \eqref{eq:SES}, we obtain a short exact sequence of abelian groups
\begin{equation*}
\{1\}\longrightarrow G\overset{\Delta^t}{\longrightarrow}(\C^\times)^N\overset{E^t}{\longrightarrow}T\longrightarrow\{1\}
\end{equation*}
with $G\coloneqq (\C^\times)^M$ and $T\coloneqq (\C^\times)^{N-M}$, which precisely recovers the underlying action of $G$ on $T^*V$.

Fixing a generic effective stability parameter $\delta\in\Hom(G,\C^\times)\otimes_\Z\Q\cong\Q^M$, we obtain the symplectic resolution
\begin{equation*}
Y_\delta(\Delta)\to Y_0(\Delta)
\end{equation*}
of dimension $2(N-M)$ by Hamiltonian reduction by $G$. Then, there is a residual Hamiltonian action of $T=(\C^\times)^{N-M}$ on $Y_\delta(\Delta)$ and $Y_0(\Delta)$. This action is defined only up to the action of the torus $G=(\C^\times)^M$. It is fixed by a choice of splitting
\begin{equation*}
q\colon T\to(\C^\times)^N
\end{equation*}
of the above short exact sequence. The existence of such a splitting is guaranteed by the unimodularity assumption: any unimodular matrix can be extended to a square matrix of determinant $\pm 1$. We notice that $\epsilon\in\Hom(\C^\times,T)$ defines a Hamiltonian $\C^\times$-action on $Y_\delta(\Delta)$ and $Y_0(\Delta)$.


\subsection{Symplectic Duality (and 3d Mirror Symmetry)}\label{sec:sym-dual}

We briefly recall symplectic duality for hypertoric varieties \cite{BLPW16b}. See \cite{Kam22,WY23} for recent reviews. The closely related physical concept of 3d mirror symmetry, which we shall comment upon even more briefly, was introduced in \cite{IS96}.

Very roughly, the (expected) statement of \emph{symplectic duality} is that given some symplectic resolution $X\to X_0$ one can find a dual symplectic resolution $X^!\to X_0^!$, with certain structures of $X$ corresponding to certain other structures of $X^!$ and such that $(X^!)^!\cong X$. Most known examples of symplectic dual pairs $(X,X^!)$ come from Higgs and Coulomb branches of 3d $\mathcal{N}=4$ theories.

\medskip

While in full generality symplectic duality remains hypothetical, there are some well-understood families of examples. In fact, an early instance of symplectic duality is Gale duality of hypertoric varieties \cite{BLPW10,BLPW12}, which has been thoroughly studied. Indeed, given a hypertoric variety, its symplectic dual is also a hypertoric variety, namely the \emph{Gale dual} hypertoric variety.

We now describe the Gale dual hypertoric variety. Recall the short exact sequences
\begin{equation*}
\{0\}\longrightarrow\Z^{N-M}\overset{E}{\longrightarrow}\Z^N\overset{\Delta}{\longrightarrow}\Z^{M}\longrightarrow\{0\}
\end{equation*}
and
\begin{equation*}
\{1\}\longrightarrow G \overset{\Delta^t}{\longrightarrow}(\C^\times)^N\overset{E^t}{\longrightarrow}T \longrightarrow\{1\}
\end{equation*}
introduced in \autoref{sec:torus-act}, defining the hypertoric variety $Y_0(\Delta)$. From these, we obtain the following dual ones:
\begin{equation*}
\{0\}\longrightarrow\Z^M\overset{E^!}{\longrightarrow}\Z^N\overset{\Delta\!^!}{\longrightarrow}\Z^{N-M}\longrightarrow\{0\}
\end{equation*}
and
\begin{equation*}
\{1\}\longrightarrow G^!\overset{(\Delta\!^!)^t}{\longrightarrow}(\C^\times)^N\overset{(E^!)^t}{\longrightarrow}T^!\longrightarrow\{1\}
\end{equation*}
with $\Delta\!^!\coloneqq E^t$ and $E^!\coloneqq\Delta^t$ and $G^!\coloneqq T^\vee\cong(\C^\times)^{N-M}$ and $T^!\coloneqq G^\vee\cong(\C^\times)^M$. Here, $T^\vee$ and $G^\vee$ denote dual tori with the property $\Hom(\C^\times,T^\vee)\cong\Hom(T,\C^\times)$ and analogously for $G^\vee$.

Recall that we fixed $\delta\in\Hom(G,\C^\times)\cong\Hom(\C^\times,T^!)$, and that we can also fix $\epsilon\in\Hom(\C^\times,T)\cong\Hom(G^!,\C^\times)$. Suppose both are generic. Then
\begin{equation*}
Y_\delta(\Delta)\to Y_0(\Delta)\quad\text{and}\quad Y_\epsilon(\Delta\!^!)\to Y_0(\Delta\!^!)
\end{equation*}
are symplectic dual (here: Gale dual) symplectic resolutions, of dimension $2(N-M)$ and dimension $2M$, respectively, along with Hamiltonian $\C^\times$-actions on each. As discussed in \autoref{sec:torus-act}, the stability parameter $\delta$ becomes the parameter defining the $\C^\times$-action on the Gale dual and vice versa for $\epsilon$.

We mention that the two examples of hypertoric (quiver) varieties discussed in \autoref{sec:var1} and \autoref{sec:var2} are a Gale dual pair, as we shall explain in more detail there.

\medskip

Finally, let us comment on 3d mirror symmetry in the context of 3d abelian gauge theories. This is the statement that for any theory defined by the datum $\Delta$, there is a a mirror dual theory, defined by the datum $\Delta\!^!$, such that the Higgs branch of one theory is isomorphic, as a hyperkähler variety, to the Coulomb branch of the other theory, and vice versa (see also \autoref{sec:assvar}). Thus, the symplectic resolutions $Y_\delta(\Delta)\to Y_0(\Delta)$ and $Y_\epsilon(\Delta\!^!)\to Y_0(\Delta\!^!)$ are the Higgs and Coulomb branches and the Coulomb and Higgs branches of the theories defined by $\Delta$ and $\Delta\!^!$, respectively. 


\subsection{Poisson Isomorphisms}\label{subsec:iso} 

We continue with a few remarks on holomorphic symplectic (and therefore Poisson) isomorphisms between hypertoric varieties. Let $(\Delta^1,E^1)$ be a pair of unimodular $(M\times N)$- and $(N\times(N-M))$-matrices, respectively, entering the above short exact sequence~\eqref{eq:SES}. Recall that $\Delta^1$ and $E^1$ are defined up to a choice of basis, and different choices lead to isomorphic hypertoric varieties.

As we now explain, there are further operations that lead to Poisson isomorphism of hypertoric varieties. See, e.g., Proposition~1.3.2 in \cite{Pro08}, or the classification in~\cite{Nag21}, Theorem~4.2.
\begin{rem}\label{rem:swap-sign}
Let $(\Delta^2,E^2)$ be another pair of matrices that can be obtained from $(\Delta^1,E^1)$ by multiplying some of the columns of $\Delta^1$ and some of the rows of $E^1$ by an overall sign, or by swapping columns of $\Delta^1$ and rows of $E^1$. Then $Y_0(\Delta^1)$ is isomorphic as a holomorphic symplectic (and therefore Poisson) variety to $Y_0(\Delta^2)$, and $Y_0((\Delta^1)^!)$ is isomorphic as a holomorphic symplectic (and therefore Poisson) variety to $Y_0((\Delta^2)^!)$. Thus, the data defining the hypertoric variety $Y_0(\Delta)$ as a Poisson variety can equally well be taken to be the set of hyperplanes determined by the rows of $E$ in $\R^{N-M}$, upon fixing the standard realisation $\Z^{N-M}\subset\R^{N-M}$ without orientation.
\end{rem}

\medskip

Without loss of generality, we can further restrict the assumptions on $\Delta$ and $E$ slightly. We shall use this in \autoref{prop:specialelements}, for example. Recall that, as the map defined by $\Delta$ is surjective, it cannot have any zero row.
\begin{rem}\label{rem:zerocolumnrow}
We may also assume that $\Delta$ has no zero column. Indeed, if that were the case, we could split off the trivial (meaning ``unreduced'') hypertoric variety $T^*\C$ from $Y_\delta(\Delta)$ and from $Y_0(\Delta)$, corresponding to the short exact sequence
\begin{equation*}
\begin{tikzcd}
\{0\} \arrow[r]
& \Z^1 \arrow[r, "E=\pm1"]
& \Z^1 \arrow[r, "\Delta=0"]
& \Z^0 \arrow[r]
&\{ 0 \}
\end{tikzcd}
\end{equation*}
for $N=1$ and $M=0$.

Dually (in the symplectic duality sense), we may assume that $E$ has no zero row. If that were the case, we could split off the trivial (meaning ``fully reduced'') hypertoric variety $\{\mathrm{pt.}\}$ from $Y_\delta(\Delta)$ and from $Y_0(\Delta)$, corresponding to the  short exact sequence
\begin{equation*}
\begin{tikzcd}
\{0\} \arrow[r]
& \Z^0 \arrow[r, "E=0"]
& \Z^1 \arrow[r, "\Delta=\pm1"]
& \Z^1 \arrow[r]
&\{ 0 \}
\end{tikzcd}
\end{equation*}
for $N=1$ and $M=1$ (cf.\ \autoref{sec:smallexample}).
\end{rem}


\subsection{Coordinate Ring Generators}\label{sec:coordringgens}

We describe an explicit (finite) set of generators of the coordinate ring $\smash{\C[Y_0(\Delta)]=\C[T^*V]^G/(\{\sum_{j=1}^N\Delta_{ij}x_jy_j\}_{i=1}^M)}$ of the affine hypertoric variety $Y_0(\Delta)$. Recall the unimodular matrices $E$ and $\Delta$ entering the short exact sequence~\eqref{eq:SES}, expressed in a given basis. First, note that the bilinear expressions
\begin{equation*}
J_i\coloneqq\sum_{j=1}^NE_{ji}x_jy_j\in\C[T^*V]^G
\end{equation*}
for $i=1,\dots,N-M$ are $G$-invariant. Similarly, for any $a\in\Z^{N-M}$,
\begin{equation*}
W^a\coloneqq\prod_{\substack{j=1\\(Ea)_j>0}}^Nx_j^{(Ea)_j}\prod_{\substack{j=1\\(Ea)_j<0}}^Ny_j^{-(Ea)_j}\in\C[T^*V]^G.
\end{equation*}
We sometimes identify the $J_i,W^a\in\C[T^*V]^G$ with their images in the quotient $\C[Y_0(\Delta)]$ modulo the ideal $\smash{(\{\sum_{j=1}^N\Delta_{ij}x_jy_j\}_{i=1}^M)}$.

The following statements can be found in Lemma~2.17 and Remark~2.18 in \cite{Nag21} and in Section~6.1.2 of \cite{BDGH16}.
\begin{prop}\label{prop:coordring}
The coordinate ring of $Y_0(\Delta)$ is of the form
\begin{equation*}
\C[Y_0(\Delta)]\cong\C\bigl[\{x_jy_j\}_{j=1}^N\cup\{W^a\mid a\in\Z^{N-M}\}\bigr]\bigm/\bigl(\bigl\{\textstyle\sum_{j=1}^N\Delta_{ij}x_jy_j\bigr\}_{i=1}^M\bigr).
\end{equation*}
Noting that the rows of $\Delta$ together with the columns of $E$ span $\C^N$, this means that $\C[Y_0(\Delta)]$ is generated by the $J_i$ for $i=1,\dots,N-M$ and the $W^a$ for $a\in\Z^{N-M}$.
\end{prop}
In fact, a finite number of generators suffices.
\begin{prop}[Coordinate Ring Generators]\label{prop:ring-gen}
The coordinate ring $\C[Y_0(\Delta)]$ is generated by $J_i$ for $i=1,\dots,N-M$ and $W^a$ for a finite number of $a\in\Z^{N-M}$.
\end{prop}
\begin{proof}
To prove the assertion using \autoref{prop:coordring}, we need to show that a finite subset of the generators listed is sufficient to generate all of them, in particular, the $W^a$ for all $a\in\Z^{N-M}$.

Consider the $N$ hyperplanes that pass through the origin in $\R^{N-M}$ and whose normal vectors are given, under the standard realisation $\Z^{N-M}\subset\R^{N-M}$, by the rows of $E$. These divide $\Z^{N-M}$ into finitely many rational polyhedral cones $C$. Then, for a given $C$,
\begin{equation*}
W^aW^b=W^{a+b}
\end{equation*}
for all $a,b\in C\cap\Z^{N-M}$. It follows from Gordan's lemma that within each region~$C$ there is a finite generating set of all $W^a$ with $a\in C\cap\Z^{N-M}$. It then follows that a finite subset of the $W^{a}$, $a\in\Z^{N-M}$, generates all of them.
\end{proof}

We shall describe examples of the generating sets guaranteed by \autoref{prop:ring-gen} below in \autoref{sec:var1} and \autoref{sec:var2}.

\medskip

The following consequence of \autoref{prop:ring-gen} will be useful later (see \autoref{thm:var}) when we determine the associated variety of the hypertoric \svoa{}s constructed in this text.
\begin{cor}\label{cor:poiss-cen}
The Poisson centre $Z(\C[Y_0(\Delta) ])$ of $\C[Y_0(\Delta)]$ is trivial.
\end{cor}
\begin{proof}
This follows directly from the presentation of the ring given above.
\end{proof}
The corollary can also be given a more geometric proof; see the proof of Proposition~3.3 in \cite{AK18}.


\subsection{Unstable Locus}\label{sec:unstable}

In \autoref{sec:global} below, we want to show that a certain sheaf $\smash{\mathcal{V}_{\delta,\Delta}^{\hbar}}$ on $Y_\delta(\Delta)$ has no support outside the stable locus. To this end, here, we describe the ideal defining the unstable locus.

Recall that by \autoref{prop:stable-locus}, the unstable locus $\mathfrak{Z}=T^*V\setminus\mathfrak{X}$ is given by those points $(x_1,\dots,x_N,y_1,\dots,y_N)\in T^*V$ such that there are no $a_j\in\Q_{\geq0}$ and $b_j\in\Q_{\geq0}$ for $j=1,\dots,N$ satisfying
\begin{equation*}
\delta_i=\sum_{\substack{j=1\\x_j\neq0}}^Na_j\Delta_{ij}-\sum_{\substack{j=1\\y_j\neq0}}^Nb_j\Delta_{ij}
\end{equation*}
for all $i=1,\dots,M$. The unstable locus $\mathfrak{Z}$ is closed and we denote its defining ideal by $I(\mathfrak{Z})\subset\C[T^*V]$.

\begin{prop}[Unstable Locus]\label{prop:unstable-loc}
Fix an effective stability parameter $\delta\in\Q^M$. Then $I(\mathfrak{Z})$ contains the ideal $(m_d\mid \Delta d=\delta)$ in $\C[T^*V]$ for
\begin{equation*}
m_d\coloneqq\prod_{\substack{j=1,\dots,N\\d_j>0}}x_j\prod_{\substack{j=1,\dots,N\\d_j<0}}y_j
\end{equation*}
where $d\in\Q^N$ runs over the (nonempty) set of solutions of $\Delta d=\delta$.
\end{prop}
\begin{proof}
Fix a point $(x_1,\dots,x_N,y_1,\dots,y_N)\in T^*V$. It is in the semistable locus $\mathfrak{X}$ if and only if there are $\smash{a,b\in\Q_{\geq0}^N}$ such that
\begin{equation*}
\delta_i=\sum_{\substack{j=1\\x_j\neq0}}^N\Delta_{ij}a_j-\sum_{\substack{j=1\\y_j\neq0}}^N\Delta_{ij}b_j
\end{equation*}
for all $i=1,\dots,M$. We need to show that if $m_d(x_1,\dots,x_N,y_1,\dots,y_N)\neq0$ for one of the monomials $m_d$, $d\in\Q^N$ with $\Delta d=\delta$, then the point $(x_1,\dots,x_N,y_1,\dots,y_N)$ is already in the semistable locus $\mathfrak{X}$. Now, by definition, the assumption means that this $d\in\Q^N$ satisfies that all the $x_j$ with $d_j>0$ and $y_j$ with $d_j<0$ are nonzero. Then, for the $j\in\{1,\dots,N\}$ with $d_j>0$ and $d_j<0$ we set
\begin{equation*}
a_j\coloneqq d_j\in\Q_{>0}\qquad\text{and}\qquad b_j\coloneqq -d_j\in\Q_{>0},
\end{equation*}
respectively, and the remaining $a_j$ and $b_j$ to zero, in order to obtain a solution $\smash{a,b\in\Q_{\geq0}^N}$ to the semi-stability condition. Hence, this $(x_1,\dots,x_N,y_1,\dots,y_N)$ is in the semistable locus~$\mathfrak{X}$.
\end{proof}

\begin{rem}\label{rem:Zideal}
We can specialise the monomials $m_d$ in $I(\mathfrak{Z})$ by restricting to certain \emph{minimal} solutions of $\Delta d=\delta$.

Recall that the $(M\times N)$-matrix $\Delta$ has full rank~$M$ and is unimodular. In particular, some $M$-element subsets of its columns span $\Q^M$. Indeed, denoting as before the columns of $\Delta$ by $\Delta_j$, $j=1,\dots,N$, let $\mathcal{C}$ be the (nonempty) set of subsets $C=\{j_1<\dots<j_M\}\subset\{1,\dots,N\}$ of column indices such that the $(M\times M)$-submatrix $\Delta_C=(\Delta_{j_1},\dots,\Delta_{j_M})$ has determinant $\pm 1$.

Now, for any $C=\{j_1,\dots,j_M\}\in\mathcal{C}$, we can solve $\delta=\sum_{l=1}^Md_l\Delta_{j_l}$ by setting $\smash{d_l\coloneqq(\Delta_C^{-1}\delta)_l\in\Q}$. Extending the $d_l$ for $l=1,\dots,M$ with zeroes corresponding to the rows $\{1,\dots,N\}\setminus C$ yields a solution $d\in\Q^N$ of $\Delta d=\delta$. As it has $N-M$ zero entries, we can think of it as a minimal solution of $\Delta d=\delta$. (If $\delta$ is generic, then $\Delta_C^{-1}\delta$ has no zero entries, and so the minimal solutions $d\in\Q^N$ have exactly $N-M$ zero entries.)

Then, if we define the corresponding monomials
\begin{equation*}
m_C\coloneqq\prod_{\substack{l=1\\(\Delta_C^{-1}\delta)_l>0}}^Mx_{j_l}\prod_{\substack{l=1\\(\Delta_C^{-1}\delta)_l<0}}^My_{j_l}
\end{equation*}
for all $C\in\mathcal{C}$, these are clearly special cases of the monomials $m_d$ for $d\in\Q^N$ with $\Delta d=\delta$ that we defined in \autoref{prop:unstable-loc}. Hence, there is the inclusion of ideals
\begin{equation*}
\{m_C\mid C\in\mathcal{C}\}\subset\{m_d\mid \Delta d=\delta\}\subset I(\mathfrak{Z}).
\end{equation*}
\end{rem}

We recall the generators of the coordinate ring $\C[Y_0(\Delta)]$ from \autoref{prop:ring-gen}. We show that many of them lie in the defining ideal of the unstable locus:
\begin{prop}\label{prop:specialelements}
Let $I(\mathfrak{Z})$ be the defining ideal of the unstable locus $\mathfrak{Z}=T^*V\setminus\mathfrak{X}$. Then there exists some $b\in\Z^{N-M}\setminus\{0\}$ such that
\begin{equation*}
W^b\in I(\mathfrak{Z})\quad\text{and}\quad W^{-b}\in I(\mathfrak{Z}).
\end{equation*}
\end{prop}
\begin{proof}
Using \autoref{prop:unstable-loc}, we can rephrase the assertion as follows: we claim that there exist $b\in\Z^{N-M}$ and $d^+,d^-\in\Q^N$ such that $\Delta d^+=\Delta d^-=\delta$ and
\begin{align*}
d^+_i >0 &\;\Rightarrow\; (Eb)_i>0,& d_i^+ < 0 &\;\Rightarrow\; (Eb)_i <0,
\intertext{as well as}
d^-_i >0 &\;\Rightarrow\; -(Eb)_i>0,& d_i^- < 0 &\;\Rightarrow\; -(Eb)_i <0.
\end{align*}
Indeed, in that case, the monomials $m_{d^+}\mid W^b$ and $m_{d^-}\mid W^{-b}$, which shows that $W^b$ and $W^{-b}$ are in the ideal $(m_d\mid\Delta d=\delta)\subset I(\mathfrak{Z})$ of $\C[T^*V]$.

To prove the rephrased assertion, fix a solution $d\in\Q^N$ of $\Delta d=\delta$, which always exists as $\Delta$ has full rank. Then for any $r\in \Q$ and any vector $b\in\Z^{M-N}$, by the fact that $\Delta E=0$, $\Delta(d+rEb)=\delta$. By \autoref{rem:zerocolumnrow}, we may assume that $E$ does not have any zero row. Hence, there exists $b\in\Z^{M-N}\setminus\{0\}$ such that $Eb\in\Z^N\subset\Q^N$ has no zero entries. Then, for large enough $r\in\Q$, $d^+\coloneqq d+rEb$ has entries with the same sign as $Eb$, and $
d^-\coloneqq d-rEb$ has entries with the same sign as $-Eb$. This proves the assertion.
\end{proof}

\begin{rem}\label{rem:specialelements}
In view of \autoref{lem:lowertruncated}, one would want to prove a slightly stronger version of \autoref{prop:specialelements}, where one additionally requires that the solutions $d^+$ and $d^-$ of $\Delta d=\delta$ be minimal in the sense of \autoref{rem:Zideal}. In other words, we would like to show that it is possible to find sets $C,C'\in\mathcal{C}$ and a vector $b\in\Z^{M-N}$ such that $m_C\mid W^b$ and $m_{C'}\mid W^{-b}$.

This is likely true under no further (or only very mild) assumptions on the stability parameter $\delta\in\Q^M$, but we do not verify this here.

We can slightly reformulate the problem to finding two sets $C,C'\in\mathcal{C}$ such that the corresponding minimal solutions $d,d'$ of $\Delta d=\delta$ satisfy that $d$ and $-d'$ lie in the same closed orthant of $\Q^N$, which means that $d_j\leq0$ if and only if $-d'_j\leq0$ and $d_j\geq0$ if and only if $-d'_j\geq0$ for all $j=1,\dots,N$. Then, $v\coloneqq d-d'$ lies in the same orthant and also in the kernel of $\Delta$ and hence in the column space of $E$. By scaling, it is of the form $v=Eb$ for some $b\in\Z^{M-N}$ that then satisfies the conditions in the proof of \autoref{prop:specialelements}.
\end{rem}


\subsection{Universal Family of Poisson Deformations}\label{sec:deformation}

The hypertoric variety $Y_\delta(\Delta)$ has the universal family of Poisson deformations $\tilde{Y}_\delta(\Delta)$ over $\g^*\cong\C^M$, which we briefly review \cite{KV02,Los12,Los22,Nag21}.

Consider $\smash{\mathfrak{X}\times\g^*=(T^*V)_\delta^\mathrm{ss}\times\g^*}$ as a smooth algebraic Poisson variety, where the Poisson bracket on $\C[\g^*]=S(\g)$ is trivial. One can consider the extended moment map $\smash{\tilde\mu\colon\mathfrak{X}\times\g^*\to\g^*}$ such that the corresponding comoment map is given by $\tilde\mu^*\colon\g\to\C[\mathfrak{X}\times\g^*]=\C[\mathfrak{X}]\otimes\C[\g^*]$,
\begin{equation*}
a_i\mapsto\mu^*(a_i)-\tau_i^\Delta=\sum_{j=1}^N\Delta_{ij}x_jy_j-\tau_i^\Delta,
\end{equation*}
where $\{a_1,\dots,a_M\}$ and $\{\tau_1^\Delta,\dots,\tau_M^\Delta\}$ are the (same) standard basis of $\g\cong\C^M$ in the domain and codomain of $\tilde\mu$, respectively. Then there is an isomorphism $\tilde\mu^{-1}(0)\cong\mathfrak{X}$ that identifies $\tau_i^\Delta$ with $\mu^*(a_i)$ for $i=1,\dots,M$. The torus $G=(\C^\times)^M$ acts freely on $\mathfrak{X}\times\g^*$ and the action preserves $\tilde\mu^{-1}(0)$. We then define the Poisson manifold
\begin{equation*}
\tilde{Y}_\delta(\Delta)\coloneqq\tilde{\mu}^{-1}(0)/G\cong\mathfrak{X}/G.
\end{equation*}
The second projection map $\mathfrak{X}\times\g^*\to\g^*$ induces the morphism $\rho\colon\tilde{Y}_\delta(\Delta)\to\g^*$ of Poisson schemes, and the zero fibre is $\rho^{-1}(0)\cong Y_\delta(\Delta)$ as holomorphic symplectic manifolds. One can show that $\tilde{Y}_\delta(\Delta)$ is a universal family of filtered Poisson deformations of $Y_\delta(\Delta)$ over $\g^*\cong H^2(Y_\delta(\Delta),\C)$.

\begin{rem}\label{rem:namikawa}
The universal Poisson deformation
\begin{equation*}
\tilde{Y}_{\delta}(\Delta)\to\g^*
\end{equation*}
descends to the universal Poisson deformation
\begin{equation*}
\tilde{Y}_{0}(\Delta)\to\g^*/\mathbb{W}
\end{equation*}
where $\mathbb{W}$ is the (finite) Namikawa-Weyl group.  See \cite{Nag21}, for more details. While the Namikawa-Weyl group was used in \cite{Kuw21} to define vertex algebras associated with hypertoric varieties, we shall not make extensive use of this group for our vertex superalgebra construction over $Y_0(\Delta)$.
\end{rem}

The structure sheaf $\mathcal{O}_{\tilde{Y}_\delta(\Delta)}$ of $\tilde{Y}_\delta(\Delta)$ is obtained by Hamiltonian reduction of the structure sheaf $\mathcal{O}_{\mathfrak{X}\times\g^*}=\mathcal{O}_\mathfrak{X}\otimes_\C\mathcal{O}_{\g^*}$ of $\mathfrak{X}\times\g^*$. Indeed,
\begin{align*}
\mathcal{O}_{\tilde{Y}_\delta(\Delta)}&=\Bigl(\tilde{p}_{*}\bigl(\mathcal{O}_{\mathfrak{X}\times\g^*}\bigm/\sum_{i=1}^{M}\mathcal{O}_{\mathfrak{X}\times\g^*}\tilde\mu^*(a_i)\bigr)\Bigr)^\g\\
&=\Bigl(\tilde{p}_{*}\bigl(\mathcal{O}_{\mathfrak{X}\times\g^*}\bigm/\sum_{i=1}^{M}\mathcal{O}_{\mathfrak{X}\times\g^*}(\mu^*(a_i)-\tau_i^\Delta)\bigr)\Bigr)^\g
\end{align*}
with the projection $\tilde{p}\colon\mathfrak{X}\times\g^*\to\tilde{Y}_\delta(\Delta)$. It is an algebra over $\C[\g^*]=\C[\tau_1^\Delta,\dots,\tau_M^\Delta]$, and $Y_\delta(\Delta)$ is the fibre of $\rho\colon\tilde{Y}_\delta(\Delta)\to\g^*$ at $\tau_1^\Delta=\dots=\tau_M^\Delta=0$; cf.\ the expression for $\mathcal{O}_{Y_\delta(\Delta)}$ in \autoref{sec:hypertoric}.

For $\delta=0$, $\tilde{Y}_0(\Delta)\cong T^*V/G$ can be viewed as the (affine) Lawrence hypertoric variety (see \cite{Nag21} for details). The coordinate ring is
\begin{equation*}
\C[\tilde{Y}_0(\Delta)]=\bigl(\C[T^*V]\otimes\C[\g^*]/(\{\mu^*(a_i)-\tau^\Delta_i\}_{i=1}^M)\bigr)^G\cong\C[T^*V]^G.
\end{equation*}

\medskip

We now relate the coordinate rings of the hypertoric variety $Y_\delta(\Delta)$ and its deformation $\tilde{Y}_\delta(\Delta)$. Later, in \autoref{thm:var}, we use this to determine the associated variety of the hypertoric \svoa{}s constructed in this text.
\begin{prop}\label{prop:poiss-centre}
Let $I_{Z(\Delta)}$ be the ideal in $\C[\tilde{Y}_\delta(\Delta)]$ generated by the augmentation ideal of its Poisson centre. Then, as a Poisson algebra,
\begin{equation*}
\C[\tilde{Y}_\delta(\Delta)]/I_{Z(\Delta)}\cong\C[Y_0(\Delta)].
\end{equation*}
\end{prop}
\begin{proof}
Let $J$ be the ideal of $\C[\tilde{Y}_\delta(\Delta)]$ generated by $\rho^*(x)$ for $x\in\g$,
where we recall that $\rho\colon\tilde{Y}_{\delta}(\Delta)\to\mathfrak{g}^*$ is the fibration. Then it follows that $J\subset I_{Z(\Delta)}$ and $\C[Y_0(\Delta)]=\C[\tilde{Y}_\delta(\Delta)]/J$. But by \autoref{cor:poiss-cen}, the Poisson centre of $\C[Y_0(\Delta)]$ is trivial. Hence $J= I_{Z(\Delta)}$.
\end{proof}


\subsection{Super Variety}\label{sec:super}

Similarly to \cite{AKM23}, we also define a super version of the hypertoric variety $Y_\delta(\Delta)$. To this end, consider $\bar{\mathcal{O}}_{\mathfrak{X}}\coloneqq\mathcal{O}_{\mathfrak{X}}\otimes\C[\Pi T^*V]$ where the coordinate ring $\C[\Pi T^*V]\cong\Lambda(\psi_1,\dots,\psi_N,\phi_1,\dots,\phi_N)$ is a Poisson superalgebra, as defined in \autoref{sec:freefield}. Then $\bar{\mathcal{O}}_{\mathfrak{X}}$ is a sheaf of Poisson superalgebras on $\mathfrak{X}$.

We consider the comoment map $\bar\mu^*\colon\g\to\bar{\mathcal{O}}_{\mathfrak{X}}(\mathfrak{X})$ defined by
\begin{equation*}
a_i\mapsto\mu^*(a_i)+\sum_{j=1}^N\Delta_{ij}\psi_j\phi_j=\sum_{j=1}^N\Delta_{ij}(x_jy_j+\psi_j\phi_j)
\end{equation*}
for $i=1,\dots,M$. The comoment map $\bar\mu^*$ induces an action of $\g$ on $\bar{\mathcal{O}}_{\mathfrak{X}}$ and the corresponding $G$-action.
Define
\begin{equation*}
\bar{\mathcal{O}}_{Y_\delta(\Delta)}=\Bigl(p_{*}\bigl(\bar{\mathcal{O}}_{\mathfrak{X}}\bigm/\sum_{i=1}^{M}\bar{\mathcal{O}}_{\mathfrak{X}}\bar\mu^*(a_i)\bigr)\Bigr)^\g
\end{equation*}
with, as before, the projection $\smash{p\colon\mu^{-1}(0)\cap\mathfrak{X}\to Y_\delta(\Delta)}$. Then $\smash{\bar{\mathcal{O}}_{Y_\delta(\Delta)}}$ is a sheaf of superalgebras on $Y_\delta(\Delta)$. The Poisson bracket on $\bar{\mathcal{O}}_{\mathfrak{X}}$ induces a Poisson bracket on $\bar{\mathcal{O}}_{Y_\delta(\Delta)}$.


\subsection{Local Trivialisation}\label{sec:localtrivial}

Briefly, we recall local trivialisations of the Hamiltonian $G$-actions on $\mathfrak{X}$ and $\mathfrak{X}\times\g^*$. We refer to \cite{Kuw21} for more details.

There is an affine open covering $\smash{\mathfrak{X}=\bigcup_J\mathfrak{U}_J}$ with $\mathfrak{U}_J\cong T^*\C^{N-M}\times T^*G$ of the semistable locus $\mathfrak{X}$ that trivialises the $G$-action and the moment map $\mu$. That is, if we set $U_J=(\mu^{-1}(0)\cap\mathfrak{U}_J)/G\subset Y_\delta(\Delta)$, we obtain an affine open covering $\smash{Y_\delta(\Delta)=\bigcup_JU_J}$ with $U_J\cong T^*\C^{N-M}$.

Similarly, if we consider the Poisson deformation $\smash{\tilde{Y}_\delta(\Delta)}$ of $\smash{Y_\delta(\Delta)}$ and define $\smash{\tilde{\mathfrak{U}}_J=\mathfrak{U}_J\times\g^*\subset\mathfrak{X}\times\g^*}$, we obtain an affine open covering $\smash{\mathfrak{X}\times\g^*=\bigcup_J\tilde{\mathfrak{U}}_J}$ with $\tilde{\mathfrak{U}}_J\cong T^*\C^{N-M}\times T^*G\times\g^*$. As $G$ acts trivially on $\g^*$, the $\smash{\tilde{\mathfrak{U}}_J}$ are also preserved by the $G$-action and we set $\smash{\tilde{U}_J=(\tilde\mu^{-1}(0)\cap\tilde{\mathfrak{U}}_J)/G\subset\tilde{Y}_\delta(\Delta)}$, which yields an affine open covering $\smash{\tilde{Y}_\delta(\Delta)=\bigcup_J\tilde{U}_J}$ with $\tilde{U}_J\cong T^*\C^{N-M}\times\g^*$.


\subsection{Example: Minimal Nilpotent Orbit Closure for \texorpdfstring{$\sl_N$}{sl\_N}}\label{sec:var1}

In the following two sections, we consider two (symplectic dual) examples of hypertoric varieties. In fact, they are examples of quiver hypertoric varieties, which we describe in more detail in \autoref{sec:quiver}.\footnote{For convenience, we shall work in the language of the hypertoric quiver varieties $Y_\delta(Q)$ associated with the quiver $Q$ rather than that of the Nakajima quiver varieties associated with $\mathcal{Q}$ or $\mathcal{Q}^\#$. In particular, it will be convenient not to distinguish between usual and framing edges.}

\medskip

For the first example, we consider the quiver in \autoref{fig:ex1} with $M=1$ and for some $N\geq1$. This case is discussed in detail in Example~2.20 in \cite{Nag21}. 

As all edges in $E$ are of the form $e_{01}^k$, we simply denote them by $e^k\coloneqq e_{01}^k$ for $k=1,\dots,N$, as already indicated in \autoref{fig:ex1}:
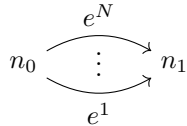
\begin{equation*}
\begin{tikzpicture}
\node (1) at (0,0) {$n_0$};
\node (2) at (2,0) {$n_1$};
\draw[<-] (2) to [bend right] node[above] {$e^N$} (1);
\path (2) -- (1) node [midway] {$\raisebox{+1.5ex}{\vdots}$};
\draw[<-] (2) to [bend left] node[below] {$e^1$} (1);
\end{tikzpicture}
\end{equation*}
The corresponding $(1\times N)$-matrix $\Delta$ and a choice of the $(N\times(N-1))$-matrix $E$ in the short exact sequence \eqref{eq:SES} are given by
\begin{equation*}
\Delta=(1,\dots,1)\quad\text{and}\quad E=\begin{pmatrix}1\\-1&1\\
&\ddots&\ddots\\
&&-1&1\\
&&&-1\end{pmatrix},
\end{equation*}
i.e.\ $\Delta$ has entries $\Delta_k\coloneqq\Delta_{1,{}_{01}^k}=1$ for $k=1,\dots,N$.

The symplectic vector space $T^*V\cong\C^{2N}$ has the coordinates $x_k\coloneqq\smash{x_{01}^{(k)}}$ and $y_k\coloneqq\smash{y_{01}^{(k)}}$ for $k=1,\ldots,N$. The action of $G=\C^\times$ on $T^*V$ is given by
\begin{equation*}
t\cdot x_k=tx_k,\quad t\cdot y_k=t^{-1} y_k,
\end{equation*}
for $t\in G$, and the corresponding moment map $\mu\colon T^*V\cong\C^{2N}\to\g^*\cong(\C)^*\cong\C$ is
\begin{equation*}
\mu((\dots,x_k,\dots,y_k,\dots))=\sum_{k=1}^Nx_ky_k
\end{equation*}
and the comoment map $\mu^*\colon\g\cong\C\to\C[T^*V]$
\begin{align*}
a&\mapsto\sum_{k=1}^Nx_ky_k,
\end{align*}
where $\{a\}$ is the standard basis of $\g\cong\C$.

We fix an effective stability parameter $\delta>0$ in $\Q$ that is generic. For $N=1$, both $Y_\delta(\Delta)$ and $Y_0(\Delta)$ are just a point, so we assume in the following that $N\geq2$. The (singular) affine hypertoric variety for the quiver $Q$ in \autoref{fig:ex1} is
\begin{equation*}
Y_0(\Delta)=Y_0(Q)\cong\{Z\in\sl_N\mid\rk(Z)\leq1\}=\overline{\mathbb{O}_\mathrm{min}(\sl_N)},
\end{equation*}
the minimal nilpotent orbit closure of $\sl_N$. The condition on the rank is equivalent to all $(2\times 2)$-minors vanishing. Recall that the minimal nilpotent orbit of $\sl_N$ is
\begin{equation*}
\mathbb{O}_\mathrm{min}(\sl_N)=\{Z\in\sl_N\mid\rk(Z)=1\}=\{Z\in \sl_N \mid Z^2=0\}.
\end{equation*}
Indeed, the coordinate ring of $Y_0(\Delta)$ is isomorphic to that of $\overline{\mathbb{O}_\mathrm{min}(\sl_N)}$,
\begin{align*}
\C[Y_0(\Delta)]&=\C[\mu^{-1}(0)]^G=\frac{\C[\{x_iy_j\}_{i,j=1,\dots,N}]}{\bigl(\sum_{k=1}^Nx_ky_k\bigr)}\\
&\cong\frac{\C[\{z_{ij}\}_{i,j=1,\dots,N}]}{\bigl(\sum_{k=1}^Nz_{kk},\{z_{ij}z_{kl}-z_{il}z_{kj}\mid 1\leq i,j,k,l\leq N\}\bigr)}=\C[\overline{\mathbb{O}_\mathrm{min}(\sl_N)}]
\end{align*}
with the identification $z_{ij}\mapsto x_iy_j$ for $i,j=1,\dots,N$, where $z_{ij}$ is the coordinate function in $\C[\gl_N]=\C[\{z_{ij}\}_{i,j=1,\dots,N}]$ corresponding to the $(i,j)$-th matrix element.

In view of \autoref{prop:ring-gen}, we point out that the finite set of generators of the coordinate ring $\C[Y_0(\Delta)]$ described there can be chosen as
\begin{equation*}
J_i=x_iy_i-x_{i+1}y_{i+1}\quad\text{for}\quad i=1,\dots,N-1
\end{equation*}
and
\begin{equation*}
W^{a(i,j)}=x_iy_j\quad\text{for}\quad i,j=1,\dots,N\text{ with }i\neq j
\end{equation*}
where $a(i,j)=(0,\dots,0,1,\dots,1,0,\dots,0)^t\in\Z^{N-1}$ with the first $1$ in the $i$-th position and the last in the $(j-1)$-st position if $i<j$ and similarly $a(i,j)=(0,\dots,0,-1,\dots,-1,0,\dots,0)^t\in\Z^{N-1}$ with the first $1$ in the $j$-th position and the last in the $(i-1)$-st position if $i>j$.

On the other hand, the smooth hypertoric quiver variety for the quiver $Q$ is the cotangent bundle of projective space
\begin{equation*}
Y_\delta(\Delta)=Y_\delta(Q)=\bigl\{(x,y)\bigm| x\neq0,\,\textstyle\sum_{k=1}^Nx_ky_k=0\bigr\}\big/\C^\times\cong T^*\mathbb{P}^{N-1},
\end{equation*}
of dimension $2(N-1)$, where we wrote $x=(x_1,\dots,x_N)$ and $y=(y_1,\dots,y_N)$.

The symplectic resolution
\begin{equation*}
\pi\colon Y_\delta(\Delta)\to Y_0(\Delta),
\end{equation*}
is concretely given by
\begin{equation*}
\pi(x,y)=\begin{pmatrix}
x_1y_1&\cdots&x_1y_M\\
\vdots&\ddots&\vdots\\
x_My_1&\cdots&x_My_M
\end{pmatrix}.
\end{equation*}
Under the isomorphisms $Y_\delta(\Delta)\cong T^*\mathbb{P}^{N-1}$ and $Y_0(\Delta)\cong \overline{\mathbb{O}_\mathrm{min}(\sl_N)}$, this recovers the well-known Springer resolution
\begin{equation*}
T^*\mathbb{P}^{N-1}\to\overline{\mathbb{O}_\mathrm{min}(\sl_N)}.
\end{equation*}


\subsection{Example: Kleinian Singularity of Type \texorpdfstring{$A_{N-1}$}{A\_(N-1)}}\label{sec:var2}

As second example, we consider the hypertoric quiver variety $Y_\delta(Q)$ associated with the quiver $Q$ in \autoref{fig:ex2} for some $N\geq2$ and with $M=N-1$. This is discussed, e.g., in Lemma~10.2 of \cite{HS02}, and in Example~2.19 of \cite{Nag21}. We label the edges $e_1,\dots,e_N$ as follows:
\begin{equation*}
\begin{tikzpicture}
\node (1) at (4,1) {$n_0$};
\node (2) at (2,0) {$n_1$};
\node (3) at (4,0) {$n_2$};
\node (4) at (6,0) {$n_{N-1}$};
\draw[<-] (2) to node[above] {$e_1$} (1);
\draw[<-] (3) to node[below] {$e_2$} (2);
\draw[<-] (4) -- (3) node [midway, fill=white] {$\dots$};
\draw[->] (4) to node[above] {$e_N$} (1);
\end{tikzpicture}
\end{equation*}
That is, we set $e_i\coloneqq e_{i-1,i}^1$ for $i=1,\dots,N$, reducing the lower right index modulo~$N$ for $i=N$. Note that, in contrast to \autoref{fig:ex2}, we reversed the orientation of the edge $e_N$ to obtain a more symmetric description. As we mentioned before, this does not affect the variety $Y_\delta(Q)$ up to isomorphism. (In \autoref{sec:quiver}, we assume that the node $n_0$ of $Q$ is a source to simplify the description as a Nakajima quiver variety. However, this is not relevant here.)

Then the $((N-1)\times N)$-matrix $\Delta$ and a choice of the $(N\times1)$-matrix $E$ in the short exact sequence \eqref{eq:SES} are given by
\begin{equation*}
\Delta^t=\begin{pmatrix}1\\-1&1\\
&\ddots&\ddots\\
&&-1&1\\
&&&-1\end{pmatrix}
\quad\text{and}\quad E^t=(1,\dots,1),
\end{equation*}
i.e.
\begin{equation*}
\Delta_{\ell i}=\delta_{\ell i}-\delta_{\ell,i-1}
\end{equation*}
for $\ell=1,\dots,N-1$ and $i=1,\dots,N$. The symplectic vector space $T^*V\cong\C^{2N}$ has coordinates $x_i\coloneqq\smash{x_{i-1,i}^{(1)}}$ and $y_i\coloneqq\smash{y_{i-1,i}^{(1)}}$ for $i=1,\dots,N$. The action of $G=(\C^\times)^{N-1}$ on $T^*V\cong\C^{2N}$ is
\begin{align*}
(t_1,\dots,t_{N-1})\cdot x_i&=t_{i-1}^{-1}t_ix_i,&(t_1,\dots,t_{N-1})\cdot y_i&=t_{i-1}t_i^{-1}y_i\\
\intertext{for $i=2,\dots,N-1$ and}
(t_1,\dots,t_{N-1})\cdot x_1&=t_1x_1,&(t_1,\dots,t_{N-1})\cdot y_1&=t_1^{-1}y_1,\\
(t_1,\dots,t_{N-1})\cdot x_N&=t_{N-1}^{-1}x_N,&(t_1,\dots,t_{N-1})\cdot y_N&=t_{N-1}y_N
\end{align*}
for $(t_1,\dots,t_{N-1})\in G$. The moment map $\mu\colon T^*V\cong\C^{2N}\to\g^*\cong(\C^{N-1})^*\cong\C^{N-1}$ is given by
\begin{equation*}
\mu((\dots,x_i,\dots,y_i,\dots))=\bigl(x_\ell y_\ell-x_{\ell+1}y_{\ell+1}
\bigr)_{1\leq\ell\leq N-1},
\end{equation*}
and the comoment map $\mu^*\colon\g\cong\C^{N-1}\to\C[T^*V]$ is
\begin{equation*}
a_\ell\mapsto x_\ell y_\ell-x_{\ell+1}y_{\ell+1},
\end{equation*}
where $\{a_1,\dots,a_{N-1}\}$ is the standard basis of $\g\cong\C^{N-1}$.

We fix a generic effective stability parameter $\delta\in\Z^{N-1}$, like $\delta=(1,\dots,1)^t$, for instance. The (singular) affine hypertoric variety can be identified as
\begin{equation*}
Y_0(\Delta)=Y_0(Q)\cong\C^2/{(\Z/N\Z)}\eqqcolon X_{A_{N-1}},
\end{equation*}
the Kleinian or du Val singularity of type $A_{N-1}$. Indeed, the coordinate ring of $Y_0(\Delta)$ is
\begin{align*}
\C[Y_0(\Delta)]&=\C[\mu^{-1}(0)]^G=\frac{\C[x_1\dots x_N,y_1\dots y_N,x_1y_1,\dots,x_Ny_N]}{(x_1y_1-x_2y_2,\dots,x_{N-1}y_{N-1}-x_Ny_N)}\\
&\cong\frac{\C[e,f,h]}{(ef-h^N)}\cong\C[x^N,y^N,xy]=\C[x,y]^{\Z/N\Z}=\C[X_{A_{N-1}}]
\end{align*}
with the identifications
\begin{equation*}
e\mapsto x_1\dots x_N,\quad f\mapsto y_1\dots y_N\quad\text{and}\quad h\mapsto x_1y_1\text{ (or any $x_iy_i$)}
\end{equation*}
and
\begin{equation*}
e\mapsto x^N,\quad f\mapsto y^N\quad\text{and}\quad h\mapsto xy
\end{equation*}
and with the action of $\Z/N\Z=\langle g\rangle$ on $\C^2=\Spec\C[x,y]$, $g\cdot(x,y)=(\xi x,\xi^{-1} y)$, where $\xi$ is a primitive $N$-th root of unity.

The finite set of generators of the coordinate ring $\C[Y_0(\Delta)]$ described in \autoref{prop:ring-gen} is given by
\begin{equation*}
J=x_1y_1+\dots+x_Ny_N
\end{equation*}
and
\begin{equation*}
W^1=x_1\dots x_N\quad\text{and}\quad W^{-1}=y_1\dots y_N,
\end{equation*}
corresponding to $Nh$, $e$ and $f$, respectively. Here, $a$ just takes values in $\Z$, and $a=\pm1$ suffice to generate the whole coordinate ring.

One can show that the smooth hypertoric quiver variety
\begin{equation*}
Y_\delta(\Delta)=Y_\delta(Q)\cong\widetilde{X}_{A_{N-1}}
\end{equation*}
gives the minimal resolution
\begin{equation*}
\widetilde{X}_{A_{N-1}}\to X_{A_{N-1}},
\end{equation*}
meaning that any other resolution of $X_{A_{N-1}}$ factors through it. It has dimension~$2$.


\subsubsection*{Special Case \texorpdfstring{$N=2$}{N=2}}

We mention that for $N=2$ the two constructions from \autoref{sec:var1} and from this section, \autoref{sec:var2}, give isomorphic symplectic varieties of dimension~$2$. Indeed, the underlying quivers, say $Q_1$ and $Q_2$, respectively, are identical up to a different choice of orientation:
\begin{equation*}
\begin{tikzpicture}
\node (1) at (0,0) {$n_0$};
\node (2) at (2,0) {$n_1$};
\draw[<-] (2) to [bend right] node[above] {$e^2$} (1);
\draw[<-] (2) to [bend left] node[below] {$e^1$} (1);
\node (3) at (5,0.5) {$n_0$};
\node (4) at (5,-0.5) {$n_1$};
\draw[<-] (4) to [bend right] node[right] {$e_1$} (3);
\draw[->] (4) to [bend left] node[left] {$e_2$} (3);
\node at (-1,0) {$Q_1$:};
\node at (3.5,0) {$Q_2$:};
\end{tikzpicture}
\end{equation*}
The corresponding weight matrices are
\begin{equation*}
\Delta^1=(1,1)\qquad\text{and}\qquad\Delta^2=(1,-1).
\end{equation*}
Indeed, for both quivers, $T^*V\cong\C^4$ with coordinates $x_1,x_2,y_1,y_2$ and symplectic form $dx_1\wedge dy_2+dx_2\wedge dy_2$. The two different actions of $G=\C^\times$ on $T^*V$
\begin{equation*}
t\cdot(tx_1,tx_2,t^{-1}y_1,t^{-1}y_2)\quad\text{and}\quad t\cdot(tx_1,t^{-1}x_2,t^{-1}y_1,ty_2)
\end{equation*}
for $t\in G$ are intertwined by the $G$-equivariant linear isomorphism $T^*V\to T^*V$
\begin{equation*}
(x_2,y_2)\mapsto (-y_2,x_2),
\end{equation*}
also preserving the symplectic form. Correspondingly, the linear isomorphism also intertwines the moment maps $\mu_1,\mu_2\colon T^*V\cong\C^4\to\g^*\cong\C^*\cong\C$,
\begin{equation*}
\mu_1((x_1,x_2,y_1,y_2))=x_1y_1+x_2y_2\quad\text{and}\quad\mu_2((x_1,x_2,y_1,y_2))=x_1y_1-x_2y_2.
\end{equation*}
Hence, the corresponding (singular) affine varieties are isomorphic,
\begin{equation*}
\overline{\mathbb{O}_\mathrm{min}(\sl_2)}\cong Y_0(\Delta^1)\cong Y_0(\Delta^2)\cong\C^2/(\Z/2\Z)=X_{A_1},
\end{equation*}
and the same is true for their resolutions, the hypertoric quiver varieties
\begin{equation*}
T^*\mathbb{P}^{N-1}\cong Y_{\delta_1}(\Delta^1)\cong Y_{\delta_2}(\Delta^2)\cong\widetilde{X}_{A_{N-1}},
\end{equation*}
see Lemma~2.25 in \cite{Nag21} for more details. For example, the above isomorphisms become apparent for the coordinate rings of the singular affine varieties
\begin{align*}
\C[Y_0(\Delta^1)]=\C[\mu_1^{-1}(0)]^G&=\frac{\C[x_1y_1,x_1y_2,x_2y_1,x_2y_2]}{(x_1y_1+x_2y_2)}\cong\C[\overline{\mathbb{O}_\mathrm{min}(\sl_2)}]\\[-6pt]
&\hspace*{6.4em}\rotatebox{270}{$\cong$}\\
\C[Y_0(\Delta^2)]=\C[\mu_2^{-1}(0)]^G&=\frac{\C[x_1x_2,y_1y_2,x_1y_1,x_2,y_2]}{(x_1y_1-x_2y_2)}=\C[x,y]^{\Z/2\Z}=\C[X_{A_1}]
\end{align*}
under $(x_2,y_2)\mapsto (-y_2,x_2)$.


\subsubsection*{Symplectic Duality}

The hypertoric quiver varieties from \autoref{sec:var1} and \autoref{sec:var2} are symplectic or Gale duals; see \autoref{sec:sym-dual}. Indeed, consider the short exact sequence
\begin{equation*}
\{0\}\longrightarrow\Z^{N-1}\overset{E}{\longrightarrow}\Z^N\overset{\Delta}{\longrightarrow}\Z^1\longrightarrow\{0\}
\end{equation*}
with the $(1\times N)$- and $(N\times(N-1))$-matrices
\begin{equation*}
\Delta=(1,\dots,1)\quad\text{and}\quad E=\begin{pmatrix}1\\-1&1\\
&\ddots&\ddots\\
&&-1&1\\
&&&-1\end{pmatrix}.
\end{equation*}
We also define $\Delta\!^!\coloneqq E^t$ and $E^!\coloneqq\Delta^t$. Then, it is apparent from the definitions above that
\begin{equation*}
Y_\delta(\Delta)\cong T^*\mathbb{P}^{N-1}\to\overline{\mathbb{O}_\mathrm{min}(\sl_N)}\quad\text{and}\quad Y_{\Delta\!^!}(\Delta\!^!)\cong\widetilde{X}_{A_{N-1}}\to X_{A_{N-1}}
\end{equation*}
are symplectic dual symplectic resolutions, for generic choices of effective stability parameters $\delta$ and $\Delta\!^!$, as described above. They have dimension $2(N-1)$ and $2$, respectively.


\section{\SVOA{}s for Hypertoric Varieties}\label{sec:hypertoricsvoa}

In this section, we consider two families of vertex operator (super)algebras by chiralising the construction of the (singular) affine hypertoric varieties $Y_0(\Delta)$ by Hamiltonian reduction. However, there are anomalies present that prevent the BRST differential from closing. Depending on the choice of anomaly cancellation, we call these \emph{minimal hypertoric \voa{}s} $V_\mathrm{min}(\Delta)$ and \emph{boundary hypertoric \svoa{}s} $V(\Delta)$.

The former were first defined and considered in \cite{Kuw21} and use Heisenberg \voa{}s for anomaly cancellation (see also \cite{NS25,Sas25} for the study of some examples, and \cite{CSYZ25} for the generalisation to nonabelian quivers), while the latter \cite{BF25,BCDN23,Niu23} use free fermions (see also \cite{Yos23,FS24,Sas25} for the study of some examples, and \cite{AKM23} for fermionic anomaly cancellation appearing in a different context).

As a consequence, the boundary hypertoric \svoa{}s are infinite-index conformal extensions of the minimal hypertoric \voa{}s (times a Heisenberg \voa{}). Here, we give a systematic study of these fermionic extensions. Based on this, we shall determine the associated variety of the boundary hypertoric \svoa{}s in \autoref{sec:assvar}. We shall see that they are, in contrast to the minimal hypertoric \voa{}s, chiral quantisations of the underlying affine hypertoric varieties, i.e.\ $X_{V(\Delta)}=Y_0(\Delta)$. This also means that they are quasi-lisse, while the minimal hypertoric \voa{}s in general are not.

Another interesting phenomenon related to these fermionic extensions, at least in the examples that we study, is that the characters of the hypertoric \voa{}s are upgraded from partial (or false) theta functions, which are not modular, to (quasi)modular forms, as we shall discuss in \autoref{sec:chars}.

Examples of the \svoa{} constructions in this section corresponding to the hypertoric varieties studied in \autoref{sec:var1} and \autoref{sec:var2} will be considered in \autoref{sec:examples}.

\medskip

While the \svoa{}s $V(\Delta)$ in this section are analogues (or, more precisely, chiral quantisations) of the coordinate ring $\C[Y_0(\Delta)]$ of the affine and singular hypertoric varieties $Y_0(\Delta)$, in \autoref{sec:sheaves} we shall also present a vertex algebraic analogue of the smooth versions $Y_\delta(\Delta)$. More precisely, we define certain sheaves $\smash{\mathcal{V}_{\delta,\Delta}^\hbar}$ of ($\hbar$-adic) \svoa{}s over the hypertoric varieties $Y_\delta(\Delta)$. By contrast, the sheaves constructed in \cite{Kuw21} using minimal anomaly cancellation live over the universal family $\tilde{Y}_\delta(\Delta)$ of Poisson deformations of the hypertoric variety rather than the variety itself.

In \autoref{sec:global} we compare these two constructions. Classically, the global sections $\mathcal{O}_{Y_\delta(\Delta)}(Y_\delta(\Delta))$ of the structure sheaf $\mathcal{O}_{Y_\delta(\Delta)}$ coincide with the coordinate ring of $Y_0(\Delta)$, $\C[Y_0(\Delta)]\cong\mathcal{O}_{Y_\delta(\Delta)}(Y_\delta(\Delta))$. Using the faithfulness theorem in \cite{ADS26}, we shall prove a vertex algebraic analogue of this statement.


\subsection{Free-Field Vertex Operator Superalgebras}\label{sec:freefield}

For the construction of the hypertoric \svoa{}s by quantum Hamiltonian reduction, we need the Weyl and Clifford vertex operator (super)algebras, also called $\beta\gamma$-system and $bc$-system, respectively.

\medskip

Let $V=\C^N$ for $N\in\N$ be a finite-dimensional vector space. As in \autoref{sec:hypertoric}, consider the cotangent bundle $T^*V=V\oplus V^*$ equipped with the natural symplectic form. Recall that $x_1,\dots,x_N$ and $y_1,\dots,y_N$ are the standard coordinate functions of $T^*V$; they are Darboux coordinates with respect to the symplectic form. The coordinate ring $\C[T^*V]\cong\C[x_1,\dots,x_N,y_1,\dots,y_N]$ is a Poisson algebra with Poisson bracket $\{y_i,x_j\}=\delta_{ij}$ and $\{x_i,x_j\}=\{y_i,y_j\}=0$ for $i,j=1,\dots,N$.

The Weyl vertex algebra or $\beta\gamma$-system $\mathcal{D}^\mathrm{ch}(T^*V)$ associated with $T^*V$ is the vertex algebra strongly generated by the fields $x_i(z)$ and $y_i(z)$ for $i=1,\dots,N$ subject to the operator product expansions
\begin{equation*}
x_i(z)y_j(w)\sim -\frac{\delta_{ij}}{(z-w)}
\end{equation*}
and $x_i(z)x_j(w)\sim 0$ and $y_i(z)y_j(w)\sim 0$. As a vector space,
\begin{equation*}
\mathcal{D}^\mathrm{ch}(T^*V)=\C[\{x_{i(-n)},y_{i(-n)}\mid \substack{i=1,\dots,N\\n\in\Ns}\}].
\end{equation*}

The $\beta\gamma$-system $\mathcal{D}^{\mathrm{ch}}(T^*V)$ admits an $N$-parameter family of conformal structures of central charge $\smash{c=\sum_{i=1}^N2(1-6a_i+6a_i^2)}$ defined by the conformal vectors
\begin{equation*}
\omega^{a}=\sum_{i=1}^N\bigl(a_i\partial x_i\,y_i-(1-a_i)x_i\partial y_i\bigr)
\end{equation*}
indexed by $a\in\C^N$. Each endows $\mathcal{D}^{\mathrm{ch}}(T^*V)$ with the structure of a (generalised) conformal vertex algebra. With respect to this conformal structure, the strong generators have $L_0$-weights
\begin{equation*}
\wt(x_i)=1-a_i\quad\text{and}\quad\wt(y_i)=a_i
\end{equation*}
for $i=1,\dots,N$.

In \autoref{sec:conf-vect}, we shall discuss what choices of the conformal vectors lie in the kernel of the BRST differential and hence descend to conformal structures on the BRST cohomology. One of these structures is the \emph{natural} conformal structure, defined by
\begin{align*}
\omega=\omega^{1/2}=\sum_{i=1}^N(\partial x_i\,y_i-x_i\partial y_i)/2
\end{align*}
for $a=(1/2,\dots,1/2)^t$. For this choice, $\mathcal{D}^{\mathrm{ch}}(T^*V)$ is a (generalised) \voa{} that is $\frac{1}{2}\N$-graded by $L_0$-weights and of central charge $c=-N$. The strong generators $x_i$ and $y_i$ all have $L_0$-weights $1/2$ in this case. If not otherwise noted, this is the conformal structure that we consider.

On the other hand, to prove certain technical results, we shall also make use of the \emph{auxiliary} conformal structures defined by the conformal vectors $\omega^1$ and $\omega^0$ for $a=(1,\dots,1)^t$ and $a=(0,\dots,0)^t$, respectively.

\medskip

Denote by $\Pi T^*V$ the odd vector space corresponding to $T^*V$. Let $\psi_1,\dots,\psi_N$ and $\phi_1,\dots,\phi_N$ denote the standard Poisson coordinates on $\Pi T^*V$, i.e.\ the coordinate ring $\C[\Pi T^*V]\cong\Lambda(\psi_1,\dots,\psi_N,\phi_1,\dots,\phi_N)$ is a Poisson superalgebra with Poisson bracket $\{\psi_i,\phi_j\}=\delta_{ij}$ and $\{\psi_i,\psi_j\}=\{\phi_i,\phi_j\}=0$ for $i,j=1,\dots,N$.

The Clifford vertex superalgebra or $bc$-system $\mathcal{C}\ell(\Pi T^*V)$ associated with $\Pi T^*V$ is the vertex superalgebra strongly generated by the odd fields $\psi_i(z)$ and $\phi_i(z)$ for $i=1,\dots,N$ satisfying the operator product expansions
\begin{equation*}
\psi_i(z)\phi_j(w)\sim\frac{\delta_{ij}}{(z-w)}
\end{equation*}
and $\psi_i(z)\psi_j(w)\sim 0$ and $\phi_i(z)\phi_j(w)\sim 0$. As a super vector space,
\begin{equation*}
\mathcal{C}\ell(\Pi T^*V)=\Lambda(\{\psi_{i(-n)},\phi_{i(-n)}\mid\substack{i=1,\dots,N\\n\in\Ns}\}).
\end{equation*}

Also the Clifford vertex superalgebra $\mathcal{C}\ell(\Pi T^*V)$ admits an $N$-parameter family of conformal structures of central charge $\smash{c=\sum_{i=1}^N2(-1+6b_i-6b_i^2)}$ defined by the conformal vectors
\begin{equation*}
\omega^b=\sum_{i=1}^N\bigl(b_i\partial\psi_i\,\phi_i-(1-b_i)\psi_i\partial\phi_i\bigr)
\end{equation*}
for $b\in\C^N$. Each endows $\mathcal{C}\ell(\Pi T^*V)$ with the structure of a (generalised) \svoa{}. With respect to this conformal structure, the strong generators have $L_0$-weights
\begin{equation*}
\wt(\psi_i)=1-b_i\quad\text{and}\quad\wt(\phi_i)=b_i
\end{equation*}
for $i=1,\dots,N$.

We shall discuss in \autoref{sec:conf-vect} what choices of the conformal vectors are good in the sense that they are in the kernel of the BRST cohomology. Among them is the \emph{natural} conformal structure defined by the conformal vector
\begin{equation*}
\omega=\omega^{1/2}=\sum_{i=1}^N(\partial\psi_i\,\phi_i-\psi_i\partial\phi_i)/2
\end{equation*}
for $b=(1/2,\dots,1/2)^t$. This makes $\mathcal{C}\ell(\Pi T^*V)$ a \svoa{} (of ``correct statistics'') of central charge $c=N$, where the strong generators $\psi_i$ and $\phi_i$ both have $L_0$-weight $1/2$.  If not mentioned otherwise, this is the conformal structure that we consider.

Again, to prove certain technical results, we shall also make use of the \emph{auxiliary} conformal structures defined by the conformal vectors $\omega^1$ and $\omega^0$ for $b=(1,\dots,1)^t$ and $b=(0,\dots,0)^t$, respectively.

The \svoa{} $\mathcal{C}\ell(\Pi T^*V)\cong V_{\Z^N}$ is isomorphic to the lattice \svoa{} $V_{\Z^N}$ (see also below) associated with the rank-$N$ odd \emph{standard lattice} $\Z^N=\langle\tau_1,\dots,\tau_N \rangle_\Z$ with bilinear form $\langle\tau_i,\tau_j \rangle=\delta_{ij}$, equipped with the standard conformal vector for a lattice vertex operator (super)algebra so that the central charge is $c=\rk(\Z^N)=N$. To make this isomorphism explicit, if we consider the lattice \svoa{} construction (see, e.g., \cite{Kac98}) involving the twisted group algebra $\C_\eps[\Z^N]$ with the explicit choice of $2$-cocycle such that $\eps(\tau_i,\tau_i)=1$, then $\psi_i\mapsto e_{\tau_i}$ and $\phi_i\mapsto e_{-\tau_i}$ (implying that $\tau_i=\psi_i\phi_i$) is a way to define this isomorphism.

Throughout, we identify $\mathcal{C}\ell(\Pi T^*V)=V_{\Z^N}$ via $\psi_i=e_{\tau_i}$ and $\phi_i=e_{-\tau_i}$. The conformal vector satisfies $\smash{\omega=\sum_{i=1}^N(\partial\psi_i\,\phi_i-\psi_i\partial\phi_i)/2=\sum_{i=1}^N\tau_i^2/2}$.

\medskip

To construct the BRST cohomology, we shall also need the \emph{ghost} Clifford \svoa{}. It only differs from the above \svoa{} by the choice of conformal vector. However, it shall be convenient to keep the notation somewhat separate. As in \autoref{sec:hypertoric}, consider the abelian Lie algebra $\g=(\gl_1)^M\cong\C^M$ for $M\in\N$ and the corresponding super vector space $\Pi T^*\g$ with standard coordinate functions $c_i\in\g^*\subset\C[\g]$ and $b_i\in\g\subset\C[\g^*]$. Then $\mathcal{C}\ell(\Pi T^*\g)=\Lambda(\{c_{i(-n)},b_{i(-n)}\mid\substack{i=1,\dots,M\\n\in\Ns}\})$ is strongly generated by the odd fields $c_i(z)$ and $b_i(z)$ for $i=1,\dots,M$ subject to the operator product expansions
\begin{equation*}
c_i(z)b_j(w)\sim\frac{\delta_{ij}}{(z-w)}
\end{equation*}
and $c_i(z)c_j(w)\sim 0$ and $b_i(z)b_j(w)\sim 0$. On the ghost Clifford vertex superalgebra we make the choice
\begin{equation*}
\omega=\sum_{i=1}^M\partial c_i\, b_i
\end{equation*}
of conformal vector, making it into a \svoa{} ($\Z$-graded by $L_0$-weights) of central charge $c=-2M$. The $L_0$-weights of $c_i$ and $b_i$ are $0$ and $1$, respectively.

By setting $\deg(c_i)=1$ and $\deg(b_i)=-1$, we also equip $\mathcal{C}\ell(\Pi T^*\g)$ with the ghost (or cohomological) grading, making $\mathcal{C}\ell(\Pi T^*\g)=\bigoplus_{\bullet\in\Z}\mathcal{C}\ell^\bullet(\Pi T^*\g)$ a $\Z$-graded \svoa{}.

\medskip

Finally, we introduce the notion of generalised half-lattice vertex (super)algebras. These interpolate between Heisenberg \voa{}s and lattice vertex (super)algebras (like $V_{\Z^N}$ mentioned above), which are infinite-order simple-current extensions of the former.

Let $\h=(\h,\langle\cdot,\cdot\rangle)$ be a finite-dimensional $\C$-vector space equipped with a nondegenerate,\footnote{Below, in \autoref{sec:hypertoricSVOA}, we also consider a Heisenberg vertex algebra equipped with the zero bilinear form. In that case, it becomes a commutative vertex algebra.} symmetric bilinear form $\langle\cdot,\cdot\rangle$. Associated with $\h$ is the usual Heisenberg (or free boson) \voa{} $\pi^\h$. Given a basis $\{h_1,\dots,h_N\}$ of $\h$, $\pi^\h$ is the vertex algebra strongly generated by the fields $h_i(x)$ for $i=1,\dots,N$ subject to the operator product expansions
\begin{equation*}
h_i(z)h_j(w)\sim \frac{\langle h_i,h_j\rangle}{(z-w)^2}.
\end{equation*}
As a vector space,
\begin{equation*}
\pi^\h=\C[\{h_{i(-n)}\mid \substack{i=1,\dots,N\\n\in\Ns}\}].
\end{equation*}
With the standard choice of conformal vector $\omega=\sum_{i=1}^Nh_i^2/2$, it has central charge $c=\dim(\h)$, but we may also consider shifted conformal structures. It has the irreducible Fock modules $\smash{\pi_\lambda^\h}$ for all $\lambda\in\h$.

Then, let $K\subset\h$ be a (rational) lattice inside $\h$, i.e.\ a free $\Z$-module such that the restriction of the bilinear form to $K$ takes values in $\Q$. Crucially, we allow the rank of $K$ to be smaller than $\dim(\h)$ and the restriction of the bilinear form to $K$ to be degenerate. We further assume that $K$ is even, i.e.\ that $\langle\lambda,\lambda\rangle\in2\Z$ for all $\lambda\in K$, or at least that $K$ is integral, i.e.\ that $\langle\lambda,\mu\rangle\in\Z$ for all $\lambda,\mu\in K$. Then the direct sum of Fock modules for $\pi^\h$
\begin{equation*}
\HL_K^\h\coloneqq\bigoplus_{\lambda\in K}\pi_\lambda^\h,
\end{equation*}
an infinite simple-current extension of $\pi^\h$ (see, e.g., \cite{GLM24a}), can be endowed with the (unique up to isomorphism) structure of a (simple) conformal vertex algebra if $K$ is even, or that of a conformal vertex superalgebra if $K$ is odd. We call $\HL_K^\h$ the \emph{Heisenberg-lattice vertex (super)algebra} or \emph{generalised half-lattice vertex (super)algebra}.

In the special case when $K=\{0\}$ is the zero lattice, this recovers the Heisenberg \voa{} $\smash{\HL_{\{0\}}^\h=\pi^\h}$ itself. In the special case when $K$ has full rank $\rk(K)=\dim(\h)$, i.e.\ when $\h=K\otimes_\Z\C$, we obtain the usual lattice vertex (super)algebra $\HL_K^\h=V_K$. We remark that, in general, a Heisenberg-lattice vertex (super)algebra cannot be decomposed into a tensor product of a lattice and a Heisenberg vertex (super)algebra. To understand this, consider the example of $\h=\C^2$ with basis $\{x,y\}$ and bilinear form matrix $\diag(1,-1)$, and choose the isotropic lattice $K=\Z(x+y)$. In fact, this recovers the usual half-lattice vertex algebra $\Pi$ \cite{BDT02}.

Given a half-lattice vertex (super)algebra $\HL_K^\h$, any coset $\gamma+K$ for $\gamma\in\h$ defines the generalised Fock module
\begin{equation*}
\HL_{\gamma+K}^\h\coloneqq\bigoplus_{\lambda\in\gamma+K}\pi_\lambda^\h,
\end{equation*}
which is an irreducible module for $\HL_K^\h$. We note that $\gamma$ does not necessarily have to lie in $K\otimes_\Z\C$, which can be a proper subspace of $\h$.


\subsection{Minimal and Boundary Hypertoric \SVOA{}s}\label{sec:hypertoricSVOA}

In the following, we define the two version of the hypertoric vertex operator (super)algebras by quantum Hamiltonian reduction of certain free-field vertex operator (super)algebras: the smaller \emph{minimal} hypertoric \voa{}s $V_\mathrm{min}(\Delta)$ and the larger \emph{boundary} hypertoric \svoa{}s $V(\Delta)$.

\medskip

We begin with the construction of the boundary hypertoric \svoa{}s, which are the main objects in this text. We let $M,N\in\Ns$ with $M\leq N$ and $\Delta=(\Delta_{ij})_{1\leq i\leq M,1\leq j\leq N}$ as in \autoref{sec:hypertoric}, in particular \autoref{conv:main}. We consider the ``matter'' free-field \svoa{}
\begin{equation*}
M\coloneqq\mathcal{D}^\mathrm{ch}(T^*V)\otimes\mathcal{C}\ell(\Pi T^*V)
\end{equation*}
of central charge $c=-N+N=0$ for $V\cong\C^N$ and the ``ghost'' \svoa{} $\mathcal{C}\ell(\Pi T^*\g)$ of central charge $c=-2M$ for $\g\cong\C^M$ (with the natural conformal structures defined in the previous section). Then, consider the tensor product \svoa{}
\begin{equation*}
\tilde{C}=M\otimes\mathcal{C}\ell(\Pi T^*\g)
\end{equation*}
of central charge $c=-2M$, which inherits the ghost $\Z$-grading $\tilde{C}=\bigoplus_{\bullet\in\Z}\tilde{C}^\bullet$ from $\mathcal{C}\ell(\Pi T^*\g)$ with $\tilde{C}^\bullet=M\otimes\mathcal{C}\ell^\bullet(\Pi T^*\g)$.

\medskip

In order to define the BRST cohomology, we consider the following $2N$ bosonic fields of $L_0$-weight~$1$ in $M$:
\begin{equation*}
\sigma_i\coloneqq x_iy_i,\quad \tau_i\coloneqq\psi_i\phi_i\quad\text{for }i=1,\dots,N.
\end{equation*}
They satisfy the operator product expansions
\begin{equation*}
\sigma_i(z)\sigma_j(w)\sim-\frac{\delta_{ij}}{(z-w)^2}\quad\text{and}\quad \tau_i(z)\tau_j(w)\sim\frac{\delta_{ij}}{(z-w)^2}
\end{equation*}
and $\sigma_i(z)\tau_j(w)\sim0$ for $i,j=1,\dots,N$. That is, they generate a Heisenberg vertex algebra $\pi^S\otimes\pi^T\subset M$ associated with the $2N$-dimensional vector space with nondegenerate bilinear form
\begin{equation*}
S\oplus T,\quad\diag(-1,\dots,-1,1,\dots,1)
\end{equation*}
with respect to the basis $\{\sigma_1,\dots,\sigma_n\}\cup\{\tau_1,\dots,\tau_N\}$. Below, in \autoref{sec:freefield2}, we shall complement the Heisenberg \voa{} $\pi^S$ by another Heisenberg \voa{} $\pi^R$ (with opposite bilinear form) in order to define a free-field embedding $\mathcal{D}^\mathrm{ch}(T^*V)\hookrightarrow\HL_K^{R\oplus S}$ into a half-lattice \voa{} extending $\pi^R\otimes\pi^S$.

The standard choice of conformal vector $\smash{\sum_{i=1}^N\tau_i^2/2}$ of the fermionic Heisenberg vertex algebra $\pi^T$ coincides with the conformal vector on $\mathcal{C}\ell(\Pi T^*V)$, i.e.\ the latter is a conformal extension of the former. It is an infinite-order simple-current extension, corresponding to the integral lattice $\Z^N$, as described above.

\medskip

Recall the short exact sequence \eqref{eq:SES}, $\{0\}\longrightarrow\Z^{N-M}\overset{E}{\longrightarrow}\Z^N\overset{\Delta}{\longrightarrow}\Z^M\longrightarrow\{0\}$ from \autoref{sec:torus-act}. Because $\Delta E=0$, the columns of $E$ are orthogonal to the rows of $\Delta$ with respect to the standard bilinear form $\diag(1,\dots,1)$ on $\C^N$ (or its negative). Moreover, they span all of $\C^N$. Using this, we make a base change in the bosonic and fermionic Heisenberg fields by defining the following fields
\begin{equation*}
\sigma^\Delta_i\coloneqq\sum_{j=1}^N\Delta_{ij}\sigma_j\quad\text{and}\quad\tau^\Delta_i\coloneqq\sum_{j=1}^N\Delta_{ij}\tau_j
\end{equation*}
for $i=1,\dots,M$, corresponding to the rows of $\Delta$, and
\begin{equation*}
\sigma^E_i\coloneqq\sum_{j=1}^NE_{ji}\sigma_j\quad\text{and}\quad\tau^E_i\coloneqq\sum_{j=1}^NE_{ji}\tau_j
\end{equation*}
for $i=1,\dots,N-M$, corresponding to the columns of $E$. We then define the following pairwise orthogonal spaces:
\begin{align*}
S^\Delta&\coloneqq\langle\sigma_1^\Delta,\dots,\sigma_M^\Delta\rangle_\C,&S^E&\coloneqq\langle\sigma_1^E,\dots,\sigma_{N-M}^E\rangle_\C,\\
T^\Delta&\coloneqq\langle\tau_1^\Delta,\dots,\tau_M^\Delta\rangle_\C,&T^E&\coloneqq\langle\tau_1^E,\dots,\tau_{N-M}^E\rangle_\C.
\end{align*}
(We can identify $T^\Delta$ with the abelian Lie algebra $\g\cong\C^M$; see also \autoref{sec:deformation}.) These spaces are equipped with the nondegenerate bilinear forms $-\Delta\Delta^t$, $-E^tE$, $\Delta\Delta^t$ and $E^tE$, respectively. Hence, the above Heisenberg vertex algebra is of the form
\begin{equation*}
\pi^S\otimes\pi^T\cong\pi^{S^\Delta}\otimes\pi^{S^E}\otimes\pi^{T^\Delta}\otimes\pi^{T^E}\subset M.
\end{equation*}

\smallskip

We introduce a chiralisation of the comoment map $\mu^*\colon\g\to\C[T^*V]$. Recall that the abelian Lie algebra $\g=\gl_1^M\cong\C^M$ with basis $\{a_1,\dots,a_M\}$ is naturally equipped with the zero invariant bilinear form, the Killing form. The corresponding affine vertex algebra is the commutative vertex algebra $V(\g)=\C[\{a_{i(-n)}\mid \substack{i=1,\dots,M\\n\in\Ns}\}]$. It is isomorphic to the rank-$M$ Heisenberg vertex algebra ``of level~$0$'' (i.e.\ with zero bilinear form, cf.\ \autoref{sec:freefield}).

The chiral comoment map is the $\C[\partial]$-linear map $\bar\mu^*_\mathrm{ch}\colon V(\g)\to M$ defined by
\begin{equation*}
\bar\mu^*_\mathrm{ch}(a_i)=\sum_{j=1}^N\Delta_{ij}(x_jy_j+\psi_j\phi_j)=\sum_{j=1}^N\Delta_{ij}(\sigma_j+\tau_j)=\sigma^\Delta_i+\tau^\Delta_i
\end{equation*}
for $i=1,\dots,M$. A direct calculation shows:
\begin{prop}
The chiral comoment map $\bar\mu^*_\mathrm{ch}\colon V(\g)\to M$ is a homomorphism of vertex superalgebras.
\end{prop}
\begin{proof}
We verify the operator product expansions $\bar\mu_\mathrm{ch}(a_i)(z)\bar\mu_\mathrm{ch}(a_j)(w)\sim0$ for $i,j=1,\dots,M$.
\end{proof}

\begin{rem}
The naive chiral comoment map $\smash{a_i\mapsto\sum_{j=1}^N\Delta_{ij}x_jy_j}$ would not lead to a BRST differential $d$ satisfying $d^2=0$. (In other words, the representation it defines does not have the correct level.) It hence must be modified by adding free fermions. This is called \emph{anomaly cancellation} in physics (see, e.g., \cite{CSYZ25} for a discussion in the context of quiver \svoa{}s).

Instead of $M=\mathcal{D}^\mathrm{ch}(T^*V)\otimes\mathcal{C}\ell(\Pi T^*V)$ it suffices to consider $\smash{\mathcal{D}^\mathrm{ch}(T^*V)\otimes\pi^{T^\Delta}}$ (which we call $M_\mathrm{min}$ below), a \voa{} of central charge $c=-N+M$ with the standard conformal vector on $\smash{\pi^{T^\Delta}}$, in order to define the above chiral comoment map. This is exactly the approach taken in \cite{Kuw21} to the define what we call minimal hypertoric \voa{}s $V_\mathrm{min}(\Delta)$. We shall argue in this work that it is beneficial, for several reasons, to consider instead the extension
\begin{equation*}
\pi^{T^\Delta}\subset\pi^{T^\Delta}\otimes\pi^{T^E}\cong\pi^T\subset\mathcal{C}\ell(\Pi T^*V)\cong V_{\Z^N}.
\end{equation*}
This lattice extension will partially survive also after the BRST reduction and will make the resulting \svoa{} $V(\Delta)$ better behaved than $V_\mathrm{min}(\Delta)$.
\end{rem}

\medskip

To define the BRST differential, we consider the odd vector
\begin{equation}\label{eq:Q}
Q\coloneqq\sum_{i=1}^M\bar\mu^*_\mathrm{ch}(a_i)c_i=\sum_{i=1}^M\sum_{j=1}^N\Delta_{ij}(\sigma_j+\tau_j)c_i=\sum_{i=1}^M(\sigma^\Delta_i+\tau^\Delta_i)c_i.
\end{equation}
Then $Q$ has $L_0$-weight~$1$ and ghost degree~$1$ and satisfies $Q_{(0)}Q=0$. We define $d\coloneqq Q_{(0)}$, which satisfies
\begin{equation*}
d^2=(Q_{(0)})^2=(Q_{(0)}Q)_{(0)}/2=0.
\end{equation*}
Moreover, $d$ has ghost degree~$1$ and $L_0$-weight~$0$. That is, $d$ commutes with $L_0$. More generally, one can show that the conformal vector $\omega$ of $\tilde{C}$ is in $\ker(d)$, which, by the Borcherds identity, is equivalent to $[d,L_n]=0$ for all $n\in\Z$. It follows that $(\tilde{C},d)$ is cochain complex of \svoa{}s of central charge $c=-2M$; see \autoref{sec:conf-vect} for details.

It will be advantageous to consider the relative cohomology. To this end, we consider the conformal subalgebra $C\subset\tilde{C}$ given by the kernel under the zero modes of $b_i$ and of $d\, b_i=Q_{(0)}b_i=\bar\mu^*_\mathrm{ch}(a_i)$ for all $i=1,\dots,M$, i.e.
\begin{equation*}
C\coloneqq\bigcap_{i=1}^M\ker(b_{i(0)})\cap\bigcap_{i=1}^M\ker((d\,b_i)_{(0)})=\bar{C}\cap\bigcap_{i=1}^M\ker((d\,b_i)_{(0)})\subset\tilde{C}
\end{equation*}
where $\bar{C}\coloneqq\bigcap_{i=1}^M\ker(b_{i(0)})$ is simply the usual rank-$M$ symplectic fermion \svoa{} strongly generated by $\partial c_i$ and $b_i$ for $i=1,\dots,M$.
\begin{prop}
$(\tilde{C},d)$ is a cochain complex of \svoa{}s of central charge $c=-2M$, and $(C,d)$ is a subcomplex of $(\tilde{C},d)$ with the same conformal structure.
\end{prop}
Again, we refer to \autoref{sec:conf-vect} for details regarding the conformal structure.
\begin{proof}
The first assertion is immediate. For the second assertion we observe that, as $d=Q_{(0)}$ is a derivation,
\begin{align*}
\mu_\mathrm{ch}(A)_{(0)} Q_{(0)} v &= Q_{(0)} \mu_\mathrm{ch}(A)_{(0)} v - (Q_{(0)} \mu_\mathrm{ch}(A))_{(0)} v = (Q_{(0)}^2 b)_{(0)} v = 0,\\
b_{(0)} Q_{(0)} v &= - Q_{(0)} b_{(0)} v + (Q_{(0)} b)_{(0)} v = \mu_\mathrm{ch}(A)_{(0)} v = 0
\end{align*}
for any $v\in C$.
\end{proof}

The following is the main object of this text:
\begin{defi}[Boundary Hypertoric \SVOA{}]\label{defi:svoa}
We consider the relative cohomology
\begin{equation*}
V(\Delta)\coloneqq H_\mathrm{BRST}^{\infty/2+\bullet}(\g,M)\coloneqq H^\bullet(C,d).
\end{equation*}
It follows from standard arguments (see, e.g., \cite{FBZ04}) that $V(\Delta)$ is a ($\Z$-graded by cohomological or ghost degree) \svoa{} of central charge $c=-2M$. We shall see in \autoref{prop:vanish} below that only the zeroth degree of $V(\Delta)$ is nonzero. We call $V(\Delta)=H^\bullet(C,d)=H^0(C,d)$ the \emph{boundary hypertoric \svoa{}}.
\end{defi}

\smallskip

As mentioned above, in order to make the connection to \cite{Kuw21}, we also consider the minimal BRST reduction with $\mathcal{C}\ell(\Pi T^*V)$ replaced by the Heisenberg vertex subalgebra $\smash{\pi^{T^\Delta}}$. That is, we replace $M=\mathcal{D}^\mathrm{ch}(T^*V)\otimes\mathcal{C}\ell(\Pi T^*V)$ by
\begin{equation*}
M_\mathrm{min}\coloneqq\mathcal{D}^\mathrm{ch}(T^*V)\otimes\pi^{T^\Delta},
\end{equation*}
$\tilde{C}=M\otimes\mathcal{C}\ell(\Pi T^*\g)$ by $\tilde{C}_\mathrm{min}\coloneqq M_\mathrm{min}\otimes\mathcal{C}\ell(\Pi T^*\g)$ and $C$ by $C_\mathrm{min}=C\cap\tilde{C}_\mathrm{min}$, all \voa{}s of central charge $c=-N-M$. We obtain the cochain complex $(\tilde{C}_\mathrm{min},d)$ and the subcomplex $(C_\mathrm{min},d)$.

\begin{defi}[Minimal Hypertoric \SVOA{}]
We define
\begin{equation*}
V_\mathrm{min}(\Delta)\coloneqq H_\mathrm{BRST}^{\infty/2+\bullet}(\g,M_\mathrm{min})\coloneqq H^\bullet(C_\mathrm{min},d),
\end{equation*}
a ($\Z$-graded by cohomological or ghost degree) \voa{} of central charge $c=-N-M$, which we call \emph{minimal hypertoric \voa{}}.
\end{defi}
 Also $V_\mathrm{min}(\Delta)=H^\bullet(C_\mathrm{min},d)=H^0(C_\mathrm{min},d)$ satisfies a vanishing of nonzero cohomologies. In \cite{Kuw21}, this \voa{} is simply called hypertoric \voa{}. We shall further explain how to describe $V(\Delta)$ as a (conformal) lattice extension of $\smash{V_\mathrm{min}(\Delta)\otimes\pi^{T^E}}$, both vertex operator (super)algebras of central charge $c=-2M$.


\subsection{Vanishing Theorem}\label{sec:vanish}

In this section, we study the minimal and hypertoric vertex operator (super)algebras $V_\mathrm{min}(\Delta)$ and $V(\Delta)$ concretely, i.e.\ we compute the relative BRST cohomologies. Along the way, we also prove a vanishing theorem for nonzero cohomologies, i.e.\ $H^{\infty/2+n}(\g,M_\mathrm{min})=\{0\}$ and $H^{\infty/2+n}(\g,M)=\{0\}$ for all $n\neq0$ \cite{Vor94} (see also \cite{BF25}).


\subsubsection{Free-Field Embedding}\label{sec:freefield2}

In order to describe the precise shape of the hypertoric vertex operator (super)algebras $V(\Delta)$ and $V_\mathrm{min}(\Delta)$, it is advantageous to embed the $\beta\gamma$-system into a Heisenberg-lattice vertex algebra
\begin{equation*}
\mathcal{D}^\mathrm{ch}(T^*V)\hookrightarrow\HL_K^{R\oplus S}
\end{equation*}
where $R\oplus S$ is the $2N$-dimensional vector space spanned by $\{\rho_i,\sigma_i\}_{i=1}^N$ equipped with the indefinite, symmetric bilinear form
\begin{equation*}
\langle\rho_i,\rho_j\rangle=\delta_{ij},\quad\langle\sigma_i,\sigma_j\rangle=-\delta_{ij}\quad\text{and}\quad\langle\rho_i,\sigma_j\rangle=0
\end{equation*}
and $K\subset R\oplus S$ is the isotropic sublattice of rank $N$ given by
\begin{equation*}
K\coloneqq\langle\{-\rho_i+\sigma_i\}_{i=1}^N\rangle_\Z.
\end{equation*}
Then $\mathcal{D}^\mathrm{ch}(T^*V)$ can be obtained as the intersection
\begin{equation*}
\ker(e^{\rho_1}_{(0)})\cap\dots\cap\ker(e^{\rho_N}_{(0)})
\end{equation*}
of kernels of screening operators in $\HL_K^{R\oplus S}$.

\begin{rem}\label{rem:localise}
As we shall explain in more detail in \autoref{sec:freefieldrealisation} below, the extension to $\smash{\HL_K^{R\oplus S}}$ corresponds to localising $\mathcal{D}^\mathrm{ch}(T^*V)$ from $T^*V\cong T^*\C^N$ to $T^*(\C^\times)^N$. In other words, the Heisenberg-lattice \voa{} $\smash{\HL_K^{R\oplus S}}$ is obtained by adjoining the inverse fields $x_i(z)^{-1}$ to the fields $x_i(z)$ and $y_i(z)$ that strongly generate $\mathcal{D}^\mathrm{ch}(T^*V)$ (see \cite{MSV99}).
\end{rem}

Overall, by again identifying $\mathcal{C}\ell(\Pi T^*V)\cong V_{\Z^N}$ with a lattice \svoa{} as in \autoref{sec:freefield} (for the lattice $J\coloneqq\langle\{\tau_i\}_{i=1}^N\rangle_\Z$, which we introduce below), in the boundary case, we obtain an embedding, defined by the screening kernels $\smash{\ker(e^{\rho_1}_{(0)})\cap\dots\cap\ker(e^{\rho_N}_{(0)})}$, of the free-field \svoa{}
\begin{equation*}
M=\mathcal{D}^\mathrm{ch}(T^*V)\otimes\mathcal{C}\ell(\Pi T^*V)\hookrightarrow\HL_L^\h
\end{equation*}
into a Heisenberg-lattice vertex superalgebra of rank $(2N,3N)$. Here, the underlying $3N$-dimensional quadratic space $\h\coloneqq R\oplus S\oplus T$ of the Heisenberg \voa{} is spanned by $\{\rho_i,\sigma_i,\tau_i\}_{i=1}^N$ with the indefinite bilinear form
\begin{equation*}
\langle\rho_i,\rho_j\rangle=\delta_{ij},\qquad\langle\sigma_i,\sigma_j\rangle=-\delta_{ij},\qquad
\langle\tau_i,\tau_j\rangle=\delta_{ij}
\end{equation*}
and the remaining inner products vanishing. Moreover, the (integral) lattice $L$ of rank $2N$ is given by
\begin{equation*}
L\coloneqq\langle\{-\rho_i+\sigma_i,\tau_i\}_{i=1}^N\rangle_\Z.
\end{equation*}
\begin{rem}[Parity]\label{rem:parity}
As the lattice $L$ is in general not even, but integral, it contains vectors $x$ of norm $\langle x,x\rangle/2\in\Z$ and those of norm $\langle x,x\rangle/2\in\Z+1/2$. Fock modules (and the submodules obtained by applying screening kernels) corresponding to the former have even parity, those corresponding to the latter have odd parity. As the $-\rho_i+\sigma_i$ are isotropic, all of this only depends on the $\tau_i$. In other words, Fock modules with momentum $\dots+\sum_{i=1}^N n_i\tau_i$ have parity $(-1)^{n_1+\dots+n_N}$. Cf.\ \autoref{rem:boundaryschar} below.
\end{rem}

\begin{rem}
In the minimal case, we replace $M$ in the input of the relative BRST cohomology by the vertex subalgebra $\smash{M_\mathrm{min}=\mathcal{D}^\mathrm{ch}(T^*V)\otimes\pi^{T^\Delta}}$, but it is sometimes more convenient to work with $\smash{M_\mathrm{min}\otimes\pi^{T^E}=\mathcal{D}^\mathrm{ch}(T^*V)\otimes\pi^T}$ instead. Then the localisation of $\mathcal{D}^\mathrm{ch}(T^*V)$, as described above, yields an embedding
\begin{equation*}
\mathcal{D}^\mathrm{ch}(T^*V)\otimes\pi^T\hookrightarrow \HL_K^\h,
\end{equation*}
with $\h$ of dimension $3N$ and $K$ an isotropic lattice of rank $N$, as above.
\end{rem}

\medskip

The usefulness of the above free-field embedding comes from the fact that it commutes in a certain sense with the BRST reduction. On the other hand, applying the BRST reduction to Fock modules turns out to be easy to compute, thanks to a vanishing result by Voronov \cite{Vor94} applied to the case of abelian Lie algebras $\g$ (see also \cite{BF25}).

We first summarise the result in the boundary case. The details will be given in the remainder of this section.
\begin{prop}[Free-Field Embedding]\label{prop:commute1}
Taking screening kernels and taking relative BRST cohomology commute in the following sense:
\begin{equation*}
\begin{tikzcd}
M\arrow[r,"\text{BRST}"]\arrow[d,hook,"\ker(e^{\rho_i}_{(0)})"]
&V(\Delta)=H_\mathrm{BRST}^{\infty/2+0}(\g,M)\arrow[d,hook,"\ker(e^{p^\bot(\rho_i)}_{(0)})"]\\
\HL_L^\h\arrow[r,"\text{BRST}"]
&\HL_{L^\bot}^{\h^\bot}=H_\mathrm{BRST}^{\infty/2+\bullet}(\g,\HL_L^\h)
\end{tikzcd}
\end{equation*}
\end{prop}
The \svoa{} in the bottom right of the diagram is a certain Heisenberg-lattice vertex superalgebra of rank $(2N-M,3N-2M)$ that we describe below. Crucially, the vanishing result of \cite{Vor94} says that only the zeroth relative cohomology is nonzero and that applying it to $L\subset\h$ results in another Heisenberg-lattice vertex superalgebra obtained by collapsing certain $2M$ directions in the vector space and $M$ in the lattice. The commutativity of the diagram then says that we can determine $V(\Delta)$ by applying the screening kernels $\smash{\ker(e^{p^\bot(\rho_1)}_{(0)})\cap\dots\cap\ker(e^{p^\bot(\rho_N)}_{(0)})}$ to the Heisenberg-lattice vertex superalgebra $\smash{\HL_{L^\bot}^{\h^\bot}}$. Here, the $p^\bot(\rho_i)$ are certain projections of the $\rho_i$, which we describe below. The result will be that we can write $V(\Delta)$ as a conformal extension of a certain tensor product of singlet \voa{}s for $p=2$ \cite{Ada03} and Heisenberg \voa{}s.

Similarly, in the minimal case, one gets:
\begin{prop}[Free-Field Embedding]\label{prop:commute2}
Taking screening kernels and taking relative BRST cohomology commute in the following sense:
\begin{equation*}
\begin{tikzcd}
M_\mathrm{min}\otimes\pi^{T^E}\arrow[r,"\text{BRST}"]\arrow[d,hook,"\ker(e^{\rho_i}_{(0)})"]
&V_\mathrm{min}(\Delta)\otimes\pi^{T^E}=H^{\infty/2+0}(\g,M_\mathrm{min})\otimes\pi^{T^E}\arrow[d,hook,"\ker(e^{p^\bot(\rho_i)}_{(0)})"]\\
\HL_K^\h\arrow[r,"\text{BRST}"]
&\HL_{K^\bot}^{\h^\bot}=H^{\infty/2+0}(\g,\HL_K^\h)
\end{tikzcd}
\end{equation*}
\end{prop}
Here, on the bottom right, we consider a Heisenberg-lattice vertex algebra of rank $(N-M,3N-2M)$. We note that in all four corners, we can split off $\smash{\pi^{T^E}}$ as a tensor factor, which is not affected by the screening operators or the relative BRST cohomology. However, we keep this factor in order to be able to better compare the minimal and the boundary case (so that one is a conformal extension of the other).


\subsubsection{Smallest Example}\label{sec:smallexample}

We first describe the special case of $N=M=1$ with $\Delta=(1)$ of size $1\times 1$ and $E$ being of size $1\times 0$ (cf.\ \autoref{rem:zerocolumnrow}). This case can be seen as building block for the general case, which we describe below. Further examples will be considered later in \autoref{sec:examples}.

\medskip

We begin with the observation that the $\beta\gamma$-system $\mathcal{D}^\mathrm{ch}(T^*\C)$ contains the dual pair $\mathcal{M}(2)\otimes\pi^{\C\sigma}$ consisting of the singlet \voa{} $\mathcal{M}(2)=M_0$ for $\smash{p=2}$, i.e.\ with central charge $c=1-6(p-1)^2/p=-2$, \cite{Ada03} (see also \cite{Zam85,Wan98a,Wan98b}) and the Heisenberg \voa{} $\smash{\pi^{\C\sigma}=\pi_0^{\C\sigma}}$ for $S=\C\sigma$ with $\langle\sigma,\sigma\rangle=-1$.

Recall that the singlet \voa{} $\mathcal{M}(2)=M_0$ has certain \emph{atypical} irreducible modules $M_m\coloneqq\mathcal{M}_{n+1,1}$ for $m\in\Z$ that appear in the decomposition of the symplectic fermions $\mathsf{SF}\cong\bigoplus_{m\in\Z}M_m$ according to fermion number;\footnote{Similarly, the even subalgebra of the symplectic fermion \svoa{} $\mathsf{SF}$ is the triplet \voa{} $\mathcal{W}(2)$ for $p=2$. It is known to decompose as $\smash{\mathcal{W}(2)\cong\bigoplus_{m\in2\Z}M_m}$ into modules for the singlet \voa{} $\mathcal{M}(2)=M_0$.} these modules are all simple currents with fusion rules $M_m\boxtimes M_n\cong M_{m+n}$, but the complete representation theory is much more complicated \cite{AM17,CMY21,CMY23,BCDN23}. Then, $\mathcal{D}^\mathrm{ch}(T^*\C)$ decomposes as a module for $\mathcal{M}(2)\otimes\pi^{\C\sigma}$ as
\begin{equation*}
\mathcal{D}^\mathrm{ch}(T^*\C)=\bigoplus_{m\in\Z}M_m\otimes\pi_{m\sigma}^{\C\sigma},
\end{equation*}
where both the atypical singlet modules $M_m$ and the Fock modules $\pi_{m\sigma}^{\C\sigma}$ are simple currents with $\Z$-fusion rules. Now, localising $T^*\C$ to $T^*\C^\times$ (see \autoref{rem:localise}) corresponds to the screening kernel
\begin{equation*}
M_m=M_m^\rho\overset{\ker(e^\rho_{(0)})}{\hooklongrightarrow}\pi_{-m\rho}^{\C\rho}
\end{equation*}
with $\langle\rho,\rho\rangle=1$ and orthogonal to $\sigma$. Here, we add the upper index $\rho$ to the singlet modules $M_m$ to ``remember'' the embedding into the Heisenberg \voa{}. Overall, this yields the embedding
\begin{equation*}
\mathcal{D}^\mathrm{ch}(T^*\C)=\bigoplus_{m\in\Z}M^\rho_m\otimes\pi_{m\sigma}^{\C\sigma}\overset{\ker(e^\rho_{(0)})}{\hooklongrightarrow}\bigoplus_{m\in\Z}\pi_{-m\rho}^{\C\rho}\otimes\pi_{m\sigma}^{\C\sigma}=\HL_{\Z(-\rho+\sigma)}^{\C\rho\oplus\C\sigma}
\end{equation*}
into a Heisenberg-lattice vertex algebra, as described above.

\medskip

We apply the relative BRST cohomology defined by the chiral comoment map
\begin{equation*}
\bar\mu^*_\mathrm{ch}(a)=(xy+\psi\phi)=(\sigma+\tau)
\end{equation*}
(where $a$ is the standard basis vector of $\gl_1\cong\C$) to both sides of the above embedding, each tensored with the Heisenberg \voa{} $\smash{\pi^{\C\tau}}$, where $\langle\tau,\tau\rangle=1$ and orthogonal to both $\rho$ and $\sigma$. This leads to the following commutative diagram corresponding to the minimal hypertoric \voa{}:
\begin{equation*}
\begin{tikzcd}
\mathcal{D}^\mathrm{ch}(T^*\C)\otimes\pi^{\C\tau}=\bigoplus_{m\in\Z}M^\rho_m\otimes\pi_{m\sigma}^{\C\sigma}\otimes\pi_0^{\C\tau}\arrow[r,"\text{BRST}"]\arrow[d,hook,"\ker(e^\rho_{(0)})"]
&M_0^{\rho-\sigma-\tau}=\mathcal{M}(2)=V_\mathrm{min}(\Delta)\arrow[d,hook,"\ker(e^{\rho-\sigma-\tau}_{(0)})"]\\
\HL_{\Z(-\rho+\sigma)}^{\C\rho\oplus\C\sigma\oplus\C\tau}=\bigoplus_{m\in\Z}\pi_{-m\rho}^{\C\rho}\otimes\pi_{m\sigma}^{\C\sigma}\otimes\pi_0^{\C\tau}\arrow[r,"\text{BRST}"]
&\pi_0^{\C(-\rho+\sigma+\tau)}
\end{tikzcd}
\end{equation*}
The last line is precisely the statement of the vanishing result from \cite{Vor94}.

We explain this in more detail. Consider the $3$-dimensional space $\h=\langle\rho,\sigma,\tau\rangle_\C$ with the nondegenerate symmetric bilinear form defined by $\diag(1,-1,1)$, underlying the Heisenberg \voa{} and Fock modules in the input of the BRST reduction. The BRST differential defines the isotropic subspace
\begin{equation*}
\mathfrak{i}=\langle\sigma+\tau\rangle_\C,
\end{equation*}
which by the cohomology is projected to
\begin{equation*}
\mathfrak{i}^\bot/\mathfrak{i}=\langle\rho,\sigma+\tau\rangle_\C/\langle\sigma+\tau\rangle_\C,
\end{equation*}
where $\mathfrak{i}^\bot$ denotes the orthogonal complement of $\mathfrak{i}$ in $\h$. More precisely, $\mathfrak{i}^\bot/\mathfrak{i}$ is the new space underlying all the Fock modules, while at the same time only Fock modules with momenta in $\mathfrak{i}^\bot$ survive, meaning that the underlying lattice $\Z(-\rho+\sigma)$ is intersected with $\mathfrak{i}^\bot$. The space $\mathfrak{i}^\bot/\mathfrak{i}$ is $1$-dimensional, equipped with a nondegenerate symmetric bilinear form, isometric to, e.g., $\mathfrak{i}^\bot/\mathfrak{i}\cong\langle\rho\rangle_\C\cong\langle-\rho+\sigma+\tau\rangle_\C$. While the identification $\mathfrak{i}^\bot/\mathfrak{i}\cong\langle\rho\rangle_\C$ is certainly simpler in the present case (with $N=M=1$), it will turn out to be advantageous in the general case to rather identify
\begin{equation*}
\mathfrak{i}^\bot/\mathfrak{i}\cong\langle-\rho+\sigma+\tau\rangle_\C\eqqcolon\h^\bot.
\end{equation*}
This corresponds to choosing an orthogonal decomposition $\h=\h^\|\oplus\h^\bot$ where $\h^\|=\langle\sigma+\tau,-\rho+\sigma\rangle$ is a hyperbolic plane that is collapsed by the BRST reduction. Finally note that, modulo $\mathfrak{i}=\langle\sigma+\tau\rangle_\C$, the screening operator $\smash{e^\rho_{(0)}}$ equals $\smash{e^{\rho-\sigma-\tau}_{(0)}}$, which is the right form for our choice of $\h^\bot$. This gives the above commutative diagram, but as we just explained, one could replace all occurrences of $\rho-\sigma-\tau$ by $\rho$ om the right-hand side.

\medskip

Now, we can compute the \voa{} of interest (in the top right of the diagram) by applying the screening kernel $\smash{\ker(e^{\rho-\sigma-\tau}_{(0)})}$ to the Heisenberg \voa{} $\smash{\pi_0^{\C(\rho-\sigma-\tau)}}$. This then proves that the minimal hypertoric \voa{} $V_\mathrm{min}(\Delta)$ for $N=M=1$ and $\Delta=(1)$ is simply the singlet \voa{} $\mathcal{M}(2)$ \cite{Ada03}.

Moreover, \cite{Vor94} tells us that for the relative BRST complex, any cohomology besides the zeroth vanishes.

\medskip

We can extend the above picture to the boundary case, where we obtain, noting that (cf.\ \autoref{rem:parity}) the parity is given the parity of the coefficient of $\tau$ in the momentum, i.e.\ $(-1)^n$ on the left-hand side and $(-1)^m$ on the right-hand side:
\begin{equation*}
\begin{tikzcd}
\begin{tabular}{l}
$\mathcal{D}^\mathrm{ch}(T^*\C)\otimes\mathcal{C}\ell(\Pi T^*V)$\\
$=\bigoplus_{m,n\in\Z}M^\rho_m\otimes\pi_{m\sigma}^{\C\sigma}\otimes\pi_{n\tau}^{\C\tau}$
\end{tabular}
\arrow[r,"\text{BRST}"]\arrow[d,hook,"\ker(e^\rho_{(0)})"]
&\bigoplus_{m\in\Z}M_m^{\rho-\sigma-\tau}=\mathsf{SF}=V(\Delta)\arrow[d,hook,"\ker(e^{\rho-\sigma-\tau}_{(0)})"]\\
\begin{tabular}{l}
$\HL_{\Z(-\rho+\sigma)\oplus\Z\tau}^{\C\rho\oplus\C\sigma\oplus\C\tau}$\\
$=\bigoplus_{m,n\in\Z}\pi_{-m\rho}^{\C\rho}\otimes\pi_{m\sigma}^{\C\sigma}\otimes\pi_{n\tau}^{\C\tau}$
\end{tabular}
\arrow[r,"\text{BRST}"]
&\HL_{(-\rho+\sigma+\tau)\Z}^{\C(-\rho+\sigma+\tau)}\cong V_\Z
\end{tikzcd}
\end{equation*}
The vanishing result \cite{Vor94} works exactly as before, just that now the lattice direction $-\rho+\sigma+\tau$ in the orthogonal complement $\mathfrak{i}^\bot$ survives. Effectively, the relative BRST cohomology imposes the relation $m=n$ in the direct sum on the left-hand side, in addition to collapsing the ambient vector space from $\h$ to $\h^\bot$.

The boundary hypertoric \svoa{} $V(\Delta)$ (in the top right of the diagram) is now obtained by applying the screening kernel $\smash{\ker(e^{\rho-\sigma-\tau}_{(0)})}$ to the one-dimensional lattice \voa{} $\smash{\HL_{(-\rho+\sigma+\tau)\Z}^{\C(-\rho+\sigma+\tau)}}$. This gives the symplectic fermion \svoa{} $\mathsf{SF}$ as a simple-current extension of the singlet \voa{} $\mathcal{M}(2)$, as described above. We note that $(-\rho+\sigma+\tau)\Z$ is isomorphic to the odd standard lattice $\Z$. In the general case described below, $V(\Delta)$ will be a simple-current extension of Heisenberg and singlet \voa{}s. (Again, note that one could replace all occurrences of $\rho-\sigma-\tau$ by $\rho$ on the right-hand side of the above diagram.)

Finally, we point out that the boundary hypertoric \svoa{} $V(\Delta)$ is a simple-current extension of the minimal one $V_\mathrm{min}(\Delta)$ corresponding exactly to the odd standard lattice $\Z$. A similar statement will hold in the general case, though we have to consider $\smash{V_\mathrm{min}(\Delta)\otimes\pi^{T^E}}$, which happens to coincide with $V_\mathrm{min}(\Delta)$ here.


\subsubsection{Orthogonal Projections}

In the next section, we shall describe the relative BRST cohomology on the level of Fock modules in the general case. As in the example above, with the vanishing result from \cite{Vor94,BF25} this reduces to a problem in linear algebra, which we describe in the following.

\medskip

Consider the $3N$-dimensional space $\h\coloneqq R\oplus S\oplus T$ with basis $\{\rho_i,\sigma_i,\tau_i\}_{i=1}^N$ and bilinear form matrix
\begin{equation*}
\begin{pmatrix}
\id&0&0\\0&-\id&0\\0&0&\id
\end{pmatrix}.
\end{equation*}
The BRST differential defines the $M$-dimensional isotropic subspace
\begin{equation*}
\mathfrak{i}=\langle\{\sigma_i^\Delta+\tau_i^\Delta\}_{i=1}^M\rangle_\C,
\end{equation*}
and BRST cohomology will project this to the quotient space
\begin{align*}
\mathfrak{i}^\bot/\mathfrak{i}=\langle\{\rho_i^E,\sigma_i^E,\tau_i^E\}_{i=1}^{N-M}\cup\{\rho_i^\Delta,\sigma_i^\Delta+\tau_i^\Delta\}_{i=1}^M\rangle_\C/\langle\{\sigma_i^\Delta+\tau_i^\Delta\}_{i=1}^M\rangle_\C\\
\cong\langle\{\rho_i^E,\sigma_i^E,\tau_i^E\}_{i=1}^{N-M}\cup\{-\rho_i^\Delta+\sigma_i^\Delta+\tau_i^\Delta\}_{i=1}^M\rangle_\C\eqqcolon\h^\bot,
\end{align*}
which is $(3N-2M)$-dimensional, again equipped with a nondegenerate symmetric bilinear form. Here, $\rho_i^\Delta$ and $\rho_i^E$ are defined like $\sigma_i^\Delta,\tau_i^\Delta$ and $\sigma_i^E,\tau_i^E$ above, i.e.\ $\smash{\rho^\Delta_i\coloneqq\sum_{j=1}^N\Delta_{ij}\rho_j}$ for $i=1,\dots,M$ and $\smash{\rho^E_i\coloneqq\sum_{j=1}^NE_{ji}\rho_j}$ for $\smash{i=1,\dots,N-M}$. The above choice of quotient space corresponds to choosing the orthogonal decomposition $\h=\h^\|\oplus\h^\bot$ into the $2M$-dimensional (hyperbolic) subspace $\h^\|$ collapsed by the BRST reduction and spanned by
\begin{align*}
\sigma_i^\Delta+\tau_i^\Delta&\text{ for }i=1,\dots,M,\\
-\rho_i^\Delta+\sigma_i^\Delta&\text{ for }i=1,\dots,M
\end{align*}
and the $(3N-2M)$-dimensional orthogonal complement $\h^\bot$ spanned by
\begin{align*}
\rho_i^E,\sigma_i^E,\tau_i^E&\text{ for }i=1,\dots,N-M,\\
-\rho_i^\Delta+\sigma_i^\Delta+\tau_i^\Delta&\text{ for }i=1,\dots,M.
\end{align*}
These choices will be convenient in the following. (But note that we could have equally well chosen $\h^\bot$ with the basis vectors $-\rho_i^\Delta+\sigma_i^\Delta+\tau_i^\Delta$ replaced by $\rho_i^\Delta$ for all $i=1,\dots,M$ and with the corresponding choice of $\h^\|$.)

We denote by $p^\|$ and $p^\bot$ the orthogonal projections onto the subspaces $\h^\|$ and~$\h^\bot$, respectively. Explicitly, we may write them in the basis $\{\rho_i,\sigma_i,\tau_i\}_{i=1}^N$ as
\begin{equation*}
p^\bot=
\begin{pmatrix}
\id&p^\Delta&-p^\Delta\\-p^\Delta&p^E{-}p^\Delta&p^\Delta\\-p^\Delta&-p^\Delta&\id
\end{pmatrix}
\end{equation*}
where
\begin{equation*}
p^\Delta=\Delta^t(\Delta\Delta^t)^{-1}\Delta\quad\text{and}\quad p^E=E(E^t E)^{-1}E^t
\end{equation*}
are the orthogonal projections onto column spaces of $\Delta^t$ and $E$, respectively. These column spaces are orthogonal because $\Delta E=0$. Now, applying the relative BRST cohomology \cite{Vor94} projects from $\h$ to $\mathfrak{i}^\bot/\mathfrak{i}$, and, with the above choice of orthogonal decomposition $\h=\h^\|\oplus\h^\bot$, collapses $\h^\|$ while preserving $\h^\bot$ (see \autoref{prop:vanish_lattice} and \autoref{prop:vanish_lattice2} below).

Recall the lattices $\smash{K=\langle\{-\rho_i+\sigma_i\}_{i=1}^N\rangle_\Z}$ and $\smash{L=\langle\{-\rho_i+\sigma_i,\tau_i\}_{i=1}^N\rangle_\Z}$ in the minimal and boundary case, respectively. The former is isotropic, the latter integral (but degenerate). For later use, we also introduce the positive-definite, integral lattice $J\coloneqq\langle\{\tau_i\}_{i=1}^N\rangle_\Z$. Then $L=K\oplus J$.
The cohomology will reduce them to
{\allowdisplaybreaks
\begin{align*}
K^\bot=K\cap\h^\bot=K\cap p^\bot(K),\\
J^\bot=J\cap\h^\bot=J\cap p^\bot(J),\\
L^\bot=L\cap\h^\bot=L\cap p^\bot(L).
\end{align*}
}%
Strictly speaking, we should look at the intersections $K\cap\i^\bot$, $J\cap\i^\bot$ and $L\cap\i^\bot$ and then consider them modulo $\mathfrak{i}$, but with our choice of quotient $\h^\bot=\mathfrak{i}^\bot/\mathfrak{i}$, this amounts to the above expressions. In fact, $\h^\bot$ was chosen precisely so that these lattices have nice bases, as we shall see momentarily.

We now determine bases of these lattices. Here we use that the matrix $E$ is injective and unimodular.
\begin{lem}\label{lem:latbasis}
The lattices $K^\bot$, $J^\bot$ and $L^\bot$ of rank $N-M$, $N-M$ and $2N-M$ are of the form
{\allowdisplaybreaks
\begin{align*}
K^\bot&=\langle\{-\rho_i^E+\sigma_i^E\}_{i=1}^{N-M}\rangle_\Z,\\
J^\bot&=\langle\{\tau_i^E\}_{i=1}^{N-M}\rangle_\Z,\\
L^\bot&=\langle\{-\rho_i^E+\sigma_i^E\}_{i=1}^{N-M}\cup\{-\rho_i+\sigma_i+\tau_i\}_{i=1}^N\rangle_\Z,
\end{align*}
}%
respectively.
\end{lem}
\begin{proof}
First, we describe the lattice $K^\bot$. The inclusion $K^\bot\supseteq\langle\{-\rho_i^E+\sigma_i^E\}_{i=1}^{N-M}\rangle_\Z$ is immediate. It remains to show the other inclusion. First, we consider
\begin{equation*}
p^\bot(K)=\langle\{p^\bot(-\rho_i+\sigma_i)\}_{i=1}^N\rangle_\Z=\langle\{p^E(-\rho_i+\sigma_i)\}_{i=1}^N\rangle_\Z.
\end{equation*}
To simplify notation, we identify $K=\langle\{-\rho_i+\sigma_i\}_{i=1}^N\rangle_\Z$ with $\Z^N$, written as column vectors. Then $p^\bot(K)=p^E\Z^N$ and
\begin{align*}
K^\bot&=K\cap p^\bot(K)=\Z^N\cap p^E\Z^N=\Z^N\cap E(E^tE)^{-1}E^t\Z^N\\
&\subset\Z^N\cap E\Q^{N-M}.
\end{align*}
Now suppose that the vector $b$ belongs to $\Z^N\cap E\Q^{N-M}$. In other words, $b\in\Z^N$ and $b=Ex$ for some $x\in\Q^{N-M}$. Since $E$ is injective, its $N-M$ columns are linearly independent and we can find $N-M$ rows that are also linearly independent. Call $\tilde{E}$ the corresponding square submatrix of $E$. Since $E$ is unimodular, $\det(\tilde{E})\in\{\pm1\}$. The linear system of equations $Ex=b$ implies $\tilde{E}x=\tilde{b}$ where $b\in\Z^{N-M}$ is a subvector of $b\in\Z^N$ corresponding to the same rows. Because $\tilde{E}$ is invertible, we can use Cramer's rule to solve for $x$, and since $\tilde{E}$ and $\tilde{b}$ have only integer entries and because $1/\det(\tilde{E})\in\Z$, it follows that $x\in\Z^{N-M}$. Hence, $b\in E\Z^N$. We have hence shown that
\begin{equation*}
K^\bot\subset\Z^N\cap E\Q^{N-M}\subset E\Z^{N-M}=\langle\{-\rho_i^E+\sigma_i^E\}_{i=1}^{N-M}\rangle_\Z,
\end{equation*}
proving the assertion for $K^\bot$. In the last equality we use that, by definition, the $-\rho_i^E+\sigma_i^E$ for $i=1,\dots,N-M$ expressed in the basis $\{-\rho_i+\sigma_i\}_{i=1}^N$ are simply the columns of $E$.

We proceed with determining a basis of the lattice $L^\bot$. It is helpful to slightly rewrite the lattice $L$ as $L=\langle\{-\rho_i+\sigma_i,-\rho_i+\sigma_i+\tau_i\}_{i=1}^N\rangle_\Z$. Then
\begin{align*}
p^\bot(L)&=\langle\{p^\bot(-\rho_i+\sigma_i),p^\bot(-\rho_i+\sigma_i+\tau_i)\}_{i=1}^N\rangle_\Z\\
&=\langle\{p^E(-\rho_i+\sigma_i),-\rho_i+\sigma_i+\tau_i\}_{i=1}^N\rangle_\Z.
\end{align*}
The same argument as before then yields
\begin{equation*}
L^\bot=L\cap p^\bot(L)=\langle\{-\rho_i^E+\sigma_i^E\}_{i=1}^{N-M}\cup\{-\rho_i+\sigma_i+\tau_i\}_{i=1}^N\rangle_\Z.
\end{equation*}

Finally, we consider the lattice $J^\bot$. Here, we see that
\begin{equation*}
J^\bot=J\cap\h^\bot=J\cap T\cap\h^\bot=J\cap T^E=\langle\{\tau_i^E\}_{i=1}^{N-M}\rangle_\Z,
\end{equation*}
where in the last step we again argued like for the lattice $K$ to obtain the nontrivial inclusion $J\cap T^E\subset\langle\{\tau_i^E\}_{i=1}^{N-M}\rangle_\Z$.
\end{proof}
For the application of the screening operators, we can choose a more convenient basis of $\h^\bot=p^\bot(R)\oplus S^E\oplus T^E$ (an orthogonal direct sum) as follows:
\begin{align*}
p^\bot(\rho_i)&\text{ for }i=1,\dots,N,\\
\sigma_i^E,\tau_i^E&\text{ for }i=1,\dots,N-M.
\end{align*}
We note that the $p^\bot(\rho_i)$, which form a basis of $p^\bot(R)$, satisfy $\langle p^\bot(\rho_i),p^\bot(\rho_j)\rangle=\delta_{ij}$, like the original $\rho_i$ in $R$. This will play a role in \autoref{sec:vanishfock}, when we apply the screening operators.

Moreover, the lattice basis vectors of $K^\bot$ and $L^\bot$ can be written in terms of the new basis as
\begin{equation*}
-\rho_i^E+\sigma_i^E=-\sum_{j=1}^NE_{ji}p^\bot(\rho_j)+\sigma_i^E
\end{equation*}
for $i=1,\dots,N-M$ and
\begin{equation*}
-\rho_i+\sigma_i+\tau_i=-p^\bot(\rho_i)+\sum_{j=1}^{N-M}((E^tE)^{-1}E^t)_{ji}(\sigma_j^E+\tau_j^E)
\end{equation*}
for $i=1,\dots,N$, where in each case both summands are orthogonal.

\begin{rem}\label{rem:biggerK}
Even though $K\oplus J=L$ by definition, $K^\bot\oplus J^\bot\subset L^\bot$ is in general only a sublattice of $L^\bot$ of smaller rank (infinite index). However, we can consider the orthogonal decomposition $\h^\bot=(p^\bot(R)\oplus S^E)\oplus T^E$ and based on this the integral lattices
\begin{align*}
\overline{K^\bot}&\coloneqq L\cap(p^\bot(R)\oplus S^E)=L^\bot\cap(p^\bot(R)\oplus S^E),\\
J^\bot&=L\cap T^E=L^\bot\cap T^E
\end{align*}
of rank $N$ and $N-M$, respectively. Then $\smash{\overline{K^\bot}\oplus J^\bot\subset L^\bot}$ is a sublattice of $L^\bot$ of finite index. Cf.\ \autoref{rem:intermediate}.
\end{rem}


\subsubsection{Vanishing Theorem for Fock Modules}\label{sec:vanishfock}

We now come to the main result of this section, the vanishing result of \cite{Vor94}, applied to Fock modules (see also \cite{BF25}). Recall from \autoref{sec:freefield2} that we embedded $\smash{M_\mathrm{min}\otimes\pi^{T^E}}$ and $M$, the inputs of the BRST reduction, into Heisenberg-lattice vertex (super)algebras $\HL_K^\h$ and $\HL_L^\h$, respectively, i.e.\ into a direct sum of Fock modules for the Heisenberg \voa{} $\pi^\h=\pi^R\otimes\pi^S\otimes\pi^S$. Also recall the choice of quotient space $\h^\bot=\mathfrak{i}^\bot/\mathfrak{i}$.

\begin{prop}[Vanishing Theorem, Heisenberg Modules]\label{prop:vanish_Heisenberg}
The relative cohomologies of the Fock modules $\pi_\lambda^\h$ satisfy, for $\lambda\in\h$ and $n\in\Z$,
\begin{equation*}
H^{\infty/2+n}(\g,\pi_\lambda^\h)=\delta_{n,0}\delta_{\lambda\in\h^\bot}\pi_\lambda^{\h^\bot}.
\end{equation*}
\end{prop}
The vanishing result states that all relative cohomologies (even the zeroth) vanish unless $\lambda\in\h^\bot$. In that case, the Fock module survives, but with a reduction in rank from $\h$ to $\h^\bot$, i.e.\ from $3N$ to  $3N-2M$.

\medskip

From this, using the commutativity of the diagrams in \autoref{sec:freefield2}, zoomed in to the level of Fock modules, we can also determine the BRST cohomology of the direct summands of $M_\mathrm{min}\otimes\pi^{T^E}$ and $M$. We can then use this in \autoref{sec:vanish2} to compute the BRST cohomology of $M_\mathrm{min}\otimes\pi^{T^E}$ and $M$.

The ``matter'' free-field vertex operator (super)algebras in the input of the BRST reduction are conformal extensions
\begin{equation*}
M\supseteq M_\mathrm{min}\otimes\pi^{T^E}\supseteq M^R\otimes\pi^S\otimes\pi^T,
\end{equation*}
which we embed via the screening kernel
\begin{equation*}
\ker(e^{\rho_1}_{(0)})\cap\dots\cap\ker(e^{\rho_N}_{(0)})
\end{equation*}
into the rank-$(2N,3N)$ and rank-$(N,3N)$ Heisenberg-lattice vertex algebras
\begin{equation*}
\HL_L^\h\supseteq\HL_K^\h\supseteq\pi^R\otimes\pi^S\otimes\pi^T.
\end{equation*}
Here, we wrote for short $M^R=\smash{M_0^{\rho_1}\otimes\dots\otimes M_0^{\rho_n}}\cong\mathcal{M}(2)^{\otimes N}$ for $N$ tensor copies of the singlet \voa{}, with the superscript remembering the embedding into the Heisenberg \voa{} $\pi^R$ via the screening kernels.

Now, using that the relative BRST cohomology and applying the screening kernels commute, in the sense stated in \autoref{prop:commute1} and \autoref{prop:commute2}, which holds on the level of Fock modules, we obtain as a corollary the following vanishing result, again with the above choice of quotient space $\h^\bot=\mathfrak{i}^\bot/\mathfrak{i}$.
\begin{prop}[Vanishing Theorem, Singlet-Heisenberg Modules]\label{prop:vanish_singlet}
Consider $\lambda=\lambda_r+\lambda_s+\lambda_t\in\h=R\oplus S\oplus T$, and suppose that $\lambda_r$ is an integer linear combination of the $\rho_i$. We consider the module $M_{\lambda_r}^R\otimes\pi_{\lambda_s}^S\otimes\pi_{\lambda_t}^T$ for the singlet-Heisenberg \voa{} $M^R\otimes\pi^S\otimes\pi^T$. The relative cohomologies satisfy
\begin{equation*}
H^{\infty/2+n}(\g,M_{\lambda_r}^R\otimes\pi_{\lambda_s}^S\otimes\pi_{\lambda_t}^T)=\delta_{n,0}\delta_{\lambda\in\h^\bot}M_{p^\bot(\lambda_r)}^{p^\bot(R)}\otimes\pi_{p^E(\lambda_s)}^{S^E}\otimes\pi_{p^E(\lambda_t)}^{T^E}
\end{equation*}
for all $n\in\Z$. If the condition $\lambda\in\h^\bot$ is satisfied, then $\lambda=p^\bot(\lambda)=p^\bot(\lambda_r)+p^\bot(\lambda_s)+p^\bot(\lambda_t)=p^\bot(\lambda_r)+p^E(\lambda_s)+p^E(\lambda_t)\in\h^\bot= p^\bot(R)\oplus S^E\oplus T^E$, which is an orthogonal direct sum.
\end{prop}
Here, we again wrote for short $M^{p^\bot(R)}=\smash{M_0^{p^\bot(\rho_1)}\otimes\dots\otimes M_0^{p^\bot(\rho_n)}}\cong\mathcal{M}(2)^{\otimes N}$ for $N$ tensor copies of the singlet \voa{} embedded into the Heisenberg \voa{} $\smash{\pi^{p^\bot(R)}}$. We use an analogous notation for the nontrivial simple-current modules of $M^R$ and $\smash{M^{p^\bot(R)}}$. Note that $p^\bot(\lambda_r)$ is still an integer linear combination of the $p^\bot(\rho_i)$, which form a basis of $p^\bot(R)$ and satisfy $\langle p^\bot(\rho_i),p^\bot(\rho_j)\rangle=\delta_{ij}$. So, this notation makes sense.

The expressions for $V_\mathrm{min}(\Delta)\otimes\pi^{T^E}$ and $V(\Delta)$ in \autoref{sec:vanish2} below will then be obtained from the above modules by summing over the lattices $K$ and $L$ in $\h$, respectively.


\subsubsection{Vanishing Theorem}\label{sec:vanish2}

We apply the vanishing result, \autoref{prop:vanish_singlet}, to the free-field vertex operator (super)algebras $\smash{M_\mathrm{min}\otimes\pi^{T^E}}$ and $M$ in order to obtain the minimal and boundary hypertoric vertex operator(super)algebras $\smash{V_\mathrm{min}(\Delta)\otimes\pi^{T^E}}$ and $V(\Delta)$, respectively. We shall see that they are certain simple-current extensions of Heisenberg and singlet \voa{}s.

\medskip

We begin with the computation of the BRST cohomology in the minimal hypertoric case. The free-field \voa{} in the input of the BRST reduction is of the form
\begin{align*}
M_\mathrm{min}\otimes\pi^{T^E}&=\mathcal{D}^\mathrm{ch}(T^*V)\otimes\pi^T\\
&=\bigoplus_{m_1,\dots,m_N\in\Z}M_{\sum_{i=1}^Nm_i\rho_i}^R\otimes\pi_{\sum_{i=1}^Nm_i\sigma_i}^S\otimes\pi_0^T,
\end{align*}
which we embed via the screening kernel
\begin{equation*}
\ker(e^{\rho_1}_{(0)})\cap\dots\cap\ker(e^{\rho_N}_{(0)})
\end{equation*}
into the rank-$(N,3N)$ Heisenberg-lattice vertex algebra
\begin{equation*}
\HL_K^\h=\bigoplus_{m_1,\dots,m_N\in\Z}\pi_{-\sum_{i=1}^Nm_i\rho_i}^R\otimes\pi_{\sum_{i=1}^Nm_i\sigma_i}^S\otimes\pi_0^T.
\end{equation*}
Recall the $3N$-dimensional quadratic space $\h=\langle\{\rho_i,\sigma_i,\tau_i\}_{i=1}^N\rangle_\C$ and the isotropic rank-$N$ lattice $K=\langle\{-\rho_i+\sigma_i\}_{i=1}^N\rangle_\Z$ from above.

After choosing the quotient space $\h^\bot=\mathfrak{i}^\bot/\mathfrak{i}$, the relative BRST cohomology (see \autoref{prop:vanish_Heisenberg}) collapses $\h^\|$ (as well as the lattice directions $-\rho_i^\Delta+\sigma_i^\Delta$ for $i=1,\dots,M$), while preserving $\h^\bot$. That is, when considering the Heisenberg-lattice extension $\HL_K^\h$ of $\smash{M_\mathrm{min}\otimes\pi^{T^E}=\mathcal{D}^\mathrm{ch}(T^*V)\otimes\pi^T}$, the vanishing theorem can be stated very simply as:
\begin{prop}[Vanishing Theorem, Lattice]\label{prop:vanish_lattice}
The relative cohomologies satisfy
\begin{equation*}
H^{\infty/2+n}(\g,\HL_K^\h)=\delta_{n,0}\HL_{K^\bot}^{\h^\bot}
\end{equation*}
for $n\in\Z$, where $K^\bot=K\cap\h^\bot$.
\end{prop}
Here,
\begin{equation*}
K^\bot=K\cap\h^\bot=\langle\{-\rho_i^E+\sigma_i^E\}_{i=1}^{N-M}\rangle_\Z
\end{equation*}
is an isotropic lattice of rank $N-M$ (see \autoref{lem:latbasis}). To obtain $\smash{V_\mathrm{min}(\Delta)\otimes\pi^{T^E}}$ it remains to apply the screening operators $\smash{e^{\rho_i}_{(0)}}$, but we need to project them to $\h^\bot$ first by means of the orthogonal projection $p^\bot$. However, as $\rho_i-p^\bot(\rho_i)=p^\Delta(\sigma_i+\tau_i)\in\mathfrak{i}$ for all $i=1,\dots,N$, these are really the same screening operators in the cohomology.

We can write the rank-$(N-M,3N-2M)$ Heisenberg-lattice vertex algebra $H^{\infty/2+0}(\g,\HL_K^\h)=\HL_{K^\bot}^{\h^\bot}$ as
\begin{align*}
\HL_{K^\bot}^{\h^\bot}&=\bigoplus_{k_1,\dots,k_{N-M}\in\Z}\pi_{\sum_{i=1}^{N-M}k_i(-\rho_i^E+\sigma_i^E)}^{\h^\bot}\\
&=\bigoplus_{k_1,\dots,k_{N-M}\in\Z}\pi_{-\sum_{j=1}^N\sum_{i=1}^{N-M}E_{ji}k_ip^\bot(\rho_j)}^{p^\bot(R)}\otimes\pi_{\sum_{i=1}^{N-M}k_i\sigma_i^E}^{S^E}\otimes\pi_0^{T^E}.
\end{align*}
With this presentation, we can apply the screening kernels $\smash{\ker(e^{\rho_1}_{(0)})\cap\dots\cap\ker(e^{\rho_N}_{(0)})}$, which modulo $\mathfrak{i}$ are
\begin{equation*}
\ker(e^{p^\bot(\rho_1)}_{(0)})\cap\dots\cap\ker(e^{p^\bot(\rho_N)}_{(0)}),
\end{equation*}
to obtain, using \autoref{prop:commute1}:
\begin{prop}[Main Vanishing Theorem]\label{prop:vanish_main1}
The minimal hypertoric \voa{} is of the form
\begin{align*}
&H^{\infty/2+n}(\g,M_\mathrm{min})\otimes\pi^{T^E}=\delta_{n,0}V_\mathrm{min}(\Delta)\otimes\pi^{T^E}\\
&=\delta_{n,0}\bigoplus_{k_1,\dots,k_{N-M}\in\Z}M_{\sum_{j=1}^N\sum_{i=1}^{N-M}E_{ji}k_ip^\bot(\rho_j)}^{p^\bot(R)}\otimes\pi_{\sum_{i=1}^{N-M}k_i\sigma_i^E}^{S^E}\otimes\pi_0^{T^E},
\end{align*}
a $\Z^{N-M}$-simple current extension of $N$ tensor copies of the singlet \voa{} $\mathcal{M}(2)$ for $p=2$ and a rank-$2(N-M)$ Heisenberg \voa{}.
\end{prop}
In other words, the result is obtained by applying \autoref{prop:vanish_singlet} to the above decomposition of $\smash{M_\mathrm{min}\otimes\pi^{T^E}}$ as singlet-Heisenberg modules.

\medskip

We then come to the boundary hypertoric \svoa{}, which we describe in a similar way. The ``matter'' free-field \svoa{} in the input of the BRST reduction is of the form
\begin{equation*}
M=\bigoplus_{\substack{m_1,\dots,m_N\in\Z\\n_1,\dots,n_N\in\Z}}M_{\sum_{i=1}^Nm_i\rho_i}^R\otimes\pi_{\sum_{i=1}^Nm_i\sigma_i}^S\otimes\pi_{\sum_{i=1}^Nn_i\tau_i}^T.
\end{equation*}
It embeds via the screening kernels $\smash{\ker(e^{\rho_1}_{(0)})\cap\dots\cap\ker(e^{\rho_N}_{(0)})}$ into the rank-$(2N,3N)$ Heisenberg-lattice vertex superalgebra
\begin{equation*}
\HL_L^\h=\bigoplus_{\substack{m_1,\dots,m_N\in\Z\\n_1,\dots,n_N\in\Z}}\pi_{-\sum_{i=1}^Nm_i\rho_i}\otimes\pi_{\sum_{i=1}^Nm_i\sigma_i}^S\otimes\pi_{\sum_{i=1}^Nn_i\tau_i}^T,
\end{equation*}
recalling the integral (and degenerate) lattice $L=\langle\{-\rho_i+\sigma_i,\tau_i\}_{i=1}^N\rangle_\Z$. In both cases, by \autoref{rem:parity}, the parity is given by $(-1)^{n_1+\dots+n_N}$. We consider the same orthogonal decomposition $\h=\h^\|\oplus\h^\bot$ as before, corresponding to a choice of quotient space $\h^\bot=\mathfrak{i}^\bot/\mathfrak{i}$. Analogously to \autoref{prop:vanish_lattice}, the vanishing result of \cite{Vor94} states:
\begin{prop}[Vanishing Theorem, Lattice]\label{prop:vanish_lattice2}
The relative cohomologies satisfy
\begin{equation*}
H^{\infty/2+n}(\g,\HL_L^\h)=\delta_{n,0}\HL_{L^\bot}^{\h^\bot}
\end{equation*}
for $n\in\Z$, where $L^\bot=L\cap\h^\bot$.
\end{prop}
This follows directly from the vanishing result for the Fock modules stated in \autoref{prop:vanish_lattice}. Here
\begin{equation*}
L^\bot=L\cap\h^\bot=\langle\{-\rho_i^E+\sigma_i^E\}_{i=1}^{N-M}\cup\{-\rho_j+\sigma_j+\tau_j\}_{j=1}^N\rangle_\Z
\end{equation*}
is an integral (and degenerate) lattice of rank $2N-M$ (see \autoref{lem:latbasis}). In order to apply the projected screening operators $\smash{e^{p^\bot(\rho_j)}_{(0)}}$ (which, modulo $\mathfrak{i}$, coincide with the $\smash{e^{\rho_j}_{(0)}}$) for $j=1,\dots,N$ to obtain $V(\Delta)$, we write $\smash{H^{\infty/2+0}(\g,\HL_L^\h)=\HL_{L^\bot}^{\h^\bot}}$ as
{\allowdisplaybreaks
\begin{align*}
\HL_{L^\bot}^{\h^\bot}&=\bigoplus_{\substack{k_1,\dots,k_{N-M}\in\Z\\l_1,\dots,l_N\in\Z}}\pi_{\sum_{i=1}^{N-M}k_i(-\rho_i^E+\sigma_i^E)+\sum_{j=1}^Nl_i(-\rho_j+\sigma_j+\tau_j)}^{\h^\bot}\\
&=\bigoplus_{\substack{k_1,\dots,k_{N-M}\in\Z\\l_1,\dots,l_N\in\Z}}\pi_{-\sum_{j=1}^N\bigl(\sum_{i=1}^{N-M}E_{ji}k_i+l_j\bigr)p^\bot(\rho_j)}^{p^\bot(R)}\\
&\quad\otimes\pi_{\sum_{i=1}^{N-M}\bigl(k_i+\sum_{j=1}^N((E^tE)^{-1}E^t)_{ij}l_j\bigr)\sigma_i^E}^{S^E}\\
&\quad\otimes\pi_{\sum_{i=1}^{N-M}\sum_{j=1}^N((E^tE)^{-1}E^t)_{ij}l_j\tau_i^E}^{T^E}.
\end{align*}
}%
Recall from \autoref{rem:parity} that the parity in this \svoa{} is determined by the parity of the sum of the coefficients of the $\tau_i$, i.e.\ $(-1)^{l_1+\dots+l_N}$. Finally, we can apply the screening kernels
\begin{equation*}
\ker(e^{p^\bot(\rho_1)}_{(0)})\cap\dots\cap\ker(e^{p^\bot(\rho_N)}_{(0)})
\end{equation*}
to obtain, using \autoref{prop:commute2}, the hypertoric \svoa{}:
\begin{prop}[Main Vanishing Theorem]\label{prop:vanish_main2}
The boundary hypertoric \svoa{} is of the form
\begin{align*}
&H^{\infty/2+n}(\g,M)=\delta_{n,0}V(\Delta)\\
&=\delta_{n,0}\bigoplus_{\substack{k_1,\dots,k_{N-M}\in\Z\\l_1,\dots,l_N\in\Z}}M_{\sum_{j=1}^N\bigl(\sum_{i=1}^{N-M}E_{ji}k_i+l_j\bigr)p^\bot(\rho_j)}^{p^\bot(R)}\\
&\quad\otimes\pi_{\sum_{i=1}^{N-M}\bigl(k_i+\sum_{j=1}^N((E^tE)^{-1}E^t)_{ij}l_j\bigr)\sigma_i^E}^{S^E}\otimes\pi_{\sum_{i=1}^{N-M}\sum_{j=1}^N((E^tE)^{-1}E^t)_{ij}l_j\tau_i^E}^{T^E},
\end{align*}
a $\Z^{2N-M}$-simple current extension of $N$ tensor copies of the singlet \voa{} $\mathcal{M}(2)$ and a rank-$2(N-M)$ Heisenberg \voa{}, with parity given by $(-1)^{l_1+\dots+l_N}$.
\end{prop}
Again, one can see the assertion also by applying \autoref{prop:vanish_singlet} to the above decomposition of $M$ as singlet-Heisenberg modules.

\begin{cor}[Vanishing Theorem]\label{prop:vanish}
The relative cohomologies satisfy
\begin{equation*}
H^{\infty/2+n}(\g,M)=\{0\}\quad\text{and}\quad H^{\infty/2+n}(\g,M_\mathrm{min})=\{0\}
\end{equation*}
for all $n\neq0$. In other words, the hypertoric vertex operator (super)algebras $V(\Delta)$ and $V_\mathrm{min}(\Delta)$ are only concentrated in ghost degree~$0$.
\end{cor}

\begin{rem}[Alternative Presentation]\label{rem:altrep}
We can also give a slightly different presentation of $V_\mathrm{min}(\Delta)$ and $V(\Delta)$. Indeed, using the exactness of the sequence $\smash{\{0\}\to\Z^{N-M}\overset{E}{\to}\Z^N\overset{\Delta}{\to}\Z^M\to\{0\}}$, we write the lattices $K^\bot$, $J^\bot$ and $L^\bot$ as
\begin{align*}
K^\bot&=\{\textstyle\sum_{m_1,\dots,m_N\in\Z}m_i(-\rho_i+\sigma_i)\mid\Delta m=0\},\\
J^\bot&=\{\textstyle\sum_{n_1,\dots,n_N\in\Z}n_i\tau_i\mid\Delta n=0\},\\
L^\bot&=\{\textstyle\sum_{m_1,\dots,m_N\in\Z}m_i(-\rho_i+\sigma_i)+\sum_{n_1,\dots,n_N\in\Z}n_i(-\rho_i+\sigma_i+\tau_i)\mid\Delta m=0\}\\
&=\{\textstyle\sum_{m_1,\dots,m_N\in\Z}m_i(-\rho_i+\sigma_i)+\sum_{n_1,\dots,n_N\in\Z}n_i\tau_i\mid\Delta m=\Delta n\}
\end{align*}
where $m=(m_1,\dots,m_N)^t\in\Z^N$ and $n=(n_1,\dots,n_N)^t\in\Z^N$. Note that the condition $\Delta m=\Delta n$ is exactly the condition $\lambda\in\h^\bot$ in \autoref{prop:vanish_Heisenberg} and \autoref{prop:vanish_singlet} for $\lambda\in\h$ of the form $\smash{\lambda=\sum_{m_1,\dots,m_N\in\Z}m_i(-\rho_i+\sigma_i)+\sum_{n_1,\dots,n_N\in\Z}n_i\tau_i}$. The sublattice $K^\bot\oplus J^\bot\subset L^\bot$ differs from $L^\bot$ by the condition $\Delta m=0=\Delta n$ instead of just $\Delta m=\Delta n$, which implies that its rank is $M$ smaller. In fact, by the short exact sequence, the quotient is $L^\bot/(K^\bot\oplus J^\bot)\cong\Z^M$. Then
\begin{align*}
V_\mathrm{min}(\Delta)\otimes\pi^{T^E}&=\!\!\!\!\bigoplus_{\substack{m_1,\dots,m_N\in\Z\\\Delta m=0}}\!\!\!\!M_{\sum_{i=1}^Nm_ip^\bot(\rho_i)}^{p^\bot(R)}\otimes\pi_{\sum_{i=1}^Nm_i\sigma_i}^{S^E}\otimes\pi_0^{T^E}\\
&=\!\!\!\!\bigoplus_{\substack{m_1,\dots,m_N\in\Z\\\Delta m=0}}\!\!\!\!M_{\sum_{i=1}^Nm_ip^\bot(\rho_i)}^{p^\bot(R)}\otimes\pi_{\sum_{i=1}^{N-M}\sum_{j=1}^N((E^tE)^{-1}E^t)_{ij}m_j\sigma_i^E}^{S^E}\otimes\pi_0^{T^E}
\end{align*}
and
\begin{align*}
V(\Delta)&=\!\!\bigoplus_{\substack{m_1,\dots,m_N\in\Z\\n_1,\dots,n_N\in\Z\\\Delta m=\Delta n}}\!\!M_{\sum_{i=1}^Nm_ip^\bot(\rho_i)}^{p^\bot(R)}\otimes\pi_{\sum_{i=1}^N\sum_{j=1}^N(p^E)_{ij}m_j\sigma_i}^{S^E}\otimes\pi_{\sum_{i=1}^N\sum_{j=1}^N(p^E)_{ij}n_j\tau_i}^{T^E}\\
&=\!\!\bigoplus_{\substack{m_1,\dots,m_N\in\Z\\n_1,\dots,n_N\in\Z\\\Delta m=\Delta n}}\!\!M_{\sum_{i=1}^Nm_ip^\bot(\rho_i)}^{p^\bot(R)}\otimes\pi_{\sum_{i=1}^{N-M}\sum_{j=1}^N((E^tE)^{-1}E^t)_{ij}m_j\sigma_i^E}^{S^E}\\
&\quad\otimes\pi_{\sum_{i=1}^{N-M}\sum_{j=1}^N((E^tE)^{-1}E^t)_{ij}n_j\tau_i^E}^{T^E}.
\end{align*}
Here, the parity is given by $(-1)^{n_1+\dots+n_N}$.
\end{rem}

\begin{rem}[Intermediate \SVOA{}s]\label{rem:intermediate}
It will be convenient to also consider the following intermediate \svoa{}s
\begin{equation*}
V_\mathrm{min}(\Delta)\otimes\pi^{T^E}\hookrightarrow V_\mathrm{min}(\Delta)\otimes V_{J^\bot}\hookrightarrow\overline{V_\mathrm{min}(\Delta)}\otimes V_{J^\bot}\hookrightarrow V(\Delta)
\end{equation*}
with
\begin{align*}
V_\mathrm{min}(\Delta)\otimes V_{J^\bot}&=\biggl(\bigoplus_{k_1,\dots,k_{N-M}\in\Z}M_{\sum_{j=1}^N\sum_{i=1}^{N-M}E_{ji}k_ip^\bot(\rho_j)}^{p^\bot(R)}
\otimes\pi_{\sum_{i=1}^{N-M}k_i\sigma_i^E}^{S^E}\biggr)\\
&\quad\otimes\biggl(\bigoplus_{l_1,\dots,l_{N-M}\in\Z}\pi_{\sum_{i=1}^{N-M}l_i\tau_i^E}^{T^E}\biggr)
\end{align*}
with parity $(-1)^{\sum_{j=1}^N\sum_{i=1}^{N-M}l_iE_{ji}}=(-1)^{\sum_{j=1}^{N}(El)_j}$ or, in the alternative presentation of the lattices in \autoref{rem:altrep},
{\allowdisplaybreaks
\begin{align*}
V_\mathrm{min}(\Delta)\otimes V_{J^\bot}&=\biggl(\bigoplus_{\substack{m_1,\dots,m_N\in\Z\\\Delta m=0}}M_{\sum_{i=1}^Nm_ip^\bot(\rho_i)}^{p^\bot(R)}
\otimes\pi_{\sum_{i=1}^Nm_i\sigma_i}^{S^E}\biggr)\\
&\quad\otimes\biggl(\bigoplus_{\substack{n_1,\dots,n_N\in\Z\\\Delta n=0}}\pi_{\sum_{i=1}^Nn_i\tau_i}^{T^E}\biggr)\\
&=\biggl(\bigoplus_{\substack{m_1,\dots,m_N\in\Z\\\Delta m=0}}M_{\sum_{i=1}^Nm_ip^\bot(\rho_i)}^{p^\bot(R)}
\otimes\pi_{\sum_{i=1}^{N-M}\sum_{j=1}^N((E^tE)^{-1}E^t)_{ij}m_j\sigma_i^E}^{S^E}\biggr)\\
&\quad\otimes\biggl(\bigoplus_{\substack{n_1,\dots,n_N\in\Z\\\Delta n=0}}\pi_{\sum_{i=1}^{N-M}\sum_{j=1}^N((E^tE)^{-1}E^t)_{ij}n_j\tau_i^E}^{T^E}\biggr),
\end{align*}
}%
with parity $(-1)^{n_1+\dots+n_N}$. This \svoa{} involves the lattice \svoa{}
\begin{equation*}
V_{J^\bot}=\bigoplus_{l_1,\dots,l_{N-M}\in\Z}\pi_{\sum_{i=1}^{N-M}l_i\tau_i^E}^{T^E}
\end{equation*}
or, written differently,
\begin{equation*}
V_{J^\bot}=\bigoplus_{\substack{n_1,\dots,n_N\in\Z\\\Delta n=0}}\pi_{\sum_{i=1}^Nn_i\tau_i}^{T^E}=\bigoplus_{\substack{n_1,\dots,n_N\in\Z\\\Delta n=0}}\pi_{\sum_{i=1}^{N-M}\sum_{j=1}^N((E^tE)^{-1}E^t)_{ij}n_j\tau_i^E}^{T^E}.
\end{equation*}
The lattice $J^\bot\subset T^E$ is positive-definite and integral with bilinear form matrix $E^tE$ with respect to the basis $\{\tau_i^E\}_{i=1}^{N-M}$ of $J^\bot$. In particular, $V_{J^\bot}$ is \strat{}. This will be useful in determining the associated variety of $V(\Delta)$; see \autoref{thm:var} below. As in \autoref{rem:parity}, the parity of an element of $V_{J^\bot}$ is determined by the norm of the momentum of the Fock module in which it lies.

We note that $V_\mathrm{min}(\Delta)\otimes V_{J^\bot}$ is in general not a dual pair in $V(\Delta)$. It can be useful to consider the double commutant $\smash{\overline{V_\mathrm{min}(\Delta)}\coloneqq\Com_{V(\Delta)}(\Com_{V(\Delta)}(V_\mathrm{min}(\Delta)))}$, which is an infinite simple-current extension of $V_\mathrm{min}(\Delta)$. Then, $\overline{V_\mathrm{min}(\Delta)}\otimes V_{J^\bot}$ is a dual pair in $V(\Delta)$, and the latter is a finite simple-current extension of the former.

This observation is, before applying the screening kernels, mirrored by \autoref{rem:biggerK}. $K^\bot\oplus J^\bot\subset L^\bot$ does not have finite index (i.e.\ the same rank) in general, but we can consider the infinite-order extension $\smash{\overline{K^\bot}}$ of $K^\bot$ so that $\smash{\overline{K^\bot}}\oplus J^\bot$ is a finite-index (same-rank) sublattice of $L^\bot$. We forego a more detailed discussion here, and rather refer to the examples in \autoref{sec:examples}.
\end{rem}


\subsubsection{Summary}\label{sec:summary}

We summarise the constructions in this section in the following diagram, with the second and fourth of the four diagonal slices corresponding to \autoref{prop:commute1} and \autoref{prop:commute2}, respectively.

\begin{equation*}
\adjustbox{scale=0.75,center}{%
\begin{tikzcd}[row sep={35,between origins}, column sep={60,between origins}]
& M^{p^\bot(R)}\otimes\pi^{S^E}\otimes\pi^{T^E} \ar[hook,"\Z^{N-M}"]{rr}\ar[hook,"\ker(e^{p^\bot(\rho_i)}_{(0)})"' near end]{dd}\ar[<-,"\text{BRST}",sloped]{dl} && V_\mathrm{min}(\Delta)\otimes\pi^{T^E} \ar[hook,"\Z^{N-M}"]{rr}\ar[hook,"\ker(e^{p^\bot(\rho_i)}_{(0)})"' near end]{dd}\ar[<-,"\text{BRST}",sloped]{dl} && V_\mathrm{min}(\Delta)\otimes V_{J^\bot} \ar[hook,"\Z^M"]{rr}\ar[hook,"\ker(e^{p^\bot(\rho_i)}_{(0)})"' near end]{dd}\ar[<-,"\text{BRST}",sloped]{dl} && V(\Delta) \ar[hook,"\ker(e^{p^\bot(\rho_i)}_{(0)})"' near end]{dd}\ar[<-,"\text{BRST}",sloped]{dl} \\
M^R\otimes\pi^S\otimes\pi^T \ar[hook,crossing over,"\Z^N" near start]{rr}\ar[hook,"\ker(e^{\rho_i}_{(0)})"' near end]{dd} && M_\mathrm{min}\otimes\pi^{T^E} \ar[hook,crossing over,"\Z^{N-M}" near start]{rr} && M_\mathrm{min}\otimes V_{J^\bot} \ar[hook,crossing over,"\Z^M" near start]{rr} && M\\
& \pi^{\h^\bot} \ar[hook,"\Z^{N-M}" near start]{rr}\ar[<-,"\text{BRST}",sloped]{dl} && V_{K^\bot}^{\h^\bot} \ar[hook,"\Z^{N-M}" near start]{rr}\ar[<-,"\text{BRST}",sloped]{dl} && V_{K^\bot\oplus J^\bot}^{\h^\bot} \ar[hook,"\Z^M" near start]{rr}\ar[<-,"\text{BRST}",sloped]{dl} && V_{L^\bot}^{\h^\bot} \ar[<-,"\text{BRST}",sloped]{dl}\\
\pi^\h \ar[hook,"\Z^N"]{rr} && V_K^\h \ar[hook,"\Z^{N-M}"]{rr}\ar[hook,from=uu,crossing over,"\ker(e^{\rho_i}_{(0)})"' near end] && V_{K\oplus J^\bot}^\h \ar[hook,"\Z^M"]{rr}\ar[hook,from=uu,crossing over,"\ker(e^{\rho_i}_{(0)})"' near end] && V_L^\h \ar[hook,from=uu,crossing over,"\ker(e^{\rho_i}_{(0)})"' near end]
\end{tikzcd}}
\end{equation*}

The horizontal arrows describe infinite simple-current extensions, with the arrows labelled by the corresponding abelian group. For convenience, we record the lattice and Heisenberg rank of the Heisenberg-lattice vertex (super)algebras on the bottom of the diagram.
\begin{equation*}
\adjustbox{scale=0.75,center}{%
\begin{tikzcd}[row sep={35,between origins}, column sep={60,between origins}]
& (0,3N{-}2M) \ar[hook,"\Z^{N-M}"]{rr}\ar[<-,"\text{BRST}",sloped]{dl} && (N{-}M,3N{-}2M) \ar[hook,"\Z^{N-M}"]{rr}\ar[<-,"\text{BRST}",sloped]{dl} && (2N{-}2M,3N{-}2M) \ar[hook,"\Z^M"]{rr}\ar[<-,"\text{BRST}",sloped]{dl} && (2N{-}M,3N{-}2M) \ar[<-,"\text{BRST}",sloped]{dl} \\
(0,3N) \ar[hook,"\Z^N"]{rr} && (N,3N) \ar[hook,"\Z^{N-M}"]{rr} && (2N{-}M,3N) \ar[hook,"\Z^M"]{rr} && (2N,3N)
\end{tikzcd}}
\end{equation*}


\subsection{Simplicity and Simple-Current Extensions}\label{sec:simple}

In \autoref{sec:vanish} we showed that the minimal hypertoric \voa{} $V(\Delta)$ and the boundary hypertoric \svoa{}s $V_\mathrm{min}(\Delta)$ are simple-current extensions of tensor products of singlet and Heisenberg \voa{}s (see also \cite{BCDN23,Niu23}). This implies:
\begin{prop}\label{prop:symplicity}
The minimal and boundary hypertoric (vertex)operator superalgebras $V_\mathrm{min}(\Delta)$ and $V(\Delta)$ are simple vertex operator (super)algebras.
\end{prop}
\begin{proof}
Both the singlet \voa{} $\mathcal{M}(2)$ \cite{Ada03} and the Heisenberg \voa{} are simple, and hence so are tensor products of them. The vertex operator (super)algebras in question are then simple-current extensions of such a tensor product. Hence, they must also be simple. The simplicity of simple-current extensions is well-established in the literature in the finite-index case, but the argument carries over to countable extensions without complication.
\end{proof}
The simplicity of $V_\mathrm{min}(\Delta)$ was also proved in \cite{Kuw21} using a commutant description \cite{Lin09}.

The simplicity of the boundary hypertoric \svoa{}s also follows from \cite{BG26}. They show that a cohomological \svoa{} such as $V(\Delta)$ is weakly graded-unitary because $M=\mathcal{D}^\mathrm{ch}(T^*V)\otimes\mathcal{C}\ell(\Pi T^*V)$ has this property, which is preserved by the relative BRST cohomology in good cases. Moreover, $V(\Delta)$ is of CFT-type by construction. Then, by Remark~2.18 in \cite{BG26}, $V(\Delta)$ is simple, self-contragredient and of CFT-type. The last statement also holds for $V_\mathrm{min}(\Delta)$. We summarise this is the following proposition:
\begin{prop}\label{prop:properties}
The minimal and boundary hypertoric (vertex)operator superalgebras $V_\mathrm{min}(\Delta)$ and $V(\Delta)$ are simple, self-contragredient and of CFT type.
\end{prop}
The $L_0$-grading (see also \autoref{cor:conformalvector} below) of $V_\mathrm{min}(\Delta)$ and $V(\Delta)$ takes values in $(1/2)\N$, but in the case of $V(\Delta)$ there is in general no relation (``correct statistics'') between the parity and integrality of the $L_0$-grading. These properties are directly inherited from $M$.

\begin{rem}\label{rem:ext}
The results in \autoref{sec:vanish} can also be used to describe the boundary hypertoric \svoa{}s $V(\Delta)$ as a $\Z^N$-simple-current extension of the minimal one, or more precisely of $\smash{V_\mathrm{min}(\Delta)\otimes\pi^{T^E}}$. This allows us to relate the results in this work to those in \cite{Kuw21}. Indeed,
\begin{align*}
V(\Delta)&=\bigoplus_{l_1,\dots,l_N\in\Z}(V_\mathrm{min}(\Delta)\otimes\pi^{T^E})\\
&\quad\boxtimes\Bigl(M_{\sum_{j=1}^Nl_jp^\bot(\rho_j)}^{p^\bot(R)}\otimes\pi_{\sum_{i=1}^{N-M}\sum_{j=1}^N((E^tE)^{-1}E^t)_{ij}l_j(\sigma_i^E+\tau_i^E)}^{S^E\oplus T^E}\Bigr),
\end{align*}
with parity $(-1)^{l_1+\dots+l_N}$, recalling that
\begin{equation*}
V_\mathrm{min}(\Delta)\otimes\pi^{T^E}=\bigoplus_{k_1,\dots,k_{N-M}\in\Z}M_{\sum_{j=1}^N\sum_{i=1}^{N-M}E_{ji}k_i}^{p^\bot(R)}\otimes\pi_{\sum_{i=1}^{N-M}k_i\sigma_i^E}^{S^E}\otimes\pi_0^{T^E},
\end{equation*}
and where the fusion product $\boxtimes$ is taken over $\mathcal{M}(2)^{\otimes N}\otimes\pi_0^{S^E\oplus T^E}$, the underlying singlet-Heisenberg \voa{}. Writing $V_\mathrm{min}(\Delta)$ as an extension of $\smash{V_\mathrm{min}(\Delta)\otimes\pi^{T^E}}$ corresponds to the lattice quotient
\begin{align*}
L^\bot/K^\bot&=\langle\{-\rho_i^E+\sigma_i^E\}_{i=1}^{N-M}\cup\{-\rho_i+\sigma_i+\tau_i\}_{i=1}^N\rangle_\Z\bigm/\langle\{-\rho_i^E+\sigma_i^E\}_{i=1}^{N-M}\rangle_\Z\\
&\cong\langle\{-\rho_i+\sigma_i+\tau_i\}_{i=1}^N\rangle_\Z\cong\Z^N.
\end{align*}
\end{rem}

\begin{rem}\label{rem:intermediatesce}
We can also view $V(\Delta)$ as a $\Z^M$-simple-current extension of the intermediate \svoa{} $V_\mathrm{min}(\Delta)\otimes V_{J^\bot}$ from \autoref{rem:intermediate}. Indeed, the corresponding lattice quotient is
{\allowdisplaybreaks
\begin{align*}
L^\bot/(K^\bot\oplus J^\bot)&=\langle\{-\rho_i^E+\sigma_i^E\}_{i=1}^{N-M}\cup\{-\rho_i+\sigma_i+\tau_i\}_{i=1}^N\rangle_\Z\\
&\quad\bigm/\langle\{-\rho_i^E+\sigma_i^E\}_{i=1}^{N-M}\cup\{\tau_i^E\}_{i=1}^{N-M}\rangle_\Z\\
&=\langle\{-\rho_i^E+\sigma_i^E\}_{i=1}^{N-M}\cup\{-\rho_i+\sigma_i+\tau_i\}_{i=1}^N\rangle_\Z\\
&\quad\bigm/\langle\{-\rho_i^E+\sigma_i^E\}_{i=1}^{N-M}\cup\{-\rho_i^E+\sigma_i^E+\tau_i^E\}_{i=1}^{N-M}\rangle_\Z\\
&\cong\langle\{-\rho_i+\sigma_i+\tau_i\}_{i=1}^N\rangle_\Z\bigm/\langle\{-\rho_i^E+\sigma_i^E+\tau_i^E\}_{i=1}^{N-M}\rangle_\Z\\
&=\Z^N/E\Z^{N-M}\cong\Delta\Z^N=\Z^M
\end{align*}
}%
by the exactness of the sequence $\smash{\{0\}\to\Z^{N-M}\overset{E}{\to}\Z^N\overset{\Delta}{\to}\Z^M\to\{0\}}$. Then
\begin{align}
\label{eq:V-dec}
V(\Delta)&=\bigoplus_{(l_1,\dots,l_N)^t\in\Z^N/E\Z^{N-M}}\bigl(V_\mathrm{min}(\Delta)\otimes V_{J^\bot}\bigr)\\
\nonumber
&\quad\boxtimes\Bigl(M_{\sum_{j=1}^Nl_jp^\bot(\rho_j)}^{p^\bot(R)}\otimes\pi_{\sum_{i=1}^{N-M}\sum_{j=1}^N((E^tE)^{-1}E^t)_{ij}l_j(\sigma_i^E+\tau_i^E)}^{S^E\oplus T^E}\Bigr),
\end{align}
recalling that
\begin{align*}\nonumber
&V_\mathrm{min}(\Delta)\otimes V_{J^\bot}\\
&=\bigoplus_{\substack{k_1,\dots,k_{N-M}\in\Z\\l_1,\dots,l_{N-M}\in\Z}}M_{\sum_{j=1}^N\sum_{i=1}^{N-M}E_{ji}k_ip^\bot(\rho_j)}^{p^\bot(R)}
\otimes\pi_{\sum_{i=1}^{N-M}k_i\sigma_i^E}^{S^E}\otimes\pi_{\sum_{i=1}^{N-M}l_i\tau_i^E}^{T^E}.
\end{align*}
With the exact sequence, we can restate the sum $(l_1,\dots,l_N)^t\in\Z^N/E\Z^{N-M}$ as running over $s(\Z^M)\subset\Z^N$, where $s\colon\Z^M\to\Z^N$ is a choice of section of the map $\smash{\Z^N\overset{\Delta}{\to}\Z^M}$.

Based on the alternative presentation in \autoref{rem:altrep}, we can also write
\begin{align*}
V(\Delta)&=\bigoplus_{l_1,\dots,l_M\in\Z^M}\Biggl(\bigoplus_{\substack{m_1,\dots,m_N\in\Z\\\Delta m=l}}\!\!M_{\sum_{i=1}^Nm_ip^\bot(\rho_i)}^{p^\bot(R)}\otimes\pi_{\sum_{i=1}^{N-M}\sum_{j=1}^N((E^tE)^{-1}E^t)_{ij}m_j\sigma_i^E}^{S^E}\Biggr)\\
&\quad\otimes\Biggl(\bigoplus_{\substack{n_1,\dots,n_N\in\Z\\\Delta n=l}}\pi_{\sum_{i=1}^{N-M}\sum_{j=1}^N((E^tE)^{-1}E^t)_{ij}n_j\tau_i^E}^{T^E}\Biggr),
\end{align*}
with parity given by $(-1)^{n_1+\dots+n_N}$, recalling that $\Delta\Z^N=\Z^M$. This is an extension of $V_\mathrm{min}(\Delta)\otimes V_{J^\bot}$, which is recovered by collapsing the outer sum by setting $\smash{l_1=\dots=l_M=0}$ (cf.\ \autoref{rem:intermediate}).
\end{rem}


\subsection{Conformal Structures}\label{sec:conf-vect}

Recall from \autoref{sec:freefield} that we can equip the vertex superalgebra $\tilde{C}=M\otimes\mathcal{C}\ell^\bullet(\Pi T^*\g)=\mathcal{D}^\mathrm{ch}(T^*V)\otimes\mathcal{C}\ell(\Pi T^*V)\otimes\mathcal{C}\ell^\bullet(\Pi T^*\g)$ in the input of the BRST reduction with a $2N$-parameter family of conformal structures defined by
\begin{align*}
\omega^{a,b}=\sum_{i=1}^N\bigl(a_i\partial x_i\,y_i-(1-a_i)x_i\partial y_i\bigr)+\sum_{i=1}^N\bigl(b_i\partial\psi_i\,\phi_i-(1-b_i)\psi_i\partial\phi_i\bigr)+\sum_{i=1}^M\partial c_i\, b_i
\end{align*}
for $a,b\in\C^N$. Only some of these will descend to the BRST cohomology. Indeed, a direct computation shows:
\begin{prop}\label{prop:conformalkernel}
The conformal vector $\omega^{a,b}$ for $a,b\in\C^N$ on $\tilde{C}$ lies in $C$, and it lies in $\ker(d)$ if and only if
\begin{equation*}
\Delta(a-b)=0\in\C^M.
\end{equation*}
In that case, $\omega^{a,b}$ defines a conformal structure (of the same central charge) in the relative cohomology $\smash{V(\Delta)=H_\mathrm{BRST}^{\infty/2+\bullet}(\g,M)=H^\bullet(C,d)}$.
\end{prop}
We remark that the conformal structure on the ghost Clifford vertex superalgebra $\mathcal{C}\ell^\bullet(\Pi T^*\g)$ has to be chosen exactly as it is. That is, even though $\mathcal{C}\ell^\bullet(\Pi T^*\g)$ also supports an $M$-parameter family of conformal structures, only one of them will yield a conformal vector in $\ker(d)$.

Further note that many of these conformal structures will define the same conformal structure in the cohomology, namely when they differ by a coboundary, $\smash{\omega^{a,b}-\omega^{a',b'}\in\im(d)}$.

\medskip

In this paper, unless mentioned otherwise, we specialise to the \emph{natural} choice of conformal structure on $\tilde{C}=M\otimes\mathcal{C}\ell^\bullet(\Pi T^*\g)=\mathcal{D}^\mathrm{ch}(T^*V)\otimes\mathcal{C}\ell(\Pi T^*V)\otimes\mathcal{C}\ell^\bullet(\Pi T^*\g)$, namely the one for $a=b=(1/2,\dots,1/2)^t$, which satisfies the condition in \autoref{prop:conformalkernel}. (In fact, we have already done so in our discussion above, cf.\ \autoref{prop:properties}.) Specifically, we consider the conformal vector
\begin{equation}\label{eq:conf}
\begin{split}
\omega&=\sum_{i=1}^N(\partial x_i\,y_i-x_i\partial y_i)/2+\sum_{i=1}^N(\partial\psi_i\,\phi_i-\psi_i\partial\phi_i)/2+\sum_{i=1}^M\partial c_i\, b_i\\
&=\sum_{i=1}^N(\partial x_i\,y_i-x_i\partial y_i)/2+\sum_{i=1}^N\tau_i^2/2+\sum_{i=1}^M\partial c_i\, b_i,
\end{split}
\end{equation}
where the Heisenberg subalgebra $\pi^T\subset\mathcal{C}\ell(\Pi T^*V)$ is generated by $\tau_i=\psi_i\phi_i$ for $i=1,\dots,N$. With regard to this conformal vector, the free-field generators in $\tilde{C}$ have $L_0$-weights
\begin{equation*}
\wt(x_j)=\wt(y_j)=1/2,\quad\wt(\psi_j)=\wt(\phi_j)=1/2,\quad\wt(c_i)=0,\wt(b_i)=1
\end{equation*}
for $j=1,\dots,N$ and $i=1,\dots,M$. This makes $\tilde{C}$ (and $C$) a \svoa{} ($\frac{1}{2}\N$-graded by $L_0$-weights) of central charge $\smash{c=-N+N-2M=-2M}$. Then, \autoref{prop:conformalkernel} implies:
\begin{cor}\label{cor:conformalvector}
The conformal vector \eqref{eq:conf} satisfies $\omega\in C\cap\ker(d)$. Hence, it defines a conformal vector, again of central charge $c=-2M$, in the relative cohomology $V(\Delta)=H_\mathrm{BRST}^{\infty/2+\bullet}(\g,M)=H^\bullet(C,d)$.
\end{cor}

\smallskip

Corresponding to $M\supseteq\mathcal{D}^\mathrm{ch}(T^*V)\otimes\pi^T=\mathcal{D}^\mathrm{ch}(T^*V)\otimes\pi^{T^\Delta}\otimes\pi^{T^E}=M_\mathrm{min}\otimes\pi^{T^E}$, we decompose $\omega$ as
\begin{equation*}
\omega=\omega_\mathrm{min}+\omega_{T^E}
\end{equation*}
with
\begin{align*}
\omega_\mathrm{min}&=\sum_{i=1}^N(\partial x_i\,y_i-x_i\partial y_i)/2+\sum_{i=1}^M\tau_i^\Delta\tau^i_\Delta/2+\sum_{i=1}^M\partial c_i\, b_i,\\
\omega_{T^E}&=\sum_{i=1}^{N-M}\tau_i^E\tau^i_E/2.
\end{align*}
Here, $\{\tau^i_\Delta\}_{i=1}^M$ and $\{\tau^i_E\}_{i=1}^{N-M}$ are dual bases for $\{\tau_i^\Delta\}_{i=1}^M$ and $\{\tau_i^E\}^{N-M}$ with respect to the restriction of the positive-definite bilinear form $\langle\cdot,\cdot\rangle$ on $T$ to the orthogonal subspaces $T^\Delta$ and $T^E$, respectively. That is, $\smash{\langle\tau^i_\Delta,\tau_j^\Delta\rangle=\delta_{ij}}$ for $i,j=1\dots,M$ and analogously for $T^E$. Explicitly, $\tau^i_\Delta$ is defined like $\tau_i^\Delta$, but with $\Delta$ replaced by $(\Delta\Delta^t)^{-1}\Delta$, and similarly for $\tau^i_E$.

\begin{rem}\label{rem:confext}
It is not difficult to see that $\omega_\mathrm{min}$ is a conformal vector for $V_\mathrm{min}(\Delta)$ of central charge $c=-N-M$ (this is the conformal structure in \cite{Kuw21}), $\omega_{T^E}$ one for $\smash{\pi^{T^E}}$ of central charge $c=N-M$ and that (cf.\ \autoref{rem:intermediate})
\begin{equation*}
V_\mathrm{min}(\Delta)\otimes\pi^{T^E}\hookrightarrow V_\mathrm{min}(\Delta)\otimes V_{J^\bot}\hookrightarrow V(\Delta)
\end{equation*}
are conformal extensions with conformal vector $\omega$.
\end{rem}

\smallskip

Although we endow the hypertoric \svoa{} with the conformal structure defined by $\omega$ for $a=b=(1/2,\dots,1/2)^t$, later, in \autoref{sec:global}, it will also be useful to consider other, \emph{auxiliary} conformal structures.

\begin{rem}[Auxiliary Conformal Structures]\label{rem:other-conf-str}
When $0\leq a_i\leq 1$ and $0\leq b_i\leq 1$ for $a,b\in\C^N$ satisfying the condition of \autoref{prop:conformalkernel}, the conformal structure $\omega^{a,b}$ is such that $C$ and hence $V(\Delta)$ is $\R_{\geq0}$-graded by $L_0$-weights.

In particular, for any subset $I\subset\{1,\dots,N\}$ the conformal vectors
\begin{equation*}
\omega^1\coloneqq\omega^{a,a}\quad\text{and}\quad\omega^0\coloneqq\omega^{a',a'}
\end{equation*}
with $a=(\delta_{1\in I},\dots,\delta_{N\in I})$ and $a'=(\delta_{1\notin I},\dots,\delta_{N\notin I})$ satisfy this condition. With respect to $\omega^1$, the free-field generators for $i=1,\dots,N$ have $L_0$-weights
\begin{align*}
\wt(x_i)=0,\wt(y_i)=1,\quad\wt(\psi_i)=0,\wt(\phi_i)=1
\end{align*}
if $i\in I$ and
\begin{align*}
\wt(x_i)=1,\wt(y_i)=0,\quad\wt(\psi_i)=1,\wt(\phi_i)=0
\end{align*}
if $i\notin I$, and vice versa for $\omega^0$. (The ghosts $c_i$ and $b_i$ have the same weights as before.)
\end{rem}

\smallskip

In the following, we rewrite the conformal vector $\omega$ up to coboundaries in $\im(d)$. This will be useful in \autoref{sec:assvar} below when we determine the associated variety of $V(\Delta)$.
\begin{prop}\label{prop:confvect}
The conformal vector $\omega$ of $V(\Delta)$ is of the form
\begin{equation*}
\omega=\sum_{i=1}^{N}\bigl((x_i\phi_i)(y_i\psi_i)-(y_i\psi_i)(x_i\phi_i)\bigr)/2 +\sum_{i=1}^{N-M}(\sigma_i^E + \tau_i^E)\tau_E^i,
\end{equation*}
where $\sigma_i^E+\tau_i^E\in V(\Delta)$ for all $i=1,\dots,N-M$ and $x_i\phi_i,y_i\psi_i\in V(\Delta)$ for all $i=1,\dots,N$.
\end{prop}
\begin{proof}
The fields $\sigma_i+\tau_i$, $x_i\phi_i$ and $y_i\psi_i$ are in $C\cap\ker(d)$ for all $i=1,\dots,N$, and therefore so are the fields mentioned in the proposition. They hence define fields in the relative cohomology $V(\Delta)$. To show that $\omega$ can be written in the asserted form, notice that for all $i=1,\dots,N$,
\begin{align*}
&(\partial x_i\,y_i-x_i\partial y_i)/2+(\partial\psi_i\, \phi_i-\psi_i\partial \phi_i)/2=(\partial x_i\,y_i-x_i\partial y_i)/2+\tau_i^2/2\\
&=\bigl((x_i\phi_i)(y_i\psi_i)-(y_i\psi_i)(x_i\phi_i)\bigr)/2+(\sigma_i+\tau_i)\tau_i.
\end{align*}
It follows that
\begin{align*}
\omega&=\sum_{i=1}^{N}\bigl((x_i\phi_i)(y_i\psi_i)-(y_i\psi_i)(x_i\phi_i)\bigr)/2+\sum_{i=1}^N (\sigma_i+\tau_i)\tau_i+\sum_{i=1}^M\partial c_i\,b_i\\
&=\sum_{i=1}^{N}\bigl((x_i\phi_i)(y_i\psi_i)-(y_i\psi_i)(x_i\phi_i)\bigr)/2+\sum_{i=1}^{N-M}(\sigma^E_i+\tau^E_i)\tau_E^i\\
&\quad+\sum_{i=1}^M\bigl((\sigma^\Delta_i+\tau^\Delta_i)\tau_\Delta^i+\partial c_i\,b_i\bigr).
\end{align*}
Using that, for all $i=1,\dots,M$,
\begin{equation*}
(\sigma^\Delta_i+\tau^\Delta_i)\tau_\Delta^i+\partial c_i\,b_i=d(\tau^i_\Delta b_i)\in\im(d)
\end{equation*}
is a coboundary, the assertion follows.
\end{proof}
Then, by \autoref{rem:confext}, the conformal vectors $\omega_\mathrm{min}$ and $\omega_{T^E}$ of $V_\mathrm{min}(\Delta)$ and $\smash{\pi^{T^E}\subset V_{J^\bot}}$ (with $\omega=\omega_\mathrm{min}+\omega_{T^E}$) are given by
\begin{align*}
\omega_\mathrm{min}&=\sum_{i=1}^{N}\bigl((x_i\phi_i)(y_i\psi_i)-(y_i\psi_i)(x_i\phi_i)\bigr)/2+\sum_{i=1}^{N-M}(\sigma_i^E+\tau_i^E)^E_i\tau_E^i-\sum_{i=1}^{N-M}\tau_i^E\tau^i_E/2,\\
\omega_{T^E}&=\sum_{i=1}^{N-M}\tau_i^E\tau^i_E/2,
\end{align*}
respectively.

\begin{rem}
The above identities imply that $\omega$, $\omega_\mathrm{min}$ and $\omega_{T^E}$ all have Sugawara-like expressions in $V(\Delta)$, meaning that they can all be written in terms of weight-$1$ fields in $V(\Delta)$. This means that when one describes the strong generators of $V(\Delta)$, one does not need introduce the conformal vector as an additional strong generator.

On the other hand, because in the expression for $\omega_{\mathrm{min}}$ fields in $V_{J^\bot}$ are used, this may not always be the case within $V_\mathrm{min}(\Delta)$ itself. We will see an example of this below in \autoref{sec:kleinian_svoa}. In fact, to be more precise, to make this comment for $\omega_\mathrm{min}$, we should apply the reformulation in the proof of \autoref{prop:confvect} only to the ``$\Delta$-part'' of $\mathcal{D}^\mathrm{ch}(T^*V)$, that is, schematically
\begin{equation*}
\omega_\mathrm{min}=\sum_{i=1}^{M}\bigl((x_i\phi_i)(y_i\psi_i)-(y_i\psi_i)(x_i\phi_i)\bigr)_\Delta/2+\sum_{i=1}^{N-M}(\partial x_i\,y_i-x_i\partial y_i)_E/2.
\end{equation*}
Now, the expression on the left-hand side is still clearly not generated by fields that are in $V_\mathrm{min}(\Delta)$, so the argument persists. In some sense, this is the more natural expression for $\omega_\mathrm{min}$ as it does not depend on fields in $\smash{\pi^{T^E}}$.
\end{rem}


\section{BRST Cohomology Sheaves over Symplectic Resolutions}\label{sec:sheaves}

In this section, following the fermionic construction developed in \cite{AKM23}, we define a sheaf $\smash{\mathcal{V}_{\delta,\Delta}^\hbar}$ of $\hbar$-adic (i.e.\ microlocalised) \svoa{}s over the smooth (projective) hypertoric variety $\smash{Y_\delta(\Delta)}$ by quantum Hamiltonian reduction. In contrast to \cite{Kuw21}, this sheaf will be defined over the hypertoric variety (or a nilpotent thickening of it) rather than over its universal family of Poisson deformations $\smash{\tilde{Y}_\delta(\Delta)}$ (cf.\ \autoref{sec:deformation} and \autoref{sec:super}). This is due to a different anomaly cancellation, replacing the Heisenberg \voa{} with free fermions (see also \cite{Niu23}).

In the end, we will see in \autoref{sec:global} that the global sections $\smash{\mathcal{V}_{\delta,\Delta}^\hbar(Y_\delta(\Delta))}$, after removing the microlocalisation parameter $\hbar$ by considering the $\mathbb{S}$-invariant structure for a certain action by the torus $\mathbb{S}=\C^\times$, coincide with the boundary hypertoric \svoa{} $V(\Delta)$. That is, we show that the natural map
\begin{equation*}
V(\Delta)\to\bigl[\mathcal{V}_{\delta,\Delta}^\hbar(Y_\delta(\Delta))\bigr]^\mathbb{S}
\end{equation*}
is an isomorphism. To prove this, we use a faithfulness theorem from \cite{ADS26}. Similarly, the global sections of the sheaf over $\smash{\tilde{Y}_\delta(\Delta)}$ defined in \cite{Kuw21} reduce to the minimal hypertoric \voa{} $V_\mathrm{min}(\Delta)$, as was shown there.


\subsection{Sheaves of \texorpdfstring{$\hbar$}{ℏ}-Adic Vertex Operator Superalgebras}

We recall basic examples of sheaves of $\hbar$-adic vertex operator (super)algebras. We refer to \cite{AKM15,Kuw21,AKM23} for details of this microlocalisation technique for vertex (super)algebras.

\medskip

Let $V=\C^N$ and $T^*V=V\oplus V^*$ be equipped with the standard symplectic form, as in \autoref{sec:hypertoric}. We denote by $\smash{\mathcal{D}^\mathrm{ch}_{T^*V,\hbar}}$ the microlocalisation of the $\beta\gamma$-system or Weyl vertex algebra on $T^*V$. In particular, $\smash{\mathcal{D}^\mathrm{ch}_{T^*V,\hbar}}$ is a sheaf of $\hbar$-adic \voa{}s on $T^*V$ whose global sections $\smash{\mathcal{D}^\mathrm{ch}_{T^*V,\hbar}(T^*V)=\mathcal{D}^\mathrm{ch}(T^*V)_\hbar}$ are exactly the Weyl vertex algebra on $T^*V$ (see \autoref{sec:freefield}), in the $\hbar$-adic version, i.e.\ with operator product expansions
\begin{equation*}
x_i(z)y_j(w)\sim -\hbar\frac{\delta_{ij}}{(z-w)}
\end{equation*}
for $i,j=1,\dots,N$. Without the microlocalisation parameter $\hbar$, the infinite sums appearing in the operator product expansions of certain localisations, such as those involving both $x_i^{-1}$ and $y_i^{-1}$, would not be well-defined. If one only localises on $V$ rather than $T^*V$, then microlocalisation is not necessary and one arrives at the notion of the chiral de Rham complex on $V$ \cite{MSV99}.

Recall that the $\infty$-jet bundle (or arc space) $\smash{\mathcal{O}_{J_\infty T^*V}\coloneqq J_\infty\mathcal{O}_{T^*V}}$ naturally has the structure of a sheaf of vertex Poisson algebras on $T^*V$. The sheaf $\smash{\mathcal{D}^\mathrm{ch}_{T^*V,\hbar}}$ is a quantisation of $\smash{\mathcal{O}_{J_\infty T^*V}}$ in the sense that, as vertex Poisson algebras,
\begin{equation*}
\mathcal{D}^\mathrm{ch}_{T^*V,\hbar}/\hbar\mathcal{D}^\mathrm{ch}_{T^*V,\hbar}\cong\mathcal{O}_{J_\infty T^*V}.
\end{equation*}

\medskip

Similarly, let $W=(\C^M,\langle\cdot,\cdot\rangle)$ be a vector space equipped with a (typically nondegenerate) symmetric bilinear form $\langle\cdot,\cdot\rangle\colon W\times W\to\C$. We consider the corresponding sheaf $\smash{\mathcal{H}^W_\hbar}$ of $\hbar$-adic \voa{}s on $W^*$ whose global sections $\smash{\mathcal{H}^W_\hbar(W^*)=\pi^W_\hbar}$ are the $\hbar$-adic version of the Heisenberg \voa{} $\pi^W$ on $W$, which has the operator product expansions
\begin{equation*}
h(z)k(w)\sim\hbar^2\frac{\langle h,k\rangle}{(z-w)^2}
\end{equation*}
for $h,k\in W$. Again, $\smash{\mathcal{H}^W_\hbar}$ is a quantisation of the $\infty$-jet bundle $\smash{\mathcal{O}_{J_\infty W^*}}$ of the structure sheaf $\smash{\mathcal{O}_{W^*}}$ in the sense that there is a natural isomorphism of sheaves of vertex Poisson algebras
\begin{equation*}
\mathcal{H}^W_\hbar/\hbar\mathcal{H}^W_\hbar\cong\mathcal{O}_{J_\infty W^*}.
\end{equation*}

\medskip

Finally, given the odd vector space $\Pi T^*V$, we also consider the $\hbar$-adic version $\mathcal{C}\ell(\Pi T^*V)_\hbar$ of the Clifford \svoa{} or $bc$-system $\mathcal{C}\ell(\Pi T^*V)$ (see \autoref{sec:freefield}), with operator product expansions
\begin{equation*}
\psi_i(z)\phi_j(w)\sim\hbar\frac{\delta_{ij}}{(z-w)}
\end{equation*}
for $i=1,\dots,N$. Again, it also comes in the incarnation of the \emph{ghost} Clifford \svoa{} $\mathcal{C}\ell(\Pi T^*\g)_\hbar$ with
operator product expansions
\begin{equation*}
c_i(z)b_j(w)\sim\hbar\frac{\delta_{ij}}{(z-w)}
\end{equation*}
for $i=1,\dots,M$ and with a different choice of conformal structure. It is equipped with the ghost (or cohomological) grading $\mathcal{C}\ell(\Pi T^*\g)_\hbar=\bigoplus_{\bullet\in\Z}\mathcal{C}\ell^\bullet(\Pi T^*\g)_\hbar$.

We note that $\smash{\mathcal{C}\ell(\Pi T^*V)_\hbar/\hbar\mathcal{C}\ell(\Pi T^*V)_\hbar\cong\Lambda^\mathrm{vert}(\Pi T^*V)}$ as vertex Poisson algebras, where $\smash{\Lambda^\mathrm{vert}(\Pi T^*V)}$ is ``the arc space'' over the supervariety $\Pi T^*V$ with $\C[\Pi T^*V]=\Lambda(\Pi T^*V)$ the usual exterior algebra of the (odd) vector space $T^*V$. We say that $\smash{\mathcal{C}\ell(\Pi T^*V)_\hbar}$ quantises $\smash{\Lambda^\mathrm{vert}(\Pi T^*V)}$.


\subsection{Semi-Infinite BRST Reduction}\label{sec:sheafBRST}

First, even though we do not use it here, we recall the construction \cite{Kuw21} of a sheaf of $\hbar$-adic \voa{}s $\smash{\mathcal{W}^\hbar_{\delta,\Delta}}$ on a certain deformation $\smash{\tilde{Y}_\delta(\Delta)}$ of the hypertoric variety $Y_\delta(\Delta)$. Then, by modifying the \emph{anomaly cancellation} (to ensure that the BRST differential closes), we define a closely related sheaf of vertex operator \emph{super}algebras $\smash{\mathcal{V}^\hbar_{\delta,\Delta}}$ on the hypertoric variety $Y_\delta(\Delta)$ itself. While $\smash{\mathcal{W}^\hbar_{\delta,\Delta}}$ recovers the minimal hypertoric \voa{} $V_\mathrm{min}(\Delta)$ in the global sections \cite{Kuw21}, we shall show that $\smash{\mathcal{V}^\hbar_{\delta,\Delta}}$ recovers the boundary hypertoric \svoa{} $V(\Delta)$; see \autoref{sec:global} below.


\subsubsection{Over Poisson Deformation}

Recall from \autoref{sec:deformation} the universal family of Poisson deformations $\smash{\tilde{Y}_\delta(\Delta)}$ of the hypertoric variety $Y_\delta(\Delta)$. In \cite{Kuw21}, a sheaf of $\hbar$-adic \voa{}s $\smash{\mathcal{W}^\hbar_{\delta,\Delta}}$ on $\smash{\tilde{Y}_\delta(\Delta)}$ is constructed (called $\smash{\tilde{\mathcal{D}}_{X,\hbar}^{ch}}$ there) by chiralising the construction of $\smash{\tilde{Y}_\delta(\Delta)}$. We briefly recall this construction, with the slight modification that we take the \emph{relative} BRST cohomology. The last point is not essential, but it removes spurious states in positive cohomological degrees.

Consider the sheaf of $\hbar$-adic \voa{}s $\smash{\mathcal{D}^\mathrm{ch}_{T^*V,\hbar}\hatotimes\mathcal{H}^{\g^*}_\hbar}$ on $T^*V\times\g^*$. Here, in the construction of $\smash{\mathcal{H}^{\g^*}_\hbar}$ (or rather $\smash{\mathcal{H}^{T^\Delta}_\hbar}$, following notation of \autoref{sec:hypertoricsvoa}) we consider the vector space $\g^*$ equipped with the bilinear form $\Delta\Delta^t$ in the basis $\smash{\{\tau^\Delta_i\}_{i=1}^M}$, which encodes the (second order poles of the) operator-product expansion. By restriction, we obtain the sheaf
\begin{equation*}
\mathcal{M}_\mathrm{min,ss}^\hbar\coloneqq\mathcal{D}^\mathrm{ch}_{\mathfrak{X},\hbar}\hatotimes\mathcal{H}^{\g^*}_\hbar\quad\text{on}\quad\mathfrak{X}\times\g^*,
\end{equation*}
where $\mathfrak{X}=(T^*V)_\delta^\mathrm{ss}\subset T^*V$ is the semistable locus (see \autoref{sec:hypertoric}).

The chiral comoment map is the homomorphism of $\hbar$-adic vertex algebras
\begin{equation*}
\tilde\mu^*_\mathrm{ch}\colon V(\g)_\hbar\to\mathcal{M}^\hbar_\mathrm{min,ss},\quad\tilde\mu^*_\mathrm{ch}(a_i)=\sum_{j=1}^N\Delta_{ij}x_jy_j+\tau^\Delta_i
\end{equation*}
for $i=1,\dots,M$, where $V(\g)_\hbar$ is the (commutative) $\hbar$-adic universal affine \voa{} for $\smash{\g=\gl_1^M}$, equipped with the zero bilinear form.

To define the BRST complex, we tensor with the ghost Clifford \svoa{} $\mathcal{C}\ell(\Pi T^*\g)_\hbar$ to define $\smash{\tilde{\mathcal{C}}_\mathrm{min}^\hbar\coloneqq\mathcal{M}^\hbar_\mathrm{min,ss}\hatotimes\mathcal{C}\ell(\Pi T^*\g)_\hbar}$, a sheaf of $\hbar$-adic \svoa{}s on $\mathfrak{X}\times\g^*$. This sheaf has a $\Z$-grading by (ghost) degrees $\smash{\tilde{\mathcal{C}}_\mathrm{min}^\hbar=\prod_{\bullet\in\Z}\tilde{\mathcal{C}}_\mathrm{min}^{\bullet,\hbar}}$, where $\smash{\tilde{\mathcal{C}}_\mathrm{min}^{\bullet,\hbar}\coloneqq\mathcal{M}^\hbar_\mathrm{min,ss}\hatotimes\mathcal{C}\ell^\bullet(\Pi T^*\g)_\hbar}$.

The BRST differential is defined as $\smash{d\coloneqq\frac{1}{\hbar}Q_{(0)}}$, where $Q$ is the odd global section defined in equation~\eqref{eq:Q}. It has ghost degree~$1$. In contrast to \cite{Kuw21}, we consider the relative subsheaf $\smash{\mathcal{C}_\mathrm{min}^\hbar\subset\tilde{\mathcal{C}}_\mathrm{min}^\hbar}$ by taking the kernel under the zero modes of the global sections $b_i$ and $\smash{Q_{(0)}b_i=\bar\mu^*_\mathrm{ch}(a_i)}$ for all $i=1,\dots,M$. Then, for every open set $\tilde{\mathfrak{U}}\subset\mathfrak{X}\times\g^*$, $\smash{(\tilde{\mathcal{C}}_\mathrm{min}^\hbar(\tilde{\mathfrak{U}}),d)}$ is a cochain complex and $\smash{(\mathcal{C}_\mathrm{min}^\hbar(\tilde{\mathfrak{U}}),d)}$ is a subcomplex. The relative BRST cohomology is then defined by
\begin{equation*}
H_\mathrm{BRST}^{\infty/2+\bullet}(\g,\mathcal{M}_\mathrm{min,ss}^\hbar(\tilde{\mathfrak{U}}))\coloneqq H^\bullet(\mathcal{C}_\mathrm{min}^\hbar(\tilde{\mathfrak{U}}),d)
\end{equation*}
for $\tilde{\mathfrak{U}}\subset\mathfrak{X}\times\g^*$ open, a presheaf of $\hbar$-adic \voa{}s on $\mathfrak{X}\times\g^*$.

Then, one defines a cohomology sheaf on $\smash{\tilde{Y}_\delta(\Delta)}$ as follows. For an open subset $\smash{\tilde{U}\subset\tilde{Y}_\delta(\Delta)}$, let $\tilde{\mathfrak{U}}\subset\mathfrak{X}\times\g^*$ be open such that $\tilde{\mathfrak{U}}$ is closed under the $G$-action and $\smash{\tilde{p}^{-1}(\tilde{U})=\tilde{\mathfrak{U}}\cap\tilde{\mu}^{-1}(0)}$ with the projection map $\smash{\tilde{p}\colon\mathfrak{X}\times\g^*\to\tilde{Y}_\delta(\Delta)}$. Then one can show that the presheaf $\smash{\tilde{\mathfrak{U}}\mapsto H_\mathrm{BRST}^{\infty/2+\bullet}(\g,\mathcal{M}_\mathrm{min,ss}^\hbar(\tilde{\mathfrak{U}}))}$ on $\mathfrak{X}\times\g^*$ is supported on $\smash{\tilde{\mu}^{-1}(0)}$ and hence does not depend on the choice of $\tilde{\mathfrak{U}}$ for a given $\tilde{U}$. Finally, the relative cohomology sheaf
\begin{equation*}
H_\mathrm{BRST}^{\infty/2+\bullet}(\g,\mathcal{M}_\mathrm{min,ss}^\hbar)\quad\text{over}\quad\tilde{Y}_\delta(\Delta)
\end{equation*}
is defined as the sheafification of the presheaf $\smash{\tilde{U}\mapsto H_\mathrm{BRST}^{\infty/2+\bullet}(\g,\mathcal{M}_\mathrm{min,ss}^\hbar(\tilde{\mathfrak{U}}))}$ on $\smash{\tilde{Y}_\delta(\Delta)}$. The sheaf $\smash{\mathcal{W}^\hbar_{\delta,\Delta}}$ of $\hbar$-adic \voa{}s on $\smash{\tilde{Y}_\delta(\Delta)}$ is then defined as the zeroth cohomology
\begin{equation*}
\mathcal{W}^\hbar_{\delta,\Delta}\coloneqq H_\mathrm{BRST}^{\infty/2+0}(\g,\mathcal{M}_\mathrm{min,ss}^\hbar).
\end{equation*}
It is shown in Proposition~5.10 of \cite{Kuw21} that the negative cohomologies vanish. Because we are taking the relative cohomology, the positive cohomologies vanish as well.

Finally, we note that the $\hbar$-adic analogue of the conformal vector from \autoref{rem:confext} (see \cite{Kuw21} for details) descends to a global section on $\mathcal{W}^\hbar_{\delta,\Delta}$, which then becomes a sheaf of $\hbar$-adic \voa{}s of central charge $c=-N-M$ on $\smash{\tilde{Y}_\delta(\Delta)}$.

\medskip

Based on the local trivialisation and the affine open covering $\smash{\tilde{Y}_\delta(\Delta)=\bigcup_J\tilde{U}_J}$ with $\tilde{U}_J\cong T^*\C^{N-M}\times\g^*$ in \autoref{sec:localtrivial}, one can show that
\begin{equation}\label{eq:free-field}
\mathcal{W}^\hbar_{\delta,\Delta}(\tilde{U}_J)=H_\mathrm{BRST}^{\infty/2+0}(\g,\mathcal{M}_\mathrm{min,ss}^\hbar(\mathfrak{U}_J\times\g^*))\cong\mathcal{D}^\mathrm{ch}(T^*\C^{N-M})_\hbar\hatotimes\pi^\g_\hbar
\end{equation}
(again, $\smash{\pi^\g_\hbar}$ is called $\smash{\pi^{T^\Delta}_\hbar}$ in the notation of \autoref{sec:hypertoricsvoa}), which as a $\C[[\hbar]]$-module is $\smash{\mathcal{O}_{J_\infty\tilde{Y}_\delta(\Delta)}(\tilde{U}_J)[[\hbar]]}$. However, this isomorphism depends on the local trivialisation and does not induce an isomorphism of sheaves. But, we obtain an isomorphism of sheaves of vertex Poisson algebras
\begin{equation}\label{eq:minimal_quant}
\mathcal{W}^\hbar_{\delta,\Delta}/\hbar\mathcal{W}^\hbar_{\delta,\Delta}\cong\mathcal{O}_{J_\infty\tilde{Y}_\delta(\Delta)}.
\end{equation}
That is, the sheaf $\mathcal{W}^\hbar_{\delta,\Delta}$ of $\hbar$-adic \voa{}s quantises the sheaf of vertex Poisson algebras $\mathcal{O}_{J_\infty\tilde{Y}_\delta(\Delta)}$.

\medskip

The $\hbar$-adic \voa{} of global sections $\mathcal{W}^\hbar_{\delta,\Delta}(\tilde{Y}_\delta(\Delta))$ is of particular interest. It will turn out that, after setting $\hbar=1$ in a certain sense by using a conical torus action (cf.\ the notion of conical symplectic resolution; see \autoref{sec:conical} and \cite{BLPW16a,BLPW16b}), it coincides with the minimal hypertoric \voa{} $V_\mathrm{min}(\Delta)$.

There is an action of a one-dimensional torus $\mathbb{S}=\C^\times$ on $\mathfrak{X}\times\g^*$, such that the corresponding action on the structure sheaf $\smash{\mathcal{O}_{\mathfrak{X}\times\g^*}=\mathcal{O}_\mathfrak{X}\otimes_\C\mathcal{O}_{\g^*}}$ has weights
\begin{equation*}
\wt_\mathbb{S}(x_j)=\wt_\mathbb{S}(y_j)=1/2\quad\text{and}\quad\wt_\mathbb{S}(\tau_i^\Delta)=1
\end{equation*}
for $j=1,\dots,N$ and $i=1,\dots,M$, where we recall that $\{\tau_1^\Delta,\dots,\tau_M^\Delta\}$ is the standard basis of $\g$. Under this action, the Poisson bracket on $\mathcal{O}_{\mathfrak{X}\times\g^*}$ is homogeneous of weight $-1$. Since the $\mathbb{S}$-action commutes with the $G$-action, there is an induced $\mathbb{S}$-action on $\tilde{Y}_\delta(\Delta)$.

Similarly, there is an equivariant $\mathbb{S}$-action on the sheaf $\smash{\mathcal{M}_\mathrm{min,ss}^\hbar=\mathcal{D}^\mathrm{ch}_{\mathfrak{X},\hbar}\hatotimes\mathcal{H}^{\g^*}_\hbar}$ over $\C$ such that the weights of the generating fields are verbatim those above and $\wt_\mathbb{S}(\partial)=0$ and $\wt_\mathbb{S}(\hbar)=1$. The operator product expansions are homogeneous with respect to this $\mathbb{S}$-action. We extend this action to the BRST complex $\smash{\mathcal{C}_\mathrm{min}^\hbar\subset\tilde{\mathcal{C}}_\mathrm{min}^\hbar}$ by $\wt_\mathbb{S}(c_i)=0$ and $\wt_\mathbb{S}(b_i)=1$. Hence, the BRST differential $\smash{d=\frac{1}{\hbar}Q_{(0)}}$ is homogeneous of weight~$0$, so the cohomology sheaf $\smash{H_\mathrm{BRST}^{\infty/2+\bullet}(\g,\mathcal{M}_\mathrm{min,ss}^\hbar)}$ is also equipped with the induced equivariant $\mathbb{S}$-action over $\tilde{Y}_\delta(\Delta)$.

By looking at the affine open covering $\smash{\tilde{Y}_\delta(\Delta)=\bigcup_J\tilde{U}_J}$, one can show that the global sections $\smash{\mathcal{W}^\hbar_{\delta,\Delta}(\tilde{Y}_\delta(\Delta))}$ of the zeroth cohomology $\smash{\mathcal{W}^\hbar_{\delta,\Delta}=H_\mathrm{BRST}^{\infty/2+0}(\g,\mathcal{M}_\mathrm{min,ss}^\hbar)}$ are a direct product of weight spaces for $\mathbb{S}$ (cf.\ \autoref{lem:lowertruncated} below):
\begin{equation*}
\mathcal{W}^\hbar_{\delta,\Delta}(\tilde{Y}_\delta(\Delta))=\prod_{m\in(1/2)\N}\mathcal{W}^\hbar_{\delta,\Delta}(\tilde{Y}_\delta(\Delta))_m
\end{equation*}
with $\mathcal{W}^\hbar_{\delta,\Delta}(\tilde{Y}_\delta(\Delta))_0=\C\vac$, in which we consider the direct sum
\begin{equation*}
\mathcal{W}^\hbar_{\delta,\Delta}(\tilde{Y}_\delta(\Delta))_\mathrm{fin}\coloneqq\bigoplus_{m\in(1/2)\N}\mathcal{W}^\hbar_{\delta,\Delta}(\tilde{Y}_\delta(\Delta))_m.
\end{equation*}
This subspace is a $\C[\hbar]$-module since the weights are nonnegative and $\wt_\mathbb{S}(\hbar)=1$, and it is preserved by operator product expansions. Finally, we set
\begin{equation*}
\tilde{V}^\delta_\mathrm{min}(\Delta)\coloneqq\bigl[\mathcal{W}^\hbar_{\delta,\Delta}(\tilde{Y}_\delta(\Delta))\bigr]^\mathbb{S}\coloneqq\mathcal{W}^\hbar_{\delta,\Delta}(\tilde{Y}_\delta(\Delta))_\mathrm{fin}\big/(\hbar-1),
\end{equation*}
the quotient space by the ideal generated by $\hbar-1$, which is a (non-$\hbar$-adic) \voa{}.

\medskip

We can also consider the (nonlocalised) $\hbar$-adic \voa{}
\begin{equation*}
H_\mathrm{BRST}^{\infty/2+0}(\g,\mathcal{M}_\mathrm{min,ss}^\hbar(\mathfrak{X}\times\g^*))
\end{equation*}
and reduce it in the same way to a \voa{}
\begin{align*}
&\bigl[H_\mathrm{BRST}^{\infty/2+0}(\g,\mathcal{M}_\mathrm{min,ss}^\hbar(\mathfrak{X}\times\g^*))\bigr]^\mathbb{S}=\bigl[H^0(\mathcal{C}_\mathrm{min}^\hbar(\mathfrak{X}\times\g^*),d)\bigr]^\mathbb{S}\\
&=H^0(C_\mathrm{min},d)=H_\mathrm{BRST}^{\infty/2+0}(\g,M_\mathrm{min})=V_\mathrm{min}(\Delta),
\end{align*}
which gives the minimal hypertoric \voa{} constructed in \autoref{sec:hypertoricsvoa}.

Now, there is a homomorphism of $\hbar$-adic \voa{}s
\begin{equation*}
H_\mathrm{BRST}^{\infty/2+0}(\g,\mathcal{M}_\mathrm{min,ss}^\hbar(\mathfrak{X}\times\g^*))\to\mathcal{W}^\hbar_{\delta,\Delta}(\tilde{Y}_\delta(\Delta))=H_\mathrm{BRST}^{\infty/2+0}(\g,\mathcal{M}_\mathrm{min,ss}^\hbar)(\tilde{Y}_\delta(\Delta)).
\end{equation*}
This descends to a homomorphism of \voa{}s
\begin{equation}\label{eq:minimal_iso}
V_\mathrm{min}(\Delta)\to\tilde{V}^\delta_\mathrm{min}(\Delta)=\bigl[\mathcal{W}^\hbar_{\delta,\Delta}(\tilde{Y}_\delta(\Delta))\bigr]^\mathbb{S}.
\end{equation}
One of the main results of \cite{Kuw21} is that both these maps are isomorphisms. But note that the injectivity follows immediately from the simplicity of $V_\mathrm{min}(\Delta)$. In \autoref{thm:glob-sec}, we show the analogous statement for our fermionic analogue of the construction in \cite{Kuw21}.

\medskip

We end the discussion by noting, in view of \autoref{sec:assvar} below, that it was shown in \cite{Kuw21} that there is an injective homomorphism of Poisson algebras
\begin{equation*}
R_{V_\mathrm{min}(\Delta)}\hookrightarrow\mathcal{O}_{\tilde{Y}_\delta(\Delta)}(\tilde{Y}_\delta(\Delta))=\C[\tilde{Y}_\delta(\Delta)];
\end{equation*}
see \autoref{prop:toshiro-subalg}. Here, $R_{V_\mathrm{min}(\Delta)}$ denotes the $C_2$-algebra of $V_\mathrm{min}(\Delta)$, which we introduce below. The argument uses properties of the sheaf $\smash{\mathcal{W}^\hbar_{\delta,\Delta}}$ and the isomorphism~\eqref{eq:minimal_iso}.


\subsubsection{Over Supervariety}\label{sec:oursheaf}

In the following, we shall present a slightly different construction of a sheaf $\smash{\mathcal{V}^\hbar_{\delta,\Delta}}$ of $\hbar$-adic vertex operator \emph{super}algebras on the hypertoric variety $\smash{Y_\delta(\Delta)}$ itself, in analogy to \cite{AKM23}.

We define the sheaf $\smash{\mathcal{V}^\hbar_{\delta,\Delta}}$ of $\hbar$-adic \svoa{}s on $Y_\delta(\Delta)$ by chiralising the construction of the sheaf of superalgebras $\smash{\bar{\mathcal{O}}_{Y_\delta(\Delta)}}$ in \autoref{sec:super}, following the construction in \cite{AKM23}.

Consider the sheaf of $\hbar$-adic \svoa{}s $\smash{\mathcal{D}^\mathrm{ch}_{T^*V,\hbar}\hatotimes\mathcal{C}\ell(\Pi T^*V)_\hbar}$ on $T^*V$. By restriction, we obtain the sheaf
\begin{equation*}
\mathcal{M}_\mathrm{ss}^\hbar\coloneqq\mathcal{D}^\mathrm{ch}_{\mathfrak{X},\hbar}\hatotimes\mathcal{C}\ell(\Pi T^*V)_\hbar\quad\text{on}\quad\mathfrak{X}.
\end{equation*}

The chiral comoment map is the homomorphism of $\hbar$-adic vertex superalgebras
\begin{equation*}
\bar\mu^*_\mathrm{ch}\colon V(\g)_\hbar\to\mathcal{M}^\hbar_\mathrm{ss},\quad\bar\mu^*_\mathrm{ch}(a_i)=\sum_{j=1}^N\Delta_{ij}(x_jy_j+\psi_j\phi_j)
\end{equation*}
for $i=1,\dots,M$, where $\smash{V(\g)_\hbar}$ is the (commutative) $\hbar$-adic universal affine \voa{} for $\smash{\g=\gl_1^M}$, equipped with the zero bilinear form.

To define the BRST complex, we tensor with the ghost Clifford \svoa{} $\smash{\mathcal{C}\ell(\Pi T^*\g)_\hbar}$ to define $\smash{\tilde{\mathcal{C}}^\hbar\coloneqq\mathcal{M}^\hbar_\mathrm{ss}\hatotimes\mathcal{C}\ell(\Pi T^*\g)_\hbar}$, a sheaf of $\hbar$-adic \svoa{}s on $\mathfrak{X}$. This sheaf has a $\Z$-grading by (ghost) degrees $\smash{\tilde{\mathcal{C}}^\hbar=\prod_{\bullet\in\Z}\tilde{\mathcal{C}}^{\bullet,\hbar}}$, where $\smash{\tilde{\mathcal{C}}^{\bullet,\hbar}\coloneqq\mathcal{M}^\hbar_\mathrm{ss}\hatotimes\mathcal{C}\ell^\bullet(\Pi T^*\g)_\hbar}$.

The BRST differential (of ghost degree~$1$) is defined as $\smash{d\coloneqq\frac{1}{\hbar}Q_{(0)}}$, where $Q$ is the odd global section defined in equation~\eqref{eq:Q}. It is the same BRST differential as in the minimal case under the inclusion $\smash{\pi^{T^\Delta}_\hbar\subset\mathcal{C}\ell(\Pi T^*V)_\hbar}$ given by $\smash{\tau^\Delta_i=\sum_{j=1}^N\Delta_{ij}\psi_j\phi_j}$. We consider the subsheaf $\smash{\mathcal{C}^\hbar\subset\tilde{\mathcal{C}}^\hbar}$ by taking the kernel under the zero modes of the global sections $b_i$ and $\smash{Q_{(0)}b_i=\bar\mu^*_\mathrm{ch}(a_i)}$ for all $i=1,\dots,M$.

For every open set $\mathfrak{U}\subset\mathfrak{X}$, $(\tilde{\mathcal{C}}^\hbar(\mathfrak{U}),d)$ is a cochain complex and $(\mathcal{C}^\hbar(\mathfrak{U}),d)$ is a subcomplex. The relative BRST cohomology is then defined by
\begin{equation*}
H_\mathrm{BRST}^{\infty/2+\bullet}(\g,\mathcal{M}_\mathrm{ss}^\hbar(\mathfrak{U}))\coloneqq H^\bullet(\mathcal{C}^\hbar(\mathfrak{U}),d)
\end{equation*}
for $\smash{\mathfrak{U}\subset\mathfrak{X}}$ open, a presheaf of $\hbar$-adic \svoa{}s on $\mathfrak{X}$.

Then, one defines a cohomology sheaf on $Y_\delta(\Delta)$ as follows. For an open subset $U\subset Y_\delta(\Delta)$, let $\mathfrak{U}\subset\mathfrak{X}$ be open such that $\mathfrak{U}$ is closed under the $G$-action and $p^{-1}(U)=\mathfrak{U}\cap\mu^{-1}(0)$ with the projection map $p\colon\mu^{-1}(0)\cap\mathfrak{X}\to Y_\delta(\Delta)$. Hence, we obtain the presheaf $\smash{\mathfrak{U}\mapsto H_\mathrm{BRST}^{\infty/2+\bullet}(\g,\mathcal{M}_\mathrm{ss}^\hbar(\mathfrak{U}))}$ on $\mathfrak{X}$. Then, the relative cohomology sheaf
\begin{equation*}
H_\mathrm{BRST}^{\infty/2+\bullet}(\g,\mathcal{M}_\mathrm{ss}^\hbar)\quad\text{over}\quad Y_\delta(\Delta)
\end{equation*}
is defined as the sheafification of the presheaf $\smash{U\mapsto H_\mathrm{BRST}^{\infty/2+\bullet}(\g,\mathcal{M}_\mathrm{ss}^\hbar(\mathfrak{U}))}$ on $Y_\delta(\Delta)$. The sheaf $\smash{\mathcal{V}^\hbar_{\delta,\Delta}}$ of $\hbar$-adic \svoa{}s on $Y_\delta(\Delta)$ is then defined as the zeroth cohomology
\begin{equation*}
\mathcal{V}^\hbar_{\delta,\Delta}\coloneqq H_\mathrm{BRST}^{\infty/2+0}(\g,\mathcal{M}_\mathrm{ss}^\hbar).
\end{equation*}
In fact, the global sections of the cohomology sheaf $\smash{H_\mathrm{BRST}^{\infty/2+i}(\g,\mathcal{M}_\mathrm{ss}^\hbar)}$ vanish for $i\neq0$; see \autoref{rem:vanish}.

Finally, we note that the $\hbar$-adic analogue of the conformal vector from \autoref{cor:conformalvector} descends to a global section on $\smash{\mathcal{V}^\hbar_{\delta,\Delta}}$, which then becomes a sheaf of $\hbar$-adic \svoa{}s of central charge $c=-2M$ on $Y_\delta(\Delta)$.

\medskip

Based on the local trivialisation and the affine open covering $\smash{Y_\delta(\Delta)=\bigcup_JU_J}$ with $U_J\cong T^*\C^{N-M}$ in \autoref{sec:localtrivial}, it follows in analogy to \eqref{eq:free-field} from \cite{Kuw21} that $\smash{\mathcal{V}^\hbar_{\delta,\Delta}(U_J)=H_\mathrm{BRST}^{\infty/2+0}(\g,\mathcal{M}_\mathrm{ss}^\hbar(U_J))}$ contains $\mathcal{D}^\mathrm{ch}(T^*\C^{N-M})_\hbar$ generated by the same fields as the corresponding local sections in \eqref{eq:free-field}, complemented by $N$ pairs of free fermions. There is an isomorphism of sheaves of vertex Poisson superalgebras $\mathcal{V}^\hbar_{\delta,\Delta}/\hbar\mathcal{V}^\hbar_{\delta,\Delta}\cong\bar{\mathcal{O}}_{J_\infty Y_\delta(\Delta)}$. That is, the sheaf $\smash{\mathcal{V}^\hbar_{\delta,\Delta}}$ of $\hbar$-adic \svoa{}s quantises the sheaf of vertex Poisson superalgebras $\smash{\bar{\mathcal{O}}_{J_\infty\tilde{Y}_\delta(\Delta)}}$.

\medskip

We turn our attention to the global sections $\mathcal{V}^\hbar_{\delta,\Delta}(Y_\delta(\Delta))$. Ultimately, we want to show that they, after setting $\hbar=1$ in a certain sense by using a conical torus action, coincide with the boundary hypertoric \svoa{} $V(\Delta)$, in parallel to the discussion for $V_\mathrm{min}(\Delta)$. However, our proof in \autoref{sec:global} will follow a different strategy from the one employed in \cite{Kuw21}, namely it will use the faithfulness theorem from \cite{ADS26} (together with the result from \cite{Kuw21}).

There is an action of a one-dimensional torus $\mathbb{S}=\C^\times$ on $\mathfrak{X}$, such that the corresponding action on the structure sheaf $\smash{\mathcal{O}_\mathfrak{X}}$ has weights
\begin{equation*}
\wt_\mathbb{S}(x_j)=\wt_\mathbb{S}(y_j)=1/2
\end{equation*}
for $j=1,\dots,N$. Under this action, the Poisson bracket on $\mathcal{O}_\mathfrak{X}$ is homogeneous of weight $-1$. Since the $\mathbb{S}$-action commutes with the $G$-action, there is an induced $\mathbb{S}$-action on $Y_\delta(\Delta)$. We can extend this $\mathbb{S}$-action to $\bar{\mathcal{O}}_{\mathfrak{X}}\coloneqq\mathcal{O}_{\mathfrak{X}}\otimes\C[\Pi T^*V]$ via
\begin{equation*}
\wt_\mathbb{S}(\psi_j)=\wt_\mathbb{S}(\phi_j)=1/2
\end{equation*}
for $j=1,\dots,N$, which then descends to an $\mathbb{S}$-action on the sheaf $\bar{\mathcal{O}}_{Y_\delta(\Delta)}$ of  superalgebras on $Y_\delta(\Delta)$.

Following Section~5.1 in \cite{AKM23}, there is an equivariant $\mathbb{S}$-action on the sheaf $\smash{\mathcal{M}_\mathrm{ss}^\hbar=\mathcal{D}^\mathrm{ch}_{\mathfrak{X},\hbar}\hatotimes\mathcal{C}\ell(\Pi T^*V)_\hbar}$ over $\C$ such that the weights of the generating fields are verbatim those above and $\wt_\mathbb{S}(\partial)=0$ and $\wt_\mathbb{S}(\hbar)=1$. The operator product expansions are homogeneous with respect to the $\mathbb{S}$-action. We extend this action to the BRST complex $\smash{\mathcal{C}^\hbar\subset\tilde{\mathcal{C}}^\hbar}$ by setting $\wt_\mathbb{S}(c_i)=0$ and $\wt_\mathbb{S}(b_i)=1$. Hence, the BRST differential $\smash{d=\frac{1}{\hbar}Q_{(0)}}$ is homogeneous of weight~$0$, and so the cohomology sheaf $\smash{H_\mathrm{BRST}^{\infty/2+\bullet}(\g,\mathcal{M}_\mathrm{ss}^\hbar)}$ is also equipped with the induced equivariant $\mathbb{S}$-action over $Y_\delta(\Delta)$.

Now, similarly to \cite{Kuw21} (see the discussion above) and \cite{AKM23}, by using the affine open covering $\smash{Y_\delta(\Delta)=\bigcup_JU_J}$, it follows that the global sections $\smash{\mathcal{V}^\hbar_{\delta,\Delta}(Y_\delta(\Delta))}$ of the zeroth cohomology $\smash{\mathcal{V}^\hbar_{\delta,\Delta}=H_\mathrm{BRST}^{\infty/2+0}(\g,\mathcal{M}_\mathrm{ss}^\hbar)}$ are a direct product of weight spaces for $\mathbb{S}$:
\begin{equation*}
\mathcal{V}^\hbar_{\delta,\Delta}(Y_\delta(\Delta))=\prod_{m\in(1/2)\N}\mathcal{V}^\hbar_{\delta,\Delta}(Y_\delta(\Delta))_m
\end{equation*}
with $\mathcal{V}^\hbar_{\delta,\Delta}(Y_\delta(\Delta))_0=\C\vac$, in which we consider the direct sum
\begin{equation*}
\mathcal{V}^\hbar_{\delta,\Delta}(Y_\delta(\Delta))_\mathrm{fin}\coloneqq\bigoplus_{m\in(1/2)\N}\mathcal{V}^\hbar_{\delta,\Delta}(Y_\delta(\Delta))_m.
\end{equation*}
This subspace is a $\C[\hbar]$-module since the weights are nonnegative and $\wt_\mathbb{S}(\hbar)=1$, and it is preserved by operator product expansions. Finally, we set
\begin{equation*}
\tilde{V}^\delta(\Delta)\coloneqq\bigl[\mathcal{V}^\hbar_{\delta,\Delta}(Y_\delta(\Delta))\bigr]^\mathbb{S}\coloneqq\mathcal{V}^\hbar_{\delta,\Delta}(Y_\delta(\Delta))_\mathrm{fin}\big/(\hbar-1),
\end{equation*}
the quotient space by the ideal generated by $\hbar-1$, which is a \svoa{}.

\medskip

We can also consider the (nonlocalised) $\hbar$-adic \svoa{}
\begin{equation*}
H_\mathrm{BRST}^{\infty/2+0}(\g,\mathcal{M}_\mathrm{ss}^\hbar(\mathfrak{X}))
\end{equation*}
and reduce it in the same way to a \svoa{}
\begin{align*}
&\bigl[H_\mathrm{BRST}^{\infty/2+0}(\g,\mathcal{M}_\mathrm{ss}^\hbar(\mathfrak{X}))\bigr]^\mathbb{S}=\bigl[H^0(\mathcal{C}^\hbar(\mathfrak{X}),d)\bigr]^\mathbb{S}\\
&=H^0(C,d)=H_\mathrm{BRST}^{\infty/2+0}(\g,M)=V(\Delta),
\end{align*}
which gives the boundary hypertoric \svoa{} constructed in \autoref{sec:hypertoricsvoa}.

Now, there is a homomorphism of $\hbar$-adic \svoa{}s
\begin{equation*}
H_\mathrm{BRST}^{\infty/2+0}(\g,\mathcal{M}_\mathrm{ss}^\hbar(\mathfrak{X}))\to\mathcal{V}^\hbar_{\delta,\Delta}(Y_\delta(\Delta))=H_\mathrm{BRST}^{\infty/2+0}(\g,\mathcal{M}_\mathrm{ss}^\hbar)(Y_\delta(\Delta)).
\end{equation*}
This descends to a homomorphism of \svoa{}s
\begin{equation}\label{eq:boundary_iso}
V(\Delta)\hookrightarrow\tilde{V}^\delta(\Delta)=\bigl[\mathcal{V}^\hbar_{\delta,\Delta}(Y_\delta(\Delta))\bigr]^\mathbb{S}.
\end{equation}
This map must be injective because $V(\Delta)$ is simple by \autoref{prop:symplicity}. In \autoref{sec:global}, we shall show that this map is, in fact, an isomorphism of \svoa{}s.

\medskip

As one of the main results of this paper, in \autoref{sec:assvar} we determine the associated variety of the the boundary hypertoric \svoa{} $\smash{V(\Delta)\cong\tilde{V}^\delta(\Delta)}$. That is, we show that $(R_{V(\Delta)})_\mathrm{red}\cong\C[Y_0(\Delta)]$. We shall base the computation on the corresponding result for $V_\mathrm{min}(\Delta)$ established in \cite{Kuw21}. At this point, we only establish the existence of a natural map $R_{V(\Delta)}\to\C[Y_0(\Delta)]$.

For a sheaf of $\hbar$-adic \voa{}s $\mathcal{V}_\hbar$, define the presheaf of commutative Poisson $\C[[\hbar]]$-algebras $\smash{U\mapsto R_{\mathcal{V}_\hbar(U)}}$ where $\smash{R_{\mathcal{V}_\hbar(U)}=\mathcal{V}_\hbar(U)/\mathcal{V}_\hbar(U)_{(-2)}\mathcal{V}_\hbar(U)}$ is the $C_2$-algebra of $\mathcal{V}_\hbar(U)$. (We shall describe the $C_2$-algebra in more detail in \autoref{sec:assvar} below, but we note that here, in the $\hbar$-adic setting, the Poisson bracket is $\{a,b\}=\hbar^{-1}a_{(0)}b$.) We let $\smash{\mathcal{R}_{\mathcal{V}_\hbar}}$ be the sheafification of this presheaf. In other words, this is the quotient sheaf $\smash{\mathcal{R}_{\mathcal{V}_\hbar}=\mathcal{V}_\hbar/\mathcal{V}_{\hbar\,{(-2)}}\mathcal{V}_\hbar}$.

Now, analogously to the minimal case in \cite{Kuw21}, there is an isomorphism of sheaves of $\C[[\hbar]]$-algebras
\begin{equation*}
\mathcal{R}_{\mathcal{V}^\hbar_{\delta,\Delta}}\cong\bar{\mathcal{O}}_{Y_\delta(\Delta)}[[\hbar]],
\end{equation*}
which gives a map
\begin{equation*}
R_{\mathcal{V}^\hbar_{\delta,\Delta}(Y_\delta(\Delta))}\to\mathcal{R}_{\mathcal{V}^\hbar_{\delta,\Delta}}(Y_\delta(\Delta))\cong\bar{\mathcal{O}}_{Y_\delta(\Delta)}(Y_\delta(\Delta))[[\hbar]],
\end{equation*}
which then descends to
\begin{equation*}
R_{V(\Delta)}\to\bar{\mathcal{O}}_{Y_\delta(\Delta)}(Y_\delta(\Delta))
\end{equation*}
and finally to
\begin{equation*}
(R_{V(\Delta)})_\mathrm{red}\to\mathcal{O}_{Y_\delta(\Delta)}(Y_\delta(\Delta))\cong\C[Y_0(\Delta)].
\end{equation*}


\subsection{Global Sections}\label{sec:global}

Above, we established the natural map \eqref{eq:boundary_iso},
\begin{equation*}
\phi\colon V(\Delta)\hookrightarrow\tilde{V}^\delta(\Delta)=\bigl[\mathcal{V}^\hbar_{\delta,\Delta}(Y_\delta(\Delta))\bigr]^\mathbb{S}.
\end{equation*}
Since $V(\Delta)$ is simple, the map is injective. We view $\smash{\tilde{V}^\delta(\Delta)}$ as a module for $V(\Delta)$. In the following, we show that the map is an isomorphism using the faithfulness theorem in \cite{ADS26}.
\begin{thm}[Global Sections]\label{thm:glob-sec}
As \svoa{}s,
\begin{equation*}
V(\Delta)\cong\tilde{V}^\delta(\Delta).
\end{equation*}
\end{thm}
The rest of the section is dedicated to the proof of this theorem.

\medskip

First, we recall the following truncation condition on the $L_0$-weight grading for modules of a conformal vertex algebra $V$. A $V$-module $M$ is called \emph{lower-truncated} if for any $m\in M$ there exists an integer $k\in\Ns$ such that for any homogeneous elements $v^1,\dots,v^s$ and integers $i_1,\dots,i_s$ satisfying $\sum_{j=1}^s(-\wt(v^j)+1+i_j)>k$, it follows that $\smash{v^1_{(i_1)}\dots v^s_{(i_s)}=0}$. This is satisfied in particular when the $L_0$-grading on $M$ is bounded from below.

We then recall the notions of (lower-truncated) \emph{vertex $A$-superalgebras} and \emph{$A$-loop modules} from \cite{Bor01b} (see also \cite{ADS26}). They allow us to localise vertex superalgebras and their modules by weight-$0$ elements, in analogy to the localisation of rings and their modules.

Let $V$ be a $\smash{\Q_{\geq0}}$-graded (by $L_0$-weights) conformal vertex superalgebra. Then the even weight-zero space $\smash{R\coloneqq V_0\cap V^{\bar0}}$ forms a unital, commutative, associative $\C$-algebra under the $(-1)$-st mode. Let $A$ be a subring of $R$. Then $V$ can be viewed as an $A$-loop module and hence as vertex $A$-superalgebra. In that situation, given a multiplicative system $S\subset A$, we may define the localisation $V_S$ of $V$ by $S$, which is then a vertex $A_S$-superalgebra.

Similarly, in the above situation, if $M$ is a (lower-truncated) module over the vertex superalgebra $V$, then $M$ is naturally a (lower-truncated) $A$-loop module. If $M$ is lower-truncated, then for a multiplicative system $S\subset A$ we may again define the localisation $M_S$ of $M$ by $S$, which is a lower-truncated $A_S$-loop module.

\medskip

The main tool in this section will be the following result from \cite{ADS26}.
\begin{prop}[Faithfulness Theorem]\label{prop:faithful}
Let $V$ be a vertex (super)algebra equipped with two conformal structures determined by conformal elements $\omega^1$ and $\omega^0$. Suppose that $V$ is $\Q_{\geq0}$-graded with respect to both conformal elements. Let $f,g\in V$ be two even elements such that
\begin{enumerate}[label=(\alph*)]
\item\label{item:a} $L^1_0f=0$ and $L^0_0g=0$,
\item\label{item:b} $L^0_nf=0$ for all $n\in\Ns$,
\item\label{item:c} $L^0_0f=\lambda f$ for some $\lambda\in\Q_{>0}$, which we write as $\lambda=p/q$ with $p,q\in\Ns$ and $(p,q)=1$,
\item\label{item:d} $(f^{nq})_{(np-1)}g^{nq}\vac=c_n\vac$ for some constant $c_n\neq0$ for all sufficiently large $n\in\N$.
\end{enumerate}
In particular, we may consider the localisation $V_f$ of $V$ by $f$ (which is in $\smash{V_0\cap V^{\bar0}}$ with respect to $\omega^1$) or $V_g$ by $g$ (which is in $\smash{V_0\cap V^{\bar0}}$ with respect to $\omega^0$). Similarly, given a lower-truncated $V$-module $M$ (for both $\omega^1$ and $\omega^0$), we may consider the localisations $M_f$ and $M_g$. Then, the following assertions hold:
\begin{enumerate}
\item Let $M$ be a lower-truncated $V$-module with respect to both $\omega^1$ and $\omega^0$. Suppose that $M_f=\{0\}$ and $M_g=\{0\}$. Then already $M=\{0\}$.
\item Let $M'$ be another such lower-truncated $V$-module and let $\phi\colon M\to M'$ be a homomorphism such that the induced maps $M_f\to M'_f$ and $M_g\to M'_g$ are injective or surjective. Then $\phi$ itself is injective or surjective, respectively.
\item If $V_f$ and $V_g$ are simple, then $V$ itself is simple.
\end{enumerate}
\end{prop}
In this work, we shall only need to use the surjectivity in part~(2) of the theorem.

\medskip

In the following, we shall show that the faithfulness theorem is applicable to the vertex superalgebra $V(\Delta)$ and the module $\smash{\tilde{V}^\delta(\Delta)}$.

First, we define certain elements $f,g\in V(\Delta)$. The key point is that they descend via the BRST reduction from elements in $M=\mathcal{D}^\mathrm{ch}(T^*V)\otimes\mathcal{C}\ell(\Pi T^*V)$ corresponding to elements $f,g\in I(\mathfrak{Z})\subset\C[T^*V]$, where $I(\mathfrak{Z})$ is the defining ideal of the unstable locus $\mathfrak{Z}=T^*V\setminus\mathfrak{X}$.

We assert in \autoref{lem:gen-in-svoa} below, for $a\in\Z^{N-M}$, that the fields
\begin{equation*}
W^a=\prod_{\substack{j=1\\(Ea)_j>0}}^Nx_j^{(Ea)_j}\prod_{\substack{j=1\\(Ea)_j<0}}^Ny_j^{-(Ea)_j}
\end{equation*}
in $\mathcal{D}^\mathrm{ch}(T^*V)$ descend to fields (of the same name, by abuse of notation) in the BRST cohomology $V(\Delta)$.

Consider a pair of elements $f\coloneqq W^a$ and $g\coloneqq W^{-a}$, for now for some arbitrary $a\in\Z^{N-M}\setminus\{0\}$. In the following lemmata, we show that they satisfy the assumptions of the faithfulness theorem.
\begin{lem}\label{lem:faithful1}
For $a\in\Z^{N-M}\setminus\{0\}$, the elements $f=W^a$ and $g=W^{-a}$ in $V(\Delta)$ satisfy condition~\ref{item:d} in \autoref{prop:faithful}.
\end{lem}
\begin{proof}
By construction, for each pair $\{x_i,y_i\}$, $x_i$ (respectively, $y_i$) appears in $W^a$ if and only if $y_i$ (respectively, $x_i$) appears in $W^{-a}$ (but some $\{x_i,y_i\}$ may not appear at all). Without loss of generality, in the following argument, we assume that $W^a$ contains only $x_i$ and $W^{-a}$ only $y_i$. Hence,
\begin{equation*}
f=W^a=\prod_{i=1}^Nx_i^{p_i}\quad\text{and}\quad g=W^{-a}=\prod_{i=1}^Ny_i^{p_i}
\end{equation*}
for some $p_i\in\N$ for all $i=1,\dots,N$ with $p\coloneqq\sum_{i=1}^N p_i>0$ (since $a\neq0$).

First, we consider the fields $f$ and $g$ in $\mathcal{D}^\mathrm{ch}(T^*V)\subset M\subset\tilde{C}$. We note that $\smash{f^n=\prod_{i=1}^Nx_i^{np_i}}$ and $\smash{g^n=\prod_{i=1}^Ny_i^{np_i}}$. Like in \autoref{lem:gen-in-svoa}, this simple formula is true because in this particular case the bracketing in the normally ordered products $f^n$ and $g^n$ does not matter. Then,
\begin{align*}
(f^n)_{(np-1)}(g^n)=\prod_{i=1}^N(x_i^{np_i})_{(np_i-1)}(y_i^{np_i})=\prod_{i=1}^N(-1)^{np_i}(np_i)!\vac,
\end{align*}
which is a nonzero multiple of the vacuum $\vac$ for all $n\in\N$. These identities hold in $\smash{\ker(d)\cap C\subset\tilde{C}}$. But then, quotienting by $\im(d)$, they must also hold for the images of the fields $f$ and $g$ in $V(\Delta)=H^\bullet(C,d)$, because the vacuum vector cannot be zero up to a coboundary. Finally note that $p\in\Ns$ is precisely the weight $\lambda$ in condition~\ref{item:c} in \autoref{prop:faithful} if one considers the auxiliary conformal structure $\omega^0$ below.
\end{proof}

\smallskip

To establish the remaining assumptions of the faithfulness theorem, we need to define two auxiliary conformal structures on $V(\Delta)$. Recall from \autoref{rem:other-conf-str}, for a given choice of subset $I\subset\{1,\dots,N\}$, the two conformal structures $\omega^1$ and $\omega^0$ on $\tilde{C}$ that descend to well-defined conformal structures on $C$ and on the relative cohomology $V(\Delta)=H^\bullet(C,d)$. As explained in \autoref{rem:other-conf-str}, $V(\Delta)$ is evidently $\Q_{\geq0}$-graded (in fact, $\N$-graded) with respect to both conformal structures.

Now, given the two elements $f=W^a$ and $g=W^{-a}$ for some $a\in\Z^{N-M}\setminus\{0\}$, we define the subset $I\coloneqq\bigl\{j\in\{1,\dots,N\}\bigm|x_j\mid W^a\bigr\}=\bigl\{j\in\{1,\dots,N\}\bigm|(Ea)_j>0\bigr\}$ of $\{1,\dots,N\}$. Then:
\begin{lem}\label{lem:faithful2}
For $a\in\Z^{N-M}\setminus\{0\}$, the elements $f=W^a$ and $g=W^{-a}$ together with the conformal vectors $\omega^1$ and $\omega^0$ of $V(\Delta)$ for the subset $I\subset\{1,\dots,N\}$ satisfy conditions \ref{item:a}, \ref{item:b} and \ref{item:c} in \autoref{prop:faithful}.
\end{lem}
\begin{proof}
The set $I$ and the conformal elements $\omega^1$ and $\omega^0$ are defined precisely such that these properties are satisfied. (Note that in the special case considered in the proof of \autoref{lem:faithful1}, the constant $\lambda$ is given by $p\in\Ns$.)
\end{proof}

Overall, we have thus shown that for any $a\in\Z^{N-M}\setminus\{0\}$, the vertex superalgebra $V(\Delta)$ with elements $f,g,\omega^1,\omega^0\in V(\Delta)$ satisfies the conditions of the faithfulness theorem, \autoref{prop:faithful}. But it remains to establish the lower-truncation of the $V(\Delta)$-module $\smash{\tilde{V}^\delta(\Delta)}$, for which we need to specialise the choice of $a\in\Z^{N-M}\setminus\{0\}$.

\medskip

Recall the description of the unstable locus $\mathfrak{Z}=T^*V\setminus\mathfrak{X}$ in \autoref{sec:unstable}. In \autoref{prop:specialelements} we identified vectors $b\in\Z^{N-M}\setminus\{0\}$ such that $f=W^b$ and $g=W^{-b}$ are both in $I(\mathfrak{Z})\subset\C[T^*V]$, the defining ideal in of the unstable locus. In other words, the principal open sets in $T^*V$ are contained in the semistable locus, $\mathfrak{D}_f\subset\mathfrak{X}$ and $\mathfrak{D}_g\subset\mathfrak{X}$.

With the choice of $f=W^b$ and $g=W^{-b}$ and the corresponding choice of conformal structures $\omega^1$ and $\omega^0$, as described above, we can show that $\smash{\tilde{V}^\delta(\Delta)}$ is lower-truncated.

\begin{lem}\label{lem:lowertruncated}
The vertex superalgebra $\smash{\tilde{V}^\delta(\Delta)}$ has lower-bounded $L_0$-grading (and is in particular a lower-truncated $V(\Delta)$-module) with respect to both conformal structures $\omega^1$ and $\omega^0$ for the subset $\smash{I\coloneqq\bigl\{i\in\{1,\dots,N\}\bigm|x_i\mid W^b\bigr\}}$.
\end{lem}
\begin{proof}
Recall from the proof of \autoref{prop:specialelements} that the element $b\in\Z^{N-M}\setminus\{0\}$ was chosen such that there are monomials $m_{d^+}\mid W^b$ and $m_{d^-}\mid W^{-b}$ for some $d^+,d^-\in\Q^N$ with $\Delta d^\pm=\delta$.

Without loss of generality, we assume that these solutions are minimal in the sense of \autoref{rem:specialelements}. That is, we assume that there are subsets $C,C'\subset\{1,\dots,N\}$ of cardinality $M$ such that $\Delta_C$ and $\Delta_{C'}$ have determinant $\pm1$ and such that the monomials $m_C\mid W^b$ and $m_{C'}\mid W^{-b}$. If this stronger condition is not satisfied, the following argument still works, but instead of considering just one open set for each conformal structure $\omega^1$ and $\omega^0$, we may need to consider several ones to establish the lower-boundedness of the corresponding $L_0$-gradings.

\smallskip

We shall show that for $\omega^1$ and $\omega^0$, there are open sets $U_J$ and $U_{J'}$ from the affine open covering $\smash{Y_\delta(\Delta)=\bigcup_JU_J}$ with $U_J\cong T^*\C^{N-M}$ described in \autoref{sec:localtrivial} such that the local sections $\smash{\mathcal{V}_{\delta,\Delta}^{\hbar}(U_J)}$ and $\smash{\mathcal{V}_{\delta,\Delta}^{\hbar}(U_{J'})}$ have lower-bounded $L_0$-grading for $\omega^1$ and $\omega^0$, respectively. Then the sheaf restriction maps
\begin{equation*}
\mathcal{V}_{\delta,\Delta}^{\hbar}(Y_\delta(\Delta))\hookrightarrow\mathcal{V}_{\delta,\Delta}^{\hbar}(U_J),\quad\mathcal{V}_{\delta,\Delta}^{\hbar}(Y_\delta(\Delta))\hookrightarrow\mathcal{V}_{\delta,\Delta}^{\hbar}(U_{J'})
\end{equation*}
imply that the global sections have lower-bounded $L_0$-grading for both $\omega^1$ and $\omega^0$. The same is then true for $\smash{\tilde{V}^\delta(\Delta)=\bigl[\mathcal{V}^\hbar_{\delta,\Delta}(Y_\delta(\Delta))\bigr]^\mathbb{S}}$.

For the minimal hypertoric sheaf $\smash{\mathcal{W}^\hbar_{\delta,\Delta}}$ from \cite{Kuw21} over the deformation $\tilde{Y}_{\delta}(\Delta)$, the local sections for the corresponding affine open covering $\smash{\tilde{Y}_\delta(\Delta)=\bigcup_J\tilde{U}_J}$ with $\tilde{U}_J\cong T^*\C^{N-M}\times\g^*$ are described very explicitly. Recall that $J=\{j_1,\dots,j_M\}$ is a subset of $\{1,\dots,N\}$ such that the corresponding $(M\times M)$-minor $\Delta_J$ of $\Delta$ has determinant $\pm1$. The local sections on $\tilde{U}_J$ are generated by the $\hbar$-adic Weyl vertex algebra in the $2(N-M)$ fields $a_j^*$ and $a_j$ for $j\notin J$, tensored with $M$ $\hbar$-adic Heisenberg fields. The Heisenberg generators always have $L_0$-weight~$1$, so to study the properties of the conformal grading it suffices to consider the fields $a_j^*$ and $a_j$.

Now, for the local sections of our fermionic sheaf $\smash{\mathcal{V}_{\delta,\Delta}^{\hbar}}$, the same fields $\smash{a_j^*}$ and $\smash{a_j}$ (which only contain the fields $x_i$ and $y_i$) are evidently still local sections, but they are complemented by free fermionic fields rather than by Heisenberg fields (see \autoref{sec:localtrivial}). Regardless of the choice of conformal structure, the latter will always have a lower-bounded (and in particular lower-truncated) conformal grading, so that it again suffices to study the fields $a_j^*$ and $a_j$. We describe these fields in the following.

Recall from \cite{Kuw21} that, given $J\subset\{1,\dots,N\}$ as above, we define the $(M\times M)$-matrix $\smash{\Delta_J=(\Delta_{j_1},\dots\Delta_{j_M})}$ from the columns of $\Delta$, which has determinant $\pm1$ and is hence invertible over $\Z$. Let $\lambda^t\coloneqq\Delta_J^{-1}$. Moreover, consider the disjoint union $J=J_1\sqcup J_2$ where $\smash{J_1=\{j\in J\mid(\Delta_J^{-1}\delta)_j>0\}}$ and $\smash{J_2=\{j\in J\mid(\Delta_J^{-1}\delta)_j<0\}}$. Note that $(\Delta_J^{-1}\delta)_j$ is never zero because $\delta$ is generic. Then the above mentioned fields are of the form
\begin{equation*}
a_j^*=x_jT_1^{-\lambda_{1j}}\dots T_M^{-\lambda_{Mj}},\qquad a_j=y_jT_1^{\lambda_{1j}}\dots T_M^{\lambda_{Mj}}
\end{equation*}
for $j\in\{1,\dots,N\}\setminus J$ where
\begin{equation*}
T_i=\prod_{j\in J_1}x_j^{\lambda_{ij}}\prod_{j\in J_2}y_j^{-\lambda_{ij}}
\end{equation*}
for $i=1,\dots,M$ are invertible local sections.

\smallskip

Now, by the above assumption, there are subsets $C,C'\subset\{1,\dots,N\}$ of cardinality $M$ such that $\Delta_C$ and $\Delta_{C'}$ have determinant $\pm1$ and such that the monomials $m_C\mid W^b$ and $m_{C'}\mid W^{-b}$. These monomials were defined as
\begin{equation*}
m_C\coloneqq\prod_{j\in C_1}x_j\prod_{j\in C_2}y_j\quad\text{and}\quad m_{C'}\coloneqq\prod_{j\in C'_1}x_j\prod_{j\in C'_2}y_j
\end{equation*}
with $C=C_1\sqcup C_2$ and $C'=C'_1\sqcup C'_2$ as above. In other words, recalling the definition of $W^b$ and $W^{-b}$, the divisibility statement is equivalent to
\begin{align*}
C_1,C'_2&\subset\bigl\{j\in\{1,\dots,N\}\bigm|(Eb)_j>0\bigr\},\\
C_2,C'_1&\subset\bigl\{j\in\{1,\dots,N\}\bigm|(Eb)_j<0\bigr\}.
\end{align*}

Now, the auxiliary conformal structure $\omega^1$ is chosen such that $x_j$ has $L_0$-weight~$0$ for $(Eb)_j>0$ (in particular for $j\in C_1$) and $y_j$ has $L_0$-weight~$0$ for $(Eb)_j<0$ (in particular for $j\in C_2$); see \autoref{rem:other-conf-str}. Then, looking at the definition of the $T_i$ on the open set $U_C$ (i.e.\ for $J=C$), we see that $\wt(T_i)=0$ with respect to $\omega^1$ for all $i=1,\dots,M$. Hence,
\begin{equation*}
\wt(a_j^*)=\wt(x_j)\in\{0,1\}\quad\text{and}\quad\wt(a_j)=\wt(y_j)\in\{1,0\}
\end{equation*}
for $j\in\{1,\dots,N\}\setminus C$ with respect to $\omega^1$. This shows that on $U_C$, the local sections $\smash{\mathcal{V}_{\delta,\Delta}^{\hbar}(U_C)}$ have lower-bounded $L_0$-grading, even when accounting for the free fermions, which we omitted from this discussion.

Repeating the same argument for the conformal structure $\omega^0$, which assigns $L_0$-weight~$0$ to $x_j$ for $(Eb)_j<0$ (in particular for $j\in C'_1$) and $L_0$-weight~$0$ to $y_j$ for $(Eb)_j>0$ (in particular for $j\in C'_2$), we see that $\smash{\mathcal{V}_{\delta,\Delta}^{\hbar}(U_{C'})}$ has lower-bounded $L_0$-grading with regard to $\omega_0$.
\end{proof}
We have thus established that for $b\in\Z^{N-M}$ from \autoref{prop:specialelements}, the vertex superalgebra $V(\Delta)$ with elements $f,g,\omega^1,\omega^0\in V(\Delta)$ and the $V(\Delta)$-module $\smash{\tilde{V}^\delta(\Delta)}$ satisfy the conditions of the faithfulness theorem, \autoref{prop:faithful}.

\medskip

We are now in a position to prove \autoref{thm:glob-sec}. For the argument, it is important that there are the invariant polynomials $f,g\in\C[T^*V]^G$, so that they descend to elements of $\C[Y_0(\Delta)]$, that also lie in the ideal $I(\mathfrak{Z})\subset\C[T^*V]$ of the stable locus and moreover are compatible with our auxiliary conformal structures.

\begin{proof}[Proof of \autoref{thm:glob-sec}]
We study the map $\smash{\phi\colon V(\Delta)\hookrightarrow\tilde{V}^\delta(\Delta)}$ and view $\smash{\tilde{V}^\delta(\Delta)}$ as a module for $V(\Delta)$. As stated above, $V(\Delta)$ is a $\N$-graded \svoa{} for the conformal structures $\omega^1$ and $\omega^0$, and $\smash{\tilde{V}^\delta(\Delta)}$ is a lower-truncated $V(\Delta)$-module.

With the choices of the fields $f=W^b$ and $g=W^{-b}$ in $V(\Delta)$, all the assumptions in the faithfulness theorem, \autoref{prop:faithful} are satisfied by \autoref{lem:faithful1} and \autoref{lem:faithful2}. Hence, if we can show that the induced maps $\smash{\phi_f\colon V(\Delta)_f\hookrightarrow\tilde{V}^\delta(\Delta)_f}$ and $\smash{\phi_g\colon V(\Delta)_g\hookrightarrow\tilde{V}^\delta(\Delta)_g}$ are both surjective, then the assertion follows. (The same holds for injectivity, but we do not need to prove it as $V(\Delta)$ is simple.)

Consider the symplectic resolution $\pi\colon Y_\delta(\Delta)\to Y_0(\Delta)$. Let $Y_0(\Delta)_\mathrm{sm}$ denote the smooth locus of $Y_0(\Delta)$. The restriction $\pi\colon\pi^{-1}(Y_0(\Delta)_\mathrm{sm})\to Y_0(\Delta)_\mathrm{sm}$ is an isomorphism.

Let $f=W^b$ and $g=W^{-b}$ in $\C[T^*V]^G$ be the invariant polynomials corresponding to the above choices of fields (of the same names). They also lie in the ideal $I(\mathfrak{Z})\subset\C[T^*V]$. As mentioned in \autoref{sec:coordringgens}, we also identify them with their images in the quotient $\smash{\C[Y_0(\Delta)]=\C[T^*V]^G/(\{\sum_{j=1}^N\Delta_{ij}x_jy_j\}_{i=1}^M)}$.

For the remainder of the proof, we restrict our discussion to $f$, noting that all statements hold analogously for $g$. Viewing $f\in\C[Y_0(\Delta)]$, we consider the corresponding principal open subset $D_f\subset Y_0(\Delta)_\mathrm{sm}\subset Y_0(\Delta)$. Consider the preimage $U\coloneqq\pi^{-1}(D_f)\subset Y_\delta(\Delta)$. Then $U\cong D_f$ is an affine open subset of $Y_\delta(\Delta)$. Then, we consider an open set $\mathfrak{U}\subset\mathfrak{X}\subset T^*V$, closed under the $G$-action, such that $p^{-1}(U)=\mathfrak{U}\cap\mu^{-1}(0)$ with the projection map $p\colon\mu^{-1}(0)\cap\mathfrak{X}\to Y_\delta(\Delta)$. In fact, we can take $\mathfrak{U}=\mathfrak{D}_f\subset\mathfrak{X}$, now viewing $f\in I(\mathfrak{Z})\subset\C[T^*V]$.

The sheaf $\smash{\mathcal{V}^\hbar_{\delta,\Delta}}$ was defined as the cohomology sheaf $\smash{\mathcal{V}^\hbar_{\delta,\Delta}=H_\mathrm{BRST}^{\infty/2+0}(\g,\mathcal{M}_\mathrm{ss}^\hbar)}$ over $Y_\delta(\Delta)$, where $\smash{\mathcal{M}_\mathrm{ss}^\hbar}$ is the free-field sheaf restricted to $\mathfrak{X}$. Then, because $U\subset Y_\delta(\Delta)$ is affine, it follows that
\begin{equation*}
\mathcal{V}^\hbar_{\delta,\Delta}(U)=H_\mathrm{BRST}^{\infty/2+0}(\g,\mathcal{M}_\mathrm{ss}^\hbar)(U)=H_\mathrm{BRST}^{\infty/2+0}(\g,\mathcal{M}_\mathrm{ss}^\hbar(\mathfrak{U})).
\end{equation*}
Further, because $\mathfrak{U}=\mathfrak{D}_f\subset\mathfrak{X}$ is principal open and $\mathcal{M}^\hbar$ is a sheaf of free-field \svoa{}s on $T^*V$ in the sense of \cite{MSV99},
\begin{equation*}
\mathcal{M}_\mathrm{ss}^\hbar(\mathfrak{U})=\mathcal{M}^\hbar(\mathfrak{U})=\mathcal{M}^\hbar(T^*V)_f.
\end{equation*}
Here, we note that $f$ has weight~$0$ and $\mathcal{M}^\hbar(T^*V)$ is nonnegatively graded with respect to $\omega^1$, and the localisation by $f$ is in the sense of \cite{MSV99,Bor01}, and does not need microlocalisation, in principle. We also note that localisation by $f$ commutes with the BRST reduction, as proved in \cite{ADS26}, so that we obtain
\begin{equation*}
\mathcal{V}_{\delta,\Delta}^{\hbar}(Y_\delta(\Delta))_f=\mathcal{V}^\hbar_{\delta,\Delta}(U)=H_\mathrm{BRST}^{\infty/2+0}(\g,\mathcal{M}^\hbar(T^*V)_f)=H_\mathrm{BRST}^{\infty/2+0}(\g,\mathcal{M}^\hbar(T^*V)^\hbar)_f.
\end{equation*}
The \svoa{} on the right-hand side is simply the $\hbar$-adic version of the unlocalised vertex superalgebra $V(\Delta)=\smash{H_\mathrm{BRST}^{\infty/2+\bullet}(\g,M)}$, localised by $f$. Hence, considering the $\mathbb{S}$-invariant structure, we obtain
\begin{equation*}
\tilde{V}^\delta(\Delta)_f=[\mathcal{V}^\hbar_{\delta,\Delta}(U)]^\mathbb{S}=V(\Delta)_f.
\end{equation*}
Repeating the same argument for $g$ and $\omega^0$, the assertion follows.
\end{proof}

\begin{rem}\label{rem:vanish}
We can apply the same argument as in the proof of \autoref{thm:glob-sec} to the nonzero cohomologies $\smash{H_\mathrm{BRST}^{\infty/2+i}(\g,\mathcal{M}^\hbar)}$. Then the global sections coincide with $\smash{H_\mathrm{BRST}^{\infty/2+i}(\g,M)=\{0\}}$ for all $i\neq 0$. Hence, the vanishing theorem, \autoref{prop:vanish}, also holds for the global sections of the sheaf $\smash{H_\mathrm{BRST}^{\infty/2+i}(\g,\mathcal{M}^\hbar)}$.
\end{rem}


\section{Properties of Hypertoric \SVOA{}s}

In this section, we study the main properties of the hypertoric \svoa{}s $V(\Delta)$ constructed in \autoref{sec:hypertoricsvoa}. Most importantly, we show that the associated variety $X_{V(\Delta)}$ recovers the underlying affine hypertoric variety $Y_0(\Delta)$. The modularity of the characters will be studied in \autoref{sec:chars} below.


\subsection{Some Bosonic Fields}\label{sec:somefields}

In the following, we describe some even fields of the hypertoric \svoa{} $V(\Delta)$. We shall use them in \autoref{sec:assvar} to study the associated variety of $V(\Delta)$.

Recall the description of the coordinate ring $\C[Y_0(\Delta)]=\C[\mu^{-1}(0)]^G$. In particular, \autoref{prop:ring-gen} provides elements $J_i$ for $i=1,\dots,N-M$ and $W^a$ for $a\in\Z^{N-M}$, which generate $\C[Y_0(\Delta)]$. In fact, a finite subset of these elements suffices.

Now, by replacing the coordinates $x_j$, $y_j$ for $j=1,\dots,N$ by the fields of the same name and products with normally ordered products, we obtain in particular the following even fields in $\mathcal{D}^\mathrm{ch}(T^*V)\subset M\subset\tilde{C}$:
\begin{equation*}
J_i\coloneqq\sum_{j=1}^NE_{ji}x_jy_j,\qquad W^a\coloneqq\prod_{\substack{j=1\\(Ea)_j>0}}^Nx_j^{(Ea)_j}\prod_{\substack{j=1\\(Ea)_j<0}}^Ny_j^{-(Ea)_j}
\end{equation*}
for $i=1,\dots,N-M$ and $a\in\Z^{N-M}$, slightly abusing notation by denoting these fields by the same names as their counterparts in $\C[Y_0(\Delta)]$.

In principle, we need to care about the bracketing in the normally-ordered products appearing in the $W^a$, but since only one field from each pair $\{x_j,y_j\}$ appears in them, this bracketing does not matter here. Hence, the above expressions are well-defined. Similarly, the $J_i$ only contain two fields, so there is no bracketing to begin with.

\begin{prop}\label{lem:gen-in-svoa}
The fields $J_i$ for $i=1,\dots,N-M$ and $W^a$ for $a\in\Z^{N-M}$ in $\mathcal{D}^\mathrm{ch}(T^*V)$ also define fields in the hypertoric \svoa{} $V(\Delta)$.
\end{prop}
\begin{proof}
A direct computation shows that these fields are in the relative subcomplex $C\subseteq\tilde{C}$ and are BRST closed. Hence, they descend to even fields $J_i$ and $W^a$ in $V(\Delta)=H^\bullet(C,d)$.
\end{proof}


\subsection{Associated Variety}\label{sec:assvar}

We prove that the associated variety $X_{V(\Delta)}$ of the boundary hypertoric \svoa{} $V(\Delta)$ constructed in \autoref{sec:hypertoricsvoa} recovers the affine and singular hypertoric variety $Y_0(\Delta)$, $X_{V(\Delta)}\cong Y_0(\Delta)$. This also proves the 3d Higgs branch conjecture \cite{BF25} for a large class of boundary \svoa{}s. In other words, $V(\Delta)$ is a \emph{chiral quantisation} of the coordinate ring $\C[Y_0(\Delta)]$ of $Y_0(\delta)$. In particular, $V(\Delta)$ is quasi-lisse, in contrast to the minimal hypertoric \voa{} constructed in \cite{Kuw21}.

\medskip

We recall some relevant definitions. For any vertex superalgebra $V$, there is a decreasing filtration $V=F_0(V)\supset F_1(V)\supset\dots$, and we can consider the corresponding associated graded object $\operatorname{gr}^F(V)\coloneqq\bigoplus_{n=0}^\infty F_n(V)/F_{n+1}(V)$, which carries the structure of a vertex Poisson algebra \cite{Li05}. The lowest-degree component
\begin{equation*}
R_V\coloneqq F_0(V)/F_1(V)=V/F_1(V)=V/V_{(-2)}V
\end{equation*}
is the $C_2$-algebra \cite{Zhu96}, a Poisson superalgebra with the (super-commutative and associative) product induced from the $(-1)$-st vertex algebra product and the Poisson bracket from the $0$-th one. $R_V$ generates $\operatorname{gr}^F(V)$ as a differential algebra, i.e.\ there is a surjection $\phi\colon J_\infty R_V\twoheadrightarrow\operatorname{gr}^F(V)$ from the arc space $J_\infty R_V$ of $R_V$. Often, this induces an isomorphism $\phi_\mathrm{red}\colon J_\infty(R_V)_\mathrm{red}\cong\operatorname{gr}^F(V)_\mathrm{red}$.

Following \cite{Ara12}, one then defines the \emph{associated variety} $X_V$ of $V$ as the affine Poisson variety
\begin{equation*}
X_V\coloneqq\Specm(R_V)=\Spec((R_V)_\mathrm{red}),
\end{equation*}
where $(R_V)_\mathrm{red}=R_V/N$ is the reduced algebra obtained as the quotient of $R_V$ by its nilradical $N$. A vertex superalgebra $V$ is called \emph{chiral quantisation} of an (affine) Poisson variety $X$, if $X\cong X_V$. In other words (assuming that $\phi_\mathrm{red}$ is an isomorphism), $V$ can be viewed as a quantisation of the arc space $J_\infty X$.

If the associated Poisson variety $X_V$ is a symplectic variety (in the sense of having only finitely many symplectic leaves), then $V$ is called \emph{quasi-lisse} \cite{AK18}.

\medskip

In the context of the 3d/2d-correspondence \cite{Gai19,CG19,CCG19}, given a 3d $\mathcal{N}=4$ supersymmetric field theory $\mathcal{T}$, one can associate with it the $A$-twisted (or $H$-twisted) \emph{boundary \svoa{}} $V^A(\mathcal{T})$, which is thought to live on the boundary of the topological $A$-twist of $\mathcal{T}$. Further, one can extract from the theory $\mathcal{T}$ the \emph{Higgs branch} $\mathcal{M}_H(\mathcal{T})$, a certain symplectic variety appearing in the moduli space of vacuum expectation values of $\mathcal{T}$. Mathematically, the Higgs branch $\mathcal{M}_H(\mathcal{T})$ is defined as a holomorphic symplectic quotient \cite{HKLR87} by the gauge group $G$ (see, e.g., \cite{Kam22,WY23} for recent surveys). Often, this is a symplectic singularity in the sense of \cite{Bea00}. The (3d version of the) \emph{Higgs branch conjecture} \cite{BF25} states that in good cases, the associated variety of $V^A(\mathcal{T})$
\begin{equation}\label{eq:Higgs}
X_{V^A(\mathcal{T})}\cong\mathcal{M}_H(\mathcal{T})
\end{equation}
recovers the Higgs branch $\mathcal{M}_H(\mathcal{T})$ of $\mathcal{T}$. This also leads to the expectation that the \svoa{}s $V^A(\mathcal{T})$ are quasi-lisse.

Hence, one of the key properties to study when mathematically constructing the boundary \svoa{} $V^A(\mathcal{T})$ is its associated variety. In \autoref{thm:var}, we verify the Higgs branch conjecture for 3d $\mathcal{N}=4$ supersymmetric field theories with \emph{abelian gauge group} $G=(\C^\times)^M$, i.e.\ when $V^A(\mathcal{T})=V(\Delta)$ and $\mathcal{M}_H(\mathcal{T})=Y_0(\Delta)$. That is, we prove that $X_{V(\Delta)}\cong Y_0(\Delta)$ as Poisson varieties. In particular, $V(\Delta)$ is quasi-lisse.

We remark that a proof of the conjecture for the special case of $\Delta=(1,\dots,1)$ was already given in \cite{FS24}, where $V(\Delta)$ was identified (for $N\neq 2$) with the simple affine \svoa{} $L_1(\psl_{N|N})$. We shall return to this example in \autoref{sec:nilpotent_svoa}.

\medskip

In order to prove the Higgs branch conjecture for the boundary hypertoric \svoa{} $V(\Delta)$, we first recall the following related result for the associated variety $X_{V_\mathrm{min}(\Delta)}$ of the minimal hypertoric \voa{} $V_\mathrm{min}(\Delta)$, obtained in~\cite{Kuw21}:
\begin{prop}[\cite{Kuw21}, Proposition~9.4]\label{prop:toshiro-subalg}
There is an injective homomorphism of Poisson algebras
\begin{equation*}
R_{V_\mathrm{min}(\Delta)}\hookrightarrow\C[\tilde{Y}_{\delta}(\Delta)].
\end{equation*}
\end{prop}
Here, we recall from \autoref{sec:deformation} that $\tilde{Y}_{\delta}(\Delta)$ is a universal family of Poisson deformations of $Y_{\delta}(\Delta)$ over $\g^*\cong\C^M$. We point out that because $\C[\tilde{Y}_{\delta}(\Delta)]$ is reduced, so is $R_{V_\mathrm{min}(\Delta)}$. The proposition implies that there is a dominant regular morphism
\begin{equation*}
\tilde{Y}_{\delta}(\Delta)\to X_{V_{\mathrm{min}}(\Delta)}.
\end{equation*}
However, note that these Poisson varieties are in general not the same (see \cite{Kuw21}, Lemma~10.10).

To prove the Higgs branch conjecture, \autoref{thm:var}, we first relate $R_{V(\Delta)}$ to $R_{V_\mathrm{min}(\Delta)}$ so that we can make use of \autoref{prop:toshiro-subalg}. More precisely, we shall show in \autoref{prop:vmin_surj} that in the conformal embedding $V_\mathrm{min}(\Delta)\otimes V_{J^\bot}\hookrightarrow V(\Delta)$, the images of the generators of $V_\mathrm{min}(\Delta)$ already generate the reduced $C_2$-algebra $(R_{V(\Delta)})_\mathrm{red}$, i.e.\ that there is a surjection
\begin{equation*}
R_{V_\mathrm{min}(\Delta)}=(R_{V_\mathrm{min}(\Delta)})_\mathrm{red}\twoheadrightarrow(R_{V(\Delta)})_\mathrm{red}.
\end{equation*}
This means that $X_{V(\Delta)}$ is a closed subvariety of $X_{V_\mathrm{min}(\Delta)}$.

\medskip

We show the existence of this surjection by means of a few lemmata:
\begin{lem}\label{lem:posnilp}
Let $(\h,\langle\cdot,\rangle)$ be a vector space equipped with a nondegenerate (possibly indefinite) symmetric bilinear form, and consider the corresponding Heisenberg \voa{} $\smash{\pi^\h}$. Let $h\in\h$ with $\langle h,h\rangle>0$, and suppose that $\smash{f\in\pi_h^\h}$ is in the corresponding Fock module. Then, the mode $f_{(-1)}$ of the intertwining operator $\smash{\mathcal{Y}_h(f,z)}$ of the Fock module $\smash{\pi_h^\h}$ acting on the direct sum of all Fock modules for $\smash{\pi^\h}$ is locally nilpotent, i.e.\ for each $w$ in some Fock module for $\smash{\pi^\h}$, there exists an $n\in\N$ such that $(f_{(-1)})^nw=0$.
\end{lem}
Note that the intertwining operator $\mathcal{Y}_h(\cdot,z)$ is unique up to a scaling by a nonzero scalar on each of the $1$-dimensional components of type $\smash{\binom{h+\alpha}{h\;\alpha}}$.
\begin{proof}
We shall only prove the assertion for $\smash{w=\vac\in\pi_0^\h=\pi^\h}$, as this is the case that we need later, but the proof easily generalises to arbitrary $w$. Without loss of generality, we may assume that $f$ is homogenous for the $L_0$-grading.

The $L_0$-weight of $\smash{(f_{(-1)})^n\vac\in\pi_{nh}^\h}$ is given by $\smash{n\wt(f)=n(m+\langle h,h\rangle/2)}$ for some $m\in\N$. On the other hand, the lowest $L_0$-weight of the Fock module $\smash{\pi_{nh}^\h}$ is $\smash{\langle nh,nh\rangle/2=n^2\langle h,h\rangle/2}$. It is clear that $\smash{n(m+\langle h,h\rangle/2)<n^2\langle h,h\rangle/2}$ for $n$ large enough, which means that $(f_{(-1)})^n\vac=0$.
\end{proof}

Recall from \autoref{sec:vanish} that the hypertoric \svoa{}
\begin{align*}
V(\Delta)&=\bigoplus_{\substack{k_1,\dots,k_{N-M}\in\Z\\l_1,\dots,l_N\in\Z}}M_{\sum_{j=1}^N\bigl(\sum_{i=1}^{N-M}E_{ji}k_i+l_j\bigr)p^\bot(\rho_j)}^{p^\bot(R)}\\
&\quad\otimes\pi_{\sum_{i=1}^{N-M}\bigl(k_i+\sum_{j=1}^N((E^tE)^{-1}E^t)_{ij}l_j\bigr)\sigma_i^E}^{S^E}\otimes\pi_{\sum_{i=1}^{N-M}\sum_{j=1}^N((E^tE)^{-1}E^t)_{ij}l_j\tau_i^E}^{T^E}
\end{align*}
is a subalgebra of the Heisenberg-lattice vertex superalgebra
{\allowdisplaybreaks
\begin{align*}
\HL_{L^\bot}^{\h^\bot}&=\bigoplus_{\substack{k_1,\dots,k_{N-M}\in\Z\\l_1,\dots,l_N\in\Z}}\pi_{\sum_{i=1}^{N-M}k_i(-\rho_i^E+\sigma_i^E)+\sum_{i=1}^Nl_i(-\rho_i+\sigma_i+\tau_i)}^{\h^\bot}\\
&=\bigoplus_{\substack{k_1,\dots,k_{N-M}\in\Z\\l_1,\dots,l_N\in\Z}}\pi_{-\sum_{j=1}^N\bigl(\sum_{i=1}^{N-M}E_{ji}k_i+l_j\bigr)p^\bot(\rho_j)}^{p^\bot(R)}\\
&\quad\otimes\pi_{\sum_{i=1}^{N-M}\bigl(k_i+\sum_{j=1}^N((E^tE)^{-1}E^t)_{ij}l_j\bigr)\sigma_i^E}^{S^E}\\
&\quad\otimes\pi_{\sum_{i=1}^{N-M}\sum_{j=1}^N((E^tE)^{-1}E^t)_{ij}l_j\tau_i^E}^{T^E}
\end{align*}
}%
under the application of certain screening kernels.
\begin{lem}\label{lem:pos}
In the above decomposition of $\smash{\HL_{L^\bot}^{\h^\bot}}$ into Fock modules for $\smash{\pi^{\h^\bot}}$, the momentum
\begin{equation*}
h\coloneqq\sum_{i=1}^{N-M}k_i(-\rho_i^E+\sigma_i^E)+\sum_{i=1}^Nl_i(-\rho_i+\sigma_i+\tau_i)
\end{equation*}
satisfies $\langle h,h\rangle>0$ if and only if $(l_1,\dots,l_N)\neq(0,\dots,0)$.
\end{lem}
\begin{proof}
The assertion follows from $\langle h,h\rangle=\sum_{i=1}^Nl_i^2$.
\end{proof}
\begin{lem}\label{lem:nilpotent}
Let $f\in V(\Delta)$ be in one of the direct summands of $V(\Delta)$, as presented above. If $(l_1,\dots,l_N)\neq(0,\dots,0)$, then $(f_{(-1)})^n\vac=0$ for some $n\in\N$.
\end{lem}
\begin{proof}
Combining \autoref{lem:posnilp} and \autoref{lem:pos}, it follows that a vector $\smash{f\in\HL_{L^\bot}^{\h^\bot}}$ in one of the direct summands satisfies $(f_{(-1)})^n\vac=0$ for some $n\in\N$ provided that $(l_1,\dots,l_N)\neq(0,\dots,0)$. Then, the same statement is true in the subalgebra $\smash{V(\Delta)\subset\HL_{L^\bot}^{\h^\bot}}$.
\end{proof}
So, we have shown that any element $f\in V(\Delta)$ that does not satisfy $(f_{(-1)})^n\vac=0$ for some $n\in\N$ must contain a nonzero contribution from
\begin{equation*}
V_\mathrm{min}(\Delta)\otimes\pi^{T^E}=\bigoplus_{k_1,\dots,k_{N-M}\in\Z}M_{\sum_{j=1}^N\sum_{i=1}^{N-M}E_{ji}k_ip^\bot(\rho_j)}^{p^\bot(R)}
\otimes\pi_{\sum_{i=1}^{N-M}k_i\sigma_i^E}^{S^E}\otimes\pi_0^{T^E},
\end{equation*}
which is the result of restricting the above sum in $V(\Delta)$ to $l_1=\dots=l_n=0$. Finally, this implies:
\begin{prop}\label{prop:vmin_surj}
The reduced $C_2$-algebra $(R_{V(\Delta)})_\mathrm{red}$ is generated by the image of $V_\mathrm{min}(\Delta)$, i.e.\ the induced map $(R_{V_\mathrm{min}(\Delta)})_\mathrm{red}\twoheadrightarrow(R_{V(\Delta)})_\mathrm{red}$ is surjective.
\end{prop}
\begin{proof}
By \autoref{lem:nilpotent}, it follows that any contribution to an element of the $C_2$-algebra $R_{V(\Delta)}=V(\Delta)/C_2(V(\Delta))$ not in the image of $\smash{V_\mathrm{min}(\Delta)\otimes\pi^{T^E}}$ is nilpotent. Hence, the reduced $C_2$-algebra $(R_{V(\Delta)})_\mathrm{red}$ of $V(\Delta)$ is generated by the image of $\smash{V_\mathrm{min}(\Delta)\otimes\pi^{T^E}}$. In particular, it is generated by the image of the larger \voa{} $\smash{V_\mathrm{min}(\Delta)\otimes V_{J^\bot}\supseteq V_\mathrm{min}(\Delta)\otimes\pi^{T^E}}$. Hence, $(R_{V(\Delta)})_\mathrm{red}$ is a quotient of $\smash{(R_{V_\mathrm{min}(\Delta)\otimes V_{J^\bot}})_\mathrm{red}=(R_{V_\mathrm{min}(\Delta)})_\mathrm{red}}$, recalling that $\smash{(R_{V_{J^\bot}})_\mathrm{red}}$ is trivial as the lattice \svoa{} $V_{J^\bot}$ is $C_2$-cofinite.
\end{proof}
As an immediate consequence, we obtain:
\begin{cor}
The associated variety $X_{V(\Delta)}$ is a closed subvariety of $X_{V_\mathrm{min}(\Delta)}$.
\end{cor}

\smallskip

We further show the existence of yet another surjection of Poisson algebras:
\begin{lem} There is a surjective homomorphism of Poisson algebras
\begin{equation*}\label{prop:var-first-incl}
(R_{V(\Delta)})_{\mathrm{red}}\twoheadrightarrow\C[Y_0(\Delta)].
\end{equation*}
\end{lem}
\begin{proof}
In \autoref{sec:oursheaf}, we introduced the natural map $(R_{V(\Delta)})_{\mathrm{red}}\to\C[Y_0(\Delta)]$. In the following, we show that this map is indeed a surjection. Recall that \autoref{prop:ring-gen} provides a list of generators $J_i$ and $W^a$ (monomials in $x_j$ and $y_j$) for the ring $\C[Y_0(\Delta)]$. In \autoref{lem:gen-in-svoa} we promoted these expressions to normally-ordered products of the fields $x_j(z)$ and $y_j(z)$. Considering the images of these fields in $R_{V(\Delta)}$ and then in $(R_{V(\Delta)})_{\mathrm{red}}$ provides a preimage for each generator of $\C[Y_0(\Delta)]$. This shows that the above map is surjective.
\end{proof}

In order to prove \autoref{thm:var}, we need a final result on the image of the conformal vector $\omega$ in $R_{V(\Delta)}$. We recall from \autoref{sec:conf-vect} that $\omega = \omega_{\mathrm{min}}+\omega_{T^E}$, with
\begin{align*}
\omega_\mathrm{min}&=\sum_{i=1}^{N}\bigl((x_i\phi_i)(y_i\psi_i)-(y_i\psi_i)(x_i\phi_i)\bigr)/2+\sum_{i=1}^{N-M}(\sigma_i^E+\tau_i^E)^E_i\tau_E^i-\sum_{i=1}^{N-M}\tau_i^E\tau^i_E/2,\\
\omega_{T^E}&=\sum_{i=1}^{N-M}\tau_i^E\tau^i_E/2.
\end{align*}

\begin{lem}\label{lem:nilpotent-stress}
The conformal vector $\omega$ of $V(\Delta)$, and its summands $\omega_\mathrm{min}$ and $\omega_{T^E}$, are mapped to nilpotent elements in the $C_2$-algebra $R_{V(\Delta)}=V(\Delta)/C_2(V(\Delta))$.
\end{lem}
\begin{proof}
The lattice \svoa{} $\smash{V_{J^\bot}}$ associated with the integral, positive-definite lattice $J^\bot\subset T^E$ is $C_2$-cofinite. Hence, the image of $\smash{\omega_{T^E}}$ in $\smash{R_{V_{J^\bot}}}$ is nilpotent (see Lemma~4.11 in \cite{AM24}). In particular, the image of $\omega_{T^E}$ in $R_{V(\Delta)}$ is nilpotent.

In fact, by the same argument, we see that any $\tau_i^E$ and $\tau^i_E$ for $i=1,\dots,N-M$ is mapped to a nilpotent element in $\smash{R_{V(\Delta)}}$. Moreover, note that the images of $(x_i\phi_i)$ and $(y_i\psi_i)$ in the supercommutative algebra $\smash{R_{V(\Delta)}}$ are odd and hence nilpotent. But then $(x_i\phi_i)(y_i\psi_i)$ is mapped to the product of these odd elements, which is also nilpotent. Hence, all summands in the images of $\omega_\mathrm{min}$ and $\omega$ in $R_{V(\Delta)}$ are nilpotent.
\end{proof}

\smallskip

We now prove the main result of this section. In particular, it shows that the boundary hypertoric \svoa{} $V(\Delta)$ recovers the Higgs branch of the corresponding 3d $\mathcal{N}=4$ supersymmetric field theory.
\begin{thm}[Higgs Branch Conjecture]\label{thm:var}
As Poisson varieties,
\begin{equation*}
X_{V(\Delta)}\cong Y_0(\Delta).
\end{equation*}
\end{thm}
\begin{proof}
\autoref{prop:toshiro-subalg}, \autoref{prop:vmin_surj} and \autoref{prop:var-first-incl} yield the diagram
\begin{center}
\begin{tikzcd}
(R_{V_{\mathrm{min}}(\Delta)})_{\mathrm{red}}\arrow[r,two heads,"s_1"]\arrow[d,hook,"i"]&(R_{V(\Delta)})_{\mathrm{red}}\arrow[r,two heads,"s_2"]&\C[Y_0(\Delta)]\\
\C[\tilde{Y}_\delta (\Delta)].
\end{tikzcd}
\end{center}
Also recall from \autoref{prop:poiss-centre} the Poisson centre $Z(\C[\tilde{Y}_\delta(\Delta)])$ of $\C[\tilde{Y}_\delta(\Delta)]$, its augmentation ideal $I_{Z(\Delta)}$ and that $\C[\tilde{Y}_\delta(\Delta)]/I_{Z(\Delta)}\cong\C[Y_0(\Delta)]$.

The ideal $i^{-1}(I_{Z(\Delta)})$ is evidently in the augmentation ideal of the Poisson centre of $(R_{V_{\mathrm{min}}(\Delta)})_{\mathrm{red}}$. Consider its image
\begin{equation*}
s_1(i^{-1}(I_{Z(\Delta)}))\subset(R_{V(\Delta)})_{\mathrm{red}}.
\end{equation*}
Since $s_1$ is a surjective Poisson homomorphism, this image must be contained in the augmentation ideal generated by the Poisson centre $\smash{Z((R_{V(\Delta)})_{\mathrm{red}})}$ of $\smash{(R_{V(\Delta)})_{\mathrm{red}}}$. However, we know from \autoref{lem:nilpotent-stress} that $\omega$ is nilpotent in $\smash{R_{V(\Delta)}}$. It then follows from Theorem~4.5 in \cite{AM24} that the augmentation ideal of the Poisson centre $\smash{Z(R_{V(\Delta)})}$ is contained in the nilradical of $\smash{R_{V(\Delta)}}$, which implies that the augmentation ideal of the Poisson centre of $\smash{(R_{V(\Delta)})_{\mathrm{red}}}$ is trivial. Therefore, $i^{-1}(I_{Z(\Delta)})$ is in the kernel of $s_1$, so that $s_1$ factors through the quotient.

Hence, taking the quotient by $I_{Z(\Delta)}$ and $i^{-1}(I_{Z(\Delta)})$ on the left-hand side, we arrive at the diagram
\begin{center}
\begin{tikzcd}
(R_{V_{\mathrm{min}}(\Delta)})_{\mathrm{red}}/i^{-1}(I_{Z(\Delta)})\arrow[r,two heads,"s_1"]\arrow[d,hook,"i"]&(R_{V(\Delta)})_{\mathrm{red}}\arrow[r,two heads,"s_2"]& \C[Y_0(\Delta)]\\
\C[Y_0(\Delta)].
\end{tikzcd}
\end{center}
Since $Y_0(\Delta)$ is conical, $\C[Y_0(\Delta)]$ is a direct sum of finite-dimensional homogeneous components, and all maps above respect the grading. Hence, we conclude from the injectivity of $i$ and the surjectivity of $s_2\circ s_1$ that all the maps above are isomorphisms. In particular, $s_2\colon(R_{V(\Delta)})_{\mathrm{red}}\to\C[Y_0(\Delta)]$ is an isomorphism.
\end{proof}
As the hypertoric varieties considered here are symplectic singularities in the sense of Beauville \cite{Bea00}, it follows from the work of Kaledin \cite{Kal06} that they have finitely many symplectic leaves. Thus, by definition, we can conclude:
\begin{cor}\label{cor:quasilisse}
$V(\Delta)$ is quasi-lisse.
\end{cor}


\subsection{Automorphisms and Symplectic Duality}\label{sec:auto}

Recall from \autoref{sec:sym-dual} that given a hypertoric variety $Y_0(\Delta)$ associated with a unimodular matrix $\Delta$, one can always obtain a symplectic (or Gale) dual variety $Y_0(\Delta\!^!)$. This implies that we can consider pairs of \svoa{}s, $V(\Delta)$ and $V(\Delta\!^!)$, such as the two examples studied in \autoref{sec:examples}. By \autoref{thm:var}, these \svoa{}s are chiral quantisations of $Y_0(\Delta)$ and $Y_0(\Delta\!^!)$, respectively. Thus, these \svoa{}s can be thought of as symplectic dual pairs of \svoa{}s.\footnote{Based on our comments on 3d mirror symmetry in \autoref{sec:sym-dual}, they are also the $A$-twisted boundary vertex algebras of two 3d mirror dual supersymmetric quantum field theories.}

We discussed in \autoref{sec:sym-dual} that $Y_0(\Delta)$ and $Y_0(\Delta\!^!)$ are proved to have certain properties swapped. For instance, the Lie algebra of (infinitesimal) torus symplectic symmetries of one is in correspondence with the Lie algebra of stability parameters of the other. This fact is reflected in automorphisms of $V(\Delta)$ and $V(\Delta\!^!)$.

\begin{prop}[\cite{BF25,BCDN23}]
The weight-$1$ fields $\tau_i^E$ for $i=1,\dots,N-M$ descend to inner automorphisms $\smash{e^{(\tau_i^E)_{(0)}}}$ of $V(\Delta)$. The zero modes of the fields $\smash{\tau^\Delta_i}$ for $i=1,\dots,M$ define noninner automorphisms on $V(\Delta)$.
\end{prop}
The proof is elementary. The weight-$1$ fields $\tau_i^E$, $i=1,\dots,N-M$, are elements of $V_{J^\bot}\subset  V(\Delta)$ and therefore define inner automorphisms. On the other hand, the fields $\tau_i^\Delta$, $i=1,\dots,M$, define inner automorphisms of $\mathcal{C}\ell(\Pi T^*V)\subset M$ but do not belong to $V(\Delta)$ as they are not annihilated by the BRST operator. Their zero modes, however, are annihilated and therefore preserve its kernel. Hence, they descend to noninner automorphisms of $V(\Delta)$. 

By applying this proposition to $V(\Delta\!^!)$, we conclude that $V(\Delta)$ and $V(\Delta\!^!)$ have indeed these inner and noninner torus automorphisms swapped. These tori are precisely in correspondence with those discussed in \autoref{sec:sym-dual}. 

\medskip

It was further shown in \cite{BF25}, Proposition~2, that whenever a rank-$(\ell-1)$ subtorus of these infinitesimal symplectic symmetries of $Y_0(\Delta\!^!)$, for $\ell>1$, is part of $\sl_\ell$ symplectic symmetries (as investigated for instance in \cite{GW09}), then the corresponding noninner automorphisms of $V(\Delta)$ are also part of $\sl_\ell$ noninner automorphisms. We now briefly recall how this works, and prove a symplectic dual statement.

To do so, we first recall some definitions from \autoref{sec:vanish2}, and introduce some new ones. The \svoa{} $V(\Delta)$ is constructed as the kernel of the screening operators
\begin{equation*}
\ker(e^{p^\bot(\rho_1)}_{(0)})\cap\dots\cap\ker(e^{p^\bot(\rho_N)}_{(0)})
\end{equation*}
acting on the Heisenberg-lattice \svoa{} $\HL_{L^\bot}^{\h^\bot}$. Here,
\begin{equation*}
L^\bot=\langle\{-\rho_i^E+\sigma_i^E\}_{i=1}^{N-M}\cup\{-\rho_j+\sigma_j+\tau_j\}_{j=1}^N\rangle_\Z\qquad\text{and}\qquad\h^\bot = L^\bot \otimes \C.
\end{equation*}
Moreover, an explicit expression for the screening operators can be obtained from
\begin{equation*}
p^\bot(\rho_i)=\rho_i-\sigma_i-\tau_i +\sum_{j=1}^{N-M}((E^tE)^{-1}E^t)_{ji}(\sigma_j^E+\tau_j^E).
\end{equation*}
Recall further the definitions
\begin{align*}
J&=\langle\{\tau_j\}_{j=1}^N\rangle_\Z\cong\Z^N,\\
J^\bot &=J\cap\h^\bot=\langle\{\tau_i^E\}_{i=1}^{N-M}\rangle.
\intertext{Similarly, we define}
J^\parallel&\coloneqq J\cap\h^\parallel=\langle\{\tau_i^\Delta\}_{i=1}^M\rangle.
\end{align*}
For $i=1,\dots,N$, we also define
\begin{equation*}
\tilde{\tau}_i\coloneqq-\rho_i+\sigma_i+\tau_i.
\end{equation*}
Notice that since the $-\rho_i+\sigma_i$ are isotropic,
\begin{equation*}
\langle \tilde{\tau}_i , \tilde{\tau}_j \rangle =     \langle \tau_i , \tau_j \rangle = \delta_{ij}.
\end{equation*}
We then define
\begin{equation*}
\tilde{J}\coloneqq \langle\{\tilde{\tau}_i\}_{i=1}^N\rangle_\Z \cong \Z^N,\qquad\tilde{J}^\bot \coloneqq\langle\{\tilde{\tau}_i^E\}_{i=1}^{N-M}\rangle,\qquad\tilde{J}^\parallel \coloneqq\langle\{\tilde{\tau}_i^\Delta\}_{i=1}^M\rangle
\end{equation*}
where for $i=1,\dots , M$
\begin{equation*}
\tilde{\tau}^\Delta_i \coloneqq \sum_{j=1}^N \Delta_{ij}(-\rho_i + \sigma_i +\tau_i) 
\end{equation*}
and for $i=1,\dots , N-M$
\begin{equation*}
\tilde{\tau}^E_i \coloneqq \sum_{j=1}^N E_{ji}(-\rho_i + \sigma_i +\tau_i) .
\end{equation*}
Notice that by construction $\tilde{J}=\tilde{J}\cap\h^\bot$, and so in particular $\tilde{J}^\|\cap\h^\bot=\tilde{J}^\|$.

The lattice \svoa{}s $V_J$ and $V_{\tilde{J}}$ constructed from these lattices are both isomorphic to $V_{\Z^N}$. The whole affine vertex algebra generated by the weight-$1$ space of $V_{\Z^N}$ is $L_1(\mathfrak{so}_{2N})=L_1(D_N)$, with special cases for $N<4$ (see, e.g., \cite{HM23}). Here, we restrict to a slightly smaller affine subalgebra $L_1(\gl_N)$. In the lattice picture, focusing on $V_J$ for notational simplicity, it is isomorphic to the Heisenberg-lattice vertex algebra
\begin{equation*}
\HL_{A_{N-1}}^T\cong\HL_{A_{N-1}}\otimes\pi^{\C(\tau_1+\dots+\tau_N)}
\end{equation*}
where $T=\langle\tau_1,\dots,\tau_N\rangle=J\otimes_\Z\C$ and $A_{N-1}\subset J\subset T$ is the standard realisation (spanned by the vectors $\tau_i-\tau_j$) of the root lattice $A_{N-1}$. This expression, which is evidently isomorphic to $L_1(\gl_N)$, holds for all $N\in\Ns$ if we consider $A_0$ to be the empty root system. Here, the $N$ weight-$1$ fields $\tau_i$ span the Heisenberg \voa{} associated with $T$, and the further $N^2-N$ weight-$1$ fields correspond to the root spaces associated with the roots $\smash{\tau_i-\tau_j\in A_{N-1}}$. Denoting by $\smash{\tilde{A}_{N-1}}$ the respective $A_{N-1}$ root system in the lattice $\smash{\tilde{J}}$, we obtain:
\begin{lem}
The \svoa{}s $V_J$ and $V_{\tilde{J}}$ both have an affine $L_1(\sl_N)$ subalgebra (corresponding to inner automorphisms), isomorphic to lattice \voa{}s $\smash{V_{A_{N-1}}}$ and $\smash{V_{\tilde{A}_{N-1}}}$.
\end{lem}

By considering the conformal embeddings
\begin{equation*}
V_{J^{\| }}\otimes V_{J^\bot} \subset V_J \quad\text{and}\quad V_{\tilde{J}^{\|}}\otimes V_{\tilde{J}^\bot} \subset V_{\tilde{J}},
\end{equation*}
we can further show the following:
\begin{lem}\label{lem:root-sub}
The only root subsystems in $J^\|\cap A_{N-1}$, $J^\bot\cap A_{N-1}$ are of the form
\begin{align*}
A_{\ell^\|-1} \subset J^\| \cap A_{N-1}\quad\text{and}\quad A_{\ell^\bot-1} \subset J^\bot \cap A_{N-1}
\end{align*}
for some $1\leq\ell^\|,\ell^\perp < N$, where by $A_0$ we mean the trivial (empty) root system. These define affine vertex subalgebras
\begin{equation*}
L_1(\sl_{\ell^\bot})\otimes L_1(\sl_{\ell^\|}) \subset V_{J^\|} \otimes V_{J^\perp}\subset V_J\subset M.
\end{equation*}
The same holds for the root subsystems $\tilde{J}^\|\cap \tilde{A}_{N-1}$, $\tilde{J}^\bot\cap \tilde{A}_{N-1}$, which similarly define affine subalgebras of $V_{\tilde{J}}$.
\end{lem}
\begin{proof}
The only root subsystems of an $A$-type root system are again $A$-type root systems. Moreover, the affine subalgebras of $L_1(\sl_N)$ must be again simple quotients.
\end{proof}

It follows from the above presentation of $V(\Delta)$, the fact that $\Delta\!^! = E^t$ and the fact that the screening operators commute with $V_J$ that
\begin{equation*}
L_1(\sl_{\ell^\parallel}) \subset V_{J^\|}  \subset V(\Delta\!^!) \quad \text{and}\quad  L_1(\sl_{\ell^\bot}) \subset V_{J^\bot} \subset  V(\Delta).
\end{equation*}
We therefore get the following proposition.
\begin{prop}\label{prop:inner}
Let $A_{\ell^\|-1}$ and $A_{\ell^\bot-1}$ be as in \autoref{lem:root-sub}. Then $V(\Delta)$ has an $L_1(\sl_{\ell^\|})$ affine subalgebra coming from $L_1(\gl_N)\subset V_J$. Dually, $V(\Delta\!^!)$ has an $L_1(\sl_{\ell^{\bot}})$ affine subalgebra.
\end{prop}

\smallskip

Like for the Heisenberg fields $\tau^\Delta_i$, $i=1,\dots,M$, no generating field in $\smash{V_{J^\|}}$ and $\smash{V_{J^\bot}}$ is an element of $V(\Delta)$ and $V(\Delta\!^!)$, respectively, and therefore none of the affine subalgebras they may generate. In fact, they are not BRST closed by construction. By contrast, the fields generating $\smash{V_{\tilde{J}^\|}}$ and $\smash{V_{\tilde{J}^\bot}}$ are BRST closed because they are in $\smash{V_{L^\bot}^{\h^\bot}}$. However, they are not annihilated by the screening operators for $e^{\smash{p^\bot(\rho_i)}}_{(0)}$, $i=1,\dots, N$. Therefore, they also do not descend to inner automorphisms of $\smash{V(\Delta)}$ and $\smash{V(\Delta\!^!)}$, respectively. In \cite{BF25} it was shown that their zero modes map the kernel of the screening operators to itself. They hence define noninner automorphisms of $V(\Delta)$ and $V(\Delta\!^!)$, respectively.
\begin{prop}[\cite{BF25}]\label{prop:outer}
Let $\smash{\tilde{A}_{\ell^\|-1}}$ and $\smash{\tilde{A}_{\ell^\bot-1}}$ be as in \autoref{lem:root-sub}. Then $V(\Delta)$ has an infinitesimal $\sl_{\ell^\|}$ noninner automorphism symmetry. Dually, $V(\Delta\!^!)$ has an infinitesimal $\sl_{\ell^{\bot}}$ noninner automorphism symmetry.
\end{prop}
\begin{proof}
We prove the assertion for $V(\Delta)$. Suppose that $\tilde{\tau}_i - \tilde{\tau}_j \in \tilde{A}_{\ell^\|}$. Then it follows from orthogonality between $\Delta$ and $E^t$ that
\begin{equation*}
\sum_{l=1}^{N-M}((E^tE)^{-1}E^t)_{li}(\sigma_l^E+\tau_l^E) = \sum_{l=1}^{N-M}((E^tE)^{-1}E^t)_{lj}(\sigma_l^E+\tau_l^E).
\end{equation*}
Thus,
\begin{equation*}
e_{(0)}^{p^\bot (\rho_i)} e^{\tilde{\tau}_i - \tilde{\tau}_j }_{(0)} = e^{\tilde{\tau}_i - \tilde{\tau}_j }_{(0)} e_{(0)}^{p^\bot (\rho_i)}+e_{(0)}^{p^\bot (\rho_j)}.
\end{equation*}
Moreover, $\smash{e_{(0)}^{p^\bot (\rho_i)}}$ and $\smash{e_{(0)}^{\tilde{\tau}_i - \tilde{\tau}_j}}$ commute for $k\neq i,j$. Therefore, the kernel of the screening operators is preserved by the action of these zero modes, which by the discussion above is what remained to be shown.
\end{proof}

\smallskip

To summarise, it follows from the above results that the symplectic dual \svoa{}s $V(\Delta)$ and $V(\Delta\!^!)$ have certain (infinitesimal) inner and noninner automorphisms of $A$-type swapped. Moreover, the Cartan elements of these Lie algebras reflect the isomorphism between the Lie algebra of infinitesimal torus isometries and the Lie algebra of stability parameters of the two symplectic dual varieties $Y_0(\Delta)$ and $Y_0(\Delta\!^!)$.

Thus, these inner and noninner automorphisms can be thought of as a refinement, at the level of chiral quantisations, of this basic symplectic duality statement.

\begin{rem}
Not all $\sl_\ell$-type noninner automorphisms arise in this way. As an example, consider $\Delta=(1)$, which gives for $V(\Delta)$ the symplectic fermion \svoa{} (see \autoref{sec:smallexample}). Although it does not satisfy the conditions of the above results, it has an $\sl_2$ noninner automorphism group.
\end{rem}
Examples will be discussed in \autoref{sec:examples}.


\subsection{A Remark on Free-Field Realisations}\label{sec:freefieldrealisation}

Recall that the free-field realisation of $V(\Delta)$ obtained in \autoref{sec:freefield2} relies on a \svoa{} embedding
\begin{equation*}
\mathcal{D}^{\mathrm{ch}}(T^*V) \otimes \mathcal{C}\ell (\Pi T^* V)\hookrightarrow \mathcal{D}^{\mathrm{ch}}(T^*(\C^\times)^N) \otimes \mathcal{C}\ell (\Pi T^* V), 
\end{equation*}
which is a chiralisation of the embedding $T^*(\C^\times)^N \hookrightarrow T^*V$. If we perform the torus symplectic reduction on $T^*(\C^\times)^N$ instead of $T^*V$, with (abusing notation) $\mu^{-1}(0)\subset T^*(\C^\times)^N$ the inverse image of the moment map, we obtain
\begin{equation*}
U \coloneqq \mu^{-1}(0)\qquot G \cong T^*(\C^\times)^{N-M}.
\end{equation*}
It was shown in \cite{BF25} that under our assumptions on $\Delta$ there exists a generic effective stability parameter $\delta_\Delta \in \Q^M$ such that $U$ can be identified with an openly embedded subset
\begin{equation*}
U \hookrightarrow Y_{\delta_\Delta} (\Delta).
\end{equation*}
It should follow from our sheaf construction that the free-field realisation obtained in \autoref{sec:freefield2} can equally be obtained by localising our sheaf of $\hbar$-adic \svoa{}s to $U\subset Y_{\delta_\Delta} (\Delta)$, and taking global sections; we leave this to future work.

Also, note that there are two ways to obtain an embedding $T^*\C^\times\hookrightarrow T^* \C$, by removing either the zero element in the basis or the fibre. It is then in principle possible to consider $2^N$ embeddings $T^*(\C^\times)^N\hookrightarrow T^*V$. It was further shown in \cite{BF25} that some of the resulting free-field realisations are better suited to manifest the noninner automorphisms studied in \autoref{sec:auto}. Others are better suited to find free-field realisations based on localisations of the form
\begin{equation*}
T^*\C^{N-M-L}\times T^*(\C^{\times})^L\hookrightarrow T^*V
\end{equation*}
for $1 \leq L \leq N-M$.


\section{Examples}\label{sec:examples}

In this section, we discuss two families of examples of (quiver) boundary hypertoric \svoa{}s $V(\Delta)$, those associated with the minimal nilpotent orbit closures for $\sl_N$ (see \autoref{sec:var1}) and those associated with the Kleinian singularities of type $A_{N-1}$ (see \autoref{sec:var2}). The minimal hypertoric \voa{}s $V_\mathrm{min}(\Delta)$ were already described in Section~10 of \cite{Kuw21}. The extension to the boundary \svoa{} in the first example was first treated in \cite{FS24} (relying on results of \cite{AM21}), where the associated variety was computed using alternative methods. Some further partial results on \svoa{}s were obtained in \cite{Yos23,Sas25}. Here, we focus in particular on the fermionic extension picture developed in \autoref{sec:hypertoricsvoa}.

Recall from \autoref{prop:properties} that the vertex operator (super)algebras $V_\mathrm{min}(\Delta)$ and $V(\Delta)$ are simple, self-contragredient and of CFT-type and from \autoref{cor:quasilisse} that the $V(\Delta)$ are quasi-lisse. The characters will be described in \autoref{sec:chars}.  We also mention that the \svoa{}s $V(\Delta)$ are $(1/2)\N$-graded by $L_0$-eigenvalues, but there is in general no relation between the parity and the integrality of these eigenvalues.


\subsection{Minimal Nilpotent Orbit Closure for \texorpdfstring{$\sl_N$}{sl\_N}}\label{sec:nilpotent_svoa}

We fix $N\in\Ns$, $M=1$, the $(1\times N)$-matrix $\Delta=(1,\dots,1)$ and the corresponding $(N\times(N-1))$-matrix $E$ as in \autoref{sec:var1}. See \autoref{sec:quiver}, and in particular \autoref{fig:ex1}, for the corresponding quiver realisation.

We summarise the properties of the hypertoric vertex operator (super)algebras and their identifications with affine vertex operator (super)algebras given in \cite{Kuw21,FS24} (see also \cite{AP14,AM17b}).
\begin{prop}\label{prop:ex1}
For all $N\in\Ns$, the minimal hypertoric \voa{} $V_\mathrm{min}(\Delta)$ is simple, self-contragredient, of CFT-type and central charge $c=-N-1$. For $N\geq3$, it is isomorphic to $V_\mathrm{min}(\Delta)\cong L_{-1}(\sl_N)$, the simple quotient of the affine \voa{} for $\sl_N$ at level $-1$. For $N\geq 4$, the associated variety of $V_\mathrm{min}(\Delta)$ is given by $\overline{\mathbb{S}_\mathrm{min}(\sl_N)}$, the closure of the unique minimal Dixmier sheet containing the minimal nilpotent orbit $\mathbb{O}_\mathrm{min}(\sl_N)$.

For $N\in\Ns$, the boundary hypertoric \svoa{} $V(\Delta)$ is simple, self-contragredient, of CFT-type (and $\N$-graded) and central charge $c=-2$.  It is quasi-lisse and its associated variety is given by $\smash{X_{V(\Delta)}\cong Y_0(\Delta)\cong\overline{\mathbb{O}_\mathrm{min}(\sl_N)}}$, the minimal nilpotent orbit closure of $\sl_N$. Moreover, for $N=1$ and $N\geq3$, it is isomorphic to $V(\Delta)\cong L_1(\psl_{N|N})\cong L_{-1}(\psl_{N|N})$, the simple quotient of the affine \svoa{} for $\psl_{N|N}$ at level $\pm1$.
\end{prop}
The associated variety statement for $V(\Delta)$ holds for all $N\geq 1$, provided one identifies the minimal nilpotent orbit closure with a point for $N=1$.

We describe a few special cases for small $N$. For $N=1$, as described in \autoref{sec:smallexample}, $V_\mathrm{min}(\Delta)$ is isomorphic to the singlet \voa{} $\mathcal{M}(2)$ \cite{Wan98a}, and the associated variety is the affine line $\C$ with trivial Poisson structure. On the other hand, $V(\Delta)\cong\mathsf{SF}$ is isomorphic to the symplectic fermion \svoa{}, of central charge $c=-2$, which can be identified (somewhat trivially) with $L_{\pm1}(\psl_{1|1})$.

For $N=2$, $V(\Delta)$ is isomorphic to the simple rectangular $\mathcal{W}$-algebra of $\sl_4$ at level $-5/4$ \cite{CKLR19}, which is a conformal extension of $L_{-1}(\sl_2)$.

\medskip

In the following, for $N\geq3$, we describe how $\smash{V(\Delta)\cong L_1(\psl_{N|N})}$ can be understood as a simple-current extension of $\smash{V_\mathrm{min}(\Delta)\otimes\pi^{T^E}\cong L_{-1}(\sl_N)\otimes\pi^{\otimes(N-1)}}$, as explained in \autoref{sec:hypertoricsvoa} in the general case. Some aspects of this were also treated in \cite{AM21} for this particular example.

In \autoref{fig:ex1_voa}, we show the extension steps summarised in \autoref{sec:summary}. Like in \autoref{rem:intermediate}, we also consider the dual pair $\smash{\overline{V_\mathrm{min}(\Delta)}\otimes V_{J^\bot}}$ in $V(\Delta)$, with the double commutant $\smash{\overline{V_\mathrm{min}(\Delta)}\coloneqq\Com_{V(\Delta)}(\Com_{V(\Delta)}(V_\mathrm{min}(\Delta)))}$. In this particular case, this corresponds nicely to the vertex operator (super)algebra extension $\smash{\mathcal{U}=\overline{V_\mathrm{min}(\Delta)}}$ of $L_{-1}(\sl_N)$ defined in \cite{AM21}. (In the last row of \autoref{fig:ex1_voa}, $\alpha=\alpha^\Delta+\alpha^E$ corresponds to the orthogonal decomposition $T=T^\Delta\oplus T^E$.)
\begin{figure}[ht]
\adjustbox{scale=0.9,center}{%
\begin{tikzcd}[column sep=small]
\begin{array}{l}M_\mathrm{min}\otimes\pi^{T^E}\\[+1mm]
=\bigl(\mathcal{D}^\mathrm{ch}(T^*V)\otimes\pi^{T^\Delta}\bigr)\otimes\pi^{T^E}\end{array} \arrow[r,"\text{BRST}"]\arrow[d,hook,"A_{N-1}"]
&\begin{array}{l}V_\text{min}(\Delta)\otimes\pi^{T^E}\\[+1mm]
=L_{-1}(\sl_N)\otimes\pi^{T^E}\end{array}
\arrow[d,hook,"A_{N-1}"]\\
\begin{array}{l}M_\mathrm{min}\otimes V_{J^\bot}\\[+1mm]
=\bigl(\mathcal{D}^\mathrm{ch}(T^*V)\otimes\pi^{T^\Delta}\bigr)\otimes V_{A_{N-1}}\\[+1mm]
=\bigoplus_{\alpha\in A_{N-1}}\bigl(\mathcal{D}^\mathrm{ch}(T^*V)\otimes\pi^{T^\Delta}\bigr)\otimes\pi_\alpha^{T^E}\end{array}
\arrow[r,"\text{BRST}"]\arrow[d,hook,"\Z(N)"]
&\begin{array}{l}V_\mathrm{min}(\Delta)\otimes V_{J^\bot}\\[+1mm]
=L_{-1}(\sl_N)\otimes V_{A_{N-1}}\end{array}
\arrow[d,hook,"\Z(N)"]\\
\begin{array}{l}\bigl(\mathcal{D}^\mathrm{ch}(T^*V)\otimes V_{\Z(N)}\bigr)\otimes V_{A_{N-1}}\\[+1mm]
=\bigoplus_{(\alpha,\alpha')\in\Z(N)\oplus A_{N-1}}\bigl(\mathcal{D}^\mathrm{ch}(T^*V)\otimes\pi_\alpha^{T^\Delta}\bigr)\otimes\pi_{\alpha'}^{T^E}\end{array}
\arrow[r,"\text{BRST}"]\arrow[d,hook,"\Z/N\Z"]
&\begin{array}{l}\overline{V_\mathrm{min}(\Delta)}\otimes V_{J^\bot}\\[+1mm]
=\mathcal{U}\otimes V_{A_{N-1}}\\[+1mm]
=\bigoplus_{s\in N\Z}V_s\otimes V_{A_{N-1}}\end{array}
\arrow[d,hook,"\Z/N\Z"]\\
\begin{array}{l}M=\mathcal{D}^\mathrm{ch}(T^*V)\otimes\mathcal{C}\ell(\Pi T^*V)\\[+1mm]
=\bigl(\mathcal{D}^\mathrm{ch}(T^*V)\otimes V_\Z\bigr)\otimes V_{\Z^{N-1}}\\[+1mm]
=\bigoplus_{i\in\Z/N\Z}\bigl(\mathcal{D}^\mathrm{ch}(T^*V)\otimes V_{i+\Z(N)}\bigr)\otimes V_{\tau(i)+A_{N-1}}\\[+1mm]
=\bigoplus_{\alpha\in\Z^N}\bigl(\mathcal{D}^\mathrm{ch}(T^*V)\otimes\pi_{\alpha^\Delta}^{T^\Delta}\bigr)\otimes\pi_{\alpha^E}^{T^E}\end{array}
\arrow[r,"\text{BRST}"]
&\begin{array}{l}V(\Delta)=L_1(\psl_{N|N})\\[+1mm]
=\bigoplus_{i\in\Z/N\Z}\mathcal{U}_i\otimes V_{\tau(i)+A_{N-1}}\end{array}
\end{tikzcd}}
\caption{Vertex operator (super)algebras for the minimal nilpotent orbit closure for $\sl_N$ for $N\geq3$.}
\label{fig:ex1_voa}
\end{figure}

Here, we have identified $\smash{\mathcal{C}\ell(\Pi T^*V)=V_J\cong V_{\Z^N}=\bigoplus_{\alpha\in\Z^N}\pi^T_\alpha}$ with the lattice \svoa{} for the odd standard lattice $\smash{J=\langle\{\tau_i\}_{i=1}^N\rangle_\Z\cong\Z^N}$. The diagram in \autoref{fig:ex1_voa} then corresponds to the stepwise lattice extension
\begin{equation*}
\{0\}\oplus\{(0,\dots,0)^t\}\overset{A_{N-1}}{\hooklongrightarrow}\{0\}\oplus A_{N-1}\overset{\Z(N)}{\hooklongrightarrow}\Z(N)\oplus A_{N-1}\overset{\Z/N\Z}{\hooklongrightarrow}\Z^N.
\end{equation*}
The lattice $\Z(N)$ is given by $\langle\tau_1^\Delta\rangle_\Z$ with basis vector $\tau_1^\Delta=\tau_1+\dots+\tau_N$. Its inner product matrix is $\Delta\Delta^t=(N)$. Hence, it is indeed isomorphic to $\Z(N)$, the odd standard lattice with quadratic form scaled by $N$. This is an integral lattice, which is even if and only if $N$ is. On the other hand, the root lattice $A_{N-1}$ is concretely given by $\smash{J^\bot=\langle\{\tau_i^E\}_{i=1}^{N-1}\rangle_\Z}$ with basis vectors $\smash{\tau_1^E=\tau_1-\tau_2,\dots,\tau_{N-1}^E=\tau_{N-1}-\tau_N}$ and inner product matrix $E^tE$. This is the standard realisation of the even lattice $A_{N-1}\subset\R^N$.

The last extension is of finite index, along $\Z^N/(\Z(N)\oplus A_{N-1})\cong\Z/N\Z$. It can also be understood as the lattice gluing
\begin{equation*}
\Z^N=\bigcup_{i\in\Z(N)'/\Z(N)}(i+\Z(N))\oplus(\tau(i)+A_{N-1}),
\end{equation*}
where $\tau$ is an isomorphism between the finite abelian groups $\Z(N)'/\Z(N)\cong\Z/N\Z$ and $(A_{N-1})'/A_{N-1}\cong\Z/N\Z$. In fact, $\tau$ is an odd analogue of an anti-isometry between discriminant forms (which would describe the gluing to a unimodular lattice if all the involved lattices were even).

The right-hand side of the diagram shows how these extensions behave under the relative BRST cohomology. The resulting extension picture coincides nicely with the realisation of $L_{-1}(\sl_N)$ inside $L_1(\psl_{N|N})$ described in \cite{AM21} (see also \cite{FS24}). Indeed, we recover the conformal embedding
\begin{equation*}
L_1(\psl_{N|N})\supseteq L_{-1}(\sl_N)\otimes L_1(\sl_N),
\end{equation*}
owing to the fact that the even part of the Lie algebra $\psl_{N|N}$ is $\sl_N\oplus\sl_N$. Recall that $\smash{L_1(\sl_N)\cong V_{A_{N-1}}}$ is the lattice \voa{} associated with the root lattice $A_{N-1}$ of $\sl_N$.

Some aspects of the representation theory of $L_{-1}(\sl_N)$ are described in \cite{AM21}, as well as the decomposition of $L_1(\psl_{N|N})$ into modules for $L_{-1}(\sl_N)\otimes V_{A_{N-1}}$. This can be matched perfectly with our lattice extension picture before BRST reduction. Following \cite{AM21}, we denote by $V_s$, $s\in\Z$, the complete list of irreducible ordinary modules of $L_{-1}(\sl_N)=V_0$, which have group-like fusion parametrised by $\Z$ and lowest $L_0$-weights $s^2/6+|s|/2$. We can read off from the above diagram that the $V_s$ for $s\in\Z$ are exactly the relative BRST cohomologies of the $\smash{\mathcal{D}^\mathrm{ch}(T^*V)\otimes\pi_\alpha^{T^\Delta}}$ for $\alpha\in\Z(N)'=(1/N)\Z(N)$, identifying $s=N\alpha$.

Then, we consider the simple-current extension $\smash{\mathcal{U}=\bigoplus_{s\in N\Z}V_s}$, which is the BRST reduction of $\mathcal{D}^\mathrm{ch}(T^*V)\otimes V_{\Z(N)}$. This is now the commutant of $V_{A_{N-1}}$ inside $\smash{L_1(\psl_{N|N})}$, i.e.\ $\smash{\mathcal{U}=\overline{V_\mathrm{min}(\Delta)}}$. In fact, $\mathcal{U}$ and $V_{A_{N-1}}$ form a dual pair in $\smash{L_1(\psl_{N|N})}$. $\mathcal{U}$ is a simple \svoa{} and (completely parallel to $V_{\Z(N)}$ before BRST reduction) a \voa{} if $N$ is even. The $N$ many irreducible ordinary modules of $\smash{\mathcal{U}=\mathcal{U}_0}$ are given by $\smash{\smash{\mathcal{U}_i=\bigoplus_{s\in N\Z}V_{s+i}}}$ for $i\in\Z/N\Z$ with group-like fusion rules parametrised by $\Z/N\Z$.

Finally, we see that the \svoa{} $\smash{V(\Delta)=L_1(\psl_{N|N})}$ is a mirror extension (of index $\Z/N\Z$) of the dual pair $\smash{\overline{V_\mathrm{min}(\Delta)}\otimes V_{J^\bot}=\mathcal{U}\otimes V_{A_{N-1}}}$, as indicated in the above diagram. It is described by the super analogue $\tau$ of a braiding-reversing equivalence, which is induced from the corresponding lattice gluing map $\tau$ before BRST reduction.


\subsection{Kleinian Singularity of Type \texorpdfstring{$A_{N-1}$}{A\_(N-1)}}\label{sec:kleinian_svoa}

We then consider the symplectic dual example to the previous one. That is, we fix $N\in\Ns$, $M=N-1$, the $((N-1)\times N)$-matrix $\Delta$ and the $(N\times 1)$-matrix $\smash{E=(1,\dots,1)^t}$ as in \autoref{sec:var2}. Again, see \autoref{sec:quiver}, and in particular \autoref{fig:ex2}, for the corresponding quiver realisation.

We summarise the properties of the hypertoric vertex operator (super)algebras and, in the minimal case, their identifications with certain $\mathcal{W}$-algebras \cite{FS04} given in \cite{Kuw21}. See, \autoref{rem:namikawa} and \cite{Nag21,Kuw21} for details regarding the (finite) Namikawa-Weyl group $\mathbb{W}$.
\begin{prop}\label{prop:ex2}
For $N\geq2$, the minimal hypertoric \voa{} $V_\mathrm{min}(\Delta)$ is simple, self-contragredient, of CFT-type and central charge $c=-2N+1$. After taking invariants under the Namikawa-Weyl group $\mathbb{W}$, it becomes isomorphic to $V_\mathrm{min}(\Delta)^\mathbb{W}\cong\mathcal{W}_{-N+1}(\sl_N,f_\mathrm{sub})$, the simple quotient of the subregular $\mathcal{W}$-algebra for $\sl_N$ at level $-N+1$. The associated variety of $V_\mathrm{min}(\Delta)^\mathbb{W}$ is given by the Slodowy slice $\mathcal{S}_{f_\mathrm{sub}}\subset\sl_N^*$.

For $N\geq2$, the boundary hypertoric \svoa{} $V(\Delta)$ is simple, self-contragredient, of CFT-type (and $\frac{1}{2}\N$-graded) and central charge $c=-2N+2$. It is quasi-lisse and its associated variety is given by $\smash{X_{V(\Delta)}\cong Y_0(\Delta)\cong X_{A_{N-1}}}$, the Kleinian singularity of type $A_{N-1}$, which is isomorphic to $\mathcal{S}_{f_\mathrm{sub}}\cap\mathcal{N}$, the Slodowy slice intersected with the nilpotent cone of $\sl_N$
\end{prop}
The case $N=1$ could be considered as well, but is is somewhat trivial as $M=0$, so no reduction is performed at all (cf.\ \autoref{rem:zerocolumnrow}). Correspondingly, $V_\mathrm{min}(\Delta)\cong\mathcal{D}^\mathrm{ch}(T^*\C)$ is the Weyl vertex algebra of central charge $c=-1$, and $V(\Delta)\cong\mathcal{D}^\mathrm{ch}(T^*\C)\otimes\mathcal{C}\ell(\Pi T^*\C)$, tensored with a Clifford \svoa{}, has central charge $c=0$. For both, the associated variety is $Y_0(\Delta)=T^*\C$, which is not singular.

We explain how the \voa{} $V_\mathrm{min}(\Delta)$ from \cite{Kuw21}, satisfying $V_\mathrm{min}(\Delta)^\mathbb{W}\cong\mathcal{W}_{-N+1}(\sl_N,f_\mathrm{sub})$, extends to the nicer (and in particular quasi-lisse) \svoa{} $V(\Delta)$ considered in this work. Again, as an intermediate step, we take the double commutant $\smash{\overline{V_\mathrm{min}(\Delta)}\coloneqq\Com_{V(\Delta)}(\Com_{V(\Delta)}(V_\mathrm{min}(\Delta)))}$ so that $\smash{\overline{V_\mathrm{min}(\Delta)}\otimes V_{J^\bot}}$ forms a dual pair in $V(\Delta)$. This is summarised in the diagram in \autoref{fig:ex2_voa} (cf.\ \autoref{sec:summary}), where in the last row, $\alpha=\alpha^\Delta+\alpha^E$ corresponds to the orthogonal decomposition $T=T^\Delta\oplus T^E$.
\begin{figure}[ht]
\adjustbox{scale=0.9,center}{%
\begin{tikzcd}[column sep=small]
\begin{array}{l}M_\mathrm{min}\otimes\pi^{T^E}\\[+1mm]
=\bigl(\mathcal{D}^\mathrm{ch}(T^*V)\otimes\pi^{T^\Delta}\bigr)\otimes\pi^{T^E}\end{array} \arrow[r,"\text{BRST}"]\arrow[d,hook,"\Z(N)"]
&V_\text{min}(\Delta)\otimes\pi^{T^E}
\arrow[d,hook,"\Z(N)"]\\
\begin{array}{l}M_\text{min}\otimes V_{J^\bot}\\[+1mm]
=\bigl(\mathcal{D}^\mathrm{ch}(T^*V)\otimes\pi^{T^\Delta}\bigr)\otimes V_{\Z(N)}\\[+1mm]
=\bigoplus_{\alpha\in\Z(N)}\bigl(\mathcal{D}^\mathrm{ch}(T^*V)\otimes\pi^{T^\Delta}\bigr)\otimes\pi_\alpha^{T^E}\end{array}
\arrow[r,"\text{BRST}"]\arrow[d,hook,"A_{N-1}"]
&\begin{array}{l}V_\text{min}(\Delta)\otimes V_{J^\bot}\\[+1mm]
=V_\text{min}(\Delta)\otimes V_{\Z(N)}\end{array}
\arrow[d,hook,"A_{N-1}"]\\
\begin{array}{l}\bigl(\mathcal{D}^\mathrm{ch}(T^*V)\otimes V_{A_{N-1}}\bigr)\otimes V_{\Z(N)}\\[+1mm]
=\bigoplus_{(\alpha,\alpha')\in A_{N-1}\oplus\Z(N)}\bigl(\mathcal{D}^\mathrm{ch}(T^*V)\otimes\pi_\alpha^{T^\Delta}\bigr)\otimes\pi_{\alpha'}^{T^E}\end{array}
\arrow[r,"\text{BRST}"]\arrow[d,hook,"\Z/N\Z"]
&\begin{array}{l}\overline{V_\mathrm{min}(\Delta)}\otimes V_{J^\bot}\\[+1mm]
=\overline{V_\mathrm{min}(\Delta)}\otimes V_{\Z(N)}\\[+1mm]
=\bigoplus_{\alpha\in A_{N-1}}\mathcal{W}_\alpha\otimes V_{\Z(N)}\end{array}
\arrow[d,hook,"\Z/N\Z"]\\
\begin{array}{l}
M=\mathcal{D}^\mathrm{ch}(T^*V)\otimes\mathcal{C}\ell(\Pi T^*V)\\[+1mm]
=\bigl(\mathcal{D}^\mathrm{ch}(T^*V)\otimes V_{\Z^{N-1}}\bigr)\otimes V_\Z\\[+1mm]
=\bigoplus_{i\in\Z/N\Z}\bigl(\mathcal{D}^\mathrm{ch}(T^*V)\otimes V_{i+A_{N-1}}\bigr)\otimes V_{\tau(i)+\Z(N)}\\[+1mm]
=\bigoplus_{\alpha\in\Z^N}\bigl(\mathcal{D}^\mathrm{ch}(T^*V)\otimes\pi_{\alpha^\Delta}^{T^\Delta}\bigr)\otimes\pi_{\alpha^E}^{T^E}\end{array}
\arrow[r,"\text{BRST}"]
&\begin{array}{l}V(\Delta)\\[+1mm]
=\bigoplus_{i\in\Z/N\Z}\overline{V_\mathrm{min}(\Delta)}_i\otimes V_{\tau(i)+\Z(N)}\end{array}
\end{tikzcd}}
\caption{Vertex operator (super)algebras for the Kleinian singularity of type $A_{N-1}$.}
\label{fig:ex2_voa}
\end{figure}

Again, we have identified $\smash{\mathcal{C}\ell(\Pi T^*V)=V_J\cong V_{\Z^N}=\bigoplus_{\alpha\in\Z^N}\pi^T_\alpha}$ with the lattice \svoa{} for the odd standard lattice $\smash{J=\langle\{\tau_i\}_{i=1}^N\rangle_\Z\cong\Z^N}$. The left column of the diagram is based on the stepwise lattice extension
\begin{equation*}
\{(0,\dots,0)^t\}\oplus\{0\}\overset{\Z(N)}{\hooklongrightarrow}\{0\}\oplus\Z(N)\overset{A_{N-1}}{\hooklongrightarrow}A_{N-1}\oplus \Z(N)\overset{\Z/N\Z}{\hooklongrightarrow}\Z^N
\end{equation*}
with the (even) type-$\smash{A_{N-1}}$ root lattice $\smash{A_{N-1}\cong\langle\{\tau_i^\Delta\}_{i=1}^{N-1}\rangle_\Z}$ with basis vectors $\smash{\tau_1^\Delta=\tau_1-\tau_2,\dots,\tau_{N-1}^\Delta=\tau_{N-1}-\tau_N}$ and inner product matrix $\Delta\Delta^t$ and with the integral lattice $\Z(N)\cong J^\bot=\langle\tau_1^E\rangle_\Z$ with basis vector $\tau_1^E=\tau_1+\dots+\tau_N$ and inner product matrix $E^tE=(N)$, which yields an even lattice whenever $N$ is even.

The last extension step is of finite index $\Z^N/(A_{N-1}\oplus\Z(N))\cong\Z/N\Z$ and can be described by the lattice gluing
\begin{equation*}
\Z^N=\bigcup_{i\in(A_{N-1})'/A_{N-1}}(i+A_{N-1})\oplus(\tau(i)+\Z(N)),
\end{equation*}
where $\tau$ is an odd analogue of an anti-isometry between the finite abelian groups $(A_{N-1})'/A_{N-1}\cong\Z/N\Z$ and $\Z(N)'/\Z(N)\cong\Z/N\Z$.

These extensions are transferred to the right column of the diagram by applying the relative BRST cohomology. We denote by $\mathcal{W}_\alpha$ for $\alpha\in(A_{N-1})'$ the irreducible modules for $\smash{V_\mathrm{min}(\Delta)=\mathcal{W}_0}$ obtained by BRST reduction of $\smash{\mathcal{D}^\mathrm{ch}(T^*V)\otimes\pi_\alpha^{T^\Delta}}\!$. Then we take the simple-current extension $\smash{\overline{V_\mathrm{min}(\Delta)}=\bigoplus_{\alpha\in A_{N-1}}\mathcal{W}_\alpha}$ of $\smash{V_\mathrm{min}(\Delta)=\mathcal{W}_0}$. It is a \voa{} obtained as BRST reduction of $\smash{\mathcal{D}^\mathrm{ch}(T^*V)\otimes V_{A_{N-1}}}$, and together with the vertex operator (super)algebra $\smash{V_{J^\bot}=V_{\Z(N)}}$ it forms a dual pair in $V(\Delta)$. $\smash{\overline{V_\mathrm{min}(\Delta)}}$ has the irreducible modules $\smash{\overline{V_\mathrm{min}(\Delta)}_i=\bigoplus_{\alpha\in A_{N-1}}\mathcal{W}_{\alpha+i}}$ indexed by $i\in\Z/N\Z\cong(A_{N-1})'/A_{N-1}$.

Finally, the \svoa{} $\smash{V(\Delta)}$ is a mirror extension (of index $\Z/N\Z$) of the dual pair $\smash{\overline{V_\mathrm{min}(\Delta)}\otimes V_{J^\bot}=\overline{V_\mathrm{min}(\Delta)}\otimes V_{\Z(N)}}$, as indicated in the above diagram. It is described by the super analogue $\tau$ of a braiding-reversing equivalence, which is induced from the corresponding lattice gluing map $\tau$ before BRST reduction.

\medskip

The associated varieties of the minimal and hypertoric vertex operator (super)algebras in the two examples fit nicely into the the diagram in \autoref{fig:assvar} (assuming $N\geq4$ and $N\geq2$, respectively), where going to the quasi-lisse fermionic extension corresponds to intersecting with the nilpotent cone $\mathcal{N}\subset\sl_N^*$. We expect a similar behaviour to be true in general. As we explained in \autoref{sec:sym-dual}, the hypertoric varieties on the bottom are symplectic duals. In both cases, the resolution $\pi\colon Y_\delta(\Delta)\to Y_0(\Delta)$ is essentially the Springer resolution.
\begin{figure}[ht]
\begin{tikzcd}
V_\mathrm{min}(\Delta)\cong L_{-1}(\sl_N)\arrow[d,hook]&\overline{\mathbb{S}_\mathrm{min}(\sl_N)}\subset\sl_N^*\\
V(\Delta)\cong L_{\pm1}(\psl_{N|N})&\overline{\mathbb{O}_\mathrm{min}(\sl_N)}=\overline{\mathbb{S}_\mathrm{min}(\sl_N)}\cap\mathcal{N}\subset\mathcal{N}\arrow[u,hook]\\
V_\mathrm{min}(\Delta)^\mathbb{W}=\mathcal{W}_{-N+1}(\sl_N,f_\mathrm{sub})\arrow[d,hook]&\mathcal{S}_{f_\mathrm{sub}}\subset\sl_N^*\\
V(\Delta)&X_{A_{N-1}}\cong\mathcal{S}_{f_\mathrm{sub}}\cap\mathcal{N}\subset\mathcal{N}\arrow[u,hook]
\end{tikzcd}
\caption{Symplectic dual \svoa{}s and their associated varieties.}
\label{fig:assvar}
\end{figure}


\section{Characters and Modularity}\label{sec:chars}

In this section, using the Euler-Poincaré principle, we compute the characters and supercharacters of the minimal $V_\mathrm{min}(\Delta)$ and boundary hypertoric vertex operator (super)algebras $V(\Delta)$ constructed in \autoref{sec:hypertoricsvoa} and study their modular properties. By \autoref{thm:glob-sec}, these are also the characters of the global sections of the corresponding sheaves constructed in \autoref{sec:sheaves} over the smooth hypertoric varieties. We will see that in the boundary case the supercharacters (and conjecturally also the characters) are quasimodular. One the other hand, in the minimal case, they are only partial (or false) theta functions.


\subsection{General Expression}

We recall from \autoref{sec:hypertoricSVOA} the cohomological definition of the hypertoric vertex operator (super)algebras
\begin{equation*}
V_\mathrm{min}(\Delta)=H^\bullet(C_\mathrm{min},d)\quad\text{and}\quad V(\Delta)=H^\bullet(C,d).
\end{equation*}
With this, by the Euler-Poincaré principle, the characters and supercharacters are given by
\begin{align*}
\ch_{V_\mathrm{min}(\Delta)}(q)&=\tr_{V_\mathrm{min}(\Delta)}q^{L_0-c/24}=\tr_{C_\mathrm{min}}q^{L_0-c/24}(-1)^{P_\mathrm{gh}},\\
\operatorname{(s)ch}_{V(\Delta)}(q)&=\tr_{V(\Delta)}q^{L_0-c/24}(\pm1)^P=\tr_Cq^{L_0-c/24}(\pm1)^{P_\mathrm{m}}(-1)^{P_\mathrm{gh}}.
\end{align*}
Here, $P$ denotes the parity operator. On $C_\mathrm{min}\subset\tilde{C}_\mathrm{min}\coloneqq M_\mathrm{min}\otimes\mathcal{C}\ell(\Pi T^*\g)$ and $C\subset\tilde{C}=M\otimes\mathcal{C}\ell(\Pi T^*\g)$ we split the parity operator as $P=P_\mathrm{m}+P_\mathrm{gh}$, where $P_\mathrm{m}$ only acts on $M_\mathrm{min}$ or $M$ and $P_\mathrm{gh}$ only on $\mathcal{C}\ell(\Pi T^*\g)$.

Recall further that $C$ and $C_\mathrm{min}$ are given by the kernels in $\tilde{C}$ and $\tilde{C}_\mathrm{min}$, respectively, under the zero modes of $b_i$ and of $d\, b_i=Q_{(0)}b_i=\bar\mu^*_\mathrm{ch}(a_i)$ for all $i=1,\dots,M$. The former reduce the ghost free fermion \svoa{} $\mathcal{C}\ell(\Pi T^*\g)$ to the ghost symplectic fermion \svoa{} strongly generated by $\partial c_i$ and $b_i$ for $i=1,\dots,M$. Hence, $C_\mathrm{min}$ is given by the trivial subrepresentation of $\smash{\g=(\gl_1)^M}$ acting on $\bar{C}_\text{min}$ via $\smash{\mu^*_\mathrm{ch}(a_i)_{(0)}}$, and similarly $C$ the trivial subrepresentation in $\bar{C}$. The weights of the action are precisely given by the underlying torus action
\begin{align*}
\mu^*_\mathrm{ch}(a_i)_{(0)}x_j&=\Delta_{ij}x_j,&
\mu^*_\mathrm{ch}(a_i)_{(0)}y_j&=-\Delta_{ij}y_j,\\
\mu^*_\mathrm{ch}(a_i)_{(0)}\psi_j&=\Delta_{ij}\psi_j,&
\mu^*_\mathrm{ch}(a_i)_{(0)}\phi_j&=-\Delta_{ij}\phi_j
\end{align*}
for $i=1,\dots,M$ and $j=1,\dots,N$ and
\begin{equation*}
\mu^*_\mathrm{ch}(\tau_i)_{(0)}c_j=0
\end{equation*}
for $i,j=1,\dots,M$ in the minimal case, extended trivially to the ghosts.

Introducing the formal variables $z_1,\dots,z_M$ recording the weights of the torus action of $\smash{\g=(\gl_1)^M}$ acting on the free-field vertex operator (super)algebras $\bar{C}_\text{min}$ and $\bar{C}$, we can give the following closed expressions for the characters:
\begin{prop}\label{prop:char_general}
The character of the minimal hypertoric \voa{} $V_\mathrm{min}(\Delta)$ is
\begin{align*}
&\ch_{V_\mathrm{min}(\Delta)}(q)=\tr_{C_\mathrm{min}}q^{L_0-c/24}(-1)^{P_\mathrm{gh}}\\
&=\operatorname{coeff}_{z_1^0\dots z_M^0}\bigl(\tr_{\bar{C}_\mathrm{min}}q^{L_0-c/24}(-1)^{P_\mathrm{gh}}z_1^{\mu^*_\mathrm{ch}(a_1)_{(0)}}\dots z_M^{\mu^*_\mathrm{ch}(a_M)_{(0)}}\bigr)\\[-1mm]
&=q^{N/24}\eta(q)^M\operatorname{coeff}_{z_1^0\dots z_M^0}\Bigl(\prod_{j=1}^N(q^{1/2}z_1^{\Delta_{1j}}\dots z_M^{\Delta_{Mj}};q)_\infty^{-1}(q^{1/2}z_1^{\Delta_{1j}}\dots z_M^{\Delta_{Mj}};q)_\infty^{-1}\Bigr).
\end{align*}
The (super)character of the boundary hypertoric \svoa{} $V(\Delta)$ is given by
\begin{align*}
&\operatorname{(s)ch}_{V(\Delta)}(q)=\tr_{C}q^{L_0-c/24}(\pm1)^{P_\mathrm{m}}(-1)^{P_\mathrm{gh}}\\
&=\operatorname{coeff}_{z_1^0\dots z_M^0}\bigl(\tr_{\bar{C}}q^{L_0-c/24}(\pm1)^{P_\mathrm{m}}(-1)^{P_\mathrm{gh}}z_1^{\mu^*_\mathrm{ch}(a_1)_{(0)}}\dots z_M^{\mu^*_\mathrm{ch}(a_M)_{(0)}}\bigr)\\[-1mm]
&=\eta(q)^{2M}\operatorname{coeff}_{z_1^0\dots z_M^0}\biggl(\prod_{j=1}^N\frac{(\mp q^{1/2}z_1^{\Delta_{1j}}\dots z_M^{\Delta_{Mj}};q)_\infty(\mp q^{1/2}z_1^{-\Delta_{1j}}\dots z_M^{-\Delta_{Mj}};q)_\infty}{(q^{1/2}z_1^{\Delta_{1j}}\dots z_M^{\Delta_{Mj}};q)_\infty(q^{1/2}z_1^{-\Delta_{1j}}\dots z_M^{-\Delta_{Mj}};q)_\infty}\biggr).
\end{align*}
\end{prop}
Here, we understand all expressions as expanded as a formal power series in $q^{1/2}$, times $q^{-c/24}$ with $c=-N-M$ and $c=-2M$, respectively. The coefficients will then be Laurent polynomials in $z$ with integer coefficients.

In the boundary case, this is already stated in equation~(B.8) in \cite{BCDN23}. These characters are the half-indices of the 3d $\mathcal{N}=4$ supersymmetric field theories (for abelian gauge groups) whose boundary \voa{}s we have constructed.

In the following, we shall apply $\smash{\operatorname{coeff}_{z_1^0\dots z_M^0}}$ (or equivalently $M$ contour integrations) to obtain more explicit expressions for these characters. This will also allow us to prove their modularity in some cases.


\subsection{Boundary Supercharacter}\label{sec:boundaryschar}

The case of the supercharacter of the boundary \svoa{} is the easiest. Here, the denominator and numerator of the fraction in \autoref{prop:char_general} cancel and we simply obtain:
\begin{prop}\label{prop:boundary-schar}
The supercharacter of the boundary hypertoric \svoa{} $V(\Delta)$ is
\begin{equation*}
\operatorname{sch}_{V(\Delta)}(q)=\eta(q)^{2M},
\end{equation*}
which only depends on the rank $M$ of the torus $G$.
\end{prop}

\begin{ex}\label{ex:SFchar}
For $N=M=1$ and $\Delta=(1)$, we have seen in \autoref{sec:smallexample} that $V(\Delta)=\mathsf{SF}$ is the symplectic fermion \svoa{} (with two strong generators of $L_0$-weight $1$), whose supercharacter is $\operatorname{sch}_{V(\Delta)}(q)=\operatorname{sch}_{\mathsf{SF}}(q)=\eta(q)^2$.

More generally, the supercharacters for $N\geq3$, $M=1$ and $\Delta=(1,\dots,1)$, corresponding to the minimal nilpotent orbit closures for $\sl_N$, in which case the boundary hypertoric \voa{} is $V(\Delta)\cong L_1(\psl_{N|N})$ \cite{FS24} (see also \autoref{sec:nilpotent_svoa}), were already stated in \cite{AM21}.
\end{ex}

As we showed in \autoref{cor:quasilisse} that $V(\Delta)$ is quasi-lisse, it follows from \cite{Li23} (cf.\ \cite{AK18}) that the supercharacter of $V(\Delta)$ satisfies a certain modular linear differential equation. It is expected that many quasi-lisse vertex operator(super) algebras have (super)characters that belong to vector-valued modular forms of weight~$0$ for $\SLZ$ that may have $\log(q)$-terms in other components. Often, this means that they are weakly holomorphic (quasi)modular forms of mixed nonnegative weights for some congruence subgroup and possibly a character (see, e.g., \cite{AKM23}, for another example).

Indeed, the supercharacter $\smash{\operatorname{sch}_{V(\Delta)}(q)=\eta^{2M}(q)}$ is modular of weight $M$ for $\SLZ$ and some character of order dividing $12$.

For $M=1$, i.e.\ for the symplectic fermion \svoa{} (which is actually lisse and not just quasi-lisse.), the modular transformation properties are explained, for example, in \cite{FGR22}.

We shall see in \autoref{sec:boundarychar} that the expression for the character $\ch_{V(\Delta)}(q)$ of the boundary \svoa{} is quite a bit more complicated than the supercharacter (and in particular depends not only on $M$), but still satisfies good modularity properties, at least in those examples where we are able to bring the character into a nice form.

\begin{rem}\label{rem:boundaryschar}
In analogy to \autoref{prop:boundarychar} below, we can also write the boundary supercharacter as
{\allowdisplaybreaks
\begin{align*}
\operatorname{sch}_{V(\Delta)}(q)&=\frac{q^{N/8}}{\eta(q)^{3N-2M}}\operatorname{coeff}_{z_1^0\dots z_M^0}\biggl(\prod_{j=1}^N\vartheta(-z_1^{\Delta_{1j}}\dots z_M^{\Delta_{Mj}},q)\xi(z_1^{\Delta_{1j}}\dots z_M^{\Delta_{Mj}},q)\biggr)\\
&=\frac{q^{N/8}}{\eta(q)^{3N-2M}}\sum_{\substack{m_1,\dots,m_N\in\Z\\n_1,\dots,n_N\in\Z\\\Delta m=\Delta n}}(-1)^{n_1+\dots+n_N}q^{(n_1^2+\dots+n_N^2)/2}\\
&\quad\times q^{(|m_1|+\dots+|m_n|)/2}\prod_{j=1}^N\psi(-q^{|m_j|},q),
\end{align*}
}%
which we know equals $\eta(q)^{2M}$. See below, for the definitions of the functions $\vartheta(z,q)$, $\xi(z,q)$ and $\psi(z,q)$. Writing the supercharacter in this way highlights which contributions to $V(\Delta)$ have even and odd parity; cf. \autoref{rem:parity} and the alternative presentation of $V(\Delta)$ in \autoref{rem:altrep}.
\end{rem}


\subsection{Minimal Character}\label{sec:minimalchar}

We then come to the characters of the minimal hypertoric \voa{}s $V_\mathrm{min}(\Delta)$. As these \voa{}s are in general not quasi-lisse, there is no expectation that their characters should be quasimodular. Indeed, we shall see the appearance of partial (or false) theta functions, which are not modular, but are related to mock and quantum modular forms (see, e.g., \cite{BFR12,FOR13}).

\medskip

To compute the minimal character, we expand the infinite product
\begin{equation*}
\xi(z,q)\coloneqq\frac{(q;q)^2_\infty}{(q^{1/2}z;q)_\infty (q^{1/2}z^{-1};q)_\infty}
\end{equation*}
as a formal power series in $q^{1/2}$. This computation, based on a classical partial fractions decomposition, can be found, e.g., in \cite{And84,AB05} and yields the bilateral Lambert series
{\allowdisplaybreaks
\begin{align*}
\xi(z,q)&=\sum_{m=0}^\infty\frac{(-1)^mq^{m(m+1)/2}}{1-z^{-1}q^{m+1/2}}-zq^{-1/2}\sum_{m=1}^\infty\frac{(-1)^mq^{m(m+1)/2}}{1-zq^{m-1/2}}\\
&=\sum_{m\in\Z}\sum_{r=|m|}^\infty(-1)^{r+m}z^mq^{(r^2-m^2)/2+r/2}\\
&=\sum_{m\in\Z}\sum_{n=0}^\infty(-1)^nz^mq^{|m|(n+1/2)+n(n+1)/2}.
\end{align*}
}%
Introducing Ramanujan's partial theta function \cite{AB09}
\begin{equation*}
\psi(z,q)\coloneqq\sum_{n=0}^\infty z^nq^{n(n+1)/2},
\end{equation*}
we can rewrite this as
\begin{equation*}
\xi(z,q)=\sum_{m\in\Z}z^mq^{|m|/2}\psi(-q^{|m|},q).
\end{equation*}
This expression is also given in \cite{CSVY17} in a different, though not completely unrelated, context.

\medskip

With these preparations, it follows:
\begin{prop}\label{prop:min-char}
The character of the minimal hypertoric \voa{} $V_\mathrm{min}(\Delta)$ is
{\allowdisplaybreaks
\begin{align*}
&\ch_{V_\mathrm{min}(\Delta)}(q)\\
&=\frac{q^{N/8}}{\eta(q)^{2N-M}}\operatorname{coeff}_{z_1^0\dots z_M^0}\biggl(\prod_{j=1}^N\xi(z_1^{\Delta_{1j}}\dots z_M^{\Delta_{Mj}},q)\biggr)\\
&=\frac{q^{N/8}}{\eta(q)^{2N-M}}\operatorname{coeff}_{z_1^0\dots z_M^0}\biggl(\prod_{j=1}^N\sum_{m_j\in\Z}(z_1^{\Delta_{1j}}\dots z_M^{\Delta_{Mj}})^m_jq^{|m_j|/2}\psi(-q^{|m_j|},q)\biggr)\\[-1mm]
&=\frac{1}{\eta(q)^{2N-M}}\sum_{\substack{m_1,\dots,m_N\in\Z\\\Delta m=0}}q^{|m_1+\dots+m_N|/2}q^{N/8}\prod_{j=1}^N\psi(-q^{|m_j|},q).
\end{align*}
}%
\end{prop}
This is an infinite sum, over $\Z^{N-M}$, of rank-$N$ partial theta functions. The expression may seem unwieldy, but it nicely mirrors the decomposition of $V_\mathrm{min}(\Delta)$ in terms of singlet-Heisenberg modules, in the alternative presentation given in \autoref{rem:altrep}.

\begin{ex}
For $N=M=1$ and $\Delta=(1)$ (see \autoref{sec:smallexample}), we obtain
\begin{align*}
\ch_{V_\mathrm{min}(\Delta)}(q)&=\frac{1}{\eta(q)}q^{1/8}\psi(-1,q)=\frac{1}{\eta(q)}\sum_{n=0}^\infty(-1)^nq^{n(n+1)/2+1/8}\\
&=\frac{1}{\eta(q)}\sum_{n\in\Z}\operatorname{sgn}(n+1/4)q^{(4n+1)^2/8},
\end{align*}
which is the character of the singlet \voa{} $V_\mathrm{min}(\Delta)\cong\mathcal{M}(2)$ \cite{Flo96}. We note that $\psi(-1,q)$ is also called Rogers' false theta function \cite{Rog17}.

On the other hand, for $N=1$ and $M=0$ (no BRST reduction at all, though this technically does not belong to the class of examples that we consider, cf.\ \autoref{rem:zerocolumnrow}), the character is just
\begin{align*}
\ch_{V_\mathrm{min}(\Delta)}(q)&=\frac{1}{\eta(q)^2}\sum_{m\in\Z}q^{|m|/2}q^{1/8}\psi(-q^{|m|},q)=\frac{q^{1/24}}{(q^{1/2};q)_\infty (q^{1/2};q)_\infty}\\
&=\frac{\eta(q)^2}{\eta(q^{1/2})^2},
\end{align*}
the character of the $\beta\gamma$-system $V_\mathrm{min}(\Delta)\cong\mathcal{D}^\mathrm{ch}(T^*\C)$ of central charge $c=-1$. This is now modular of weight~$0$ for some congruence subgroup of $\SLZ$ of level $2$ and a character. The summation over $m\in\Z$ has recovered the modularity.
\end{ex}
In general, we expect the character of $V_\mathrm{min}(\Delta)$ to be some kind of partial or false theta function of rank $M$, times something modular.

(Higher-rank) partial or false theta functions themselves are not modular in the usual sense. However, they are closely related to mock theta functions and quantum modular forms (see, e.g., \cite{BFR12,FOR13}). As we are mainly interested in the boundary hypertoric \voa{}s in this paper, we forego a more detailed investigation of the false theta functions appearing in this work.


\subsection{Boundary Character}\label{sec:boundarychar}

Comparing the minimal character in \autoref{sec:minimalchar} to the boundary supercharacter in \autoref{sec:boundaryschar}, we saw that the partial theta functions were summed to $\eta(q)^{2M}$, which is modular of weight $M$ (though it may be more useful to interpret it as participating in a vector-valued modular form of weight~$0$ with $\log(q)$-terms in other components). While the boundary supercharacters take a very simple form, we expect the boundary characters to be more interesting, while still being (quasi)modular. In the following, we determine these characters as modular forms in some classes of examples and make a conjecture about the general behaviour.

\medskip

For the boundary character, we need to take the coefficient of $z_1^0\dots z_M^0$ of
\begin{align*}
&\eta(q)^{2M}\prod_{j=1}^N\frac{(-q^{1/2}z_1^{\Delta_{1j}}\dots z_M^{\Delta_{Mj}};q)_\infty(-q^{1/2}z_1^{-\Delta_{1j}}\dots z_M^{-\Delta_{Mj}};q)_\infty}{(q^{1/2}z_1^{\Delta_{1j}}\dots z_M^{\Delta_{Mj}};q)_\infty(q^{1/2}z_1^{-\Delta_{1j}}\dots z_M^{-\Delta_{Mj}};q)_\infty}\\
&=\frac{q^{N/8}}{\eta(q)^{3N-2M}}\prod_{j=1}^N\vartheta(z_1^{\Delta_{1j}}\dots z_M^{\Delta_{Mj}},q)\xi(z_1^{\Delta_{1j}}\dots z_M^{\Delta_{Mj}},q)
\end{align*}
with
\begin{equation*}
\xi(z,q)=\frac{(q;q)^2_\infty}{(q^{1/2}z;q)_\infty(q^{1/2}z^{-1};q)_\infty}=\sum_{m\in\Z}z^mq^{|m|/2}\psi(-q^{|m|},q)
\end{equation*}
as before and the Jacobi theta function $\vartheta(z,q)$, which obeys the Jacobi triple-product identity
\begin{equation*}
\vartheta(z,q)=(q;q)_\infty(-q^{1/2}z;q)_\infty(-q^{1/2}z^{-1};q)_\infty=\sum_{n\in\Z}z^nq^{n^2/2}.
\end{equation*}
Then it follows:
\begin{prop}\label{prop:boundary-char}
The character of the boundary hypertoric \voa{} $V(\Delta)$ can be written as
\begin{align*}
\ch_{V(\Delta)}(q)&=\frac{q^{N/8}}{\eta(q)^{3N-2M}}\operatorname{coeff}_{z_1^0\dots z_M^0}\biggl(\prod_{j=1}^N\vartheta(z_1^{\Delta_{1j}}\dots z_M^{\Delta_{Mj}},q)\xi(z_1^{\Delta_{1j}}\dots z_M^{\Delta_{Mj}},q)\biggr)\\
&=\frac{q^{N/8}}{\eta(q)^{3N-2M}}\!\!\sum_{\substack{m_1,\dots,m_N\in\Z\\n_1,\dots,n_N\in\Z\\\Delta m=\Delta n}}\!\!q^{(n_1^2+\dots+n_N^2)/2}q^{(|m_1|+\dots+|m_n|)/2}\prod_{j=1}^N\psi(-q^{|m_j|},q).
\end{align*}
\end{prop}
We can view this as an infinite sum, over $\Z^{N-M}$, of rank-$N$ theta functions for the (odd) standard lattice $\Z^N$ and rank-$N$ partial theta functions. Again, like for the minimal case in \autoref{sec:minimalchar}, the above expression mirrors the decomposition of $V(\Delta)$ in terms of singlet-Heisenberg modules, in the alternative presentation given in \autoref{rem:altrep}.

While the above result allows us to easily compute the $q$-expansion of $\ch_{V(\Delta)}$ up to a certain $q$-power, it is less suitable for proving modularity properties.

\medskip

There is another useful presentation of the character in the boundary case. To this end, note that we can expand the following quotient directly
\begin{align*}
\frac{1}{(q;q)_\infty^3}\vartheta(z,q)\xi(z,q)&=\frac{(-q^{1/2}z;q)_\infty(-q^{1/2}z^{-1};q)_\infty}{(q^{1/2}z;q)_\infty(q^{1/2}z^{-1};q)_\infty}\\
&=\frac{(-q;q)_\infty^2}{(q;q)^2_\infty}\biggl(1+2\sum_{n=1}^\infty\frac{(q^{1/2}z)^n+(q^{1/2}z^{-1})^n}{1+q^n}\biggr)
\end{align*}
as a formal power series in $q^{1/2}$, which is an instance of an identity due to Kronecker (see, e.g., \cite{Wei76}) and follows from Ramanujan’s ${}_1\psi_1$-summation formula, which first appeared, without proof, in Ramanujan’s second notebook.

With this expansion, the boundary character can be written as:
\begin{prop}\label{prop:boundarychar}
The character of the boundary hypertoric \voa{} $V(\Delta)$ is given by
{\allowdisplaybreaks
\begin{align*}
\ch_{V(\Delta)}(q)&=\eta(q)^{2M}\frac{(-q;q)_\infty^{2N}}{(q;q)^{2N}_\infty}\operatorname{coeff}_{z_1^0\dots z_M^0}\\
&\quad\biggl(\prod_{j=1}^N\Bigl(1+2\sum_{n=1}^\infty\frac{(q^{1/2}z_1^{\Delta_{1j}}\dots z_M^{\Delta_{Mj}})^n+(q^{1/2}z_1^{-\Delta_{1j}}\dots z_M^{\Delta_{-Mj}})^n}{1+q^n}\Bigr)\!\biggr)\\
&=\frac{\eta(q^2)^{2N}}{\eta(q)^{4N-2M}}\operatorname{coeff}_{z_1^0\dots z_M^0}\\
&\quad\biggl(\prod_{j=1}^N\Bigl(1+2\sum_{n=1}^\infty\frac{(q^{1/2}z_1^{\Delta_{1j}}\dots z_M^{\Delta_{Mj}})^n+(q^{1/2}z_1^{-\Delta_{1j}}\dots z_M^{\Delta_{-Mj}})^n}{1+q^n}\Bigr)\!\biggr).
\end{align*}
}%
\end{prop}

In the following, we explore the Lambert-like series term $\smash{\operatorname{coeff}_{z_1^0\dots z_M^0}(\dots)}$, with the goal of studying the modularity of the character $\ch_{V(\Delta)}(q)$. We conjecture that it is (quasi)modular of mixed weight for some congruence subgroup of $\SLZ$ (see \autoref{conj:quasimodular}), which we prove in some examples.

\begin{ex}\label{ex:minimal_nilp_char}
The characters for $N\geq1$, $M=1$ and $\Delta=(1,\dots,1)$, corresponding to the minimal nilpotent orbit closures for $\sl_N$, in which case the boundary hypertoric \svoa{} is $V(\Delta)\cong L_1(\psl_{N|N})$ \cite{FS24} (see also \autoref{sec:nilpotent_svoa}), were already studied in \cite{AM21}. They conjecture that $\ch_{V(\Delta)}(q)$ satisfies a certain modular linear differential equation of weight~$0$. In the following, we prove the modularity for some values of $N$, the only new case being $N=5$. We thank Matthias Storzer for providing the formula in this case.

For $N=M=1$ and $\Delta=(1)$, $V(\Delta)\cong\mathsf{SF}$ is just the symplectic fermion \svoa{} of central charge $c=-2$ (see \autoref{sec:smallexample}). We already discussed this case in \autoref{ex:SFchar}. For the character, by the above formula, we obtain
\begin{equation*}
\ch_{V(\Delta)}(q)=\frac{\eta(q^2)^{2}}{\eta(q)^2}\cdot 1,
\end{equation*}
which is modular of weight~$0$ for a congruence subgroup of $\SLZ$ of level $2$ with character. Of course, this expression is well-known to be the character of the symplectic fermion.

For $N=2$, using some well-known Lambert series identities, the character can be written as 
\begin{align*}
\ch_{V(\Delta)}(q)&=\frac{\eta(q^2)^{4}}{\eta(q)^6}\cdot\Bigl(1+8\sum_{n=1}^\infty\frac{q^n}{(1+q^n)^2}\Bigl)=\frac{\eta(q^2)^{4}}{\eta(q)^6}\cdot\Bigl(1-8\sum_{n=1}^\infty\frac{n(-q)^n}{1-q^n}\Bigl)\\
&=\frac{\eta(q^2)^{4}}{\eta(q)^6}\cdot\bigl(-E_2(q)/3+4E_2(q^2)/3\bigl),
\end{align*}
which is quasimodular of weight~$1$ for a congruence subgroup of $\SLZ$ of level $2$ with character. Here $E_2(q)$ is the usual (holomorphic, but not modular) weight-$2$ Eisenstein series $\smash{E_2(q)=1-24\sum_{n=1}^\infty\frac{nq^n}{1-q^n}}$.

Similarly, for $N=3$, the character is
\begin{equation*}
\ch_{V(\Delta)}(q)=\dots=\frac{\eta(q^2)^{6}}{\eta(q)^{10}}\cdot\Bigl(1+24\sum_{n=1}^\infty\frac{nq^n}{1+q^n}\Bigl)=\frac{\eta(q^2)^{6}}{\eta(q)^{10}}\cdot E_{2,2}(q),
\end{equation*}
where $E_{2,2}(q)=\vartheta_{D_4}(q)$ is a modular form of weight~$2$ for $\Gamma_0(2)$, and the theta series of the root lattice $D_4$. So, overall, $\ch_{V(\Delta)}(q)$ is modular of weight~$0$ for a congruence subgroup of $\SLZ$ of level $2$ with character.

For $N=5$, one can show that
\begin{equation*}
\ch_{V(\Delta)}(q)=\frac{\eta(q^2)^{10}}{\eta(q)^{18}}(5E_{2,2}(q)^2/6+E_4(q)/6),
\end{equation*}
which is modular of weight~$0$ for a congruence subgroup of $\SLZ$ of level $2$ with character. Here, $E_4(q)=1+240\smash{\sum_{n=1}^\infty\frac{n^3q^n}{1-q^n}}$ is the weight-$4$ Eisenstein series, which coincides with the theta series $\vartheta_{E_8}(q)$ of the $E_8$ root lattice. This also verifies Conjecture~7.7 in \cite{AM21}.
\end{ex}

In general, we conjecture:
\begin{conj}
For the minimal nilpotent orbit closure for $\sl_N$, $N\geq1$, the character of the boundary hypertoric \svoa{} $V(\Delta)$ is, without the leading $\eta(q^2)^{2N}/\eta(q)^{4N-2}$, (quasi)modular of weight $2\lfloor N/2\rfloor$, so that overall $\ch_{V(\Delta)}(q)$ is (quasi)modular of weight~$0$ if $N$ is odd, and of weight~$1$ if $n$ is even. 
\end{conj}

\smallskip

For the Kleinian singularities we are able to obtain nice expressions for the boundary characters for all $N$, as a direct consequence of \autoref{prop:boundarychar}:
\begin{prop}\label{prop:kleinianchar}
The character of the boundary hypertoric \svoa{} $V(\Delta)$ for the Kleinian singularity of type $A_{N-1}$ with $N\geq2$ and $M=N-1$ is 
\begin{equation*}
\ch_{V(\Delta)}(q)=\frac{\eta(q^2)^{2N}}{\eta(q)^{2N+2}}\cdot\Bigl(1+2^{N+1}\sum_{n=1}^\infty\frac{q^{Nn/2}}{(1+q^n)^N}\Bigr).
\end{equation*}
\end{prop}
These characters have integer $q$-exponents when $N$ is even and half-integer $q$-exponents when $N$ is odd.

In the following, we investigate the modularity of these boundary characters by explicitly writing them in terms of known modular forms, starting with the case of even $N$.
\begin{prop}\label{prop:kleinian_char}
For $N\geq2$ even, the character $\ch_{V(\Delta)}(q)$ of the boundary hypertoric \voa{} $V(\Delta)$ for the Kleinian singularity of type $A_{N-1}$ is quasimodular of mixed weights $1,3,5,\dots,N-1$ for a congruence subgroup of $\SLZ$ of level $2$ with a character. 
\end{prop}
\begin{proof}
Let $N$ be even. We consider the expansion of the Lambert-like series
{\allowdisplaybreaks
\begin{align*}
\sum_{n=1}^\infty\frac{q^{Nn/2}}{(1+q^n)^N}&=\sum_{n=1}^\infty q^{Nn/2}\sum_{k=0}^\infty(-1)^k\binom{k+N-1}{k}q^{nk}\\
&=\sum_{n=1}^\infty\sum_{k=N/2}^\infty(-1)^{k-N/2}\binom{k+N/2-1}{k-N/2}q^{nk}\\
&=(-1)^{N/2}\sum_{k=0}^\infty(-1)^k\binom{k+N/2-1}{k-N/2}\sum_{n=1}^\infty q^{nk}\\
&=(-1)^{N/2}\sum_{k=0}^\infty(-1)^k\binom{k+N/2-1}{k-N/2}\frac{q^k}{1-q^k}.
\end{align*}
}%
Now, the binomial coefficient can be written as
\begin{equation*}
\binom{k+N/2-1}{k-N/2}=\frac{(k-d)(k-d+1)\cdot\dots\cdot(k+d)}{(N-1)!}
\end{equation*}
with $d=N/2-1$, and hence it is an odd polynomial in $k$ of degree $N-1$, which we write as $a_1k+a_3k^3+\dots+a_{N-1}k^{N-1}$, where the coefficients $a_1,a_3,\dots,a_{N-1}\in\Q$ only depend on $N$. Hence, we obtain
\begin{align*}
\sum_{n=1}^\infty\frac{q^{Nn/2}}{(1+q^n)^N}&=(-1)^{N/2}\sum_{k=0}^\infty(-1)^k(a_1k+a_3k^3+\dots+a_{N-1}k^{N-1})\frac{q^k}{1-q^k}\\
&=-(-1)^{N/2}\sum_{k=0}^\infty(a_1k+a_3k^3+\dots+a_{N-1}k^{N-1})\frac{q^k}{1-q^k}\\
&\quad+2(-1)^{N/2}\sum_{k=0}^\infty(a_12k+a_32^3k^3+\dots+a_{N-1}2^{N-1}k^{N-1})\frac{(q^2)^k}{1-(q^2)^k}.
\end{align*}
By their definition, it is now apparent that
\begin{equation*}
1+2^{N+1}\sum_{n=1}^\infty\frac{q^{Nn/2}}{(1+q^n)^N}
\end{equation*}
is a linear combination of the Eisenstein series $E_2(q),E_4(q),\dots,E_N(q)$ and their scaled versions $E_2(q^2),E_4(q^2),\dots,E_N(q^2)$. These are all modular, except for $E_2(q)$ and $E_2(q^2)$, which are quasimodular, of weights $2,4,\dots,N$, respectively. (We mention that while a certain linear combination of $E_2(q)$ and $E_2(q^2)$ is actually modular, this cancellation does not happen here.)

In principle, there could also be a constant term (not coming from the constant terms in the Eisenstein series), but a direct computation shows that it cancels.

Overall, after multiplying with $\eta(q^2)^{2N}/\eta(q)^{2N+2}$, which is modular of weight~$-1$, we have shown that the boundary character $\ch_{V(\Delta)}$ is quasimodular of mixed weights $1,3,5,\dots,N-1$.
\end{proof}
For comparison, the supercharacter for the Kleinian singularity was $\eta(q)^{2(N-1)}$, which is modular of weight $N-1$.

\begin{ex}
We discuss a few examples for small values of $N$. For $N=2$, the hypertoric \voa{} for $X_{A_{N-1}}$ is isomorphic to the one for $\overline{\mathbb{O}_\mathrm{min}(\sl_N)}$, which we treated above. The character is
\begin{equation*}
\ch_{V(\Delta)}(q)=\frac{\eta(q^2)^{4}}{\eta(q)^6}\cdot\bigl(-E_2(q)/3+4E_2(q^2)/3\bigl),
\end{equation*}
which agrees with the prediction of \autoref{prop:kleinian_char}.

Similarly, for $N=4$, we obtain
\begin{equation*}
\ch_{V(\Delta)}(q)=\frac{\eta(q^2)^{8}}{\eta(q)^{10}}\cdot\bigl(-2E_2(q)/9+8E_2(q^2)/9-E_4(q)/45+16E_4(q^2)/45\bigl).
\end{equation*}
\end{ex}

For $N$ odd, we also expect the character $\ch_{V(\Delta)}$ to be (quasi)modular. First, we rewrite the sum expression in the character given in \autoref{prop:kleinianchar} as follows, similarly to the case of $N$ even,
\begin{align*}
\sum_{n=1}^\infty\frac{q^{Nn/2}}{(1+q^n)^N}&=\sum_{n=1}^\infty q^{Nn/2}\sum_{k=0}^\infty(-1)^k\binom{k+N-1}{k}q^{nk}\\
&=\sum_{n=1}^\infty\sum_{k=N/2+\N}(-1)^{k-N/2}\binom{k+N/2-1}{k-N/2}(q^{1/2})^{2nk}\\
&=(-1)^{(N+1)/2}\sum_{k=1}^\infty(-1)^{k}\binom{k+N/2-3/2}{k-N/2-1/2}\sum_{n=1}^\infty(q^{1/2})^{n(2k-1)}\\
&=(-1)^{(N+1)/2}\sum_{k=1}^\infty(-1)^{k}\binom{k+N/2-3/2}{k-N/2-1/2}\frac{(q^{1/2})^{2k-1}}{1-(q^{1/2})^{2k-1}}.
\end{align*}
The binomial coefficient is a polynomial in $k$ of degree $N-1$. For example, for $N=3$, it is $-k/2+k^2/2$ and for $N=5$, $k/12-k^2/24-k^3/12+k^4/24$. It will also be instructive to consider the character expression for $N=1$, where the binomial coefficient is just $1$.

In order to prove the modularity of this expression, we need to study the modularity of expressions of the form
\begin{equation*}
\sum_{k=1}^\infty(-1)^{k}k^d\frac{q^{2k-1}}{1-q^{2k-1}}
\end{equation*}
for $d\in\N$. We study some special cases.

\begin{ex}
For $N=1$ (which does not correspond to an actual hypertoric \svoa{}),
\begin{align*}
&1+4\sum_{n=1}^\infty\frac{q^{Nn/2}}{(1+q^n)^N}=1-4\sum_{k=1}^\infty(-1)^{k}\frac{(q^{1/2})^{2k-1}}{1-(q^{1/2})^{2k-1}}\\
&=\sum_{n=0}^\infty r_2(n)(q^{1/2})^n=\vartheta_3(q^{1/2})^2=\frac{\eta(q)^{10}}{\eta(q^{1/2})^4\eta(q^2)^4},
\end{align*}
a well-known identity for the generating series of the sum-of-two-squares function $r_2(n)$, with the Jacobi theta function $\vartheta_3(q)=\sum_{n\in\Z}q^{n^2}=\eta(q^2)^5/\eta(q)^2/\eta(q^4)^2$, which is modular of weight $1/2$.

Similarly, for $N=3$, the following identity due to Jacobi holds
{\allowdisplaybreaks
\begin{align*}
\vartheta_3(q)^6=&\sum_{n=0}^\infty r_6(n)q^n=1+16\eta(q^2)^6\eta(q)^{-4}
\eta(q^4)^4\\
&\quad+4\sum_{k=1}^\infty(-1)^k(2k^2-1)^2\frac{q^{2k-1}}{1-q^{2k-1}}\\
=&\sum_{n=0}^\infty r_6(n)q^n=1+16\eta(q^2)^6\eta(q)^{-4}
\eta(q^4)^4\\
&\quad+4\sum_{k=1}^\infty(-1)^k(4k^2-4k)\frac{q^{2k-1}}{1-q^{2k-1}}+1-\vartheta_3(q)^2,
\end{align*}
}%
with the sum-of-six-squares function $r_6(n)$ (see, e.g., \cite{ALL01,KW94b}), which we rearrange to get
\begin{equation*}
1+16\sum_{k=1}^\infty(-1)^k(k^2/2-k/2)\frac{q^{2k-1}}{1-q^{2k-1}}=\vartheta_3(q)^6/2+\vartheta_3(q)^2/2-8\eta(q^2)^6\eta(q)^{-4}
\eta(q^4)^4.
\end{equation*}
With this, the expression for the character from \autoref{prop:kleinianchar} without the eta functions is
\begin{align*}
1+16\sum_{n=1}^\infty\frac{q^{3n/2}}{(1+q^n)^3}&=1+16\sum_{k=1}^\infty(-1)^{k}(-k/2+k^2/2)\frac{(q^{1/2})^{2k-1}}{1-(q^{1/2})^{2k-1}}\\
&=\vartheta_3(q^{1/2})^6/2+\vartheta_3(q^{1/2})^2/2-8\frac{\eta(q)^6\eta(q^2)^4}{\eta(q^{1/2})^4},\\
&=\frac{1}{2}\frac{\eta(q)^{30}}{\eta(q^{1/2})^{12}\eta(q^2)^{12}}+\frac{1}{2}\frac{\eta(q)^{10}}{\eta(q^{1/2})^4\eta(q^2)^4}-8\frac{\eta(q)^6\eta(q^2)^4}{\eta(q^{1/2})^4},
\end{align*}
which is modular of mixed weights $1$ and $3$. Hence, we can write the boundary character for $N=3$ as
\begin{align*}
\ch_{V(\Delta)}(q)&=\frac{\eta(q^2)^6}{\eta(q)^8}\cdot\Bigl(\vartheta_3(q^{1/2})^6/2+\vartheta_3(q^{1/2})^2/2-8\frac{\eta(q)^6\eta(q^2)^4}{\eta(q^{1/2})^4}\Bigr)\\
&=\frac{1}{2}\frac{\eta(q)^{22}}{\eta(q^{1/2})^{12}\eta(q^2)^6}+\frac{1}{2}\frac{\eta(q)^2\eta(q^2)^2}{\eta(q^{1/2})^4}-8\frac{\eta(q^2)^{10}}{\eta(q)^2\eta(q^{1/2})^4},
\end{align*}
which is modular of weights $0$ and $2$.
\end{ex}

In general, based on the $q$-expansion of
\begin{equation*}
\vartheta_3(q)^{2N}=\sum_{n=0}^\infty r_{2N}(n)q^n,
\end{equation*}
it should be possible to prove also for odd $N$:
\begin{conj}
For the Kleinian singularity of type $A_{N-1}$ for $N\geq2$, the character $\ch_{V(\Delta)}(q)$ of the boundary hypertoric \svoa{} $V(\Delta)$ is quasimodular of weights $N-1,N-3,\dots,0\text{ or }1$ for a congruence subgroup of $\SLZ$ of level $2$ with a character.
\end{conj}
In \autoref{prop:kleinian_char} we established this for all even $N$.

\medskip

Based on the study of the (symplectic dual) examples of the minimal nilpotent orbit closures and Kleinian singularities of type $A_{N-1}$, and given that the supercharacters $\operatorname{sch}_{V(\Delta)}(q)$ are manifestly modular, we conjecture:
\begin{conj}\label{conj:quasimodular}
The characters of the boundary hypertoric \svoa{}s $V(\Delta)$ are quasimodular, possibly of mixed nonnegative weights for some congruence subgroup of $\SLZ$ of level~$2$ with character.
\end{conj}
See, e.g., \cite{AKM23}, for a similar result in the context of nonabelian 3d $\mathcal{N}=4$ gauge theories, namely associated with the Jordan quiver for $\GL(N)$.


\subsection{Comparison of Minimal and Boundary Characters}

We saw in \autoref{sec:minimalchar} and \autoref{sec:boundarychar} that the (super)characters of the minimal and boundary hypertoric vertex (operator)superalgebras $V_\mathrm{min}(\Delta)$ and $V(\Delta)$ have quite different modular behaviour. This is expected, as the latter are quasi-lisse and the former in general are not.

In the following, we explain how the partial theta series appearing in the minimal characters $\ch_{V_\mathrm{min}(\Delta)}(q)$ are paired with usual theta series (which are modular) and then summed up to a quasimodular form.

To this end, we rewrite the sum expression for the (super)character from \autoref{rem:boundaryschar} and \autoref{prop:boundarychar} as
{\allowdisplaybreaks
\begin{align*}
\operatorname{(s)ch}_{V(\Delta)}(q)&=\!\!\!\sum_{l_1,\dots,l_M\in\Z^M}\!\Biggl(\frac{q^{N/8}}{\eta(q)^{2N-M}}\!\!\sum_{\substack{m_1,\dots,m_N\in\Z\\\Delta m=l}}\!\!q^{(|m_1|+\dots+|m_n|)/2}\prod_{j=1}^N\psi(-q^{|m_j|},q)\Biggr)\\
&\quad\times\Biggl(\frac{1}{\eta(q)^{N-M}}\sum_{\substack{n_1,\dots,n_N\in\Z\\\Delta n=l}}(\pm1)^{n_1+\dots+n_N}q^{(n_1^2+\dots+n_N^2)/2}\Biggr).
\end{align*}
}%
The sum in the second row is simply a rank-$(N-M)$ theta series (with a sign factor in the case of the supercharacter) over a coset of the positive-definite, integral lattice $\smash{J^\bot=\{\textstyle\sum_{n_1,\dots,n_N\in\Z}n_i\tau_i\mid\Delta n=0\}}$. The expression in brackets in the first row is the character of an irreducible module of the minimal hypertoric \voa{} $V_\mathrm{\Delta}$. We argued in \autoref{sec:minimalchar} that it is essentially a rank-$M$ partial theta function times something modular (or, more precisely, a rank-$N$ partial theta function summed over a rank-$M$ lattice coset).

The above character decomposition is simply a consequence of \autoref{rem:intermediatesce}, where we describe $V(\Delta)$ as a $\smash{\Z^M}$-simple-current extension of the intermediate \svoa{} $\smash{V_\mathrm{min}(\Delta)\otimes V_{J^\bot}}$.


\subsection{Characters and Free-Field Realisations}

In this section, so far, we have computed the characters of the minimal and boundary hypertoric vertex operator (super)algebras $V_\mathrm{min}(\Delta)$ and $V(\Delta)$ by applying the Euler-Poincaré principle to the defining BRST reduction, i.e.\ to the upper row in the diagrams in \autoref{prop:commute1} and \autoref{prop:commute2}.

On the other hand, one can also study the characters by looking at the right column in these diagrams, i.e.\ based on the free-field realisation described by the screening kernels. This amounts to writing down Felder-style complexes and computing the characters by applying the Euler-Poincaré principle to those.

We demonstrate here how this works in the case of the (super)character of the smallest example discussed in \autoref{sec:smallexample}, i.e.\ for $N=M=1$ and $\Delta=(1)$, which already manifests the essential features. Recall that in this case $V(\Delta)=\mathsf{SF}$ is the symplectic fermion \svoa{}, whereas $V_\mathrm{min}(\Delta)=\mathcal{M}(2)$ is the $p=2$ singlet \voa{}. Thus,
\begin{equation*}
\operatorname{sch}_{V(\Delta)}(q)=\eta(q)^2=q^{1/12}(q;q)_\infty^2.
\end{equation*}
We can further introduce a fugacity for the Cartan subalgebra of the $\sl_2$ symmetry,
\begin{equation*}
\operatorname{sch}_{V(\Delta)}(y,q) = q^{1/12}(y^{-1}q;q)_\infty (yq;q)_\infty.
\end{equation*}
The free-field realisation considered in this paper embeds $\mathsf{SF}$ into the $bc$-ghost \svoa{} $V_{bc}$ with strong generators of $L_0$-weights $1$ and $0$ (see the diagram in \autoref{sec:smallexample}). Its supercharacter is
\begin{equation*}
\operatorname{sch}_{V_{bc}}(y,q)=q^{1/12}(y^{-1};q)_\infty(yq;q)_{\infty},
\end{equation*}
and so
\begin{equation*}
\operatorname{sch}_{V(\Delta)}(q)=(1-y^{-1})^{-1} \operatorname{sch}_{V_{bc}}(q).
\end{equation*}
Using the Jacobi triple-product identity,
\begin{equation*}
(q;q)_\infty(y^{-1};q)_\infty(yq;q)_{\infty}=\sum_{n\in \Z}q^{\frac{n(n+1)}{2}}y^n(-1)^n,
\end{equation*}
which can be understood in terms of the boson-fermion correspondence, we get
\begin{equation*}
\operatorname{sch}_{V_{bc}}(y,q)=\frac{q^{1/12}}{(q;q)_{\infty}}\sum_{n\in\Z}q^{\frac{n(n+1)}{2}}y^n(-1)^n.
\end{equation*}
Thus, we can conclude that
\begin{equation}\label{eq:sf-char-exp}
\operatorname{sch}_{V(\Delta)}(y,q)=\frac{q^{1/12}}{(q;q)_\infty}\left(\sum_{l=0}^\infty\sum_{n\in\Z}y^{n}(-1)^{n+l}q^{\frac{(n+l)(n+l-1)}{2}}\right).
\end{equation}
Setting $l=0$ in~\eqref{eq:sf-char-exp}, one retrieves the supercharacter of $V_{bc}$, and the terms for $l>0$ compute corrections to this character arising from the Felder complex for the screening operators.

We can finally obtain the character for $V_{\mathrm{min}}(\Delta)$ by taking the coefficient of $y^0$,
\begin{equation*}\label{eq:singlet-mod}
\ch_{V_{\mathrm{min}}(\Delta)}(q)= \frac{q^{1/12}}{(q;q)_\infty}\sum_{l=0}^\infty(-1)^{l}q^{\frac{l(l-1)}{2}}.
\end{equation*}
This is manifestly a partial or false theta function. The Felder complex interpretation can be restricted to this term: the summand for $l=0$ is a Heisenberg \voa{} character and the higher summands arise from applying the Euler-Poincaré principle to the Felder complex. For other $n\neq 0$ in~\eqref{eq:sf-char-exp}, we similarly get $p=2$ singlet module supercharacters expressed in terms of a Felder complex.


\appendix

\section{Quiver Varieties}\label{sec:quiver}

In this section, we discuss a special case of hypertoric varieties, namely those coming from quivers (see, e.g., \cite{HS02} and \cite{Nag21}). They are called \emph{hypertoric quiver varieties} and are precisely the Nakajima quiver varieties coming from quivers with $1$-dimensional gauge nodes \cite{Nak94,Nak98}.

Let $Q=(V,E)$ be a finite, directed graph with $|V|=M+1$ nodes
\begin{equation*}
V=\{n_0,n_1,\dots,n_M\}
\end{equation*}
and $|E|=N$ edges
\begin{equation*}
E=\{e^k_{ij}\coloneqq(i,j)\mid k=1,\dots,m_{ij}\},
\end{equation*}
where $m_{ij}\in\Ns$ is the multiplicity of the edge $(i,j)$. We assume that $Q$ is (weakly) connected and has no self-loops. In particular, $N\geq M$. The choice of orientation of $Q$ will not affect the resulting hypertoric variety up to isomorphism of Poisson varieties (see Remark~2.22 in \cite{Nag21}).

Then, consider the free $\Z$-modules $\Z^{M+1}$ with basis the nodes $\{n_0,\dots,n_M\}$ and $\Z^N$ with basis the edges $\{e^k_{ij}\}$. We consider the rank-$M$ submodule $\Z^M\subset\Z^{M+1}$ with basis $\{n_0-n_1,n_0-n_2,\dots,n_0-n_M\}$, i.e.\ the $\Z$-linear combinations of nodes whose coefficients sum to zero. The $(M\times N)$-matrix $\Delta$ is now defined by the map
\begin{equation*}
\Z^N\to\Z^M,\quad e^k_{ij}\mapsto n_i-n_j,
\end{equation*}
together with the given choice of bases. Because $Q$ is connected, this map is surjective. Explicitly,
\begin{equation*}
\Delta_{\ell,{}_{ij}^k}=\delta_{j\ell}-\delta_{i\ell}
\end{equation*}
where the index $\ell$ refers to the coordinate for the basis vector $n_0-n_\ell$, $\ell=1,\dots,M$ and the index ${}_{ij}^k$ to the coordinate for the edge $e_{ij}^k$ (in particular, the first term forces $j\neq0$ and the second $i\neq0$). Moreover, $\Delta$ is unimodular.

We then define a Hamiltonian action of $G=(\C^\times)^M$ on $T^*V=T^*\C^N\cong\C^{2N}$. Let the coordinates of $V$ and $V^*$ be $\smash{x_{ij}^{(k)}}$ and $\smash{y_{ij}^{(k)}}$, respectively, for edges $e_{ij}^k\in E$. The action on $T^*V\cong\C^{2N}$ is given by the action on the coordinates
\begin{align*}
(t_1,\dots,t_M)\cdot x_{ij}^{(k)}&=t_i^{-1}t_jx_{ij}^{(k)},&(t_1,\dots,t_M)\cdot y_{ij}^{(k)}&=t_it_j^{-1}y_{ij}^{(k)}\\
\intertext{if $i,j\neq0$ and}
(t_1,\dots,t_M)\cdot x_{0j}^{(k)}&=t_jx_{0j}^{(k)},&\quad (t_1,\dots,t_M)\cdot y_{0j}^{(k)}&=t_j^{-1}y_{0j}^{(k)},\\
(t_1,\dots,t_M)\cdot x_{i0}^{(k)}&=t_i^{-1}x_{i0}^{(k)},&\quad (t_1,\dots,t_M)\cdot y_{i0}^{(k)}&=t_iy_{i0}^{(k)}
\end{align*}
for $i,j\neq0$, for $(t_1,\dots,t_M)\in G=(\C^\times)^M$. Here, $t_i\in G$ corresponds to the basis vector $n_0-n_i\in\Z^M$ for $i=1,\dots,M$.

The corresponding moment map $\mu\colon T^*V\cong\C^{2N}\to\g^*\cong(\C^M)^*\cong\C^M$ is
\begin{align*}
&\mu((\dots,x_{ij}^{(k)},\dots,y_{ij}^{(k)},\dots))=\biggl(\sum_{e_{ij}^k\in E}(\delta_{j\ell}-\delta_{i\ell})x_{ij}^{(k)}y_{ij}^{(k)}\biggr)_{1\leq\ell\leq M}\\
&=\biggl(
\sum_{\substack{e_{ij}^k\in E\\i,j\neq0}}\!\!(\delta_{j\ell}-\delta_{i\ell})x_{ij}^{(k)}y_{ij}^{(k)}
+\!\!\sum_{e_{0j}^k\in E}\!\!\delta_{j\ell}x_{0j}^{(k)}y_{0j}^{(k)}
-\!\!\sum_{e_{i0}^k\in E}\!\!\delta_{i\ell}x_{i0}^{(k)}y_{i0}^{(k)}
\biggr)_{1\leq\ell\leq M}.
\end{align*}
We note here that an edge $e_{ij}^k\in E$ not incident to the node $n_0$ contributes two terms of the form $xy$, one in the component corresponding to the source node, one in the component corresponding to the target node. By contrast, edges incident to $n_0$ only appear once, namely in the component corresponding to the other node incident to that edge (recalling that we did not allow self-nodes).

The comoment map $\mu^*\colon\g\cong\C^M\to\C[T^*V]$ is
\begin{align*}
a_\ell&\mapsto\sum_{e_{ij}^k\in E}(\delta_{j\ell}-\delta_{i\ell})x_{ij}^{(k)}y_{ij}^{(k)}\\
&=\sum_{\substack{e_{i\ell}^k\in E\\i\neq0}}x_{i\ell}^{(k)}y_{i\ell}^{(k)}-\sum_{\substack{e_{\ell j}^k\in E\\j\neq0}}x_{\ell j}^{(k)}y_{\ell j}^{(k)}+\sum_{e_{0\ell}^k\in E}x_{0\ell}^{(k)}y_{0\ell}^{(k)}-\sum_{e_{\ell 0}^k\in E}x_{\ell 0}^{(k)}y_{\ell 0}^{(k)},
\end{align*}
where $\{a_1,\dots,a_M\}$ is the standard basis of $\g\cong\C^M$.

Given an effective stability parameter $\delta\in\Q^M$, we denote by
\begin{equation*}
Y_\delta(Q)\coloneqq Y_\delta(\Delta)
\end{equation*}
the (projective) \emph{hypertoric quiver variety} associated with the quiver $Q$. Recall that $\Delta$ is always unimodular in the quiver case. Hence, if $\delta$ is generic, $Y_\delta(Q)$ is smooth, of dimension $2(N-M)=2(|E|-|V|+1)$ and a symplectic resolution of $Y_0(Q)\coloneqq Y_0(\Delta)$.

Two examples of quiver hypertoric varieties are discussed in \autoref{sec:var1} and \autoref{sec:var2}; see \autoref{fig:ex1} and \autoref{fig:ex2}, respectively.

\medskip

We can bring the above hypertoric quiver varieties into the form of (framed) Nakajima quiver varieties. To this end, we consider the quiver $Q$ and ``expand'' it at $n_0$. Without loss of generality, we assume that the node $n_0$ is a \emph{source}, i.e.\ has only outgoing edges. Then, we remove the node $n_0$ and the nodes $\mathcal{V}=\{n_1,\dots,n_M\}$ become the \emph{gauge nodes} of the new quiver $\mathcal{Q}=(\mathcal{V},\mathcal{E})$. Here, $\mathcal{E}$ are the remaining edges between the gauge nodes, i.e.\ $\mathcal{E}=\{e_{ij}^k\in E\mid i,j\neq0\}\subset E$.

We assign the dimension vector $v=(1,\dots,1)^t\in\Ns^M$ to $\mathcal{Q}$, corresponding to the fact that we are acting with the group $G=\prod_{i=1,\dots,M}\GL(v_i)=\GL(1)^M=(\C^\times)^M$. For each edge in $Q$ from $n_0$ to some node, say $n_i$, we place a \emph{framing node} at the node $n_i$ with that edge, now going from the framing node to $n_i$. Thus, we end up with a framed quiver $\mathcal{Q}$ with framing vector $w=(w_1,\dots,w_M)^t\in\N^M$, where $w_i$ is the number of edges from $n_i$ to $n_0$.

Finally, we let $\mathcal{Q}^\#$ be the doubled quiver of $\mathcal{Q}$. Two examples of this procedure are shown in \autoref{fig:ex1} and \autoref{fig:ex2}. As is usual, we depict the gauge nodes by circles and the framing nodes by squares, with the dimensions $v_i$ and $w_i$, respectively, inscribed.

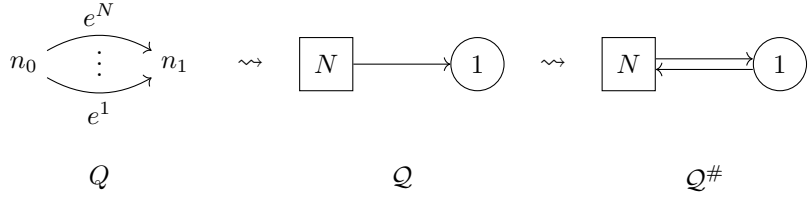
\begin{figure}[ht]
\begin{tikzpicture}
\node (1) at (0,0) {$n_0$};
\node (2) at (2,0) {$n_1$};
\draw[<-] (2) to [bend right] node[above] {$e^N$} (1);
\path (2) -- (1) node [midway] {$\raisebox{+1.5ex}{\vdots}$};
\draw[<-] (2) to [bend left] node[below] {$e^1$} (1);
\node at (3,0) {$\leadsto$};
\node (3) at (4,0) [rectangle,draw,minimum size=20pt,inner sep=0pt] {$N$};
\node (4) at (6,0) [circle,draw,minimum size=20pt,inner sep=0pt] {$1$};
\draw[<-] (4) -- (3);
\node at (7,0) {$\leadsto$};
\node (5) at (8,0) [rectangle,draw,minimum size=20pt,inner sep=0pt] {$N$};
\node (6) at (10,0) [circle,draw,minimum size=20pt,inner sep=0pt] {$1$};
\begin{scope}[transform canvas={yshift=2pt}]
\draw[->] (5) -- (6);
\end{scope}
\begin{scope}[transform canvas={yshift=-2pt}]
\draw[->] (6) -- (5);
\end{scope}
\node at (1,-1.5) {$Q$};
\node at (5,-1.5) {$\mathcal{Q}$};
\node at (9,-1.5) {$\mathcal{Q}^\#$};
\end{tikzpicture}
\caption{Quivers for the minimal nilpotent orbit closure for $\sl_N$.}
\label{fig:ex1}
\end{figure}

\begin{figure}[ht]
\begin{tikzpicture}
\node (1) at (4,1) {$n_0$};
\node (2) at (2,0) {$n_1$};
\node (3) at (4,0) {$n_2$};
\node (4) at (6,0) {$n_M$};
\draw[<-] (2) -- (1);
\draw[<-] (3) -- (2);
\draw[<-] (4) -- (3) node [midway, fill=white] {$\dots$};
\draw[<-] (4) -- (1);
\node at (4,-1) {\rotatebox{270}{$\leadsto$}};
\node (21) at (0,-2) [rectangle,draw,minimum size=20pt,inner sep=0pt] {$1$};
\node (22) at (2,-2) [circle,draw,minimum size=20pt,inner sep=0pt] {$1$};
\node (23) at (4,-2) [circle,draw,minimum size=20pt,inner sep=0pt] {$1$};
\node (24) at (6,-2) [circle,draw,minimum size=20pt,inner sep=0pt] {$1$};
\node (25) at (8,-2) [rectangle,draw,minimum size=20pt,inner sep=0pt] {$1$};
\draw[<-] (22) -- (21);
\draw[<-] (23) -- (22);
\draw[<-] (24) -- (23) node [midway, fill=white] {$\dots$};
\draw[<-] (24) -- (25);
\node at (4,-3) {\rotatebox{270}{$\leadsto$}};
\node (31) at (0,-4) [rectangle,draw,minimum size=20pt,inner sep=0pt] {$1$};
\node (32) at (2,-4) [circle,draw,minimum size=20pt,inner sep=0pt] {$1$};
\node (33) at (4,-4) [circle,draw,minimum size=20pt,inner sep=0pt] {$1$};
\node (34) at (6,-4) [circle,draw,minimum size=20pt,inner sep=0pt] {$1$};
\node (35) at (8,-4) [rectangle,draw,minimum size=20pt,inner sep=0pt] {$1$};
\begin{scope}[transform canvas={yshift=2pt}]
\draw[->] (31) -- (32);
\draw[->] (32) -- (33);
\draw[->] (33) -- (34) node [midway, fill=white] {$\dots$};
\draw[->] (34) -- (35);
\end{scope}
\begin{scope}[transform canvas={yshift=-2pt}]
\draw[->] (32) -- (31);
\draw[->] (33) -- (32);
\draw[->] (34) -- (33) node [midway, fill=white] {$\dots$};
\draw[->] (35) -- (34);
\end{scope}
\node at (-1.5,0) {$Q$};
\node at (-1.5,-2) {$\mathcal{Q}$};
\node at (-1.5,-4) {$\mathcal{Q}^\#$};
\end{tikzpicture}
\caption{Quivers for the Kleinian singularity of type $A_{N-1}$.}
\label{fig:ex2}
\end{figure}
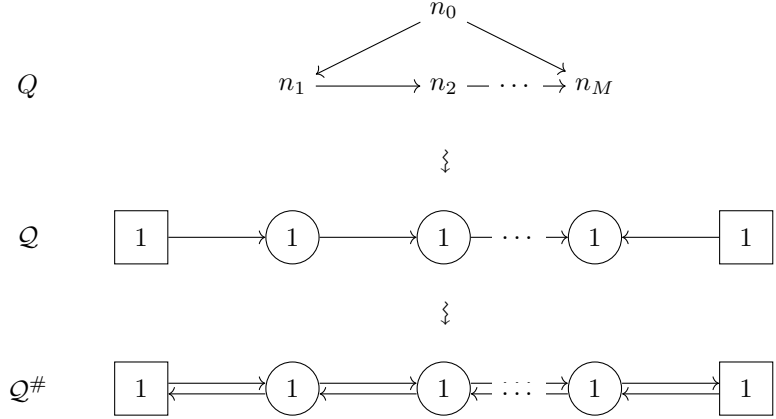

Then the hypertoric quiver variety $Y_\delta(Q)$ is exactly the Nakajima quiver variety
\begin{equation*}
\mathcal{M}_\delta(\mathcal{Q}^\#,v,w)=Y_\delta(Q)
\end{equation*}
with dimension vector $v=(1,\dots,1)^t\in\Ns^M$. Similarly, $Y_0(Q)$ coincides with the affine quiver variety
\begin{equation*}
\mathcal{M}_0(\mathcal{Q}^\#,v,w)=Y_0(Q).
\end{equation*}
The Nakajima quiver variety $\mathcal{M}_\delta(\mathcal{Q}^\#,v,w)$ has dimension
\begin{equation*}
2(N-M)=2(N_\mathrm{g}+N_\mathrm{f}-M)=2(|\mathcal{E}|-\sum_{i=1}^Mw_i-|\mathcal{V}|),
\end{equation*}
where $M=|\mathcal{V}|$ is the number of nodes in $\mathcal{Q}$, $\smash{N_\mathrm{g}\coloneqq|\mathcal{E}|=N-\sum_{i=1}^Mw_i}$ the number of edges in $\mathcal{Q}$ and $\smash{N_\mathrm{f}\coloneqq\sum_{i=1}^Mw_i}$ the number of framing edges. By definition, $N=N_\mathrm{g}+N_\mathrm{f}$.

For convenience, we describe the quiver $\mathcal{Q}=(\mathcal{V},\mathcal{E})$ and the hypertoric quiver variety $\mathcal{M}_\delta(\mathcal{Q}^\#,v,w)$ explicitly. The vector space $V$ is of the form \begin{align*}
R(\mathcal{Q},v,w)&\coloneqq\bigoplus_{e_{ij}^k\in\mathcal{E}}\Hom(\C^{v_i},\C^{v_j})\oplus\bigoplus_{n_i\in\mathcal{V}}\Hom(\C^{w_i},\C^{v_i})\\
&=\bigoplus_{e_{ij}^k\in\mathcal{E}}\Hom(\C,\C)\oplus\bigoplus_{n_i\in\mathcal{V}}\Hom(\C^{w_i},\C)\\
&\cong\bigoplus_{e_{ij}^k\in\mathcal{E}}\C\oplus\bigoplus_{n_i\in\mathcal{V}}\C^{w_i}=\C^{N_\mathrm{g}}\oplus\C^{N_\mathrm{f}}=\C^N=V,
\end{align*}
using that the dimension vector is $v=(1,\dots,1)^t$. This is the usual form of the \emph{representation space} $R(\mathcal{Q},v,w)$ of a framed quiver with dimension vector $v\in\Ns^M$ and framing vector $w\in\N^M$ (see, e.g., \cite{Kir16}).

The coordinates for this representation space $R(\mathcal{Q},v,w)\cong V$ are $\smash{x_{ij}^{(k)}}$ for $e_{ij}^k\in\mathcal{E}$ as before and $\smash{\gamma_j^{(k)}}$ corresponding what was called $\smash{x_{0j}^{(k)}}$ as coordinate for $V$.

The cotangent bundle $T^*R(\mathcal{Q},v,w)$ is canonically isomorphic to the representation space of the doubled quiver $T^*R(\mathcal{Q},v,w)\cong R(\mathcal{Q}^\#,v,w)$, where the definition of the latter implicitly assumes that also the framing edges are doubled. This is the reason why, up to isomorphism of Poisson varieties, the quiver variety only depends on $\mathcal{Q}^\#$ and not on $\mathcal{Q}$.

The additional coordinates for the cotangent space $T^*R(\mathcal{Q},v,w)\cong T^*V$ are $\smash{y_{ij}^{(k)}}$ for $e_{ij}^k\in\mathcal{E}$ as before and $\smash{\beta_j^{(k)}}$ corresponding what was called $\smash{y_{0j}^{(k)}}$ for $V^*\subset T^*V$.

With the notation set up, the action of $G=(\C^\times)^M$ on $T^*R(\mathcal{Q},v,w)$ is given by the action on the coordinates
\begin{align*}
(t_1,\dots,t_M)\cdot x_{ij}^{(k)}&=t_i^{-1}t_jx_{ij}^{(k)},&(t_1,\dots,t_M)\cdot y_{ij}^{(k)}&=t_it_j^{-1}y_{ij}^{(k)}\\
\intertext{corresponding to the edges $e_{ij}^k\in\mathcal{E}$ and}
(t_1,\dots,t_M)\cdot\gamma_j^{(k)}&=t_j\gamma_j^{(k)},&(t_1,\dots,t_M)\cdot\beta_j^{(k)}&=t_j^{-1}\beta_j^{(k)}
\end{align*}
corresponding to the framing, for $(t_1,\dots,t_M)\in G$. The corresponding comoment map $\mu^*\colon\g\cong\C^M\to\C[T^*R(\mathcal{Q},v,w)]$ is given by
\begin{equation*}
a_\ell\mapsto\sum_{e_{i\ell}^k\in\mathcal{E}}x_{i\ell}^{(k)}y_{i\ell}^{(k)}-\sum_{e_{\ell j}^k\in\mathcal{E}}x_{\ell j}^{(k)}y_{\ell j}^{(k)}+\sum_{\ell=1}^M\sum_{k=1}^{w_\ell}\gamma_{\ell}^{(k)}\beta_{\ell}^{(k)},
\end{equation*}
where $\{a_1,\dots,a_M\}$ is the standard basis of $\g\cong\C^M$ (resembling the notation in \cite{AKM23}, and in \cite{CSYZ25} up to a sign).


\bibliographystyle{alpha_noseriescomma}
\bibliography{quellen}{}

\end{document}